\documentclass[12pt]{amsart}
\usepackage{geometry,mathtools}
\usepackage[all]{xy}
\usepackage{graphicx}
\usepackage{amssymb,MnSymbol}
\usepackage{epstopdf}
\usepackage{xr}
\newtheorem{thm}{Theorem}[section]
\newtheorem{thmintro}{Theorem}

\newtheorem{question}{Question}
\newtheorem{cor}[thm]{Corollary}

\newtheorem{lem}[thm]{Lemma}
\newtheorem{prop}[thm]{Proposition}
\newtheorem{fact}[thm]{Fact}
\newtheorem{example}[thm]{Example}
\theoremstyle{definition}
\newtheorem{defn}[thm]{Definition}
\theoremstyle{remark}
\newtheorem{rem}[thm]{Remark}
\numberwithin{figure}{section}%  Fig. 4.1
\numberwithin{table}{section}
\numberwithin{equation}{section}
\newcommand{\To}{\longrightarrow}

\newcommand{\ndh}{\mathcal{N}\mathrm{Diff}^{\mathrm{hol}}}
\newcommand{\ndc}{\mathcal{N}\mathrm{Diff}^{\mathrm{const}}}

\newcommand{\indpg}{\operatorname{ind}_{\mathfrak p}^{\mathfrak g}}
\newcommand{\indbg}{\operatorname{ind}_{\mathfrak b}^{\mathfrak g}}
\newcommand{\indpgprime}{\operatorname{ind}_{\mathfrak p'}^{\mathfrak g'}}
\newcommand{\indpgtau}{\operatorname{ind}_{\mathfrak p^\tau}^{\mathfrak g^\tau}}
\newcommand{\trans}{{}^t\!}
\DeclareMathOperator{\Res}{Res}

\newcommand{\R}{\mathbb R}  %@@
\newcommand{\Q}{\mathbb Q}  %@@
\newcommand{\Z}{\mathbb Z}  %@@
\newcommand{\N}{\mathbb N}  %@@
\newcommand{\C}{\mathbb C}  %@@
\newcommand{\rightsetse}[1]{%
\hidewidth\rotatebox[origin=c]{-45}{$\xrightarrow{\kern2em}$}
     \rlap{\raisebox{1ex}
     {$\kern-.8em\scriptstyle #1$}}\hidewidth}
\newcommand{\rightsetsw}[1]{%
\hidewidth\rotatebox[origin=c]{45}{$\xleftarrow{\kern2em}$}
     \rlap{\raisebox{.1ex}
     {$\kern-.8em\scriptstyle #1$}}\hidewidth}
\newcommand{\leftsetsw}[1]{%
\hidewidth
     \llap{\raisebox{1ex}
     {$\scriptstyle #1$\kern-.8em}}
    \rotatebox[origin=c]{45}{$\xleftarrow{\kern2em}$}\hidewidth}
\newcommand{\rightsetnw}[1]{%
\hidewidth\rotatebox[origin=c]{135}{$\xrightarrow{\kern2em}$}
     \rlap{\raisebox{1ex}
     {$\kern-.8em\scriptstyle #1$}}\hidewidth}
\newcommand{\rightsetd}[1]{%
\hidewidth\rotatebox[origin=c]{-90}{$\xrightarrow{\kern2em}$}
     \rlap{{$\scriptstyle #1$}}\hidewidth}
\usepackage[breaklinks]{hyperref}
\title[Rankin--Cohen operators for symmetric pairs]
{Differential symmetry breaking operators. \\
              II. Rankin--Cohen operators for symmetric pairs}
\author{Toshiyuki Kobayashi, Michael Pevzner}

\date{\today}
\begin{document}

\begin{abstract}
Rankin--Cohen brackets are symmetry breaking operators for the tensor product of two holomorphic discrete series representations of $SL(2,\R)$.  We address a general problem to find explicit formul\ae\, for such intertwining operators
in the setting of multiplicity-free branching laws for reductive symmetric pairs.

For this purpose we use a new method (F-method) developed in \cite{PART1} and based on the \emph{algebraic
Fourier transform for generalized Verma modules}.
 The method characterizes symmetry breaking operators by means of certain systems of partial differential 
equations of second order.

We discover explicit formul\ae\, of new differential symmetry breaking
operators for all the six different complex geometries arising from semisimple symmetric
pairs of split rank one, and reveal an intrinsic reason why the coefficients of orthogonal polynomials appear
in these operators (Rankin--Cohen type)
in the three geometries and why normal derivatives are symmetry breaking operators in the other three cases. Further, we analyze a new phenomenon that the multiplicities in the
branching laws of Verma modules may jump up at singular parameters.

\vskip7pt

\end{abstract}
\maketitle
\tableofcontents
\section{Introduction} 
What kind of differential operators do preserve modularity? 
R. A. Rankin \cite{Ra} and H. Cohen \cite{C75} introduced
 a family of differential operators transforming a given pair of modular
forms into another modular form of a higher weight. Let $f_1$ and $f_2$ be holomorphic modular forms for a given arithmetic
subgroup of $SL(2,\R)$
 of weight $k_1$ and $k_2$, respectively. The bidifferential operators,
referred to as the \emph{Rankin--Cohen brackets} of degree $a$ and defined 
 by
\begin{equation}\label{rcb}
\mathcal{RC}_{k_1,k_2}^{k_3}(f_1,f_2)(z):=\sum_{\ell=0}^a(-1)^{\ell}\left(
                                         \begin{array}{c}
                                           k_1+a-1 \\
                                           \ell \\
                                         \end{array}
                                       \right)\left(
                                                \begin{array}{c}
                                                  k_2+a-1 \\
                                                a-\ell \\
                                                \end{array} \right)
f_1^{(a-\ell)}(z)f_2^{(\ell)}(z),
\end{equation} 
where $f^{(n)}(z)=\frac{d^nf}{dz^n}(z)$,
yield
holomorphic modular forms of weight $k_3=k_1+k_2+2a$
 ($a=0,1,2,\cdots$). (In the usual notation, these operators are written as $RC_{k_1,k_2}^a$.)  

The Rankin--Cohen bidifferential operators have attracted considerable attention in recent years particularly because of their applications to various areas including
\begin{itemize}
\item[-] theory of modular and quasimodular forms (special values of $L$-functions, 
the Ramanujan and Chazy differential equations,
 van der Pol and Niebur equalities) \cite{CL11,MR,Z1},
\item[-] covariant quantization \cite{BTY,CMZ,CM,OS,vDP,Pev,UU96},
\item[-] ring structures on representations spaces \cite{vDP,Z1}.
\end{itemize}

\textbf{Existing methods for the $SL(2,\R)$-case.} A prototype of the Rankin--Cohen brackets was already found by P. Gordan and S. Guldenfinger \cite{Go,Gu}
in the 19th century by using recursion relations for invariant binary forms and the Cayley processes.
For explicit constructions of the equivariant bidifferential operators 
 \eqref{rcb},
several
different methods have been developed:
\begin{itemize}
\item[-] Recurrence relations \cite{C75,ElG,HTan,PevBess, Z1}.
\item[-] Taylor expansions of Jacobi forms \cite{EZ,IKO,Kuz}.
\item[-] Reproducing kernels for Hilbert spaces \cite{PZ, UU96,Zh}.
\item[-] Dual pair correspondence \cite{Ban, Eholzer}.
\end{itemize}
\vskip7pt

In the first part of our work \cite{PART1} we proposed yet another method ({\it{F-method}})
to find differential symmetry breaking operators in a more general setting of
 branching laws for infinite-dimensional
representations, based on the algebraic Fourier transform of generalized Verma modules.
Even in the $SL(2,\R)$-case, the method is original and simple,
 and yields missing
  operators for singular parameters $(k_1, k_2, k_3)$, see Corollary \ref{cor:83} for the complete classification.
 Moreover, the F-method leads us to discover new families of covariant differential operators 
 for six different complex geometries beyond the $SL(2,\R)$ case (see Table \ref{table:intro}).
\vskip7pt

\textbf{Branching laws for symmetric pairs.}
By \emph{branching law} we mean
the decomposition of an irreducible representation $\pi$ of a group $G$ when restricted to a given subgroup $G'$. 
An important and fruitful source of examples
 is provided by pairs of groups $(G,G')$ such that $G'$ is the fixed point group of an involutive automorphism $\tau$ of $G$, called \emph{symmetric pairs}.

The decomposition of the tensor product of two representations is a special case of branching laws with respect to symmetric pairs $(G,G')$.  
Indeed,
 if $G=G_1\times G_1$ and $\tau$ is an involutive 
automorphism
 of $G$ 
given by $\tau(x,y)=(y,x)$, then $G'\simeq G_1$ and 
the restriction of the outer tensor product $\pi_1\boxtimes\pi_2$ to the subgroup $G'$
is nothing but the tensor product $\pi_1\otimes\pi_2$ of two representations $\pi_1$ and $\pi_2$ of $G_1$. 
The Littlewood--Richardson rule 
 for finite-dimensional representations is
 another classical example of branching laws with respect to the symmetric pair $(GL(p+q,\C), GL(p,\C)\times GL(q,\C))$. Our approach relies on recent progress in the theory of branching laws of infinite-dimensional representations for symmetric pairs even beyond completely reducible cases (see Section \ref{sec:8} for instance).
 
\vskip 7pt

\noindent
\textbf{Rankin--Cohen operators as intertwining operators.}
From the view point of representation theory
the Rankin--Cohen operators are intertwiners in the branching law for the tensor product 
of two holomorphic discrete series representations $\pi_{k_1}$ and $ \pi_{k_2}$ of $SL(2,\R)$. More
precisely, the discrete series representation $\pi_{k_1+k_2+2a}$ ($a\in\N$) occurs
in the following branching law \cite{Mol,Repka}:
\begin{equation}\label{eqn:absbra}
\pi_{k_1}\otimes \pi_{k_2}
\simeq
{\sum_{a\in\N}}^{\oplus}\pi_{k_1+k_2+2a},
\end{equation}
and the operator  \eqref{rcb} gives an explicit intertwining operator from $\pi_{k_1}\otimes \pi_{k_2}$ to the irreducible summand
$\pi_{k_1+k_2+2a}$.

 In our work \cite{PART1} we developed a new method to find explicit intertwining operators for
 irreducible components of branching laws in a broader setting of symmetric pairs.
 Such operators are unique up to scalars if the representation
 $\pi$ is a highest weight module of scalar type (or equivalently $\pi$ is realized in the space of holomorphic sections
 of a homogeneous holomorphic line bundle over a bounded symmetric  domain) and 
 $(G,G')$ is any symmetric pair, by the multiplicity-free theorems
(\cite{K08,K12}). 

The subject of this paper is to study concrete examples where the F-method turns out to be surprisingly
efficient. 
\vskip 7pt

Let $\mathcal V_X\to X$ be a homogeneous vector bundle
 of  a Lie group $G$
 and $\mathcal W_Y\to Y$ a  homogeneous vector bundle of  $G'$. 
Then we have a natural representation $\pi$ of $G$ on the space  $\Gamma(X,\mathcal V_X)$
of sections on $X$, and similarly that of $G'$ on $\Gamma (Y,\mathcal W_Y)$.
Assume $G'$ is a subgroup of $G$.
We address the following question:

\begin{question}\label{question:1}
 Find explicit $G'$-intertwining operators from $\Gamma(X,\mathcal V_X)$ to $\Gamma(Y,\mathcal W_Y)$.
\end{question}
To illustrate the nature of such operators we also refer to them as \emph{continuous symmetry breaking operators}. They are said to be  \emph{differential symmetry breaking operators} if the operators are differential operators.

The F-method proposed in \cite{PART1} provides necessary tools to give an answer to Question \ref{question:1}
for all symmetric pairs $(G,G')$ of split rank one inducing a holomorphic embedding $Y\hookrightarrow X$ (see Table \ref{table:1}).
We remark that the split rank one condition does not force the rank of $G/G'$
to be equal to one (see Table \ref{table:intro}
 (1), (5) below).
 \vskip7pt

\noindent\textbf{Normal derivatives and Jacobi--type differential operators.}
In representation theory,
  taking normal derivatives with respect to 
an equivariant embedding $Y\hookrightarrow X$
 is a standard tool to find abstract branching laws for representations that are realized on $X$ (see \cite{JV} for instance). 

However, we should like to emphasize that the common belief  ``normal derivatives with respect to $Y\hookrightarrow X$ are  intertwining operators in the branching laws" is not true. Actually, it already fails for the tensor product of two holomorphic 
discrete
 series of $SL(2,\R)$ where the Rankin--Cohen brackets are not normal derivatives
 with respect to the diagonal embedding $Y\hookrightarrow Y\times Y$ with $Y$ being the Poincar\'e upper half plane.

We discuss when normal derivatives
 become intertwiners 
 in the following six complex geometries
 arising from real
 symmetric pairs of split rank one:
\begin{table}[!htdp]
\begin{center}
\begin{tabular}{cr@{\ }c@{\ }lccr@{\ }c@{\ }l}
%\hline
&&\\
(1)
&$\mathbb P^n\C$&$\hookrightarrow$
&$  \mathbb P^n\C\times \mathbb P^n\C$
&&
(4)
&$\mathrm {Gr}_{p-1}(\C^{p+q})$&$\hookrightarrow$&$ \mathrm {Gr}_{p}(\C^{p+q})$
\\
(2)
&$\mathrm{LGr}(\C^{2n-2})\times \mathrm{LGr}(\C^{2})$&$\hookrightarrow$&$\mathrm{LGr}(\C^{2n})$
&&
(5)
&$\mathbb P^n\C$&$\hookrightarrow$&$\mathrm{Q}^{2n}\C$
\\
(3)
&
$\mathrm Q^{n}\C$&$\hookrightarrow$&$  \mathrm Q^{n+1}\C$
&&
(6)
&$\mathrm{IGr}_{n-1}(\C^{2n-2})$&$\hookrightarrow$&$  \mathrm{IGr}_{n}(\C^{2n})$
 \\
 &&\\
 %\hline
\end{tabular}
\end{center}
\caption{Equivariant embeddings of flag varieties}
\label{table:intro}
\end{table}
 
 Here 
$\mathrm {Gr}_{p}(\C^{n})$ is the Grassmanian of $p$-planes in $\C^n$,
 $\mathrm Q^{m}\C:=\{z\in\mathbb P^{m+1}\C\,:\, z_0^2+\cdots+z_{m+1}^2=0\}$ is the complex quadric,
 and $\mathrm{IGr}_{n}(\C^{2n}):=\{V\subset\C^{2n}\,:\,\mathrm{dim}\,V=n,
 \,Q|_V\equiv0\}$ is the Grassmanian of isotropic subspaces of $\C^{2n}$ equipped with a non-degenerate quadratic form $Q$, and
 $\mathrm{LGr}_{n}(\C^{2n}):=\{V\subset\C^{2n}\,:\,\mathrm{dim}\,V=n,
 \,\omega|_{V\times V}\equiv0\}$ is the Grassmanian of Lagrangian subspaces of $\C^{2n}$ equipped with a symplectic form $\omega$.

For  $Y\hookrightarrow X$ as in Table \ref{table:intro} and any equivariant line bundle $\mathcal L_\lambda\to X$ with sufficiently positive $\lambda$ we 
give a necessary and sufficient condition for normal derivatives to become intertwiners: 
\begin{thmintro}\label{thm:introA}
${}$
\textup{(1)} Any continuous $G'$-homomorphism from
 $\mathcal O(X,\mathcal L_\lambda)$
to $\mathcal O(Y,\mathcal W)$
 is given by normal derivatives with respect to the equivariant embedding $Y\hookrightarrow X$ if the embedding $Y\hookrightarrow X$ is of type (4), (5) or (6) in Table \ref{table:intro}.
 
\textup{(2)} None of normal derivatives of positive order is
a $G'$-homomorphism if the embedding $Y\hookrightarrow X$ is of type (1), (2) and (3)
in Table \ref{table:intro}.
\end{thmintro}
See Theorem \ref{thm:nor_der} for the precise formulation of the first statement.
For the three geometries (1), (2), and (3) in Table \ref{table:intro}, we construct explicitly all the
continuous $G'$-homomorphisms which are actually holomorphic differential operators (differential symmetry breaking operators).
For this, let $P^{\alpha,\beta}_\ell(x)$ be the Jacobi polynomial, 
  and $\widetilde C^\alpha_\ell(x)$ the normalized Gegenbauer polynomial (see Appendix \ref{sec:A2}).
We inflate them into polynomials of two variables by
$$
P_\ell^{\alpha,\beta}(x,y):=y^\ell P_\ell^{\alpha,\beta}\left(2\frac xy+1\right)\quad\mathrm{and}\,\quad
\widetilde C_\ell^\alpha(x,y):=x^{\frac\ell2} \widetilde C_\ell^\alpha\left(\frac y{\sqrt x}\right).
$$

In what follows,
 ${\mathcal{L}}_{\lambda}$ stands 
 for a homogeneous holomorphic line bundle,
 and ${\mathcal{W}}_{\lambda}^a$
 a homogeneous vector bundle
 with typical fiber $S^a({\mathbb{C}}^m)$
 ($m=n$ in (1);
$=n-1$ in (2); m=1 in (3))
 with parameter $\lambda$
 (see Lemma \ref{lem:64}
 for details).  
Then we prove:

\begin{thmintro}\label{thm:introB}
\textup{(1)} For the symmetric pair $(U(n,1)\times U(n,1), \, U(n,1))$  the differential operator 
$$
  P_a^{\lambda'-1,\, -\lambda'-\lambda''-2a+1}\left(
\sum_{i=1}^{n}v_i \frac{\partial}{\partial z_i},
\sum_{j=1}^{n}v_j \frac{\partial}{\partial z_j}
\right)
$$
is an intertwining operator from $\mathcal O(Y,\mathcal L_{(\lambda_1',\lambda_2')})\widehat\otimes \mathcal O(Y,\mathcal L_{(\lambda_1'',\lambda_2'')})$ to $\mathcal O(Y,\mathcal W_{(\lambda_1'+\lambda_1'',\lambda_2'+\lambda_2'')}^a)$, for any
 $\lambda_1',
 \lambda_1'',\lambda_2',\lambda_2''\in\Z$,  and $a\in\N$. Here  we set $\lambda'=\lambda_1'-\lambda_2'$ and
$\lambda''=\lambda_1''-\lambda_2''$.

\textup{(2)} For the symmetric pair $(Sp(n,\R), Sp(n-1,\R)\times Sp(1,\R))$
 the differential operator 
 $$
C_a^{\lambda-1}\left(
\sum_{1\leq i, j\leq n-1}2 v_iv_j
\frac{\partial^2}{\partial z_{ij}\partial z_{nn}},
\sum_{1\leq j\leq n-1}v_j\frac{\partial}{\partial z_{jn}}\right)
$$
is an intertwining operator from $\mathcal O(X,\mathcal L_\lambda)$ to $\mathcal O(Y,\mathcal W_\lambda^a)$, for any
 $\lambda\in\Z$, and $a\in\N$.\vskip7pt

\textup{(3)}  For the symmetric pair $(SO(n,2), SO(n-1,2))$ the
differential operator
 \begin{eqnarray*}
\widetilde C_a^{\lambda-\frac{n-1}2}\left(-\Delta^z_{\C^{n-1}},\frac{\partial}{\partial z_n}\right)
\end{eqnarray*}
is an intertwining operator from $\mathcal O(X,\mathcal L_\lambda)$ to $\mathcal O(Y,\mathcal L_{\lambda+a})$, for any $\lambda\in\Z$ and $a\in\N$.
\vskip7pt

  \end{thmintro}
See Theorems \ref{thm:U(n,1)}, \ref{thm:Sp},  and \ref{thm:63} 
for the precise statements, 
 respectively.  
By the localness theorem \cite[Theorem \ref{thm:C}]{PART1}, any continuous $G'$-homomorphisms are
differential operators. Then we prove that the
 above operators exhaust
 all continuous symmetry breaking operators in (2) and (3), and
 for generic parameter  $(\lambda', \lambda'')$ in (1), see \eqref{eqn:Udim} for the exact condition on the parameter. 
The first statement of Theorem \ref{thm:introB}
 corresponds to the decomposition
 of the tensor product, 
 and gives rise to the classical Rankin--Cohen brackets
 in the case where $n=1$.
 An analogous formula for Theorem \ref{thm:introB} (3) was recently found in a 
 completely different way 
by A. Juhl \cite{Juhl} 
in the setting of conformally equivariant differential operators with respect to
the embedding of Riemannian manifolds
$S^{n-1}\hookrightarrow S^{n}$. 

The proof of Theorem \ref{thm:introB} is built on the F-method, which establishes in the present
setting a bijection between the space
$$
\operatorname{Hom}_{G'}(\mathcal O(X,\mathcal L_\lambda), \mathcal O(Y,\mathcal W_\lambda^a))
$$
of symmetry breaking operators and the space of polynomial solutions to a certain ordinary differential equation, namely
\begin{eqnarray*}
&&\mathrm{Sol}_{\mathrm{Jacobi}}(\lambda'-1,-\lambda'-\lambda''-2a+1,a)\cap\mathrm{Pol}_a[s]\\
&&\mathrm{Sol}_{\mathrm{Gegen}}(\lambda-1,a)\cap\mathrm{Pol}_a[s]_{\mathrm{even}}\\
&&\mathrm{Sol}_{\mathrm{Gegen}}(\lambda-\frac{n-1}2,a)\cap\mathrm{Pol}_a[s]_{\mathrm{even}},
\end{eqnarray*}
for the geometries (1), (2), and (3) in Table \ref{table:intro}, respectively.
Here $\mathrm{Sol}_{\mathrm{Jacobi}}(\alpha,\beta,\ell)\cap\mathrm{Pol}_a[s]$ and $\mathrm{Sol}_{\mathrm{Gegen}}\cap\mathrm{Pol}_a[s]$ denote the space of polynomial solutions of degree at most $a$ to the Jacobi differential equation \eqref{eqn:JacobiDE}
and to the Gegenbauer differential equation \eqref{eqn:ODEC}, respectively. (The subscript ``even" stands for a parity condition \eqref{eqn:gs}.) \vskip 10pt

Surprisingly, the dimension of the space of symmetry breaking operators for the tensor product case (1) jumps up at some singular parameters. We illustrate this phenomenon by the
the following result in the $\mathfrak{sl}_2$-case:

\begin{thmintro}[Theorem \ref{thm:Hlambdas}]\label{thm:introC}
The following three conditions on the parameters \\ $(\lambda',\lambda'',\lambda''')\in\Z^3$ are equivalent:
\begin{enumerate}
\item[(i)] $\dim_\C\mathrm{Hom}_{SL(2,\R)}(\mathcal O(\mathcal L_{\lambda'})\widehat\otimes
\mathcal O(\mathcal L_{\lambda''}),\mathcal O(\mathcal L_{\lambda'''}))=2.$
\item[(ii)] $\dim_\C\mathrm{Hom}_{\mathfrak g}(\indbg(-\lambda'''),\indbg(-\lambda')\otimes
\indbg(-\lambda''))=2,$ where
$\indbg(-\lambda)$ is the Verma module $U(\mathfrak g)\otimes_{U(\mathfrak b)}\C_{-\lambda}$ of $\mathfrak g=\mathfrak{sl}(2,\C)$.
\item[(iii)] $\lambda',\lambda''\leq 0, 2\leq\lambda''', \lambda'+\lambda''\equiv\lambda'''$ mod $2$, $-(\lambda'+\lambda'')\geq\lambda'''-2\geq\vert\lambda'-\lambda''\vert$.
\end{enumerate}
\end{thmintro}

We also prove that the analytic continuations of the Rankin--Cohen bidifferential operators $\mathcal{RC}_{\lambda',\lambda''}^{\lambda'''}$ vanish exactly at these singular parameters $(\lambda',\lambda'',\lambda''')$ in this case. Moreover,  we construct explicitly \emph{three} symmetry breaking operators in this case, and prove that any two of the three are linearly independent. Furthermore we show that each of these three 
symmetry breaking operators factors into two natural intertwining operators as follows:
$$
\xymatrix{
& &\mathcal O(\mathcal L_{2-\lambda'})\widehat\otimes
\mathcal O(\mathcal L_{\lambda''})\ar[drr]^{\mathcal{RC}_{2-\lambda',\lambda''}^{\lambda'''}}& &\\
\mathcal O(\mathcal L_{\lambda'})\widehat\otimes
\mathcal O(\mathcal L_{\lambda''})\ar[urr]^{\left(\frac{\partial}{\partial z_1}\right)^{1-\lambda'}\otimes\,\mathrm{id}}
\ar[rr]^{\mathrm{id}\otimes\left(\frac{\partial}{\partial z_2}\right)^{1-\lambda''}}
\ar[drr]_{\mathcal{RC}_{\lambda',\lambda''}^{2-\lambda'''}} 
& & \mathcal O(\mathcal L_{\lambda'})\widehat\otimes
\mathcal O(\mathcal L_{2-\lambda''})\ar[rr]^{\mathcal{RC}_{\lambda',2-\lambda''}^{\lambda'''}} & &
\mathcal O(\mathcal L_{\lambda'''}),\\
& &\mathcal O(\mathcal L_{2-\lambda'''})\ar[urr]_{\left(\frac{d}{ dz}\right)^{\lambda'''-1}}& &
}
$$
 whereas 
the linear relation among the three is explicitly given by using Kummer's connection formula
for Gauss hypergeometric functions via the F-method.

In Section \ref{sec:5} we briefly discuss some new applications of the explicit
formul\ae\, of differential symmetry breaking operators. Namely,  we describe
an explicit construction of the discrete spectrum of complementary series
representations of $O(n+1,1)$ when restricted to $O(n,1)$ by means of
the differential operator given in Theorem \ref{thm:introB} (3).\vskip10pt

In Appendix (Section \ref{sec:10}) we collect some results on classical ordinary differential equations with focus
on singular parameters for which there exist two linearly independent polynomial solutions which
correspond, via the F-method, to the failure of multiplicity-one results in the branching laws.\vskip10pt

The authors are grateful to the referee
for his/her enlightening remarks and for suggesting to divide the 
original manuscript into two parts
 and to write more detailed proofs and explanations for the second part
 for those who are interested in analysis and also in geometric problems.
 Special thanks are also due to Dr. T. Kubo who read very carefully the 
revised manuscript and made constructive suggestions on its readability. 
\vskip10pt

Notation: $\N=\{0,1,2,\cdots\}$, $\N_+=\{1,2,\cdots\}$.
%----------------
\section{Geometric setting: Hermitian symmetric spaces}\label{sec:4}

In this section we describe the geometric setting in which Question 1
will be answered. 

\subsection{Complex submanifolds in Hermitian symmetric spaces}\label{sec:subHerm}

Let $G$ be a connected real reductive Lie group, $\theta$ a Cartan involution, and $G/K$ the associated Riemannian symmetric space.
We write ${\mathfrak {c}}({\mathfrak {k}})$
for the center
 of the complexified Lie algebra
 $\mathfrak k:=\operatorname{Lie}(K)\otimes_{\mathbb{R}} {\mathbb{C}}
 \equiv\mathfrak k(\R)\otimes_\R\C$.  
We suppose that $G/K$ is a Hermitian symmetric space. This means that
  there exists a characteristic element $H_o \in 
{\mathfrak {c}}({\mathfrak {k}})$
 such that we have an eigenspace decomposition
\[
{\mathfrak {g}}={\mathfrak {k}}+{\mathfrak {n}}_+
 +{\mathfrak {n}}_-
\]
 of ${\operatorname{ad}} (H_o)$ 
 with eigenvalues
 $0$, $1$, and $-1$, 
respectively.  
We note that ${\mathfrak {c}}({\mathfrak {k}})$ is one-dimensional if $G$ is simple. 

 Let $G_\C$ be a complex reductive Lie group with Lie algebra $\mathfrak g$, and $P_\C$ the maximal parabolic subgroup having Lie algebra
 ${\mathfrak {p}}:= {\mathfrak {k}}+{\mathfrak {n}}_+$
 with abelian nilradical ${\mathfrak {n}}_+$. The complex structure
 of the homogeneous space $G/K$ is induced from
the Borel embedding
$$
G/K\subset G_\C/K_\C\exp\mathfrak n_+ = G_{\mathbb{C}}/P_{\mathbb{C}}.
$$

Let $G'$ be a $\theta$-stable, connected reductive subgroup of $G$.
We set
 $K':=K\cap G'$ and assume 
 \begin{equation}\label{eqn:ckk}
 H_o\in\mathfrak k'.
 \end{equation}

Then the homogeneous space $G'/K'$ carries a $G'$-invariant complex structure
 such that the embedding
$G'/K' \hookrightarrow G/K$ 
 is holomorphic by the following diagram:
 \begin{equation}\label{eqn:twoflag}
  \begin{matrix}
 Y=& G'/K' &\hookrightarrow &G/K=X\\
&\llap{${}_  {\mathrm{open}}$}\,\bigcap\ & &\bigcap\,\rlap{${}_{\mathrm{open}}$}\\
  &G_\C'/P_\C'&\hookrightarrow&G_\C/P_\C,
  \end{matrix}
\end{equation}
where $G'_\C$ and $P'_\C=K'_\C\exp\mathfrak n_+'$ are
 the connected complex subgroups of $G_\C$ with Lie
algebras $\mathfrak g':=\mathrm{Lie}(G')\otimes_\R\C$ and $\mathfrak p':=\mathfrak k'+\mathfrak n'_+\equiv(\mathfrak k\cap\mathfrak g')+
(\mathfrak n_+\cap\mathfrak g')$, respectively.
\vskip 5pt

Given  
 a finite-dimensional  representation of $K$ on a complex vector space
 $V$, we extend it to a holomorphic representation of $P_\C$ 
 by letting the unipotent subgroup $\exp(\mathfrak n_+)$ act trivially,
 and form a holomorphic
 vector bundle $\mathcal V_{G_\C/P_\C}=G_\C\times_{P_\C}V$ over $G_\C/P_\C$.
 The restriction to the open set $G/K$ defines a $G$-equivariant holomorphic vector bundle $\mathcal V:=G\times_K V$. We then have
 a natural representation of $G$ on the vector space $\mathcal O(G/K,\mathcal V)$ of global holomorphic sections. 
  
 Likewise, given a finite-dimensional  representation $W$ of $K'$, we form
 the $G'$-equivariant holomorphic vector bundle $\mathcal W=G'\times_{K'}W$ and consider
 the representation of $G'$ on $\mathcal O(G'/K',\mathcal W)$.

Let $V^\vee$ and $W^\vee$ be the contragredient representations of $V$ and $W$, respectively, and we define
$\mathfrak g$- and $\mathfrak g'$-modules (\emph{generalized Verma modules}) by
\begin{eqnarray*}
\indpg(V^\vee)&:=&
U(\mathfrak g)\otimes_{U(\mathfrak p)}V^\vee,\\
\indpgprime(W^\vee)&:=&
U(\mathfrak g')\otimes_{U(\mathfrak p')}W^\vee,
\end{eqnarray*}
where $U(\mathfrak g)$ and $U(\mathfrak g')$ denote the universal enveloping algebras
of the Lie algebras $\mathfrak g$ and $\mathfrak g'$, respectively.
We endow the spaces $\mathcal O(G/K,\mathcal V)$ and $\mathcal O(G'/K',\mathcal W)$
with
 the Fr\'echet topology of uniform convergence on compact sets, and denote by $\operatorname{Hom}_{G'}(\,\cdot\,,\,\cdot\,)$ the space of continuous
symmetry breaking operators (\emph{i.e.} continuous $G'$-homomorphisms), and by $\operatorname{Diff}_{G'_\C}^{\mathrm{hol}}(\mathcal V_{G_\C/P_\C},
\mathcal W_{G_\C'/P_\C'})$ the space of $G'_\C$-equivariant holomorphic differential operators with respect to the
holomorphic map $G_\C'/P_\C'\hookrightarrow G_\C/P_\C$ (see \cite[Definition \ref{def:21}]{PART1} for the definition of differential
operators between vector bundles with different base spaces). Then 
the localness theorem \cite[Theorem \ref{thm:C}]{PART1} and the duality theorem (\emph{op. cit.}, Theorem \ref{prop:2.10}) assert:
\begin{thm}\label{thm:locGK} We have the following natural isomorphisms:
\begin{eqnarray*}
\operatorname{Hom}_{G'}(\mathcal O(G/K,\mathcal V), \mathcal O(G'/K',\mathcal W))&\simeq&
\operatorname{Diff}^{\mathrm{hol}}_{G'_\C}(\mathcal V_{G_\C/P_\C},\mathcal W_{G_\C'/P_\C'})\\
&\simeq& \operatorname{Hom}_{\mathfrak g'}(\indpgprime(W^\vee),\indpg(V^\vee)).
\end{eqnarray*}
\end{thm}

  \subsection{Semisimple symmetric pairs of holomorphic type and split rank}
  Let $\tau$ be 
an involutive automorphism of a semisimple Lie group $G$. 
 Without loss of generality we may
and do assume that
 $\tau$ commutes with the Cartan involution $\theta$ of $G$.
We define a $\theta$-stable subgroup by
$$
G^\tau:=\{g\in G:\tau g=g\}.
$$
Then the homogeneous space $G/G^\tau$ carries
a $G$-invariant pseudo-Riemannian structure $g$ induced from the Killing form of $\mathfrak g(\R)=\mathrm{Lie}(G)$,
and becomes an affine symmetric space with respect to the Levi-Civita connection.
We use the same letters $\tau$ and $\theta$
 to denote the differentials and also their complex linear extensions.
 We set $\mathfrak g(\R)^\tau:=\{Y\in\mathfrak g(\R):\tau Y=Y\}$, the Lie algebra of $G^\tau$. The pair ($\mathfrak g(\R),\mathfrak g(\R)^\tau)$ is said to be a semisimple symmetric pair. It is \emph{irreducible} if $\mathfrak g(\R)$ is simple or is a direct sum of two copies of a simple Lie algebra $\mathfrak g'(\R)$ with $\mathfrak g(\R)^\tau\simeq
 \mathfrak g'(\R)$. Then any semisimple symmetric pair is isomorphic to a direct sum of
 irreducible ones.

  \begin{defn}
Geometrically, the \emph{split rank} of the semisimple symmetric space $G/G^\tau$ is
the dimension of a maximal flat, totally geodesic submanifold $B$  in $G/G^\tau$ such that
the restriction of $g$ to $B$ is positive definite.  Algebraically, it is
 the dimension
of a maximal abelian subspace of $\mathfrak g(\R)^{-\tau,-\theta}:=\{Y\in\mathfrak g(\R)\,:
\tau Y=\theta Y=-Y\}$. The dimension is independent of the choice of the data, and the geometric and algebraic definitions coincide.
We denote it by $\mathrm{rank}_\R\, G/G^\tau$.
  \end{defn}

The automorphism $\tau\theta$ is also an involution because $\tau\theta=\theta\tau$. Since 
$$
\mathfrak g(\R)^{\tau\theta,-\theta}:=\{Y\in\mathfrak g(\R)\,:
\tau\theta Y=Y, \theta Y=-Y\}
$$
coincides with $\mathfrak g(\R)^{-\tau,-\theta}$, we have $\mathrm{rank}_\R G/G^\tau=\mathrm{rank}_\R G^{\tau\theta}$, the split rank of the reductive Lie group $G^{\tau\theta}$.

Suppose now that $G/K$ is a Hermitian symmetric space with a characteristic element $H_o$ as in
 Section \ref{sec:subHerm}.
\begin{defn}\label{def:holo}
An irreducible symmetric pair $(\mathfrak g(\R), \mathfrak g(\R)^\tau)$ (or $(G,G^\tau)$) is said to be of \emph{holomorphic type} (with respect to the complex
structure on $G/K$ defined by the characteristic element $H_o$) if $\tau(H_o)=H_o$, namely
$H_o\in\mathfrak k^\tau$.\end{defn}

If $(G,G^\tau)$ is of holomorphic type, then
$G^\tau/K^\tau$ carries a $G^\tau$-invariant complex structure such that
 the embedding $G^\tau/K^\tau\hookrightarrow G/K$ is holomorphic.

Among irreducible symmetric pairs  $(\mathfrak g(\R), \mathfrak g(\R)^\tau)$ of holomorphic type
Table \ref{table:1} gives the infinitesimal classification of those of split rank one.
\begin{table}[!Htdp]
\begin{center}
\begin{tabular}{|c|c|c|c|}
\hline
&$\mathfrak g(\R)$&$\mathfrak g(\R)^\tau$&$\mathfrak g(\R)^{\tau\theta}$\\
\hline
1&$\mathfrak{su}(n,1)\oplus\mathfrak{su}(n,1)$&$\mathfrak{su}(n,1)$&$\mathfrak{su}(n,1)$\\
\hline
2&$\mathfrak{sp}(n+1,\R)$&$\mathfrak{sp}(n,\R)\oplus\mathfrak{sp}(1,\R)$&$\mathfrak{u}(1,n)$\\
\hline
3&$\mathfrak{so}(n,2)$&$\mathfrak{so}(n-1,2)$&$\mathfrak{so}(n-1)\oplus \mathfrak{so}(1,2)$\\
\hline
4&$\mathfrak{su}(p,q)$&$\mathfrak s(\mathfrak{u}(1)\oplus\mathfrak{u}(p-1,q))$&$\mathfrak s(\mathfrak{u}(1,q)\oplus
\mathfrak u(p-1))$\\
\hline
5&$\mathfrak{so}(2,2n)$&$\mathfrak{u}(1,n)$&$\mathfrak{u}(1,n)$\\
\hline
6&$\mathfrak{so}^*(2n)$&$\mathfrak{so}(2)\oplus\mathfrak{so}^*(2n-2)$&$\mathfrak{u}
(1,n-1)$\\
\hline
\end{tabular}
\end{center}
\vskip7pt
\caption{Split rank one irreducible symmetric pairs of holomorphic type}
\label{table:1}
\end{table}

The pairs of flag varieties (see \eqref{eqn:twoflag}) associated with the six pairs $(G,G^\tau)$ in Table \ref{table:1} correspond
to the six complex parabolic geometries given in Table \ref{table:intro}.

%-------------------------------------------

\section{F-method in holomorphic setting}
\label{section:6}

In this section we reformulate the recipe
 of the F-method  (\cite[Section \ref{sec:Fmethod}]{PART1}) in the 
holomorphic setting,
that is, in the setting of Section \ref{sec:subHerm} where $G'/K'$ is a complex submanifold of the Hermitian symmetric space
$G/K$.

\subsection{F-method for Hermitian symmetric spaces}\label{sec:241}

The {algebraic Fourier transform} on a vector space $E$  is an isomorphism of the Weyl algebras
of holomorphic differential operators with polynomial coefficients
 on a complex vector spaces $E$ and its dual space $E^\vee$:
$$
\mathcal D(E)\to\mathcal D(E^\vee), \qquad T\mapsto \widehat T
$$
induced by
\begin{equation}\label{eqn:algfour}
\widehat{\frac\partial{\partial z_j}}:=-\zeta_j,\quad
\widehat z_j:=\frac\partial{\partial\zeta_j},\quad 1\leq j\leq n=\dim E,
\end{equation}
where $(z_1,\ldots,z_n)$ are coordinates on $E$ and $(\zeta_1,\ldots,\zeta_n)$ are the dual coordinates on $E^\vee$.

Let $G_{\C}$ be a connected complex reductive Lie group with Lie algebra $\mathfrak g$ and
$P_{\C} = K_{\C} N_{+,\C}$ be a parabolic subgroup with Lie algebra $\mathfrak p=\mathfrak k+\mathfrak n_+$.
Let $\lambda$ be a holomorphic representation of $K_{\C}$ on $V$. We extend it to $P_\C$ by letting $N_{+,\C}=\exp(\mathfrak n_+)$
act trivially, and form a $G_\C$-equivariant holomorphic vector bundle $\mathcal V=G_\C\times_{P_\C}V$ over $G_\C/P_\C$. Let
$\C_{2\rho}$ be the holomorphic character defined by $p\mapsto\det(\mathrm{Ad}(p):\mathfrak p\to\mathfrak p)$, and
define a twist of the contragredient representation $(\lambda^\vee, V^\vee)$ of $P_\C$ by $\lambda^*:=\lambda^\vee\otimes\C_{2\rho}$. We set $\mathcal V^*\equiv\mathcal V^\vee_{2\rho}:=G_\C\times_{P_\C}(V^\vee\otimes \C_{2\rho})$, which is isomorphic to the tensor bundle of the dual bundle $\mathcal V^\vee$ and the canonical line bundle of $G_\C/P_\C$. We shall
apply the algebraic Fourier transform to the infinitesimal representation $d\pi_{\lambda^*}$ of $\mathfrak g$ on $\mathcal O(G_\C/P_\C,\mathcal V^*)$ as follows.

We recall that the Gelfand--Naimark decomposition 
$\mathfrak g=\mathfrak n_-+\mathfrak k+\mathfrak n_+$
induces a diffeomorphism
$$
\mathfrak n_-\times K_\C\times \mathfrak n_+\to G_\C,\quad
(X,\ell, Y)\mapsto (\exp X)\ell(\exp Y),
$$
into an open dense subset, denoted by $G^{\mathrm{reg}}_\C$, of $G_\C$. Let
$
p_\pm: G^{\mathrm{reg}}_\C\to\mathfrak n_\pm, p_o:G^{\mathrm{reg}}_\C\to K_\C,
$ be the projections characterized by the identity
$$
\exp(p_-(g))p_o(g)\exp(p_+(g))=g.
$$
Furthermore, we introduce the following maps:
\begin{eqnarray}\label{eqn:alpha2}
\alpha:\mathfrak g\times \mathfrak n_-\to\mathfrak k,&\quad&
 (Y,Z)\mapsto \left.\frac {d}{dt}\right\vert_{t=0}p_o\left(e^{tY}e^Z\right),
 \\
\beta:\mathfrak g\times \mathfrak n_-\to\mathfrak n_-,&\quad&
 (Y,Z)\mapsto \left.\frac {d}{dt}\right\vert_{t=0}p_-\left(e^{tY}e^Z\right).\label{eqn:beta2}
\end{eqnarray}
For $F\in\mathcal O(\mathfrak n_-,V^\vee)\simeq\mathcal O(\mathfrak n_-)\otimes V^\vee$, we set
$f:G_\C^{\mathrm{reg}}\To V^\vee$ by $f(\exp Z p)=\lambda^*(p)^{-1}F(Z)$ for $Z\in\mathfrak n_-$ and $p\in P_\C$.
Then the infinitesimal action of $\mathfrak g$ on $\mathcal O(\mathfrak{n}_-,V^\vee)$  is given by
\begin{eqnarray} \label{eqn:35}
\left(d\pi_{\lambda^*}(Y)F\right)(Z)
&=&\left.\frac {d}{dt}\right\vert_{t=0}f(e^{-tY}e^Z)\nonumber\\
&=&\lambda^*(\alpha(Y,Z))F(Z)-(\beta(Y,\cdot\,)F)(Z)
\quad\text{for $Y \in \mathfrak{g}$},
\end{eqnarray}
where we use the same letter $\lambda^*$ to denote
 the infinitesimal action of $\mathfrak p$ on $V^\vee$.
This action yields a Lie algebra homomorphism
\begin{equation}
d\pi_{\lambda^*}:\mathfrak g\to\mathcal D(\mathfrak n_-)\otimes\operatorname{End}(V^\vee).
\end{equation}
In turn,
 we get another Lie algebra homomorphism
 by the algebraic Fourier transform
 on the Weyl algebra $\mathcal D(\mathfrak n_-)$:
\begin{equation}\label{eqn:dpihat2}
\widehat{d\pi_{\lambda^*}}:\mathfrak g\to\mathcal D(\mathfrak n_+)\otimes\operatorname{End}(V^\vee),
\end{equation}
where we identify $\mathfrak n_-^\vee$ with $\mathfrak n_+$ by a $\mathfrak g$-invariant non-degenerate bilinear form on
$\mathfrak g$ (\emph{e.g.} the Killing form).

\begin{thm}[F-method for Hermitian symmetric spaces]\label{thm:FGK}
Suppose we are in the setting of Section \ref{sec:subHerm}. 
\begin{enumerate}
\item
We have the following commutative
diagram of three isomorphisms:
\small{
\begin{equation}\label{eqn:diagram2}
\xymatrix{
&\operatorname{Hom}_{K'}(V, \mathrm{Pol}(\mathfrak n_+)\otimes W)^{\widehat{d\pi_{\lambda^*}}(\mathfrak n_+')} & {}\\
\operatorname{Hom}_{\mathfrak p'}(W^\vee,\indpg(V^\vee)) \ar[ur]^{F_c\otimes\mathrm{id}}\ar[rr]_{D_{X\to Y}}^\sim
&{} & \operatorname{Hom}_{G'}(\mathcal O(X,\mathcal V), \mathcal O(Y,\mathcal W))\ar[ul]_{\mathrm{Symb}\otimes\mathrm{id}}.
}
\end{equation}}
\item Let $\mathfrak b(\mathfrak k')$ be a Borel subalgebra of $\mathfrak k'$, and assume that $W$ is the irreducible representation of $K'$ with lowest weight $-\chi$. Then we have the following isomorphism:
$$
\mathrm{Hom}_{K'}(V, \mathrm{Pol}(\mathfrak n_+)\otimes W)^{\widehat{d\pi_{\lambda^*}}(\mathfrak n_+')}\stackrel{\sim}{\to}
\left\{P\in\mathrm{Pol}(\mathfrak n_+)\otimes V^\vee: P\,\mathrm{satisfies}\,\eqref{eqn:Q1}\,\mathrm{and}\,\eqref{eqn:Q2}\right\}
$$
\begin{alignat}{3}
\label{eqn:Q1} ZP&=\chi(Z)P,\qquad&&\mathrm{for}\,\,\mathrm{all}\,\,Z\in\mathfrak b(\mathfrak k').\\
\label{eqn:Q2} \widehat{d\pi_{\lambda^*}}({C})P&=0, &&\mathrm{for}\,\,\mathrm{all}\,\,C\in\mathfrak n'_+.
\end{alignat}

\end{enumerate}

\end{thm}
\begin{proof} 1)
The first statement follows from Theorem \ref{thm:locGK} and \cite[Corollary \ref{cor:Fmethod}]{PART1}.

2) Via the linear isomorphism $\mathrm{Hom}_{\C}(V,\mathrm{Pol}(\mathfrak n_+)\otimes W)
\simeq \mathrm{Pol}(\mathfrak n_+)\otimes\mathrm{Hom}_{\C}(V,W)$, we have an isomorphism
\begin{eqnarray*}
&&\mathrm{Hom}_{K'}(V,\mathrm{Pol}(\mathfrak n_+)\otimes W)^{\widehat{d\pi_{\lambda^*}}(\mathfrak n_+')}\\
&\simeq&\{\psi\in\mathrm{Pol}(\mathfrak n_+)\otimes \mathrm{Hom}_{\C}(V,W): \psi\,\,
\mathrm{satisfies}\,\,\eqref{eqn:Fmethod1part2}\,\mathrm{and}\,\,\eqref{eqn:Fmethod2part2}\},
\end{eqnarray*}
\begin{eqnarray}
\nu(\ell)\circ\mathrm{Ad}_\sharp(\ell)\psi\circ\lambda(\ell^{-1})=\psi
&&\mathrm{for\,\, all}\, \ell\in K',\label{eqn:Fmethod1part2}\\
(\widehat{d\pi_{\lambda^*}}( C)\otimes\operatorname{id}_W)\psi=0 &&\mathrm{for\,\, all}\, C\in\mathfrak n_+',\label{eqn:Fmethod2part2}
\end{eqnarray}
where $\mathrm{Ad}_\sharp(\ell):\mathrm{Pol}(\mathfrak n_+)\to\mathrm{Pol}(\mathfrak n_+),
\varphi\mapsto\varphi\circ \mathrm{Ad}(\ell)^{-1}$.

On the other hand, if $-\chi$ is the lowest weight of the irreducible representation $W$ of $K'$,
we have an isomorphism
\begin{equation}\label{eqn:chi}
\mathrm{Hom}_{K'}(V,\mathrm{Pol}(\mathfrak n_+)\otimes W)
\simeq(\mathrm{Pol}(\mathfrak n_+)\otimes V^\vee)_\chi,
\end{equation}
where
$$
\left(\mathrm{Pol}(\mathfrak n_+)\otimes V^\vee\right)_\chi:=\left\{ P\in\mathrm{Pol}(\mathfrak n_+)\otimes V^\vee:
P\,\,\mathrm{satisfies}\,\, \eqref{eqn:Q1}\right\}.
$$
Therefore, Theorem \ref{thm:FGK} (2) is deduced from Theorem \ref{thm:FGK} (1)
and from the following natural isomorphism:
$$
\{\psi\,\mathrm{satisfying}\, \eqref{eqn:Fmethod1part2}\,\mathrm{and}\,\eqref{eqn:Fmethod2part2}\}\stackrel{\sim}{\to}
\{P\,\mathrm{satisfying}\, \eqref{eqn:Q1}\,\mathrm{and}\,\eqref{eqn:Q2}\}.
$$
See also [\emph{op. cit.}, Lemma \ref{lem:412}].
\end{proof}
 The F-method (see [\emph{op.cit.}, Section \ref{sec:44}]) in this setting consists of the following five steps:\vskip5pt
\begin{enumerate}
\item[Step 0.] Fix a finite-dimensional  representation $(\lambda,V)$ of the maximal compact subgroup $K$. Form a $G$-equivariant holomorphic vector bundle $\mathcal V_X\equiv\mathcal V=G\times_K V$ on $X=G/K$.\vskip5pt

\item[Step 1.] 
Extend $\lambda$ to
 a representation of  the Lie algebra $\mathfrak p=\mathfrak k+\mathfrak n_+$ 
 by letting $\mathfrak n_+$ act trivially, and define another representation $\lambda^*:=\lambda^\vee\,\otimes\, \C_{2\rho}$ 
 of $\mathfrak p$
 on $V^\vee$.
Compute $
d\pi_{\lambda^*}$ and
$\widehat{d\pi_{\lambda^*}}$.\vskip5pt

\item[Step 2.] Find a finite-dimensional  representation $(\nu,W)$ of the Lie group $K'$
such that 
$$
\operatorname{Hom}_{\mathfrak g'}(\indpgprime(W^\vee),\indpg(V^\vee))\neq\{0\},
$$
or equivalently,
$$
\operatorname{Hom}_{\mathfrak k'}(W^\vee,\indpg(V^\vee))\neq\{0\}.
$$
Form a $G'$-equivariant holomorphic
vector bundle $\mathcal W_Y\equiv\mathcal W=G'\times_{K'}W$ on $Y=G'/K'$. According to the duality theorem \cite[Theorem \ref{prop:2.10}]{PART1}
the space of differential symmetry breaking operators $\operatorname{Diff}_{G'}(\mathcal V_X, \mathcal W_Y)$ is then non-trivial.\vskip5pt

\item[Step 3.]  Write down the condition
on $\operatorname{Hom}_{K'}(V, \mathrm{Pol}(\mathfrak n_+)\otimes W)^{\widehat{d\pi_{\lambda^*}}(\mathfrak n_+')}$, namely, the space of
 $\psi\in \operatorname{Pol}(\mathfrak n_+)\otimes\operatorname{Hom}_{\C}(V,W)$ satisfying
 \eqref{eqn:Fmethod1part2} and \eqref{eqn:Fmethod2part2}
or equivalently $P\in \mathrm{Pol}(\mathfrak n_+)\otimes V^\vee$ satisfying \eqref{eqn:Q1} and \eqref{eqn:Q2}. 

\vskip5pt
\item[Step 4.] 
Use the invariant theory and give a simple description of
$$
\operatorname{Hom}_{K'}(V,\mathrm{Pol}(\mathfrak n_+)\otimes W)\simeq\left(\mathrm{Pol}(\mathfrak n_+)\otimes V^\vee\right)_\chi,\quad \psi \leftrightarrow P$$
by means of ``regular functions $g(s)$ on a slice'' $S$ for generic $K'_\C$-orbits on $\mathfrak n_+$.
Induce differential equations for $g(s)$
 on $S$ 
  from (\ref{eqn:Fmethod2part2}) (or equivalently \eqref{eqn:Q2}). Concrete computations are based on the technique of the $T$-saturation of differential operators, see Section \ref{sec:Tsat}. Solve the differential equations of $g(s)$.

\vskip5pt

\item[Step 5.] Transfer a solution $g$ obtained in Step 4 into
a polynomial solution $\psi$ to (\ref{eqn:Fmethod1part2}) and (\ref{eqn:Fmethod2part2}).
 In the diagram \eqref{eqn:diagram2}, $(\operatorname {Symb}\otimes \operatorname{id})^{-1}(\psi)$ gives 
the desired differential symmetry breaking operator in the coordinates
$\mathfrak{n}_-$ of $X$. As a byproduct, obtain an explicit $K'$-type $W^\vee$ annihilated by $\mathfrak n_+'$ in $\indpgprime(V^\vee)$ (sometimes
referred to as \emph{singular vectors}) as $(F_c\otimes\mathrm{id})^{-1}(\psi)$.
 \end{enumerate}

We shall give some comments on Steps 3 and 4 in Sections \ref{sec:33part2} and \ref{sec:Tsat} respectively. For Step 2, there are two approaches: one is to use (abstract) branching laws for
the restriction of $\indpg(V^\vee)$ to the subalgebra $\mathfrak g'$ (\emph{e.g.} Fact \ref{fact:A}) or the restriction of $\mathcal O(G/K,\mathcal V)$ to the subgroup $G'$ 
(\emph{e.g.} Fact \ref{fact:B}). The other one is to apply the F-method and reduce it to a question of solving differential equations of second order. The former approach works well for generic parameters. We shall see that the latter approach is efficient for singular parameters in our setting (Theorems \ref{thm:SO(2,n)}, \ref{thm:Sp} and \ref{thm:U(n,1)},
see also \cite{KOSS}).

\subsection{T-saturation of differential operators}\label{sec:Tsat}
In order to implement Step 4, our idea is to introduce saturated differential operators as follows. 
For simplicity consider the case when $\dim_\C V=1$. 
Then $\mathrm{Hom}_{K'}(V,\mathrm{Pol}(\mathfrak n_+)\otimes W)$
is identified with a subspace of $\mathrm{Pol}(\mathfrak n_+)$ via the
isomorphism \eqref{eqn:chi}.  Let $\C(\mathfrak n_+)$ denote the field of rational functions on $\mathfrak n_+$.
Suppose that we have a morphism $T:\C[S]\To\C(\mathfrak n_+)$
such that $T$ induces an isomorphism
$$ 
T:\Gamma(S)\stackrel{\sim}{\to}\operatorname{Hom}_{K'}(V,\mathrm{Pol}(\mathfrak n_+)\otimes W)
$$
for some algebraic variety $S$ (``slice" of a generic $K'_\C$-orbit on $\mathfrak n_+$), and for some appropriate function
space $\Gamma(S)$ (\emph{e.g.} $\Gamma(S)=\mathrm{Pol}_a[t]_{\mathrm{even}}$, see \eqref{eqn:gs}).
In the special case where $V$ and $W$ are the trivial one-dimensional representations of $K$ and $K'$,
respectively, we may take $S=\mathfrak n_+//K_\C'$ (geometric quotient) and $T$ is the natural morphism
$\C[S] \stackrel{\sim}{\longrightarrow}
\C[\mathfrak n_+]^{K_\C'}$.

\begin{defn}\label{def:Tsat}
A differential operator $R$ on $\mathfrak n_+$ with rational coefficients is
\emph{T-saturated} if there exists an operator $D$ such that the following diagram
commutes:

$$
\xymatrix{
\C[S]\ar[r]^{T\quad}\ar[d]_{D}& \C(\mathfrak n_+)\ar[d]^{R}\\
\C[S]\ar[r]^{T\quad}&\C(\mathfrak n_+)
.}
$$
\end{defn}
Such an operator $D$ is unique (if exists), and we denote it by $T^\sharp R$. Then we have
\begin{equation}\label{eqn:tt}
T^\sharp(R_1\cdot R_2)=T^\sharp(R_1)T^\sharp(R_2)
\end{equation}
whenever it makes sense.

\begin{prop}\label{prop:Tsat}
Let $C_1,\cdots,C_k$ be a basis of $\mathfrak n_+'$. Suppose there exist non-zero $Q_j\in\C(\mathfrak n_+)$ such that
$Q_j\widehat{d\pi_{\lambda^*}}(C_j)$ is $T$-saturated ($1\leq j\leq k$) and set
$D_j:=T^\sharp(Q_j\widehat{d\pi_{\lambda^*}}(C_j))$. Then $T$ induces a bijection
\begin{eqnarray*}
&&\left\{ g\in\Gamma[S]: D_jg=0,\, (1\leq j\leq k)\right\}\\
&\stackrel{\sim}{\rightarrow}&
\left\{\psi\in\operatorname{Hom}_{K'}(V,\mathrm{Pol}(\mathfrak n_+)\otimes W):
\psi\,\mathrm{satisfies}\, \eqref{eqn:Fmethod1part2}\,\mathrm{and}\, \eqref{eqn:Fmethod2part2}\right\}\\
&\simeq&
\left\{P\in\left(\mathrm{Pol}(\mathfrak n_+)\otimes V^\vee\right)_\chi:
P\,\mathrm{satisfies}\, \eqref{eqn:Q2}\right\}.
\end{eqnarray*}
\end{prop}
We shall use this idea in Sections \ref{sec:son2}-\ref{sec:un1} where $S$ is one-dimensional and $D_j$ are ordinary
differential operators. We note that $D_jg=0$ ($1\leq \forall j\leq k$) is equivalent to a single equation $D_ig=0$
if $K'$ acts irreducibly on $\mathfrak n_+'$.

\subsection{Complement for the F-method in vector-valued cases and highest weight varieties}\label{sec:33part2}

 If the target $\mathcal{W}_Y$ is no longer a line bundle but a vector bundle,
{\it{i.e.}}, if $W$ is an arbitrary finite-dimensional, 
 irreducible $\mathfrak k'$-module, we recall two supplementary ingredients
 of Step 3 in the recipe by reducing \eqref{eqn:Fmethod1part2} to a simpler algebraic question on polynomial rings,  so that we can focus on the crucial part consisting of a system of differential equations of second order \eqref{eqn:Fmethod2part2}.
This construction is based on the notion of \textit{highest weight variety} of the fiber $W$
and is summarized in the following two lemmas (see \cite[Lemmas \ref{lem:412} and \ref{lem:hwv}]{PART1}.

We fix a Borel subalgebra $\mathfrak b(\mathfrak k')$ of $\mathfrak k'$.
Let $\chi: \mathfrak b(\mathfrak k')\to\C$ be a character. For a $\mathfrak{k}'$-module $U$, 
we set 
$$
U_\chi:=
 \{u\in U\,:\, Zu=\chi(Z)u \quad\text{for any $Z\in\mathfrak b(\mathfrak k')$}\}.
$$

 Suppose that $W$ is the irreducible representation of $\mathfrak k'$ with lowest weight $-\chi$. Then
 the contragredient representation $W^\vee$ has a highest weight $\chi$. We fix 
  a non-zero highest weight vector $w^\vee\in (W^\vee)_\chi$. 
 Then the contraction map
 $$
U\otimes W \to U, 
\quad
\psi\mapsto \langle\psi,w^\vee\rangle,
$$
induces a bijection between the following two subspaces:
\begin{equation}\label{eqn:111k}
(U\otimes W)^{\mathfrak{k}'}
\stackrel{\sim}{\To}
U_\chi,
\end{equation}
if $U$ is completely reducible as a $\mathfrak{k}'$-module. By using the isomorphism \eqref{eqn:111k}, we reformulate Step 3 of the recipe for the F-method as follows:

\begin{lem}\label{lem:412part2}
 Assume that $W$ is an irreducible representation of the parabolic subalgebra $\mathfrak p'$. Let $-\chi$ be the lowest weight of $W$ as a $\mathfrak k'$-module. Then we have a natural injective homomorphism
$$
\mathrm{Diff}_{G'}(\mathcal V_X,\mathcal W_Y)\hookrightarrow\left\{Q\in\left(\mathrm{Pol}(\mathfrak n_+)\otimes V^\vee\right)_\chi:
\widehat{d\pi_\mu}({C})Q=0\quad\mathrm{for}\,\mathrm{all}\,\,C\in\mathfrak n'_+\right\},
$$
which is bijective if $K'$ is connected.
\end{lem}
See \cite[Lemma \ref{lem:412}]{PART1} for the proof.

Since any non-zero vector in $W^\vee$ is cyclic, 
the next lemma explains how to recover $
 D_{X\to Y}(\varphi)$ from $Q$ given in Lemma \ref{lem:412part2}.

We assume, for simplicity,
that the $\mathfrak{k}$-module $(\lambda,V)$ lifts to $K_\C$, 
the $\mathfrak{k}'$-module $(\nu,W)$ lifts to $K'_\C$, and use the same letters to denote their liftings.
\begin{lem}\label{lem:hwvpart2}
For any $\varphi\in\operatorname{Hom}_{\mathfrak p'}(W^\vee,\indpg(V^\vee))$,
$k\in K'_\C$ and $w^\vee\in W^\vee$,
\begin{equation}\label{eqn:Scalea}
\left\langle D_{X\to Y}(\varphi), \nu^\vee(k)w^\vee\right\rangle=
(\mathrm{Ad}(k)\otimes \lambda^\vee(k))
\left\langle D_{X\to Y}(\varphi), w^\vee\right\rangle.
\end{equation}
\end{lem}

See \cite[Lemma \ref{lem:hwv}]{PART1} for the proof.

%-------------------FORMER SECTION 6 FROM PART 1
 \section{Branching laws and Hermitian symmetric spaces}\label{sec:BRL}

The existence, respectively the uniqueness
(up to scaling)
of differential symmetry breaking operators from
$\mathcal{V}_X$ to $\mathcal{W}_Y$ are subject to the conditions
\begin{equation}\label{eqn:mfbdle}
\dim\operatorname{Diff}_{G'} (\mathcal{V}_X,\mathcal{W}_Y) \ge 1, \ \text{respectively}\ \le 1.
\end{equation}
So we need to find the geometric settings
(i.e.\ the pair $Y \subset X$ of generalized flag varieties and two homogeneous
vector bundles $\mathcal{V}_X \to X$ and $\mathcal{W}_Y \to Y$)
satisfying \eqref{eqn:mfbdle}.
This is the main ingredient of Step 2 in the recipe of the F-method,
and thanks to \cite[Theorem \ref{thm:surject}]{PART1},
the existence and uniqueness are equivalent to the following
question concerning (abstract) branching laws:
 Given a $\mathfrak p$-module $V$, find all finite-dimensional  $\mathfrak p'$-modules $W$ such that
$\dim\operatorname{Hom}_{\mathfrak p'}(W^\vee,\indpg(V^\vee))=1$,
and equivalently,
\begin{equation}\label{eqn:homverma}
\dim\operatorname{Hom}_{\mathfrak g'}(\indpgprime(W^\vee),\indpg(V^\vee))
=1.
\end{equation}
This section briefly reviews what is known on this question
 (see Fact \ref{fact:A}).
 
Let $\mathfrak g$ be a complex semisimple Lie algebra, and $\mathfrak j$
a Cartan subalgebra of $\mathfrak g$. 
We fix a positive root system
$\Delta^+\equiv\Delta^+(\mathfrak g,\mathfrak j)$, write
$\rho$
for half the sum of positive roots,
$\alpha^\vee$ for the coroot for $\alpha \in \Delta$,
 and $\mathfrak g_\alpha$ for the root space.  
Define a Borel subalgebra
$\mathfrak b=\mathfrak j+\mathfrak n$ with nilradical $\mathfrak n:=\bigoplus_{\alpha\in\Delta^+}\mathfrak g_\alpha$.

The BGG category $\mathcal O$ is defined as the full subcategory of $\mathfrak g$-modules whose objects are finitely generated, $\mathfrak j$-semisimple and
locally $\mathfrak n$-finite \cite{BGG}.

As in the previous sections, fix a standard parabolic subalgebra $\mathfrak p$ with Levi
decomposition $\mathfrak p=\mathfrak k +\mathfrak n_+$ such that the Levi factor $\mathfrak k$ contains $\mathfrak j$.
 We set
$\Delta^+(\mathfrak k):=\Delta^+\cap\Delta(\mathfrak k,\mathfrak j)$.  
The parabolic BGG category $\mathcal O^{\mathfrak p}$
is defined as the full subcategory of $\mathcal O$ whose objects are locally $\mathfrak{k}$-finite.

We define
$$
\Lambda^+(\mathfrak k):=\{ \lambda\in\mathfrak j^*:\,\langle\lambda,\alpha^\vee\rangle
\in\N\text{ for any } \alpha\in\Delta^+(\mathfrak k)\},
$$
the set of linear forms $\lambda$ on $\mathfrak j$ whose restrictions to $\mathfrak j\cap[\mathfrak k,\mathfrak k]$ are dominant integral.
We write $V_\lambda$ for the finite-dimensional  simple $\mathfrak k$-module
with highest weight $\lambda$, regard it as a $\mathfrak p$-module 
by letting $\mathfrak n_+$ act trivially,
 and consider the generalized Verma module
\[
\indpg(\lambda)\equiv
\indpg(V_\lambda):=
U(\mathfrak g)\otimes_{U(\mathfrak p)}V_\lambda.
\]
Then $\indpg(\lambda)\in\mathcal O^\mathfrak p$ and any
simple object in $\mathcal O^\mathfrak p$ is the quotient of some generalized Verma module.
If 
\begin{equation}\label{eqn:condimu}
\langle\lambda,\alpha^\vee\rangle=0\qquad \mathrm{for\, all}\qquad\alpha\in\Delta(\mathfrak k),
\end{equation}
 then $V_\lambda$ is one-dimensional, to be denoted also by $\C_\lambda$.
In this case we
say $\indpg(\lambda)$ is of \textit{scalar type}.

Let $\tau\in \operatorname{Aut}(\mathfrak g)$ be an involutive automorphism of the Lie algebra $\mathfrak g$.
We write 
\[
    \mathfrak{g}^{\pm \tau}:=\{v \in \mathfrak{g}: \tau v = \pm v\}
\]
 for the $\pm 1$ eigenspaces of $\tau$, 
respectively.  
 We say that $(\mathfrak g,\mathfrak g')$ is a symmetric pair if
$\mathfrak g'=\mathfrak g^\tau$ for some $\tau$. 

For a general choice of $\tau$ and ${\mathfrak {p}}$,
 the space considered in (\ref{eqn:homverma}) 
 may be reduced to zero for all
$\mathfrak p'$-modules $W$.
Suppose $V\equiv V_\lambda$ with
$\lambda\in\Lambda^+(\mathfrak k)$ generic. 
Then a necessary and sufficient condition for the existence of $W$ such that
the left-hand side of (\ref{eqn:homverma}) is non-zero is given by the geometric 
requirement on the generalized flag variety $G_{\mathbb{C}}/P_{\mathbb{C}}$, 
namely, the set $G^\tau_\C P_\C$ is closed in $G_\C$, see \cite[Proposition 3.8]{K12}. 

Consider now the case where the nilradical $\mathfrak n_+$ of $\mathfrak p$ is abelian.
Then, the following result holds :

\begin{fact}[\cite{K12}]\label{fact:AA}
If the nilradical $\mathfrak n_+$ of $\mathfrak p$ is abelian, then for any symmetric pair $(\mathfrak g,\mathfrak g^\tau)$
the restriction of a generalized Verma module of scalar type $ \indpg(-\lambda)\vert_{\iota(\mathfrak g^\tau)}$ is multiplicity-free for any embedding $\iota:\mathfrak g^\tau\to\mathfrak g$ such that $\iota(G^\tau_\C)P_\C$ is closed in $G_\C$ and for any sufficiently positive $\lambda$.
\end{fact}

A combinatorial description of the branching law is given as follows.
Suppose 
 that ${\mathfrak {p}}$ is ${\mathfrak {g}}^{\tau}$-compatible
 (see \cite[Definition \ref{def:3.2}]{PART1}). 
Then the involution $\tau$ stabilizes
 $\mathfrak k$ and $\mathfrak{n}_+$, respectively, 
 the nilradical $\mathfrak{n}_+$ decomposes into a direct sum of eigenspaces 
$\mathfrak n_+=\mathfrak n_+^\tau+\mathfrak n_+^{-\tau}$
 and $G^\tau_\C P_\C$ is closed in $G_\C$. 
Fix a Cartan subalgebra $\mathfrak j$ of $\mathfrak k$ such that
$\mathfrak j^\tau:=\mathfrak j\cap\mathfrak g^\tau$ is a Cartan subalgebra of $\mathfrak k^\tau$.

We define $\theta \in \operatorname{End}(\mathfrak{g})$
by $\theta|_{\mathfrak{k}}=\operatorname{id}$
and $\theta|_{\mathfrak{n}_++\mathfrak{n}_-}=-\operatorname{id}$.
Then $\theta$ is an involution 
 commuting with $\tau$.  
Moreover it is an automorphism 
if $\mathfrak{n}_+$ is abelian.
The reductive subalgebra 
$
\mathfrak{g}^{\tau\theta}= {\mathfrak {k}}^{\tau}
                  + {\mathfrak {n}}_-^{-\tau}
                  + {\mathfrak {n}}_+^{-\tau}
$
decomposes into simple or abelian ideals $\bigoplus_i\mathfrak{g}_i^{\tau\theta}$,
and we write the decomposition of $\mathfrak{n}_+^{-\tau}$ as 
 ${\mathfrak {n}}_{+}^{-\tau}=\bigoplus_i{\mathfrak {n}}_{+,i}^{-\tau}$
 correspondingly.  
Each ${\mathfrak {n}}_{+,i}^{-\tau}$
 is a ${\mathfrak {j}}^{\tau}$-module,
 and we denote by 
$\Delta({\mathfrak {n}}_{+,i}^{-\tau}, {\mathfrak {j}}^{\tau})$
 the set of weights of ${\mathfrak {n}}_{+,i}^{-\tau}$
 with respect to ${\mathfrak {j}}^{\tau}$.  
The roots $\alpha$ and $\beta$ are said to be 
 {\it{strongly orthogonal}}
 if neither $\alpha + \beta$ nor $\alpha - \beta$ is a root.  
We take a maximal set of strongly orthogonal roots 
$\{\nu_1^{(i)}, \cdots, \nu_{k_i}^{(i)}\}$ 
 in  $\Delta({\mathfrak{n}}_{+,i}^{-\tau}, {\mathfrak{j}}^{\tau})$ 
 inductively as follows:
\begin{enumerate}
\item[1)]
$\nu_1^{(i)}$ is the highest root
 of $\Delta({\mathfrak {n}}_{+,i}^{-\tau}, {\mathfrak {j}}^{\tau})$.  
\item[2)]
 $\nu_{j+1}^{(i)}$ is the highest root among the elements
 in $\Delta({\mathfrak{n}}_{+,i}^{-\tau}, {\mathfrak{j}}^{\tau})$
 that are strongly orthogonal to $\nu_1^{(i)}, \cdots, \nu_{j}^{(i)}$
 ($1 \le j \le k_i-1$).   
\end{enumerate}

We define the following subset
 of ${\mathbb{N}}^k$
 ($k=\sum k_i$)
 by
\begin{equation}
\label{eqn:Aplus}
  A^+:= \prod_i A_i, \quad
  A_i :=
  \{
(a_{j}^{(i)})_{1 \le j \le k_i}\in {\mathbb{N}}^{k_i}
:
 a_1^{(i)} \ge \cdots \ge a_{k_i}^{(i)}\ge 0
\}.  
\end{equation}

Introduce the following positivity condition:
\begin{equation}\label{eqn:lpos2}
\langle\lambda-\rho_{\mathfrak{g}},\alpha\rangle>0
\quad\text{for any $\alpha \in \Delta(\mathfrak{n}_+,\mathfrak{j})$}.
\end{equation}

\begin{fact}[\cite{K08}]
\label{fact:A}
Suppose ${\mathfrak {p}}$ is ${\mathfrak {g}}^{\tau}$-compatible,
 and $\lambda$ satisfies \eqref{eqn:condimu}
 and \eqref{eqn:lpos2}.  
Then the generalized Verma module $\indpg(-\lambda)$ decomposes into a multiplicity-free direct sum of irreducible $\mathfrak g^\tau$-modules
:
\begin{equation}\label{eqn:branchV}
\indpg(-\lambda)\vert_{\mathfrak g^\tau}
\simeq
 \bigoplus
_{(a_{j}^{(i)}) \in A^+}
\indpgtau(-\lambda\vert_{\mathfrak j^\tau}
-\sum_i\sum_{j=1}^{k_i} a_j^{(i)}\nu_j^{(i)}).
\end{equation}
In particular,
 for a simple ${\mathfrak {p}}^{\tau}$-module $W$
(namely,
 a simple ${\mathfrak{k}}^{\tau}$-module
 with trivial action of ${\mathfrak {n}}^{\tau}$), 
\[
  \dim {\operatorname{Hom}}_{{\mathfrak {g}}^{\tau}}
   (\indpgtau(W^{\vee}), \indpg({\mathbb{C}}_{-\lambda}))=1
\]
if and only if the highest weight
 of the ${\mathfrak {k}}^{\tau}$-module $W$
 is of the form 
$\lambda|_{{\mathfrak {j}}^{\tau}}+\sum_i \sum_{j=1}^{k_i}
 a_j^{(i)}\nu_j^{(i)}$
 for some $(a_j^{(i)}) \in A^+$.  
\end{fact}

%------------------------------------------------------------------

Notice that when $\tau$ is a Cartan involution, $G^\tau$ is compact and $\mathfrak g^\tau=\mathfrak p^\tau$. In this case, the formula \eqref{eqn:branchV}
is due to L. K. Hua \cite{Hua} (classical case), B. Kostant (unpublished), and W. Schmid \cite{Schm}.
In general $G^\tau$ is non-compact,
and we need to consider infinite-dimensional  irreducible representations of $G^\tau$
when we consider the branching law $G\downarrow G^\tau$.
 
In remaining Sections \ref{sec:normal}, \ref{sec:son2}, \ref{sec:sp} and \ref{sec:un1} we construct a family of equivariant differential operators
for all symmetric pairs $(\mathfrak{g},\mathfrak{g}^\tau)$ with
$G^\tau$ non-compact and
 $k=1$ (in particular, $\Delta(\mathfrak n_{+,i}^{-\tau},\mathfrak j^\tau)$ is empty for all but one $i$).

In conclusion,
we recall the corresponding branching laws in the category of unitary representations,
 which are the dual 
of the formul\ae\ in Fact \ref{fact:A}.
We denote by
 $\mathcal H^2(M,\mathcal V)$  the Hilbert space
 of square integrable holomorphic sections of the Hermitian vector bundle $\mathcal V$
 over a Hermitian manifold $M$.
If the positivity condition \eqref{eqn:lpos2}
 holds,
then $\mathcal{H}^2(G/K,\mathcal{L}_\lambda)\ne\{0\}$,
and $G$ acts unitarily and irreducibly on it.

    Given $\underline{a} = (a_j^{(i)}) \in A^+$ ($\subset \mathbb{N}^k$),
we write $\mathcal W^{\underline{a}}_\lambda$
for the $G^\tau$-equivariant holomorphic vector bundle over
$G^\tau/K^\tau$ associated to the irreducible representation
 $\mathcal W^{\underline{a}}_\lambda$
of
$\mathfrak k^\tau$ with highest weight
$\lambda\vert_{\mathfrak j^\tau}+\sum_i \sum_{j=1}^{k_i}a_j^{(i)}\nu_j^{(i)}$.

\begin{fact}[\cite{K08}]\label{fact:B}
If the positivity condition \eqref{eqn:lpos2} is satisfied,
then
$\mathcal{H}^2(G^\tau/K^\tau,\mathcal{W}_\lambda^{\underline{a}})$
is non-zero and $G^\tau$ acts on it irreducibly and unitarily for any $\underline{a} \in A^+$.
Moreover, the branching law for the restriction $G\downarrow G^\tau$ is given by
\begin{equation}\label{eqn:branch}
\mathcal H^2(G/K,\mathcal L_\lambda)
\simeq \sideset{}{^\oplus}\sum_{\underline{a}\in A^+} \mathcal{H}^2
(G^\tau/K^\tau,\mathcal{W}_\lambda^{\underline{a}})
\quad \text{{\rm{(Hilbert direct sum)}}}.
\end{equation}
\end{fact}

%----------------------------------------

\section{Normal derivatives versus intertwining operators}\label{sec:normal}
Let
 $G'/K'$ be a subsymmetric space  of the Hermitian symmetric
space $G/K$ as in Section \ref{sec:subHerm}.  
Consider the Taylor expansion of any holomorphic
function (section) on $G/K$ with respect to the normal direction.
Then the coefficients give rise to holomorphic sections of a family of vector bundles over the
submanifold $G'/K'$.
This idea was used earlier by Jakobsen and Vergne \cite{JV}, and by the first author \cite{K08}
 for filtered modules to find \emph{abstract} branching laws.

However, it should be noted that normal derivatives do not always give rise to symmetry breaking operators. 
In this section we clarify the reason in the general setting, and then give a classification of all 
irreducible symmetric
pairs
($\mathfrak g(\R), \mathfrak g(\R)^\tau$) of split rank one for which it happens.

\subsection{Normal derivatives and the Borel embedding}
Suppose $E=E'\oplus E''$ is a direct sum of complex vector spaces. 
Let $\mathcal V_E:=E\times V$ and $\mathcal W_{E'}:=E'\times W$ be direct product vector bundles over $E$ and $E'$, respectively. 
Clearly, we have isomorphisms $\mathcal O(E,\mathcal V_E)\simeq \mathcal O(E)\otimes V$, and 
$\mathcal O(E',\mathcal W_{E'})\simeq \mathcal O(E')\otimes W$.

Take coordinates $y=(y_1,\cdots,y_p)$ in $E'$ and $z=(z_1,\cdots,z_n)$ in $E''$. 
The subspace $E'$ is given by the condition $z=0$ in $E=\{(y,z)\,:\,y\in E', z\in E''\}$. 
A holomorphic differential operator
$
\widetilde T:\mathcal O (E)\otimes V\To \mathcal O(E')\otimes W,\,f(y,z)\mapsto (\widetilde Tf)(y)
$ is said to be a \emph{normal derivative with respect to the decomposition $E=E'\oplus E''$}
if it is
of the form
\begin{equation}\label{eqn:normalT}
\left(\widetilde T f\right)(y)=\sum_{\alpha\in \N^q}
T_\alpha(y)\left(\frac{\partial^{|\alpha|} f(y,z)}{\partial z^\alpha}\bigg|_{z=0}\right),
\end{equation}
for some $T_\alpha\in\mathcal O(E')\otimes\mathrm{Hom}_\C(V,W)$. 

We write $\ndh(\mathcal V_E,\mathcal W_{E'})$ for the space of (holomorphic) normal derivatives.
This notion depends on the direct sum decomposition $E=E'\oplus E''$.

Since the commutative groups $E\supset E'$ act on the direct product bundles $\mathcal V_E$ and $\mathcal W_{E'}$,
respectively, we can consider symmetry breaking operators in this abelian setting, namely,
$E'$-equivariant normal derivatives, which amount to the condition that $T_\alpha(y)$ in \eqref{eqn:normalT} is a differential operator with constant coefficients
for every $\alpha\in\N^q$. We denote $\ndc(\mathcal V_E,\mathcal W_{E'})$ the subspace
of $\ndh(\mathcal V_E,\mathcal W_{E'})$ consisting of those operators.

Thus we have seen the following:
\begin{lem}\label{lem:VWE} There is a natural isomorphism:
$$
\operatorname{Hom}_\C(V,W)\otimes S(E'')
\stackrel{\sim}{\longrightarrow}\ndc(\mathcal V_E,\mathcal W_{E'}).
$$
\end{lem}
Suppose we are in the setting of Section \ref{sec:subHerm}.
 We apply the concept of normal derivatives to the subsymmetric space
 $G'/K'$ in the Hermitian symmetric space $G/K$. Let $\mathcal V$ be a homogeneous vector bundle over $X=G/K$
associated with a finite-dimensional representation $V$ of $K$. Similarly, let $\mathcal W$ be
a homogeneous vector bundle over the subsymmetric space 
$Y=G'/K'$ associated with a finite-dimensional representation $W$
of $K'$.

By using the Killing form, we take a complementary subspace $\mathfrak g''$ of $\mathfrak g'$ in $\mathfrak g$
so that $\mathfrak g=\mathfrak g'\oplus\mathfrak g''$ is a direct sum of $G'$-modules. We set $\mathfrak n_-'':=
\mathfrak n_-\cap\mathfrak g''$. Since the characteristic element $H_o\in\mathfrak g'$ (see \eqref{eqn:ckk}), we have
a direct sum decomposition of $K'$-modules:
\begin{equation}\label{eqn:n12}
\mathfrak n_-=\mathfrak n_-'\oplus\mathfrak n_-''.
\end{equation}
Accordingly, we can consider the space $\ndh(\mathcal V_{\mathfrak n_-},\mathcal W_{\mathfrak n_-'})$ of
holomorphic normal derivatives with respect to \eqref{eqn:n12}.

We write $\ndh(\mathcal V_X,\mathcal W_{Y})$ and $\ndc(\mathcal V_X,\mathcal W_{Y})$ for the images 
of $\ndh(\mathcal V_{\mathfrak n_-},\mathcal W_{\mathfrak n_-'})$ and $\ndc(\mathcal V_{\mathfrak n_-},\mathcal W_{\mathfrak n_-'})$, respectively, under the natural injective map:
$$\xymatrix{
\mathrm{Diff}^{\mathrm{hol}}(\mathcal V_{\mathfrak n_-},\mathcal W_{\mathfrak n_-'})
\ar@{^{(}->}[r]&
\mathrm{Diff}^{\mathrm{hol}}(\mathcal V_{X},\mathcal W_{Y})}
$$
induced by the following map:
\begin{align}
&\xymatrix{
\mathcal O(\mathfrak {n}_-, V)
\ar[r]
\ar@{^{(}->}[d]_{\mathrm{restriction}}
& \mathcal O(\mathfrak {n}_-', W)
\ar@{_{(}->}[d]^{\mathrm{restriction}}
\\
\mathcal O(G/K, \mathcal V)\ar@{.>}[r]
& \mathcal O(G'/K',\mathcal W).
}
\label{diag:5.1}
\end{align}

Since the trivialization of the vector bundle $G_\C\times_{P_\C} V$
$$
\xymatrix{
\mathfrak n_-\times V \ar@{^{(}->}[r] \ar[d] & G_\C\times_{P_\C} V\ar[d]& \mathcal V_X \ar[d] \ar@{_{(}->}[l]\\
\mathfrak n_- \ar@{^{(}->}[r] & G_\C/{P_\C} & X=G/K  \ar@{_{(}->}[l]
}
$$
is $K_\C$-equivariant, Lemma \ref{lem:VWE} implies:
\begin{prop}\label{prop:Knormal} There is a natural isomorphism:
$$
\mathrm{Hom}_{K'}(V,S(\mathfrak n_-'')\otimes W)
\stackrel{\sim}{\longrightarrow}
\mathcal N\mathrm{Diff}^{\mathrm{const}}_{K'}(\mathcal V_X,\mathcal W_Y).
$$
\end{prop}
We study whether or not the following two subspaces
\begin{itemize}
\item $\mathcal N\mathrm{Diff}_{K'}(\mathcal V_X,\mathcal W_Y)$ of $K'$-equivariant normal derivatives and
\item $\operatorname{Hom}_{G'}(\mathcal O(\mathcal V_X), \mathcal O(\mathcal W_Y))$ of symmetry breaking operators
\end{itemize}
coincide in $\operatorname{Hom}_{\C}(\mathcal O(\mathcal V_X), \mathcal O(\mathcal W_Y))$.
Owing to Theorem \ref{thm:FGK} and Proposition \ref{prop:Knormal}, it reduces to an algebraic
problem to compare
\begin{itemize}
\item $\operatorname{Hom}_{K'}(V, S(\mathfrak n_-'')\otimes W)$ and
\item $\operatorname{Hom}_{K'}(V, \mathrm{Pol}(\mathfrak n_+)\otimes W)^{\widehat{d\pi_{\lambda^*}}(\mathfrak n_+')}$
\end{itemize}
in $\operatorname{Hom}_{\C}(V, \mathrm{Pol}(\mathfrak n_+)\otimes W)\simeq
\operatorname{Hom}_{\C}(V, S(\mathfrak n_-)\otimes W)$. We shall see in the next subsection that they actually coincide for the
three families of symmetric pairs out of the six listed in Table \ref{table:1}.

\subsection{When are normal derivatives intertwining operators?}

Let dim $V=1$, and we write as before $\mathcal L_\lambda$  for the homogeneous line bundle over $X=G/K$
associated to the character $\C_\lambda$ of $K$.
\begin{thm}\label{thm:nor_der}
Suppose $(\mathfrak g(\R),\mathfrak g(\R)^\tau)$ is a split rank one irreducible symmetric pair of holomorphic type (see Definition \ref{def:holo}). Then, the following three
conditions on the pair $(\mathfrak g(\R),\mathfrak g(\R)^\tau)$ are equivalent:
\begin{enumerate}
\item[(i)] For any $\lambda$ satisfying the positivity condition 
 \eqref{eqn:lpos2} and for any irreducible $K^\tau$-module $W$, all continuous $G^\tau$-homomorphisms 
$$
\mathcal{O}(X,\mathcal L_\lambda)\longrightarrow
\mathcal O(Y,\mathcal W),
$$
are given by normal derivatives with respect to the decomposition $\mathfrak n_-=\mathfrak n_-^\tau\oplus\mathfrak n_-^{-\tau}$.

\item[(ii)] For some $\lambda$ satisfying \eqref{eqn:lpos2} and for some irreducible $K^\tau$-module $W$,
there exists a non-trivial  $G^\tau$-intertwining operator 
$$
\mathcal{O}(X,\mathcal L_\lambda)\longrightarrow
\mathcal O(Y,\mathcal W)
$$
which is given by normal derivatives of positive order.
\item[(iii)] The symmetric pair $(\mathfrak g(\R),\mathfrak g(\R)^\tau)$ is isomorphic to one of\/ 
$(\mathfrak{su}(p,q),\mathfrak s(\mathfrak{u}(1)\oplus\mathfrak{u}(p-1,q)))$, 
$(\mathfrak{so}(2,2n),\mathfrak{u}(1,n))$ or
$(\mathfrak{so}^*(2n),\mathfrak{so}(2)\oplus\mathfrak{so}^*(2n-2))$.
\end{enumerate}
\end{thm}

Notice that the geometric nature of embeddings $Y\hookrightarrow X$ mentioned
in the condition (iii) corresponds to the following inclusions of flag varieties:
\begin{eqnarray*}
\mathrm{Gr}_{p-1}(\C^{p+q})&\hookrightarrow& \mathrm{Gr}_{p}(\C^{p+q});\\
\mathbb P^n\C&\hookrightarrow & Q^{2n}\C;\\
\mathrm{IGr}_{n-1}(\C^{2n-2})&\hookrightarrow&\mathrm{IGr}_{n}(\C^{2n}),
\end{eqnarray*}
where $\mathrm{Gr}_{p}(\C^{k}):=\{V\subset\C^{k}\,:\,\mathrm{dim}\,V=p\}$
is the complex Grassmanian,
 $Q^m\C:=\{z\in\mathbb P^{m+1}\C\,:\, z_0^2+\cdots+z_{m+1}^2=0\}$ is the complex quadric and
 $\mathrm{IGr}_{n}(\C^{2n}):=\{V\subset\C^{2n}\,:\,\mathrm{dim}V=n,
 \,Q\big\vert_{V}\equiv0\}$ is the isotropic Grassmanian for $\C^{2n}$ equipped with a non-degenerate quadratic
form $Q$.

\subsection{Outline of the proof of Theorem \ref{thm:nor_der}}
The implication (i)$\Rightarrow$(ii) is obvious. 
On the other hand,
for split rank one symmetric spaces there are three other cases 
(i.e., (1), (2) and (3) in Table \ref{table:1})
where the $G^\tau$-intertwining operators are not given by normal derivatives. In Sections \ref{sec:son2}, \ref{sec:sp} and
\ref{sec:un1} we construct them explicitly. This will conclude the implication (ii)$\Rightarrow$(iii).
For the rest of this section we shall give a proof for the 
implication (iii)$\Rightarrow$(i).

Consider a homomorphism:
$
T:W^\vee\longrightarrow S(\mathfrak n_-^{-\tau})\otimes V^\vee.
$ We regard $S(\mathfrak n_-^{-\tau})\otimes V^\vee$ as
a subspace of $\operatorname{Pol}(\mathfrak n_+)\otimes V^\vee$ on which
the Lie algebra $\mathfrak g$ acts by $\widehat{d\pi_{\lambda^*}}$, see (\ref{eqn:dpihat2}). If $T$ is a $K^\tau$-homomorphism, the differential operator $\widetilde T:
\mathcal O(G/K,\mathcal V_X)\to\mathcal O(G^\tau/K^{\tau},\mathcal W_Y)$ is $K^\tau$-equivariant. The following
statement gives a sufficient condition for $\widetilde T$ to be $G^\tau$-equivariant.

\begin{prop}\label{prop:415}
The normal derivative
$\widetilde T\in\ndc(\mathcal V_X,\mathcal W_Y)$ induces a $G^\tau$-equivariant differential
 operator from $\mathcal V_X$ to $\mathcal W_Y$ if and only if \,
 $T$ is a $K^\tau$-homo\-morphism and
 $T(W^\vee)$ is
contained in $\left( \operatorname{Pol}(\mathfrak n_+)\otimes V^\vee\right)^{\widehat{d\pi_{\lambda^*}}(\mathfrak n_+^\tau)}$. 
\end{prop}
\begin{proof}
The proof is a direct consequence of the F-method. Indeed, by Theorem \ref{thm:FGK}, $\widetilde T\in\ndc(\mathcal V_X,\mathcal W_Y)\subset
\operatorname{Diff}^{\mathrm {const}}(\mathfrak n_-)\otimes \operatorname{Hom}_\C(V,W)$ is 
a $G^\tau$-equivariant differential operator
if and only if $(\operatorname{Symb}\otimes\operatorname{id})(\widetilde T)
\in (\operatorname{Pol}(\mathfrak n_+)\otimes \operatorname{Hom}(V,W))^{\widehat{d\pi_{\lambda^*}}(\mathfrak p^\tau)}$
where
\begin{eqnarray*}
&&(\operatorname{Pol}(\mathfrak n_+)\otimes \operatorname{Hom}(V,W))^{\widehat{d\pi_{\lambda^*}}(\mathfrak p^\tau)}\\
&=&(\operatorname{Pol}(\mathfrak n_+)\otimes \operatorname{Hom}(V,W))^{\widehat{d\pi_{\lambda^*}}(\mathfrak k^\tau)}\cap
(\operatorname{Pol}(\mathfrak n_+)\otimes \operatorname{Hom}(V,W))^{\widehat{d\pi_{\lambda^*}}(\mathfrak n^\tau_+)}.
\end{eqnarray*}
Furthermore, by Theorem \ref{thm:FGK}, 
for $\widetilde T\in\ndc(\mathcal V_X,\mathcal W_Y)$, we have
$(\operatorname{Symb}\otimes\operatorname{id})(\widetilde T)\in
(\operatorname{Pol}(\mathfrak n_+)\,\otimes\, \operatorname{Hom}(V,W))^{\widehat{d\pi_{\lambda^*}}(\mathfrak k^\tau)}$ if and only if
$T\in \mathrm{Hom}_{\mathfrak k^\tau}(W^\vee, S(\mathfrak n_-^{-\tau})\,\otimes\, V^\vee)$,
as $(\operatorname{Symb}\otimes\operatorname{id})(\widetilde T)=(F_c\otimes\operatorname{id})(T)$. Hence the statement is proved.
\end{proof}

\begin{lem}\label{lem:64}
Suppose $(\mathfrak g(\R),\mathfrak g(\R)^\tau)$ is a split rank one irreducible symmetric pair of holomorphic type and $\lambda$ satisfying \eqref{eqn:condimu} and \eqref{eqn:lpos2}.
For $a\in\N$ we define a $K^\tau$-module:
\begin{equation}\label{eqn:wa}
W^a_\lambda:= S^a(\mathfrak n_-^{-\tau})\otimes\C_\lambda.
\end{equation}

\begin{enumerate}
\item The module $W^a_\lambda$ is irreducible for any $a\in\N$.
\item If for an irreducible $K^\tau$-module
$W$ there exists a non-zero continuous $G^\tau$-homomorphism 
$\mathcal{O}(G/K,\mathcal L_\lambda) \to \mathcal O(G^\tau/K^\tau,\mathcal W)$, then the module $W$ is isomorphic
 to $W^a_\lambda$ for some $a\in\N$.

\item
 Assume that
\begin{equation}\label{eqn:homs}
\operatorname{Hom}_{\mathfrak k^\tau}(S^a(\mathfrak n_-^{-\tau}),
S^{a_1}(\mathfrak n_-^\tau)\otimes S^{a-a_1}(\mathfrak n_-^{-\tau}))=\{0\}\quad \mathrm{for\,  any}\, 1\leq a_1\leq a.
\end{equation}
Then,
 the normal derivative $\widetilde T$ corresponding to
 the natural inclusion $T:(W_\lambda^a)^\vee\to S(\mathfrak n_-^{-\tau})\otimes(\C_\lambda)^\vee$  is a 
$G^\tau$-equivariant differential operator.

\end{enumerate}
\end{lem}

\begin{proof}
If $\mathrm{rank}_\R\, G/G^\tau=1$,
then the non-compact part of $\mathfrak{g}(\mathbb{R})^{\tau\theta}$
is isomorphic to $\mathfrak{su}(1,n)$ for some $n$.
Thus the first statement follows from the observation that 
$S^a(\C^n)$ is an irreducible $\mathfrak{gl}_n(\C)$-module for any $a\in\N$ because
the action of $\mathfrak k^\tau$ on $\mathfrak n_+^{-\tau}$ corresponds to the natural action of $\mathfrak{gl}_n(\C)$ on $\C^n$.

The second statement is due to the localness theorem \cite[Theorem \ref{thm:C}]{PART1} for
$k=\mathrm{rank}_\R\, G/G^\tau=1$.

To show the third statement, observe that we have the following natural inclusions $A\subset B\supset C$, where
$$
A:=\operatorname{Pol}^a(\mathfrak n_+^{-\tau})\otimes\C_\lambda^\vee,\,
B:=\operatorname{Pol}^a(\mathfrak n_+)\otimes\C_\lambda^\vee,\, C:=
(\operatorname{Pol}^a(\mathfrak n_+)\otimes\C_\lambda^\vee)^{\widehat{d\pi_{\lambda^*}(\mathfrak n_+^\tau)}}.
$$
Therefore
$$
\operatorname{Hom}_{\mathfrak k^\tau}((W^a_\lambda)^\vee,
A)\hookrightarrow
\operatorname{Hom}_{\mathfrak k^\tau}((W^a_\lambda)^\vee,B)\hookleftarrow\operatorname{Hom}_{\mathfrak k^\tau}((W^a_\lambda)^\vee, C).
$$
By Proposition \ref{prop:Knormal} and Theorem \ref{thm:FGK}, we have
$$
\ndc_{K^\tau}(\mathcal V_X,\mathcal W_Y)\hookrightarrow
\mathrm{Hom}_{\mathfrak k^\tau}((W_\lambda^a)^\vee, B)
\hookleftarrow\mathrm{Hom}_{G'}(\mathcal O( X,\mathcal V),
\mathcal O( Y,\mathcal W)).
$$
Since $\displaystyle\mathrm{Pol}^a(\mathfrak n_+)\simeq\bigoplus_{a_1=0}^a\mathrm{Pol}^{a_1}(\mathfrak n_+^\tau)
\otimes 
\mathrm{Pol}^{a-a_1}(\mathfrak n_+^{-\tau})$, the assumption (\ref{eqn:homs}) implies
that $\mathrm{Hom}_{\mathfrak k^\tau}((W_\lambda^a)^\vee,A)\stackrel
{\sim}{\to}\mathrm{Hom}_{\mathfrak k^\tau}((W_\lambda^a)^\vee,B)$, and therefore the
 first inclusion is an isomorphism.
Moreover, since $A$ is isomorphic to the irreducible
$\mathfrak k^\tau$-module $(W_\lambda^a)^\vee$,
the first term is one-dimensional 
by Schur's lemma.
 The last one is
also one-dimensional according to the multiplicity-one decomposition given in Fact \ref{fact:A}.
Therefore, all the three terms coincide.

Hence the canonical isomorphism $T:(W_\lambda^a)^\vee\to S(\mathfrak n_-^{-\tau})\otimes(\C_\lambda)^\vee$ satisfies the assumption of Proposition
\ref{prop:415}. Thus Lemma follows.
\end{proof}

\begin{rem}
The highest weight vectors of the generalized Verma module
$\indpg(\C_\lambda^\vee)$ with respect to $\mathfrak p^\tau$ have a particularly
simple form if the condition (\ref{eqn:homs}) is satisfied. In fact, by Poincar\'e--Birkhoff--Witt theorem $\indpg(\C_\lambda^\vee)$ is isomorphic, as
a $\mathfrak k$-module, to $S(\mathfrak n_-)\otimes\C_\lambda^\vee$, when $\mathfrak n_-$ is abelian. Under the assumption \eqref{eqn:homs} we thus have
$$
\left(\indpg(\C_\lambda^\vee)\right)^{\mathfrak n_+^\tau}\simeq
\bigoplus_{a=0}^\infty S^a(\mathfrak n_-^{-\tau})\otimes \C_\lambda^\vee.
$$
This formula is an algebraic explanation of the fact that $G^\tau$-equivariant operators are given by normal derivatives in this setting.
\end{rem}

In order to conclude the proof of Theorem \ref{thm:nor_der} we have to show that in all cases mentioned in (iii) the condition (\ref{eqn:homs}) is fulfilled.
It will be done in the next subsection.

\subsection{An application of the classical branching rules}\label{section:53}

In what follows, we shall verify the condition (\ref{eqn:homs}) for the last three cases (4), (5) and (6) in Table \ref{table:1}
 by using some classical branching rules of irreducible representations of
$\mathfrak{gl}_m(\C)$.

Denote by $F(\mathfrak{gl}_m(\C),\mu)$ the finite dimensional irreducible $\mathfrak{gl}_m(\C)$-module
with highest weight $\mu$.
For example, the natural representation of the Lie algebra $\mathfrak{gl}_m(\C)$ on $\C^m$ corresponds to $F(\mathfrak{gl}_m(\C),\,(1,0,\ldots,0))$
and
its contragredient representation on $(\C^m)^\vee$ to $F(\mathfrak{gl}_m(\C),\,(0,0,\ldots,0,-1))$,
 while the action of $\mathfrak{gl}_m(\C)$ on the space of symmetric matrices
$\operatorname{Sym}(m,\C)\simeq S^2(\C^m)$
 given by $C\mapsto XC\,\trans X$ for $X\in\mathfrak{gl}_m(\C)$
and $C\in Sym(m,\C)$ corresponds to $F(\mathfrak{gl}_m(\C),\,(2,0,\ldots,0))$.
More generally, the action of $\mathfrak{gl}_m(\C)$ on the space of $i$-th symmetric tensors is no longer
irreducible and decomposes as follows:
\begin{eqnarray}\label{eqn:Si-deco}
S^i\left(\operatorname{Sym}(m,\C)\right)&\simeq &S^i\left(S^2(\C^m)\right)\nonumber
\\
&\simeq&\bigoplus_{\substack{i_1\geq\cdots\geq i_m\geq0\\
i_1+\dots+i_m=i}}F(\mathfrak{gl}_m(\C),\,(2i_1,2i_2,\ldots,2i_m)).\end{eqnarray}

In turn, classical Pieri's rule gives the following irreducible decomposition for the tensor product
of such modules:
\begin{eqnarray}\label{eqn:Pieri}
S^i\left(S^2(\C^m)\right)\otimes S^k\left(\C^m\right)
\simeq\bigoplus_{\substack{i_1\geq\cdots\geq i_m\geq0,\\
i_1+\dots+i_m=i}}
\bigoplus_{\substack{\ell_1\geq 2i_1\geq\cdots\geq\ell_m\geq2i_m,\\ \sum_{r=1}^m(\ell_r-2i_r)=k}}
F(\mathfrak{gl}_m(\C),\,(\ell_1,\ldots,\ell_m)).\nonumber
\end{eqnarray}

\begin{rem}\label{rem:pieri}
The summand of the form $F(\mathfrak{gl}_m(\C),\,(\ell,0,\ldots,0))$ occurs in the right-hand side if and only if $i_2=\cdots=i_m=0$, hence $i_1=i$
 and $\ell-2i=k$.
This remark will be 
 used in Section \ref{sec:sp}.  
\end{rem}
\vskip10pt

\begin{example}
Let $G=U(p,q)$, $G^\tau=U(1)\times U(p-1,q)$ and
$\mathfrak k^\tau=\mathfrak k^\tau(\R)\otimes_\R\C\simeq \mathfrak{gl}_1(\C) \oplus\mathfrak{gl}_{p-1}(\C)\oplus\mathfrak{gl}_q(\C)$. Then, the decomposition
$
\mathfrak n_-=\mathfrak n_-^\tau\oplus\mathfrak n_-^{-\tau}
$ as a $\mathfrak k^\tau$-module
amounts to
$$
(\C^p)^\vee\boxtimes\C^{q}
\simeq(\C\boxtimes(\C^{p-1})^\vee\boxtimes\C^q)\oplus(\C_{-1}\boxtimes
\C\boxtimes\C^q),
$$
where $\boxtimes$ stands for the outer tensor product representation.
Therefore, for $a=a_1+a_2$, 
\begin{eqnarray*}
&&\operatorname{Hom}_{\mathfrak k^\tau}
(S^a(\mathfrak n_-^{-\tau}), S^{a_1}(\mathfrak n_-^\tau)\otimes S^{a_2}(\mathfrak n_-^{-\tau}))\\
&\simeq&
\operatorname{Hom}_{\mathfrak{gl}_1(\C)}(\C_{-a},\C_{-a_2})
\otimes
\operatorname{Hom}_{\mathfrak{gl}_{p-1}(\C)}(\C,S^{a_1}((\C^{p-1})^\vee))
\otimes 
\operatorname{Hom}_{\mathfrak{gl}_{q}(\C)}(S^{a}(\C^{q}),S^{a_2}(\C^q))
\end{eqnarray*}
is not reduced to zero if and only if $a_1=0$ and $a_2=a$.
Thus, the condition \eqref{eqn:homs} is satisfied.
\end{example}

\begin{example}\label{example:SO-U} 
Let $G=SO(2,2n)$, $ G^\tau=U(1,n)$ and $\mathfrak k^\tau=\mathfrak{gl}_1(\C)\oplus\mathfrak{gl}_n(\C).$
Then the decomposition
$
\mathfrak n_-= \mathfrak n_-^\tau\oplus\mathfrak n_-^{-\tau}
$ as a $\mathfrak k^\tau$-module
amounts to
$$
\C_{-1}\boxtimes\C^{2n}\simeq(\C_{-1}\boxtimes\C^n)\oplus(\C_{-1}\boxtimes
(\C^n)^\vee).
$$
Therefore, for $a=a_1+a_2$, we have
\begin{eqnarray*}
&&\operatorname{Hom}_{\mathfrak k^\tau}
(S^a(\mathfrak n_-^{-\tau}), S^{a_1}(\mathfrak n_-^\tau)\otimes S^{a_2}(\mathfrak n_-^{-\tau}))\\
&\simeq&
\operatorname{Hom}_{\mathfrak{gl}_1(\C)}(\C_{-a},\C_{-a_1-a_2})
\otimes
\operatorname{Hom}_{\mathfrak{gl}_{n}(\C)}(S^{a}((\C^{n})^\vee),S^{a_1}(\C^{n})\otimes S^{a_2}((\C^n)^\vee))\\
&\simeq&\bigoplus_{b=0}^{\min(a_1,a_2)}
\operatorname{Hom}_{\mathfrak{gl}_{n}(\C)}(F(\mathfrak{gl}_n(\C),(0,\cdots,0,-a)),
F(\mathfrak{gl}_n(\C),(a_1-b,0,\cdots,0,-a_2+b))),
\end{eqnarray*}
where the second isomorphism follows from Pieri's rule.
Thus, the left-hand side is not reduced to zero if and only if $a_1=0$ and $a_2=a$. Hence, the condition \eqref{eqn:homs} is satisfied.
\end{example}

\begin{example}
Let $G=SO^*(2n)$, $G^\tau=SO^*(2n-2)\times SO(2)$ and $\mathfrak k^\tau=\mathfrak{gl}_{n-1}(\C)\oplus \mathfrak{gl}_1(\C)$.
In this case, the decomposition
$
\mathfrak n_-= \mathfrak n_-^\tau\oplus\mathfrak n_-^{-\tau}
$ as a $\mathfrak k^\tau$-module
amounts to
$$
(\operatorname{Alt}
({\mathbb{C}}^{n-1})^{\vee}
\boxtimes {\bf{1}}
)
\oplus 
(({\mathbb{C}}^{n-1})^{\vee}\boxtimes{\mathbb{C}}_{-1}).  
$$
Therefore, for $a=a_1+a_2$
\begin{eqnarray*}
&&\operatorname{Hom}_{\mathfrak k^\tau}
(S^a(\mathfrak n_-^{-\tau}), S^{a_1}(\mathfrak n_-^\tau)
\otimes S^{a_2}(\mathfrak n_-^{-\tau}))\\
&\simeq&
\operatorname{Hom}_{\mathfrak{gl}_{n-1}(\C)}(S^{a}((\C^{n-1})^\vee),
S^{a_1}(\operatorname{Alt}({\mathbb {C}}^{n-1})^{\vee})\otimes S^{a_2}((\C^{n-1})^\vee))\otimes\operatorname{Hom}_{\mathfrak{gl}_{1}(\C)}(\C_{-a},\C_{-a_2}).
\end{eqnarray*}

In view of the
$\mathfrak{gl}_1({\mathbb C})$-action on the right-hand side, it is non-zero only if $a_2=a$
(and therefore $a_1=0$). Thus the condition \eqref{eqn:homs} is satisfied.  
\end{example}
Hence we have verified the assumption \eqref{eqn:homs}
 for all the three symmetric pairs $(\mathfrak g(\R),\mathfrak g(\R)^\tau)$ corresponding to the three complex geometries
 (4), (5) and (6) in Table \ref{table:intro}, and have proved the implication $(\mathrm{iii})\Rightarrow(\mathrm i)$ in Theorem \ref{thm:nor_der} by Lemma \ref{lem:64} (3).

\section[Symmetry breaking for the restriction $SO(n,2)\downarrow SO(n-1,2)$]{Symmetry breaking operators for the restriction $SO(n,2)\downarrow SO(n-1,2)$}\label{sec:son2}

Let $n\ge3$.  
In what follows,
 we realize the indefinite orthogonal group
 $SO(n,2)$ 
 in a slightly non-standard way,
 namely,
 use
 a non-degenerate quadratic form on $\C^{n+2}$ defined by 
\[
\widetilde Q(w):=w_0^2+\cdots+w_n^2-w_{n+1}^2
 \quad
 \text{for $w =(w_0, \cdots, w_{n+1}) \in {\mathbb{C}}^{n+2}$,}
\] 
 and restrict it to a certain real form $E(\R)$ (see \eqref{eqn:ER} below) of ${\mathbb{C}}^{n+2}$. (The restriction
 to the standard real form $\R^{n+2}$ yields conformally covariant differential operators
corresponding to another pair of real forms
 $(SO(n+1,1),SO(n,1))$, see Remark \ref{rem:512}.)

Let $G_\C$ be the complex special orthogonal group
$SO(\C^{n+2},\widetilde Q)$ with respect to the quadratic form $\widetilde Q$.  
Then $G_{\mathbb{C}}$ acts
transitively on the isotropic cone
$$
\Xi_\C:=\{w\in\C^{n+2}\setminus\{0\}:\,\widetilde Q(w)=0\}, 
$$
and also  on the complex quadric 
$$
\mathrm Q^n\C:=\Xi_\C/\C^*\subset \mathbb P^{n+1}\C 
$$
by
$w\to g\cdot[w]:=[gw]$ for $w\in\C^{n+1}\setminus\{0\}$.
Let $w_o=\trans(1,0,\cdots,0,1) \in \Xi_\C$, and $P_{\mathbb{C}}$ be the stabilizer of the base point
$[w_o]=[1:0:\cdots:0:1] \in Q^n\C$, 
 which is a maximal parabolic subgroup
 of $G_{\mathbb{C}}$. Then
 we have an isomorphism 
 $Q^n{\mathbb{C}} \simeq G_{\mathbb{C}}/P_{\mathbb{C}}$.  
We define an embedding
\begin{equation}\label{eqn:CXi}
\iota:\C^n\to\Xi_\C,\quad z\mapsto \trans(1-Q_n(z),2z_1,\cdots,2z_n,1+Q_n(z)),
\end{equation}
where $Q_n(z):=\sum_{j=1}^n z_j^2$
 for $z=(z_1,\cdots,z_n)\in\C^n$.
Then we get coordinates on $Q^n\C$ by
\begin{equation}\label{eqn:bruhat}
\C^n\hookrightarrow\mathrm Q^ n\C,
\quad 
z\mapsto
 [\iota(z)]
\end{equation}
which define the open Bruhat cell (see \eqref{eqn:SOnGP} below).

The quadratic form $\widetilde Q$ is of signature $(n,2)$ when restricted to the
real vector space 
\begin{equation}\label{eqn:ER}
E(\R):= \sqrt{-1}\R e_0+\sum_{j=1}^{n+1}\R e_j,
\end{equation}
 where 
$\{e_j: 0\leq j\leq n+1\}$ is
 the standard basis in $\C^{n+2}$. 
Thus we have an isomorphism:
$$
SO(\C^{n+2}, \widetilde Q)\cap GL_\R(E(\R))\simeq SO(n,2).
$$
Let $G$ be its identity component $SO_o(n,2)$. 
Then the $G$-orbit through the base point $[w_o]$ in
$Q^n\C$ is still contained in $\C^n$, and is identified with the Lie ball 
$\displaystyle X:=\{z\in \C^n:\vert z\:\trans z\vert^2+1
-2\overline z \, \trans z>0,\, \vert z\:\trans z\vert<1\}\simeq G/K$ which is
the bounded Hermitian symmetric domain of type IV in the \'E. Cartan classification.

 Let $\tau$ be the involution of $GL(n+1,\C)$ by conjugation by 
 $\mathrm{diag}(1,\ldots,1,-1,1)$.
 It leaves $G$ invariant, and we denote by $G'$ the identity
component of the fixed point group $G^\tau$. The group $G'=SO_o(n-1,2)$ acts on the subsymmetric domain
$$
Y:= X\cap\{z_n=0\}.
$$
Then $Y\simeq G'/K'= SO_o(n-1,2)/SO(n-1)\times SO(2)$
 a subsymmetric space of $X$ of complex codimension one.

\vskip10pt 

We take $H_o:=E_{0,n+1}+E_{n+1,0}$.Then $H_o$ is a characteristic element as in Section \ref{sec:subHerm}. For $\lambda\in\Z$ we define a character
of $\mathfrak c(\mathfrak k)$ by $tH_o\mapsto\lambda t$, and lift it to a character $\C_\lambda$ of $K$.
Let $\mathcal L_\lambda$ be the $G$-equivariant holomorphic line bundle $G\times_K\C_\lambda$. 
The holomorphic line bundle $\mathcal L_\lambda\to X$ is trivialized by using the open Bruhat cell, and the representation
of $G$ on $\mathcal O(X,\mathcal L_\lambda)$ is identified with the multiplier representation
$\pi_\lambda\equiv\pi_\lambda^G$ of the same group on $\mathcal O(X)$
given by
\begin{equation}\label{eqn:Jlmd}
F(z)\mapsto (\pi_\lambda(g)F)(z)=J(g^{-1},z)^{-\lambda} F(g^{-1}\cdot z),
\end{equation}
where we define a map $J:G\times X\to\C^*$ by
$$
J(g,z):=\frac12\trans w_og\iota(z),\quad\mathrm{for}\,\, g\in G\,\,\mathrm{and}\,\,z\in X.
$$
Since $H_o\in\mathfrak k'$ (see \eqref{eqn:ckk}), we can also define a $G'$-equivariant holomorphic line bundle $\mathcal L_\nu=
G'\times_{K'}\C_\nu$ over $Y=G'/K'$ for $\nu\in\Z$.

Let $\widetilde G$ be the universal covering group of $G=SO_o(n,2)$. Then for any $\lambda\in\C$ one can define a 
$\widetilde G$-equivariant holomorphic line bundle $\mathcal L_\lambda=\widetilde G\times_{\widetilde K}\C_\lambda$ over
$X=G/K\simeq \widetilde G/\widetilde K$, and a representation of the same group on $\mathcal O(X,\mathcal L_\lambda)$.
Similarly, for $\nu\in\C$, the universal covering group
$\widetilde G'$ of $G'=SO_o(n-1,2)$ acts on $\mathcal O(Y,\mathcal L_\nu)$.

Here is a complete classification of symmetry breaking operators from $\mathcal O(X,\mathcal L_\lambda)$ to
$\mathcal O(Y,\mathcal L_\nu)$ with respect to the symmetric pair $\widetilde G\supset \widetilde G'$:

\begin{thm}\label{thm:SO(2,n)} Let $n\geq 3$ and $\widetilde{G'}$ be the universal covering group of $SO_o(n-1,2)$.
Suppose $\lambda,\nu\in\mathbb{C}$. Then the
 following three conditions on the parameters $(\lambda,\nu)\in\C^2$ are equivalent:
\begin{enumerate}
\item[(i)] $\mathrm{Hom}_{\widetilde{G'}} (\mathcal O(X,\mathcal L_\lambda),\mathcal O(Y,\mathcal L_{\nu}))\neq\{0\}.$
\item[(ii)] $\dim_\C\mathrm{Hom}_{\widetilde{G'}} (\mathcal O(X,\mathcal L_\lambda),\mathcal O(Y,\mathcal L_{\nu}))=1$.
\item[(iii)] $\nu-\lambda\in\N$.
\end{enumerate}
\end{thm}

\begin{rem}
The equivalence (i)$\Leftrightarrow$(ii) in Theorem \ref{thm:SO(2,n)} is not true for singular parameters $(\lambda,\nu)$ in the case of $n=2$.
This situation will be treated
carefully in Section \ref{sec:8}. In fact,
the symmetric pair
$(SO_o(2,2), SO_o(2,1))$
is locally isomorphic to the pair
$(SL(2,\R)\times SL(2,\R), \Delta(SL(2,\R))$ modulo the center. We note that $n=2$ in Theorem \ref{thm:SO(2,n)} corresponds to $\lambda'=\lambda''$ in
Theorem \ref{thm:Hlambdas}.
\end{rem}

Let $\widetilde C_\ell^\alpha(x)$ be the renormalized Gegenbauer polynomial  
 (see Appendix \ref{sec:A2}).
We inflate it to a polynomial of two variables $x$ and $y$: 
\begin{eqnarray}\label{eqn:Cxy}
\widetilde C_\ell^\alpha(x,y)&:=&x^{\frac\ell2} \widetilde C_\ell^\alpha\left(\frac y{\sqrt x}\right)\\
&=&\nonumber
\sum_{k=0}^{\left[\frac \ell2\right]}(-1)^k\frac{\Gamma(\ell-k+\alpha)}
{\Gamma\left(\alpha+\left[\frac{\ell+1}2\right]\right)\Gamma(k+1)\Gamma(\ell-2k+1)}(2y)^{\ell-2k}x^{k}.
\end{eqnarray}
For instance, $\widetilde C_0^\alpha(x,y)=1$, $\widetilde C_1^\alpha(x,y)=2 y$,
$\widetilde C_2^\alpha(x,y)=2(\alpha+1)y^2-x$, etc.
Notice that $\widetilde C_\ell^\alpha(x^2,y)$ is a homogeneous polynomial of $x$ and $y$ of degree $\ell$.\vskip10pt

\begin{thm}\label{thm:63} Retain the setting of Theorem \ref{thm:SO(2,n)}. Let
 $a:=\nu-\lambda\in\N$. Then 
the differential operator from
$\mathcal O(X)$ to $\mathcal O(Y)$
defined by
\begin{equation}\label{eqn:thm66}
D_{X\to Y,a}:=\widetilde C_a^{\lambda-\frac{n-1}2}\left(-\Delta_{\C^{n-1}}^z,\frac{\partial}{\partial z_n}\right)
\end{equation}
intertwines the restriction $\pi_\lambda^{\widetilde G}\Big\vert_{\widetilde {G'}}$ with $\pi_{\lambda+a}^{\widetilde {G'}}$ (see \eqref{eqn:Jlmd}).
Here $\Delta_{\C^{m}}^z:=\sum_{k=1}^m\frac{\partial^2}{\partial z_k^2}$ denotes the holomorphic Laplacian on $\C^m$ in the coordinates $(z_1,\cdots,z_m)$.
\end{thm}

It follows from Theorems \ref{thm:SO(2,n)} and \ref{thm:63} that
any symmetry breaking operator from
$\mathcal O(X,\mathcal L_\lambda)$ to $\mathcal O(Y,\mathcal L_{\lambda+a})$
 is proportional to $
D_{X\to Y,a}$ for any $\lambda\in\C$ and $a\in\N$.

\begin{rem}
If $\lambda\in\R$ and $\lambda>n-1$, then $\mathcal H^2(X,\mathcal L_\lambda):=
\mathcal O(X,\mathcal L_\lambda)\cap L^2(X,\mathcal L_\lambda)$
is a non-zero Hilbert space on which $\widetilde G$ acts unitarily and irreducibly, giving a holomorphic discrete series representation of $\widetilde G$ modulo the center.
 By \cite[Theorem \ref{thm:CC}]{PART1}
the same statement 
as Theorems \ref{thm:SO(2,n)} and \ref{thm:63}
remains true for symmetry breaking operators
between the unitary representations $\mathcal H^2(X,\mathcal L_\lambda)$ and
$\mathcal H^2(Y,\mathcal L_{\lambda+a})$.

\end{rem}

In order to prove Theorems \ref{thm:SO(2,n)} and \ref{thm:63} we apply the F-method (see Section \ref{sec:241}).
The Lie algebra
$\mathfrak g=\mathfrak{so}(\C^{n+2},\widetilde Q)$ has a direct sum decomposition
$$
\mathfrak g=\mathfrak n_-+\mathfrak k+\mathfrak n_+
$$
of $-1,0,$ and $1$ eigenspaces of $\mathrm{ad}(H_o)$, respectively. 
Then the maximal parabolic subgroup $P_\C$ has
 a Levi decomposition
$P_\C=K_\C N_{+,\C}$, 
where $N_{+,\C}=\exp\mathfrak n_+$. 

As Step 1 of the F-method we define a standard basis of $\mathfrak n_+\simeq \C^n$ by
\begin{eqnarray*}
C_j&:=&E_{j,0}-E_{j,n+1}-E_{0,j}-E_{n+1,j}\quad (1\leq j\leq n), 
\end{eqnarray*}
and similarly a standard basis of $\mathfrak n_-\simeq \C^n$
by
\begin{eqnarray*}
\overline C_j&:=&E_{j,0}+E_{j,n+1}-E_{0,j}+E_{n+1,j}
\quad (1\leq j\leq n).
\end{eqnarray*}

Then the decomposition $\mathfrak n_+=\mathfrak n_+^\tau\oplus\mathfrak n_+{-\tau}$ is given by
$$
\mathfrak n_+=\sum_{j=1}^{n-1}\C C_j\oplus\C C_n.
$$

Let $Z=\sum_{i=1}^n z_i\overline C_i\in\mathfrak n_-$ and
$Y=\sum_{j=1}^n y_j C_j\in\mathfrak n_+$. 
By a simple computation we have
\begin{equation}\label{eqn:SOnGP}
\exp(Z)\cdot w_o=\iota(z)\in\C^{n+2},
\end{equation}
the open Bruhat cell is given by \eqref{eqn:bruhat}. Moreover,
by using
$$
\exp(tY)\exp(Z)w_o=\iota(z)-2t
\begin{pmatrix}
(y,z)\\Q(z)y\\(y,z)
\end{pmatrix}
+o(t),
$$
we obtain formul\ae\, of the maps
\eqref{eqn:alpha2} and \eqref{eqn:beta2}, as\begin{eqnarray*} 
\alpha(Y,Z)&=&-2(z,y) H_o \quad\mathrm{mod}\,\mathfrak {so}(n,\C);\\
\beta(Y,Z)&=&2(z,y)E_z-Q_n(z)\sum_{j=1}^n y_j\frac\partial{\partial z_j},
\end{eqnarray*}
where we regard $\beta(Y,\cdot)$ as a holomorphic vector field on $\mathfrak n_-$ and
recall that $E_z:=\sum_{j=1}^{n}z_j\frac\partial{\partial z_j}$, $\displaystyle Q_n(z)=z_1^2+\cdots+z_n^2$
 and $(z,y)=z_1y_1+\cdots+ z_ny_n.$

Then the infinitesimal action $d\pi_{\lambda^*}({C_j})$ with 
$$
{\lambda^*}=\lambda^\vee\otimes\C_{2\rho}=-\lambda+n,
$$
  is given by
\begin{equation}\label{eqn:dpic22}
d\pi_{\lambda^*}(C_j)={2(\lambda-n)} z_j-
2z_jE_z+Q_n(z)\frac\partial{\partial z_j}.
\end{equation}

\begin{lem}\label{lem:5.3} For $C\in\C^n\simeq\mathfrak n_+$ and $\zeta\in\C^n\simeq\mathfrak n_-$ one has,
\begin{eqnarray*}
\widehat{d\pi_{\lambda^*}}(C_j)&=&2\lambda \frac\partial{\partial \zeta_j}+2E_\zeta
\frac\partial{\partial \zeta_j}-\zeta_j\Delta^\zeta_{\C^n},
\qquad 1\leq j\leq n,
\end{eqnarray*}
where $E_\zeta:=\sum_{i=1}^{n}\zeta_i\frac\partial{\partial \zeta_i}$ and $\Delta^\zeta_{\C^{n}}=\frac{\partial^2}{\partial \zeta_1^2}+\cdots
\frac{\partial^2}{\partial \zeta_{n}^2}$. 
\end{lem}
\begin{proof}
According to Definition \ref{eqn:algfour} we have $\widehat z_j=\frac\partial{\partial\zeta_j}$ and
hence $\widehat E_z=-E_\zeta-n$. On the other hand, using the commutation relations
of the Weyl algebra (see \emph{e.g.} \cite[\eqref{eqn:commrel}]{PART1}) we get 
$$
\Delta_{\C^n}^\zeta\zeta_j=\zeta_j\Delta_{\C^n}^\zeta+2\frac\partial{\partial\zeta_j},\qquad
\frac\partial{\partial\zeta_j}E_\zeta=E_\zeta\frac\partial{\partial\zeta_j}+\frac\partial{\partial\zeta_j}.
$$
Thus the above formula for the algebraic Fourier transform $\widehat{d\pi_{\lambda^*}}(C_j)$
of the differential operator
 \eqref{eqn:dpic22} follows.
\end{proof}

\

For Step 2 we apply Lemma \ref{lem:64} (2) and get the following.
\begin{prop}
Assume $\lambda>n-1$. If
$$
\mathrm{Hom}_{G'}(\mathcal O(G/K,\mathcal L_\lambda), \mathcal O(G'/K',\mathcal W))\neq\{0\}
$$
for an irreducible representation $W$ of $K'$, then $W$ must be one-dimensional and of the form
\begin{equation}\label{eqn:Wmuso}
W_\lambda^a:=S^a(\mathfrak n_-^{-\tau})\otimes\C_\lambda\simeq \mathrm{Pol}^a(\mathfrak n_+^{-\tau})\otimes\C_\lambda
\end{equation}
for some $a\in\N$.
\end{prop}

We denote by $\nu$ the action of $K'$ on $W_\lambda^a$.
In our setting where $\dim V=\dim W_\lambda^a=1$
we write $\zeta=(\zeta',\zeta_n)\in\C^n$ with $\zeta'=(\zeta_1,\ldots,\zeta_{n-1})\in\C^{n-1}$,
and identify  an element of $\mathrm{Hom}_\C(\C_\lambda,\mathrm{Pol}(\mathfrak n_+)\otimes W_\lambda^a)$
with a polynomial $\psi(\zeta)$ of $n$ variables. Then, for Step 3, the condition \eqref{eqn:Fmethod1part2} 
implies that $\psi(\zeta)$ is homogeneous of degree $a$ and the condition \eqref{eqn:Fmethod2part2} amounts to
the system of differential equations:
$$
\widehat{d\pi_{\lambda^*}}(C_j)\psi=
\left(2\lambda \frac\partial{\partial \zeta_j}+2E_\zeta
\frac\partial{\partial \zeta_j}-\zeta_j\Delta^\zeta_{\C^n}\right)\psi=0,
\qquad 1\leq j\leq n-1
$$
by Lemma \ref{lem:5.3}.

To be prepared for Step 4, observe that the $K'_\C$-action on $\mathfrak n_-=\mathfrak n_-^\tau\oplus\mathfrak n_-^{-\tau}$ is identified with the action of 
$SO(n-1,\C)\times SO(2,\C)$ on $\C^n$ given as 
$$
\C^{n} \boxtimes\C_{-1}\simeq(\C^{n-1}\boxtimes\C_{-1})\oplus(\C \boxtimes\C_{-1}
).
$$
Then generic $K'_\C$-orbits are of codimension one in $\mathfrak n_-$, and the $K_\C'$-orbit space in $\{\zeta\in\C^n: Q_{n-1}(\zeta')\neq0\}$ has coordinates $\displaystyle \frac{\zeta_n^2}{Q_{n-1}(\zeta')}$.

For $a\in\N$, we introduce an operator $T_a$ by
\begin{equation}\label{eqn:Dmodmorph2}
\left( T_ag\right)(\zeta):=Q_{n-1}(\zeta')^{\frac a2}g\left(\frac{\zeta_n}{\sqrt{Q_{n-1}(\zeta')}}\right),
\end{equation}
for
$g\in\C[t]$. We note that $T_ag$ is a (multi-valued) meromorphic function of $\zeta_1,\ldots,\zeta_n$.

We set
\begin{eqnarray}
\operatorname{Pol}_a[t]&:=&\C\operatorname{-span}\left\langle t^{a-i}:0\leq i\leq a\right\rangle,
\label{eqn:pola}\\
\label{eqn:gs}
\operatorname{Pol}_a[t]_{\mathrm{even}}&:=&\C\operatorname{-span}\left\langle t^{a-2j}:0\leq j\leq\left[\frac a2\right]
\right\rangle.
\end{eqnarray}
Then $\left( T_ag\right)(\zeta)$ is a homogeneous polynomial of degree $a$ if $g\in\operatorname{Pol}_a[t]_{\mathrm{even}}$.
\begin{rem}
In this section we have assumed $n\geq3$, and therefore $Q_{n-1}(\zeta')^{\frac12}=(\zeta_1^2+\cdots+\zeta_{n-1}^2)^{\frac12}$ is not a polynomial and the parity condition in \eqref{eqn:gs} is necessary.
However, for
 $n=2$, $T_ag$ is a polynomial for $g\in \operatorname{Pol}_a[t]$ as we can take a branch as $Q_1(\zeta')^{\frac12}=\zeta_1$.
\end{rem}

The first half of Step 4 is summarized in the following lemma:
\begin{lem}\label{lem:F4SO} For $n\geq 3$ we have,
$$
\operatorname{Hom}_{\mathfrak k'}
(\C_\lambda,\operatorname{Pol}(\mathfrak n_+)\otimes \C_\nu)\simeq
\left\{
\begin{array}{lll}
\{0\}& \mathrm{if}& \nu-\lambda\not\in\N,\\
T_{\nu-\lambda}(\operatorname{Pol}_{\nu-\lambda}[t]_{\mathrm{even}})& \mathrm{if}& \nu-\lambda\in\N.
\end{array}
\right.
$$
\end{lem}
\begin{proof}
As modules of $\mathfrak k'=\mathfrak{so}(n-1,\C)\oplus\mathfrak{so}(2,\C)$, we have the following 
isomorphisms:
$$
\mathrm{Pol}(\mathfrak n_+)\simeq S(\mathfrak n_-)
\simeq 
\bigoplus_{a_1,a_2\in\N} S^{a_1}(\mathfrak n_-^\tau)
\otimes S^{a_2}(\mathfrak n_-^{-\tau})\simeq
\bigoplus_{a=0}^\infty
\bigoplus_{a_1=0}^a S^{a_1}(\C^{n-1}) \boxtimes \C_{-a}.
$$
Therefore
$$
\operatorname{Hom}_{\mathfrak k'}(\C_\lambda, \operatorname{Pol}
(\mathfrak n_+)\otimes\C_\nu)\simeq 
\bigoplus_{a=0}^\infty
\bigoplus_{a_1=0}^a \left(S^{a_1}(\C^{n-1}) \right)^{SO(n-1,\C)}\boxtimes \left(\C_{\nu-a-\lambda}\right)^{SO(2,\C)}.
$$
The right-hand side is non-zero only when $\nu-\lambda\in\N$. In this case the summand is non-trivial only when
$a=\nu-\lambda$. On the other hand, since $n\geq 3$, we have
$$
S^{a_1}(\C^{n-1})^{SO(n-1,\C)}\simeq\left\{
\begin{matrix}
\C Q_{n-1}(\zeta')^{\frac{a_1}2} & \mathrm{if}& a_1\in2\N,\\
0&\mathrm{if}& a_1\not\in2\N.
\end{matrix}
\right.
$$
Hence the lemma follows.
\end{proof}

To implement the second part of Step 4 we apply Proposition \ref{prop:Tsat} to the map \eqref{eqn:Dmodmorph2}. For this
we collect some formul\ae\, for saturated differential operators that we shall use later.
\begin{lem}\label{lem:Ta}
For every $0\leq j\leq n-1$ one has:
\begin{eqnarray}
T_a^\sharp\left(\zeta_j E_{\zeta'}-Q_{n-1}(\zeta')\frac{\partial}{\partial \zeta_j}\right)&=&0,\label{eqn:TaA}\\
T_a^\sharp\left((a-1)\zeta_n-E_\zeta\frac\partial{\partial\zeta_j}\right)&=&0.
\label{eqn:TaB}
\end{eqnarray}
\end{lem}
\begin{proof}
The proof of both statements is straightforward from the definition of $T_a$.
\end{proof}
\begin{lem}\label{lem:T}
Let $T_a$ be the operator defined in \eqref{eqn:Dmodmorph2}. We write $\zeta'=(\zeta_1,\cdots,\zeta_{n-1})$ and
$\vartheta_t:=t\frac d{dt}$. One then has:
\begin{enumerate}
\item $T_a^\sharp(E_{\zeta'})=  a-\vartheta_t$.
\item $T_a^\sharp\left
(\frac{Q_{n-1}(\zeta')}{\zeta_j}\frac\partial{\partial \zeta_j}\right)
=a-\vartheta_t$, $(1\leq j\leq n-1)$.
\item $T_a^\sharp\left
(\frac{Q_{n-1}(\zeta')}{\zeta_j}E_ \zeta\frac\partial{\partial \zeta_j}\right)
=(a-1)(a-\vartheta_t)$, $(1\leq j\leq n-1)$.
\item $T_a^\sharp({\zeta_n^2\Delta^\zeta_{\C^{n-1}}})=
t^2(\vartheta_t-a)(\vartheta_t-n-a+3)$.
\item $T_a^\sharp({Q_{n-1}(\zeta')\Delta^\zeta_{\C^{n-1}}})=
(\vartheta_t-a)(\vartheta_t-n-a+3)$.
\item $T_a^\sharp({Q_{n-1}(\zeta')\frac{\partial^2}{\partial \zeta_n^2}})=
 t^{-2}(\vartheta_t^2-\vartheta_t)$.
\item $T_a^\sharp({\zeta_n\frac\partial{\partial \zeta_n}})=\vartheta_t$.
\item $T_a^\sharp({\zeta_n^2\frac{\partial^2}{\partial \zeta_n^2}})=
\vartheta_t^2-\vartheta_t$.
\end{enumerate}
\end{lem}
\begin{proof}
Notice first that the identity (1) is equivalent to (2) according to (\ref{eqn:TaA}) and that the identity (3) may be deduced from (1) or (2) by
(\ref{eqn:TaB}). 
Furthermore,
 identities (4) and (5) on the one hand and
(6) and (8) on the other are equivalent according to the definition of the $T$-saturation as $t=\frac{\zeta_n}{\sqrt{Q_{n-1}(\zeta')}}$.

Thus, it would be enough to show the identities (1), (4), (7) and (8).  
We give a proof for the first statement, and the remaining cases can be treated
in a similar way.
Let $1\leq j\leq n-1$. Then
\begin{eqnarray*}
&&\left(T_a^\sharp(E_{\zeta'})g\right)(t)=\sum_{j=1}^{n-1} \zeta_j\frac{\partial}{\partial \zeta_j}\left(Q_{n-1}(\zeta')^{\frac a2}
g\left(\frac{\zeta_n}{\sqrt{Q_{n-1}(\zeta')}}\right)\right)\\
&=&
aQ_{n-1}(\zeta')^{\frac a2-1}g\left(\frac{\zeta_n}{\sqrt{Q_{n-1}(\zeta')}}\right)\sum_{j=1}^{n-1}  \zeta_j^2
-Q_{n-1}(\zeta')^{\frac a2}g'\left(\frac{\zeta_n}{\sqrt{Q_{n-1}(\zeta')}}\right)\sum_{j=1}^{n-1} 
\frac{\zeta_j^2 \zeta_n}{\sqrt{Q_{n-1}^3(\zeta')}}\\
&=&aQ_{n-1}(\zeta')^{\frac a2}g\left(\frac{\zeta_n}{\sqrt{Q_{n-1}(\zeta')}}\right)-
\frac{\zeta_n}{\sqrt{Q_{n-1}(\zeta')}}Q_{n-1}(\zeta')^{\frac a2}g'\left(\frac{\zeta_n}{\sqrt{Q_{n-1}(\zeta')}}\right)
\\
&=&\left(a -t\frac{d}{dt}\right)
g(t).
\end{eqnarray*}
\end{proof}

For the second half of Step 4 we apply the idea of $T$-saturated differential operators (see Definition \ref{def:Tsat}).
 Although the differential operator
$\widehat{d\pi_{\lambda^*}}(C_j)$ itself is not $T_a$-saturated, we shall see that  $Q_j\widehat{d\pi_{\lambda^*}}(C_j)$
is $T_a$-saturated if we set $ Q_j=\zeta_j^{-1}Q_{n-1}(\zeta')$.
 In the following lemma, we note that the right-hand side is independent of $j$.
 
 \begin{lem}\label{lem:57}
 The $T_a$-saturation of the differential operators $\widehat{d\pi_{\lambda^*}}(C_j)$ with $C_j\in\mathfrak n_+^\tau$
 is given for any $1\leq j\leq n-1$ by
 $$
 T_a^\sharp\left(\frac{Q_{n-1}(\zeta')}{\zeta_j}\widehat{d\pi_{\lambda^*}}(C_j)\right)=\frac{-1}{t^2}\left(
 (1+t^2)\vartheta_t^2-(1-(2\lambda-n+1)t^2)\vartheta_t
-a(a+2\lambda-n+1)t^2\right).
$$
 \end{lem}
\begin{proof}
Suppose $1\leq j\leq n-1$. Applying (2), (3) and (5), (6) of Lemma \ref{lem:T}, respectively, we have following identities:
\begin{eqnarray*}
T_a^\sharp\left(\frac{ Q_{n-1}(\zeta')}{\zeta_j}\frac{\partial}{\partial \zeta_j}\right)
&=&a-\theta_t,\\
T_a^\sharp\left(\frac{ Q_{n-1}(\zeta')}{\zeta_j}E_\zeta\frac{\partial}{\partial \zeta_j}\right)&=&
 (a-1)(a-\vartheta_t),\\
T_a^\sharp\left(\frac{ Q_{n-1}(\zeta')}{\zeta_j} \zeta_j\Delta^\zeta_{\C^{n}}\right)
&=& T_a^\sharp\left(Q_{n-1}(\zeta')\left(\Delta^\zeta_{\C^{n-1}}+
\frac{\partial^2}{\partial \zeta_n^2}\right)\right)\\
&=&(\vartheta_t-a)(\vartheta_t-n+3-a)+ t^{-2}(\vartheta_t^2-\vartheta_t).
\end{eqnarray*}
We recall from Lemma \ref{lem:5.3} that 
$\widehat{d\pi_{\lambda^*}}(C_j)=2\lambda \frac\partial{\partial \zeta_j}+2E_\zeta
\frac\partial{\partial \zeta_j}-\zeta_j\Delta^\zeta_{\C^n}$. Summing up these terms we get the lemma.
\end{proof}

\begin{prop}\label{prop:55}
Let $a\in\N$, and
$T_a$ be as in \eqref{eqn:Dmodmorph2}. 
 The polynomial $\psi(\zeta)=(T_a g)(\zeta)$ of $n$ variables
 satisfies the system of partial differential equations
 \eqref{eqn:Fmethod2part2}
if and only if $g(t)$ satisfies
 the following single ordinary differential equation:
\begin{equation}\label{eqn:JacobiDE3}
\left((1-s^2)\vartheta_s^2-(1+(2\lambda-n+1)s^2)\vartheta_s
+a(a+2\lambda-n+1)s^2\right)g(-\sqrt{-1}s)=0,
\end{equation}
or equivalently, $g(t)$ is proportional to the normalized Gegenbauer polynomial\\
$\widetilde C_a^{\lambda-\frac{n-1}2}(\sqrt{-1}t)$. (For the Gegenbauer polynomial, see Section \ref{sec:A2}.)
\end{prop}
\begin{proof}
The statement follows from Lemma \ref{lem:57} after the change of variable $t=-\sqrt{-1}s$.
\end{proof}

We have carried out the crucial part 
 of the F-method.  
Let us complete the proof of Theorems \ref{thm:SO(2,n)} and \ref{thm:63}.  
\begin{proof} 
[Proof of Theorems \ref{thm:SO(2,n)} and \ref{thm:63}]
 By the general result of the F-method (see Theorem \ref{thm:FGK}), the symbol map of differential operators gives
 an isomorphism
$$
\operatorname{Hom}_{\widetilde G'}(\mathcal O(X,\mathcal L_\lambda), \mathcal O(Y,\mathcal L_\nu))\stackrel{\mathrm{Symb}}
{\stackrel{\sim}{\to}}
\operatorname{Hom}_{\mathfrak k'}(\C_\lambda,\operatorname{Pol}(\mathfrak n_+)\otimes \C_\nu)^{\widehat{d\pi_{\lambda^*}}(\mathfrak n_+')}.
$$
By Lemma \ref{lem:F4SO}, the right-hand side is reduced to zero if $\nu-\lambda\not\in\N$.
From now on, we assume $a:=\nu-\lambda\in\N$, and identify the right-hand side with a subspace of $\operatorname{Pol}(\mathfrak n_+)$. Then it follows from Lemma \ref{lem:F4SO} and Proposition \eqref{prop:55} that the bijections
\begin{eqnarray*}
\operatorname{Pol}_a[s]_{\mathrm{even}}&\stackrel{T_a}{\stackrel{\sim}{\To}}&\operatorname{Pol}_a[t]_{\mathrm{even}}\stackrel{\sim}{\To}
\operatorname{Hom}_{\mathfrak k'}(\C_\lambda,\operatorname{Pol}(\mathfrak n_+)\otimes \C_\nu)\\
h(s)&\mapsto& g(t)=h(\sqrt{-1}t)\mapsto Q_{n-1}(\zeta')^{\frac a2}g\left(\frac{\zeta_n}{\sqrt{Q_{n-1}(\zeta')}}\right)
\end{eqnarray*}
induces an isomorphism
$$
\mathrm{Sol}_{\mathrm{Gegen}}\left(\lambda-\frac{n-1}2,a\right)\cap\operatorname{Pol}_a[s]_{\mathrm{even}}\stackrel{\sim}{\To}
\operatorname{Hom}_{\mathfrak k'}(\C_\lambda,\operatorname{Pol}(\mathfrak n_+)\otimes \C_\nu)^{\widehat{d\pi_{\lambda^*}}(\mathfrak n_+')}.
$$

Since the left-hand side is always one-dimensional (see Theorem \ref{thm:Gegen} in Appendix), the first statement follows.

Furthermore, since $\mathrm{Sol}_{\mathrm{Gegen}}\left(\lambda-\frac{n-1}2,a\right)\cap\operatorname{Pol}_a[s]_{\mathrm{even}}$
is spanned by $\widetilde C_{a}^{\lambda-\frac{n-1}2}(s)$ by Theorem \eqref{thm:Gegen} (2), the space
$\operatorname{Hom}_{\widetilde G'}(\mathcal O(X,\mathcal L_\lambda), \mathcal O(Y,\mathcal L_\nu))$ is spanned by
$$
\mathrm{Symb}^{-1}\circ T_a\,\widetilde C_{a}^{\lambda-\frac{n-1}2}(\sqrt{-1}t)=
(-1)^{-\frac a2}\widetilde C_{a}^{\lambda-\frac{n-1}2}\left(-\Delta^z_{\C^{n-1}}, \frac{\partial}{\partial z_n}\right).
$$
Hence Theorems \ref{thm:SO(2,n)} and \ref{thm:63} are proved.
\end{proof}

\begin{rem}\label{rem:512}
Theorem \ref{thm:63} is a \lq\lq{holomorphic version}\rq\rq\ of the conformally covariant operator considered by A. Juhl \cite{Juhl} in the setting $S^{n-1}\hookrightarrow S^{n}$,
with equivariant actions of the pair of  groups
$SO(n,1) \subset SO(n+1,1)$, respectively.
Our proof based on the F-method is much shorter than the original proof in \cite[Chapter 6]{Juhl} that relies on 
combinatorial argument using recurrence relations of the coefficients of differential operators. 
The F-method gives a conceptual explanation for the appearance of Gegenbauer polynomials in Theorem \ref{thm:63}.
The relationship of symmetry breaking operators between real flag varieties (\emph{e.g.} \cite{Juhl,KOSS}) and the holomorphic setting is illustrated by an $SL_2$-example in \cite{KoKuPe}.
\end{rem}

\section{Symmetry breaking operators for the restriction $Sp(n,\R)\downarrow Sp(n-1,\R)\times Sp(1,\R)$}\label{sec:sp}
Let $n\geq 2$. In what follows, we realize the real symplectic group
 $G=Sp(n,\R)$ as a subgroup of the indefinite unitary group $U(n,n)$,
  so that we can directly apply the computation of
 $d\pi_{\lambda^*}({C})$ ($C\in\mathfrak n_+$) in \cite[Example \ref{example:410}]{PART1}.
 
 Let $G_\C$ be the complex symplectic group $Sp(n,\C)$ which preserves the standard symplectic
 form $\omega$ defined on $\C^{2n}$ by
 $$
 \omega(u,v):=\trans u J_nv,\quad\mathrm{for}\,\,u,v\in\C^{2n},
 $$
 where
 $J_n:=\begin{pmatrix}
0 &-I_n\\I_n&0
\end{pmatrix}$.
Let $E(\R):=\left\{\begin{pmatrix}
z\\ \bar z
\end{pmatrix} : z\in\C^n\right\}$ be a totally real vector subspace of $\C^{2n}$, and we set
$$
G:=GL_\R(E(\R))\cap Sp(n,\C)\simeq Sp(n,\R).
$$
Then the Lie algebra $\mathfrak g(\R)\simeq\mathfrak{sp}(n,\R)$ of $G$ is given by $$\mathfrak g(\R)=\mathfrak{gl}_\R(E(\R))\cap
\mathfrak{sp}(n,\C)=\left\{\begin{pmatrix}
A&B\\\overline B&\overline A
\end{pmatrix}: A=-\trans \overline A, B\in\mathrm{Sym}(n,\C)\right\},
$$
where we recall that $\mathrm{Sym}(n,\C)$ is the space of complex symmetric matrices.

Let $H_n:=\{Z\in\operatorname{Sym}(n,\C)\,:\, \Vert Z\Vert_{\mathrm{op}}<1\}$ be the
bounded symmetric domain of type CI in the \'E. Cartan classification, where $ \Vert Z\Vert_{\mathrm{op}}$
denotes the operator norm of $Z\in\mathrm{End}(\C^n)$.
 The Lie group $G=Sp(n,\R)$ acts biholomorphically on $H_n$ by
$$
g\cdot Z=(aZ+b)(cZ+d)^{-1}\quad{\mathrm{for}}\quad
g=\begin{pmatrix}a&b\\c&d\end{pmatrix}\in G,\,Z\in H_n.
$$
The isotropy subgroup $K$ of $G$ at the origin $0$ is identified with $U(n)$ by the isomorphism:
$$
K\stackrel{\sim}{\rightarrow} U(n),\quad 
\begin{pmatrix}
A&0\\0&\trans A^{-1}
\end{pmatrix}
 \mapsto A.
$$
We write $\widetilde G$ for the universal covering of $G$, and $\widetilde K$
for the connected subgroup with Lie algebra $\mathfrak k(\R)$.

Let $G'$ be the subgroup of $G=Sp(n,\R)$ that preserves the direct sum decomposition
$E(\R)\simeq \R^{2n}=\R^{2n-2}\oplus\R^2$ in the standard coordinates. Then
$G'$ is isomorphic to the connected group
$Sp(n-1,\R)\times Sp(1,\R)$. The pair $(G,G')$ is a symmetric pair as
 $G'$ is the fixed point subgroup of an involution
$\tau$ of $G$ defined by
$$
\tau(g)=
\begin{pmatrix}
I_{n-1,1} &0\\0&I_{n-1,1}
\end{pmatrix}
g
\begin{pmatrix}
I_{n-1,1} &0\\0&I_{n-1,1}
\end{pmatrix},
$$
where $I_{n-1,1}=\mathrm{diag}(1,\cdots,1,-1)$.

We set $X:=H_n\simeq G/K$ and $Y:=
X\cap \left\{ \begin{pmatrix}
a&0\\0&d
\end{pmatrix} : a\in\mathrm{Sym}(n-1,\C), d\in\C\right\}\simeq
H_{n-1}\times H_1\simeq G'/K'$. 
The symmetric pair $(G,G')$ is of holomorphic type, and
 the embedding of the complex manifold $Y\hookrightarrow X$ is $G'$-equivariant.
 
%----STEP 2-------------
Let $\mathfrak j$ be the standard Cartan subalgebra
$\sum_{i=1}^n\C(E_{ii}-E_{n+i,n+i})$
 of $\mathfrak k$, and $\{e_1,\cdots,e_n\}$ the standard
basis. Then $\mathfrak j$ is a Cartan subalgebra of $\mathfrak g$ and we choose
$\Delta^+(\mathfrak k,\mathfrak j)=\{e_i-e_j: 1\leq i<j\leq n\}$ and
$\Delta(\mathfrak n_+,\mathfrak j)=\{-(e_i+e_j): 1\leq i\leq j\leq n\}$ so that $\rho_{\mathfrak g}=(-1,-2,\cdots,-n)$.
Then we have
the following decomposition of the Lie algebra 
$$\mathfrak g=\mathfrak{sp}(n,\C)=\mathfrak n_-+\mathfrak k+\mathfrak n_+,\quad
\left(\begin{matrix}
A&B\\C&-\trans A
\end{matrix}
\right)\mapsto (B,A,C)
$$
 with $B=\trans B$ and $C=\trans C$. 
Here we have chosen a realization of $\mathfrak n_+$ in the \emph{lower} triangular matrices.
Accordingly, we adopt the following notation for characters of $\mathfrak k\simeq\mathfrak{gl}_n(\C)$: for $\lambda\in\C$
 the character $\C_\lambda$ of $\mathfrak k$ is defined by:
$$
\mathfrak k\To\C,\quad
\begin{pmatrix}
A&0\\0&-\trans A
\end{pmatrix}
\mapsto  -\lambda\,\mathrm{Trace}\, A.
$$
Its restriction to $\mathfrak j$ is given by $(-\lambda,\cdots,-\lambda)\in\mathfrak j^\vee\simeq\C^n$.

For $\lambda\in\C$,
the character $\C_\lambda$ lifts
to $\widetilde K$ and defines a $\widetilde G$-equivariant  holomorphic line bundle $\mathcal L_\lambda$ over $X= \widetilde G/\widetilde  K\simeq
G/K$. It descends to a $G$-equivariant bundle if $\lambda\in\Z$.
 In our parametrization,
$\mathcal L_{n+1}$ is the canonical line bundle of $X=G/K$, namely,
 $\C_{2\rho}=\C_{n+1}$.
 
 We shall construct differential symmetry breaking operators 
  from $\mathcal O(X,\mathcal L_\lambda)$ to $ \mathcal O(Y,\mathcal W_Y)$
  where $\mathcal W_Y$ is a $G'$-equivariant holomorphic vector bundle over $Y$.
  Unlike in the previous section, 
we have to deal with 
 vector bundles
 rather than line bundles
 because, by
   Proposition \ref{prop:64prime} below,
there exists a non-trivial $G'$-intertwining operator
 from $\mathcal O(X,\mathcal L_\lambda)$ to $ \mathcal O(Y,\mathcal W_Y)$ only if $\dim W>1$ for generic $\lambda$ except for the case when
 $\mathcal W_Y= \mathcal L_\lambda\vert_Y$ or $n=2$. 
 
 More precisely, such an irreducible representation $W$ of $\mathfrak k'\simeq\mathfrak {gl}_{n-1}(\C)\oplus\mathfrak{gl}_1(\C)$ 
 must be isomorphic to
\begin{eqnarray}\label{eqn:611}
W_\lambda^a=
F(\mathfrak{gl}_{n-1}(\C),(-\lambda,\cdots,-\lambda,-\lambda-a))\boxtimes F(\mathfrak{gl}_1(\C),(-\lambda-a)e_n),
\end{eqnarray}
for some $a\in\N$. This is a representation of $K'=GL(n-1,\C)\times GL(1,\C)$ on the space
 $\operatorname{Pol}^a[v_1,\cdots,v_{n-1}]$ of homogeneous polynomials of degree $a$ on $\C^{n-1}$
 twisted by 
 the one-dimensional representation
 $(\det_{n-1})^{-\lambda}(\det_1)^{-\lambda-a}$
 of $K'$ where $\det_k A$ denotes the determinant of $A\in M(k,\C)$.

In order to give a concrete model for the natural action of $G$ on $\mathcal O(X,\mathcal V)$
 consider an irreducible representation $\nu$ of $U(m)$ with highest weight $(\nu_1,\cdots,\nu_m)$ acting on a finite-dimensional
 complex vector space $W$. We extend it into a holomorphic representation denoted by the same letter  $\nu$ of $GL(m,\C)$ on $W$.
 Then the holomorphic vector bundle $\mathcal W=Sp(m,\R)\times_{U(m)}W$ over $H_m$ is trivialized using the open Bruhat cell, and 
 the regular representation of $Sp(m,\R)$ on $\mathcal O(H_m,\mathcal W)$ is identified with the multiplier representation of the same group on $\mathcal O(H_m)\otimes W$ given by
 \begin{equation*}
 \left( \pi_{(\nu_1,\cdots,\nu_m)}^{Sp(m,\R)}(g)F\right)(Z)=\nu\left(\trans(cZ+d)\right)F\left((aZ+b)(cZ+d)^{-1}\right), \end{equation*}
$\,\mathrm{for}\,\,
 g^{-1}=\begin{pmatrix}a&b\\c&d\end{pmatrix}\in Sp(m,\R),\,Z\in H_m.$
For $\lambda\in\Z$, the one-dimensional representation $\C_\lambda$ of $K$ has a highest weight
$(-\lambda,\cdots,-\lambda)$ and we shall simply
write $\pi_\lambda^{Sp(m,\R)}$ for the representation $\pi_{(-\lambda,\cdots,-\lambda)}^{Sp(m,\R)}$ of $Sp(m,\R)$ on $\mathcal O(H_m)$
given by
$$
\left( \pi_{\lambda}^{Sp(m,\R)}(g)F\right)(Z)=\det(cZ+d)^{-\lambda}F\left((aZ+b)(cZ+d)^{-1}\right),
$$
$\,\mathrm{for}\,
 g^{-1}=\begin{pmatrix}a&b\\c&d\end{pmatrix}\in Sp(m,\R),\,Z\in H_m.$
For $\lambda\in\C$, it gives a representation of $\widetilde{Sp(m,\R)}$ on the same space $\mathcal O(H_m)$.
 Similarly, for $a\in\N$, we denote by $\pi_{\lambda,a}^{Sp(m,\R)}$ the representation $\pi_{(0,\cdots,0,-a)+(-\lambda,\cdots,-\lambda)}^{Sp(m,\R)}$
  of the same group on $\mathcal O(H_m)\otimes\mathrm{Pol}^a[v_1,\cdots,v_m]$.

  The representation $W_\lambda^a$ may be realized on the space $\operatorname{Pol}^a[v_1,\cdots,v_{n-1}]$ where $(v_1,\cdots,v_{n-1})$
 are the standard coordinates on $\mathfrak n_-^{-\tau}\simeq\C^{n-1}$. Hence, the
differential symmetry breaking operators can be thought of
as elements of $\C\left[\frac{\partial}{\partial z_{ij}}\right]\otimes
\operatorname{Pol}^a[v_1,\ldots,v_{n-1}],$ where 
$z_{ij}$ ($1\leq i,j\leq n$) are the standard coordinates on $\mathfrak n_-\simeq\operatorname{Sym}(n,\C).$

%----THM------------
\begin{thm}\label{thm:Sp}
Let $n\geq 2$. Suppose $\lambda\in\mathbb{C}$ and
$a\in\N$.
\begin{enumerate}
\item  The vector space
$$
\operatorname{Hom}_{\widetilde{Sp(n-1,\R)\times Sp(1,\R)}}(
\mathcal O(H_n,\mathcal L_\lambda), \mathcal O(H_{n-1}\times H_1,\mathcal W_{\lambda}^a))
$$
is one-dimensional.
\item The vector-valued differential operator from $\mathcal O(X)$ to $\mathcal O(Y)\otimes W$ defined by
\begin{equation}\label{eqn:thmspn}
D_{X\to Y,a}:= \widetilde C_a^{\lambda-1}\left(
\sum_{1\leq i, j\leq n-1}2 v_iv_j
\frac{\partial^2}{\partial z_{ij}\partial z_{nn}},
\sum_{1\leq j\leq n-1}v_j\frac{\partial}{\partial z_{jn}}\right)
\in {\mathbb{C}}\left[\frac{\partial}{\partial z_{ij}}\right]
 \otimes {\operatorname{Pol}}^a[v_1, \cdots, v_{n-1}]
\end{equation}
intertwines the restriction $\pi_\lambda^{Sp(n,\R)}\Big\vert_{Sp(n-1,\R)\times Sp(1,\R)}$ and $\pi_{\lambda,a}^{Sp(n-1,\R)}\boxtimes
 \pi_{\lambda+a}^{Sp(1,\R)}$.
 Here the polynomial $\widetilde C_a^{\lambda-1}(x,y)$
 is the inflated normalized Gegenbauer polynomial defined in \eqref{eqn:Cxy}.
  \end{enumerate}
  \end{thm}
 It follows from Theorem \ref{thm:Sp} that any symmetry breaking operator from  $\mathcal O(X,\mathcal L_\lambda)$ to  $\mathcal O(Y,\mathcal W_\lambda^a)$
 is proportional to $D_{X\to Y,a}$.

\begin{rem}${}$
If $\lambda>n$ then $\mathcal H^2(X,\mathcal L_\lambda):=\mathcal O(X,\mathcal L_\lambda)\cap
L^2(X,\mathcal L_\lambda)$ is a non-zero Hilbert space on which $G$ acts unitarily and irreducibly. Then,
$\mathcal H^2(Y,\mathcal W_\lambda^a):=\mathcal O(Y,\mathcal W_\lambda^a)\cap
L^2(Y,\mathcal W_\lambda^a)\neq\{0\}$ for any $a\in\N$, and
the same statements as in Theorem \ref{thm:Sp} remain true for symmetry breaking operators
between the representation spaces $\mathcal H^2(X,\mathcal L_\lambda)$ and 
$\mathcal H^2(Y,\mathcal W_{\lambda}^a)$. 
\end{rem}

In order to prove Theorem \ref{thm:Sp} we apply the F-method. Its
 Step 1 is given   by

\begin{lem} For $\lambda\in\C$, we set $\lambda^*=\lambda^\vee\otimes
\C_{2\rho}=-\lambda+n+1$.
For $C\in Sym(n,\C)\simeq \mathfrak n_+$ and $Z\in Sym(n,\C)\simeq \mathfrak n_-$ we have
\begin{eqnarray*}
d\pi_{\lambda^*}(C )&=&(-\lambda+n+1)\operatorname{Trace}(CZ)+
\sum_{i\leq j}\sum_{k,\ell}C_{k\ell}z_{ik}z_{j\ell}\frac\partial{\partial z_{ij}},\\
\widehat{d\pi_{\lambda^*}}(C ) &=&-\lambda\sum_{i\leq j}C_{ij}\frac{\partial}{\partial\zeta_{ij}}
-\frac12\left(\sum_{i\leq k,j\leq\ell}C_{k\ell}\zeta_{ij}\frac{\partial^2}{\partial\zeta_{ik}\partial\zeta_{j\ell}}
+\sum_{i\geq k,j\geq\ell}C_{k\ell}\zeta_{ij}\frac{\partial^2}{\partial\zeta_{ik}\partial\zeta_{j\ell}}\right).
\end{eqnarray*}

\end{lem}
\begin{proof}
We embed the group $Sp(n,\R)$ into $U(n,n)$ and apply the results of
\cite[Example \ref{example:410}]{PART1} with $p=q=n$. Thus, 
the first statement follows from the formula \eqref{eqn:35}.

We consider a bilinear form
$$
\mathfrak n_+\times \mathfrak n_-\to \C,\qquad (C,Z)\mapsto \operatorname{Trace}
(C \hphantom{,} \trans Z),
$$
 where $\mathfrak n_+\simeq\operatorname{Sym}(n,\C)\simeq\mathfrak n_-$. Recall that $\zeta_{ij}$ with $1\leq i\leq j\leq n$ are the coordinates
 on $\mathfrak n_+\simeq\mathrm{Sym}(n,\C)$. 
However, it is convenient for the computations below to allow us to use
$\frac{\partial}{\partial\zeta_{ij}}$ ($i>j$) for the same 
meaning with
$\frac{\partial}{\partial \zeta_{ji}}$.
 Then
 \begin{equation*}
 \widehat{z_{ij}}=\frac12\left(1+\delta_{ij}\right)\frac{\partial}{\partial \zeta_{ij}},\quad
 \widehat{\frac{\partial}{\partial z_{ij}}}=(\delta_{ij}-2)\zeta_{ij}.
 \end{equation*}
Thus the algebraic Fourier transform
 of the first term of $d\pi_{\lambda^*}( C)$ amounts to
\begin{equation*}
({\operatorname{Trace}(CZ)})^{\widehat{}}=\frac12\sum_{i,j}C_{ij}(1+\delta_{ij})\frac{\partial}{\partial \zeta_{ij}}=\sum_{i\leq j}C_{ij}\frac{\partial}{\partial \zeta_{ij}},
\end{equation*}
whereas that of the second term of $d\pi_{\lambda^*}(C )$ amounts to
\begin{eqnarray*}
&&\left({\sum_{i\leq j}\sum_{k,\ell}C_{k\ell}z_{ik}z_{j\ell}\frac\partial{\partial z_{ij}}}\right)^{\widehat{}}=-(n+1)\sum_{i\leq j}C_{ij}\frac{\partial}{\partial\zeta_{ij}}
-\frac14\sum_{i,j,k,l}C_{kl}(1+\delta_{ik})(1+\delta_{jl})\zeta_{ij}
\frac{\partial^2}{\partial\zeta_{ik}\partial\zeta_{j\ell}}\\
&=&-(n+1)\sum_{i\leq j}C_{ij}\frac{\partial}{\partial\zeta_{ij}}
-\frac12\left(\sum_{i\leq k,j\leq\ell}C_{k\ell}\zeta_{ij}\frac{\partial^2}{\partial\zeta_{ik}\partial\zeta_{j\ell}}
+\sum_{i\geq k,j\geq\ell}C_{k\ell}\zeta_{ij}\frac{\partial^2}{\partial\zeta_{ik}\partial\zeta_{j\ell}}\right).
\end{eqnarray*}
Hence the formula for $\widehat{d\pi_{\lambda^*}}( C)$ follows.
\end{proof}

The condition \eqref{eqn:lpos2} amounts to 
$\langle(-\lambda+1,\cdots,-\lambda+n),-(e_i+e_j)\rangle>0$ for any $1\leq i\leq j\leq n$, namely $\lambda>n$.

For the Step 2 we apply Lemma \ref{lem:64}.
\begin{prop}\label{prop:64prime}
Assume $\lambda>n$. If
$$
\mathrm{Hom}_{G'}(\mathcal O(G/K,\mathcal L_\lambda), \mathcal O(G'/K',\mathcal W))\neq\{0\}
$$
for an irreducible representation $W$ of $K'$, then $W$ is of the form
$$
W=W_\lambda^a=S^a(\mathfrak n_+^{-\tau})\otimes(-\lambda\,\mathrm{Trace}_n),
$$
for some $a\in\N$ see \eqref{eqn:611}.
\end{prop}
From now on, we aim to construct (differential) symmetry breaking operators
from $\mathcal O(X,\mathcal L_\lambda)$ to $\mathcal O(Y,\mathcal W)$
in the case $W=W_\lambda^a$.

Define a Borel subalgebra $\mathfrak b(\mathfrak k')$ corresponding to the positive
root system $\Delta^+(\mathfrak k',\mathfrak j):=\Delta^+(\mathfrak k,\mathfrak j)\cap
\Delta(\mathfrak k',\mathfrak j)$.

 For Step 3 we apply Lemma \ref{lem:412part2}
and we get:
\begin{lem}\label{lem:Hom_sp}
Let $W_\lambda^a$ be the irreducible $\mathfrak k'$-module
 defined in \eqref{eqn:611}.  
\begin{enumerate}
\item The highest weight of $(W_\lambda^a)^\vee$ is given by
$$
\chi=(a,0,\ldots,0;a)+(\lambda,\ldots,\lambda;\lambda).
$$
\item For the $\mathfrak k$-module $\operatorname{Pol}(\mathfrak n_+)\otimes \C_\lambda^\vee$, the $\chi$-weight space for $\mathfrak b(\mathfrak k')$ is given by:
\begin{equation}\label{eqn:xi11}
(\operatorname{Pol}(\mathfrak n_+)\otimes \C_\lambda^\vee)_\chi\simeq\bigoplus_{2j+k=a}\C\zeta_{11}^j
\zeta_{1n}^k\zeta_{nn}^j,
\end{equation}
where we identify $\operatorname{Pol}(\mathfrak n_+)\otimes \C_\lambda^\vee$ with $\operatorname{Pol}(\mathfrak n_+)$ as vector spaces.
\end{enumerate}
\end{lem}
\begin{proof}
The statement (1) is clear from the definition of $W_\lambda^a$ given in \eqref{eqn:611}.
Notice that in our convention $\Delta(\mathfrak n_-)$ is given as $\Delta(\mathfrak n_-)=\{e_i+e_j:\,1\leq i\leq j\leq n\}$. Thus
$\mathfrak n_-$ decomposes into irreducible representations of $\mathfrak k'$ as
\begin{eqnarray}
\mathfrak n_-&\simeq&\left(\mathrm{Sym}(n-1),\C)\boxtimes\C\right)\oplus
(\C\boxtimes\C_2)\oplus(\C^{n-1}\boxtimes\C_1)\nonumber\\
&\simeq&\left( F(\mathfrak {gl}_{n-1},2e_1)\boxtimes F(\mathfrak {gl}_{1},0)\right)
\oplus
\left( F(\mathfrak {gl}_{n-1},0)\boxtimes F(\mathfrak {gl}_{1},2e_n)\right)
\label{eqn:25-1}\\
&&\oplus
\left( F(\mathfrak {gl}_{n-1},e_1)\boxtimes F(\mathfrak {gl}_{1},e_n)\right)\nonumber.
\end{eqnarray}
Accordingly we get an isomorphism of $\mathfrak k'$-modules:
\begin{eqnarray}\label{eqn:25-2}
\operatorname{Pol}(\mathfrak n_+)\simeq S(\mathfrak n_-)\simeq
 \bigoplus_{i,j,k} \left( S^i(\mathrm{Sym}(n-1),\C))\otimes S^k(\C^{n-1})\right)\boxtimes
\C_{2j+k}.
\end{eqnarray}

Since $\zeta_{11}, \zeta_{nn}$ and $\zeta_{1n}$ are highest weight vectors in the $\mathfrak k'$-module $\mathfrak n_-$ with respect to $\Delta^+(\mathfrak k')$ (see (\ref{eqn:25-1})), so is any monomial $\zeta_{11}^i\zeta_{nn}^j\zeta_{1n}^k$ in the $\mathfrak k'$-module $S(\mathfrak n_-)\simeq\operatorname{Pol}(\mathfrak n_+)$ of weight
$(2i+k)e_1+(k+2j)e_n$.

According to the irreducible decomposition (\ref{eqn:25-2}) and Remark \ref{rem:pieri}, it follows that the right-hand side of (\ref{eqn:xi11}) exhausts all highest weight vectors in 
$\operatorname{Pol}(\mathfrak n_+)$ of weight $a(e_1+e_n)$. 
Thus, 
taking into account the $\mathfrak k'$-action on $\C_\lambda^\vee\simeq
\lambda\,\mathrm{Trace}_n$, 
 we get Lemma.
\end{proof}

As Step 4, we reduce the
system of differential equations \eqref{eqn:Fmethod2part2}, \emph{i.e.}
$\widehat{d\pi}_{\lambda^*}( C)\psi=0$, to an ordinary differential equation.
For this, we identify $\mathrm{Pol}(\mathfrak n_+)\otimes V^\vee$ with the space of polynomials
in $\zeta$ on $\mathfrak n_+\simeq \mathrm{Sym}(n,\C)$. For a polynomial
$g(t)\in \mathrm{Pol}_a[t]_{\mathrm{even}}$ (see \eqref{eqn:gs}) we set
$$
\left(T_ag\right)(\zeta):=(\sqrt{2\zeta_{11}\zeta_{nn}})^ag\left(\frac{\zeta_{1n}}{\sqrt{2\zeta_{11}\zeta_{nn}}}\right).
$$

\begin{prop}\label{prop:gegpro}
Let $\chi$ be as in Lemma \ref{lem:Hom_sp} (1).
\begin{enumerate}
\item  $T_a:  \mathrm{Pol}_a[t]_{\mathrm{even}}\stackrel{\sim}{\to}
\left(\mathrm{Pol}(\mathfrak n_+)\otimes V^\vee\right)_\chi$.
\item The map $T_a$ induces an isomorphism
$$
\mathrm{Sol}_{\mathrm{Gegen}}(\lambda-1,a)\cap \mathrm{Pol}_a[t]_{\mathrm{even}}
\stackrel{\sim}{\to}
\left(\mathrm{Pol}(\mathfrak n_+)\otimes V^\vee\right)_\chi^{\widehat{d\pi_{\lambda^*}}(\mathfrak n_+')}.
$$

\item Any polynomial $\psi(\zeta)\equiv \psi(\zeta_{ij})$
 in the right-hand side of (\ref{eqn:xi11}) is given by
\begin{equation}\label{eqn:Dmodmorph1}
\psi(\zeta)=\left(T_ag\right)(\zeta):=(\sqrt{2\zeta_{11}\zeta_{nn}})^ag\left(\frac{\zeta_{1n}}{\sqrt{2\zeta_{11}\zeta_{nn}}}\right),
\end{equation}
for some $g(t)\in \mathrm{Pol}_a[t]_{\mathrm{even}}$. 
\item The polynomial $\psi(\zeta)$ 
 on $\mathrm{Sym}(n,\C)$ satisfies 
 the system of partial differential equations
 $\widehat{d\pi}_{\lambda^*}( C)\psi=0$ for any $C\in
{\mathfrak n_+'}$ if and only if $g(t)$ 
satisfies the Gegenbauer differential equation
\begin{equation}\label{eqn:JacobiDE2}
\left((1-t^2)\vartheta_t^2-(1+2(\lambda-1) t^2)\vartheta_t +a(a+2(\lambda-1))t^2\right)g(t)=0,
\end{equation}
where we denote $\vartheta_t=t\frac{d}{dt}$ as before.
\end{enumerate}
\end{prop}

\begin{proof}
The first  two statements follow from Theorem \ref{thm:FGK}, Proposition \ref{prop:Tsat} and Lemma \ref{lem:412part2}. The third statement is clear from \eqref{eqn:xi11}. The proof of the last assertion is similar to the one of Lemma \ref{lem:57} and uses the following identities for $T_a$-saturated differential operators:
$$
T_a^\sharp{\vartheta_{\zeta_{11}}}=
T_a^\sharp{\vartheta_{\zeta_{nn}}}= \frac12(a- \vartheta_t),\quad
T_a^\sharp{\vartheta_{\zeta_{1n}}}=  \vartheta_t,
$$
where $\vartheta_{\zeta_{ij}}=\zeta_{ij}\frac{\partial}{\partial\zeta_{ij}}$.
\end{proof}

We are ready to complete the proof of Theorem \ref{thm:Sp}.

\begin{proof} 
[Proof of Theorem \ref{thm:Sp}] 
By the general result of the F-method (see Theorem \ref{thm:locGK}) and 
owing to Proposition \ref{prop:Tsat} and Lemma \ref{lem:412part2}, we have the following isomorphism
$$
\mathrm{Sol}_{\mathrm{Gegen}}(\lambda-1,a)\cap\mathrm{Pol}_a[t]_{even}\simeq
\operatorname{Hom}_{\widetilde G'}(\mathcal O(X,\mathcal L_\lambda), \mathcal O(Y,\mathcal W_\lambda^a)).
$$
Hence,
the uniqueness
of the $G'$-intertwining operator amounts to the fact that the Gegenbauer differential equation has a unique polynomial 
solution up to a scalar multiple (see Theorem \ref{thm:Gegen} (2) in Appendix).

Let us prove that
$D_{X\to Y,a}$ defined in \eqref{eqn:thmspn} belongs to $\mathrm{Diff}_{G'}(\mathcal L_\lambda,\mathcal W_\lambda^a)$. Using the F-method we have
proved that if $D\in \mathrm{Diff}_{G'}(\mathcal L_\lambda,\mathcal W_\lambda^a)$ and $w^\vee$ is a highest weight vector in $(W_\lambda^a)^\vee$, then $\langle D,w^\vee\rangle$ is of the form $(\mathrm{Symb}^{-1}\otimes\mathrm{id})T_ag$, where $g(t)$ is a polynomial satisfying \eqref{eqn:JacobiDE2}.  
Hence $g(t)$ is, up to a scalar multiple, 
the Gegenbauer polynomial $\widetilde C_a^{\lambda-1}(t)$. In turn, $(T_ag)(\zeta)= \widetilde C_a^{\lambda-1}
(2\zeta_{11}\zeta_{nn},\zeta_{1n})$ up to a scalar.

Thus, in order to show $D_{X\to Y,a}\in\mathrm{Diff}_{G'}(\mathcal L_\lambda,\mathcal W_\lambda^a)$ it is sufficient to verify for all $\ell\in K_\C'$:
\begin{equation}\label{eqn:666}
(\mathrm{Symb}\otimes\mathrm{id})\langle D_{X\to Y,a},\nu^\vee(\ell^{-1})w^\vee\rangle=(\mathrm{Ad}_\sharp(\ell^{-1})\otimes\lambda^\vee(\ell^{-1}))(T_ag),
\end{equation}
by Lemma \ref{lem:hwvpart2} and by the observation that every non-zero $w^\vee\in W^\vee$ is cyclic.
 The left-hand side of \eqref{eqn:666} amounts to
\begin{eqnarray*}
&&\left\langle \widetilde C_a^{\lambda-1}\left(\sum_{1\leq i,j\leq n-1}2v_iv_j\zeta_{ij}\zeta_{nn},
\sum_{1\leq j\leq n-1}v_j\zeta_{jn}\right),\nu^\vee(\ell^{-1})w^\vee\right\rangle\\
&=& (\det\ell)^{-\lambda}
\left\langle \widetilde C_a^{\lambda-1}\left(\sum_{1\leq i,j\leq n-1}2(\ell v)_i(\ell v)_j\zeta_{ij}\zeta_{nn},
\sum_{1\leq j\leq n-1}(\ell v)_j\zeta_{jn}\right),w^\vee\right\rangle,
\end{eqnarray*}
where $v=\trans(v_1,\ldots,v_{n-1})$ stands for the column vector. Since
$\langle Q(v),w^\vee\rangle$ gives the coefficients of $v_1^a$ in the polynomial
$Q(v)$, it is equal to
\begin{eqnarray*}
&&(\det\ell)^{-\lambda}
\widetilde C_a^{\lambda-1}\left(\sum_{1\leq i,j\leq n-1}2\ell_{i1}\ell_{j1}\zeta_{ij}\zeta_{nn},
\sum_{1\leq j\leq n-1}\ell_{j1}\zeta_{jn}\right)\\
&=& (\det\ell)^{-\lambda}
\widetilde C_a^{\lambda-1}\left(\sum_{1\leq i,j\leq n-1}2(\trans\ell\zeta\ell)_{11}\zeta_{nn},
\sum_{1\leq j\leq n-1}(\trans\ell\zeta)_{1n}\right).
\end{eqnarray*}

On the other hand, the action of $\mathrm{Ad}(\ell^{-1})$ on $\operatorname{Pol}(\mathfrak n_+)$ is generated by
\begin{eqnarray*}
\zeta_{ij}\mapsto (\trans\ell\zeta\ell)_{ij},\quad
\zeta_{in}\mapsto (\trans\ell\zeta)_{in}.
\end{eqnarray*}
Hence, the right-hand side of \eqref{eqn:666} amounts to
$$
(\det\ell)^{-\lambda}
\widetilde C_a^{\lambda-1}\left(\sum_{1\leq i,j\leq n-1}2(\trans\ell\zeta\ell)_{11}\zeta_{nn},
\sum_{1\leq j\leq n-1}(\trans\ell\zeta)_{1n}\right),
$$
whence the equality \eqref{eqn:666}.

For the existence, we know that  $\mathrm{Hom}_{G'}(\mathcal O(G/K,\mathcal L_\lambda),
\mathcal O(G'/K',\mathcal W_\lambda^a))\neq\{0\}$for $\lambda>n$ by Theorem \ref{thm:locGK} and the 
branching law given by Fact \ref{fact:A}. In this case, it is given by the
differential operator \eqref{eqn:thmspn} by the F-method.  The same formula defines a non-zero
differential operator  which depends holomorphically on $\lambda\in\C$. Since  the actions of 
$\widetilde G$ on $\mathcal O(G/K,\mathcal L_\lambda)$ and that of $\widetilde G'$ on 
$\mathcal O(G'/K',\mathcal W_\lambda^a)$ can be realized on $H_n$ and $H_{n-1}\times H_1$, respectively,
by operators depending holomorphically on $\lambda\in\C$, the differential operator \eqref{eqn:thmspn}
respects the $\widetilde G'$ for all $\lambda\in\C$ by holomorphic continuation.
\end{proof}

%--------------------------------------------

\section{Symmetry breaking operators for the tensor product representations of $U(n,1)$}\label{sec:un1}
In this section we discuss a higher dimensional generalization of the Rankin--Cohen bidifferential
operators by considering the symmetric pair $(G'\times G',G')$ with $G'=U(n,1)$. First we fix some notations.
Let $U(n,1)$ be the Lie group of all matrices
 preserving the standard Hermitian form of signature $(n,1)$ on $\C^{n+1}$ given
by $I_{n,1}=\mathrm{diag}(1,\cdots,1,-1)\in GL(n+1,\C)$. 

Let $D$ be the unit ball
$\{Z\in\C^n\,:\, \Vert Z\Vert<1\}$, 
 where 
 $\Vert Z \Vert^2:=\sum_{j=1}^n |z_j|^2$
 for $Z=(z_1, \cdots, z_n)$.
It is the Hermitian symmetric domain of type AIII in $\C^n$ in \'E. Cartan classification. 
Then the Lie group $U(n,1)$ acts biholomorphically on $D$ by
$$
g\cdot Z=(aZ+b)(cZ+d)^{-1}\quad{\mathrm{for}}\quad
g=\begin{pmatrix}a&b\\c&d\end{pmatrix}\in U(n,1),\,Z\in D,
$$
and the isotropy subgroup at the origin is isomorphic to 
$U(n)\times U(1)$. 
Since $cZ+d\in GL(1,\C)$, we identify $cZ+d$ as a non-zero complex number and write $\frac{aZ+b}{cZ+d}$ instead of
$(aZ+b)(cZ+d)^{-1}$ from now on.

We adapt the same convention as in \cite[Example \ref{example:410}]{PART1}
with $p=n$ and $q=1$. In particular, we use the decomposition of the Lie algebra
$$
\mathrm{Lie}(U(n,1))\otimes_\R\C\simeq\mathfrak{gl}_{n+1}(\C)=\mathfrak n_-'+\mathfrak k'+\mathfrak n_+',\quad
\begin{pmatrix}
A&B\\C&d
\end{pmatrix}
\mapsto
(B,(A,d),C).
$$

Given a representation $\nu=\nu_1\boxtimes\nu_2$ of $U(n)\times U(1)$ on a finite-dimensional complex 
vector space $W$, we extend it to a holomorphic representation, denoted by the same
letter $\nu=\nu_1\boxtimes\nu_2$,
of $GL(n,\C)\times GL(1,\C)$ on $W$. Then the holomorphic vector bundle $\mathcal W=
U(n,1)\times_{U(n)\times U(1)}W$ over $D$ is trivialized  by using the open Bruhat cell $\mathfrak n_-'\simeq\C^n$, 
and the regular representation of $U(n,1)$ on $\mathcal O(D,\mathcal W)$ is identified with the multiplier representation $\pi_W$ of the same group on $\mathcal O(D)\otimes W$ given by
\begin{equation}\label{eqn:Un1rep}
(\pi_W(g)F)(Z):=\nu_1\left(a-\frac{(aZ+b)c}{cZ+d}\right)^{-1}\nu_2(cZ+d)^{-1}F\left(\frac{aZ+b}{cZ+d} \right),
\end{equation}
for $F\in \mathcal O(D)\otimes W, g^{-1}=
\begin{pmatrix}
a&b\\c&d
\end{pmatrix}
\in U(n,1)$ and $Z\in D$.
We note that $cZ+d\neq0$.

%----STEP 2-------------

For $\lambda_1,\lambda_2\in\C$, the map 
\begin{equation}
\label{eqn:lmd12}
     \mathfrak{gl}_n(\C)\oplus\mathfrak{gl}_1(\C)\to \C, 
     (A,d)\mapsto  -\lambda_1\,\mathrm{Trace}\, A-\lambda_2d
\end{equation}
 is a one-dimensional  representation
 of the Lie algebra $\mathfrak k'$,
 which we denote  by $\C_{(\lambda_1,\lambda_2)}$.  
The negative signature
 in \eqref{eqn:lmd12}
is chosen 
according to our realization of $\mathfrak n_+$ in the lower triangular matrices.
  For integral values of $\lambda_1$ and $\lambda_2$ the character $\C_{(\lambda_1,\lambda_2)}$ lifts
to $U(n)\times U(1)$. 
The restriction of the one-dimensional representation \eqref{eqn:lmd12} to the Cartan subalgebra $\displaystyle
\bigoplus_{i=1}^{n+1}\C E_{ii}$ is given by $(-\lambda_1,\cdots,-\lambda_1;-\lambda_2)$ in the
dual basis $\{e_1,\cdots,e_{n+1}\}$. 

For $\lambda_1,\lambda_2\in\Z$, we form 
a $U(n,1)$-equivariant holomorphic line bundle $\displaystyle\mathcal L_{\lambda_1,\lambda_2}=U(n,1)\times_{U(n)\times U(1)}\C_{(\lambda_1,\lambda_2)}$ over $D$. 
By \eqref{eqn:Un1rep}, the representation of $U(n,1)$ on $\mathcal O(D,\mathcal L_{\lambda_1,\lambda_2})$ is identified with the multiplier representation, denoted simply by $\pi_{\lambda_1,\lambda_2}$, of $U(n,1)$ on $\mathcal O(D)$ given by
$$
\left(\pi_{\lambda_1,\lambda_2}(g)F\right)(Z)=(cZ+d)^{-\lambda_1+\lambda_2}(\det g)^{-\lambda_1} F\left(\frac{aZ+b}{cZ+d}\right).
$$
In our normalization, the canonical bundle of $D$ is given by $\mathcal L_{(1,-n)}$
associated with 
$\C_{2\rho}=\mathrm{Trace} (\mathrm{ad}(\cdot):\mathfrak n_+\to
\mathfrak n_+)\simeq\C_{(1,-n)}$ with the notation of \eqref{eqn:lmd12},
and
the  dualizing bundle of $\mathcal L_{\lambda_1,\lambda_2}$ is given as 
\begin{equation}\label{eqn:dualizing}
\mathcal L_{\lambda_1,\lambda_2}^*=\mathcal L_{\lambda_1,\lambda_2}^\vee\otimes\C_{2\rho}\simeq
\mathcal L_{-\lambda_1+1,-\lambda_2-n},
\end{equation}
associated with 
$$
\C^*_{(\lambda_1,\lambda_2)}=\C_{(-\lambda_1,-\lambda_2)}\otimes\C_{2\rho}\simeq
\C_{(-\lambda_1+1,-\lambda_2-n)}.
$$

Now we consider the setting of symmetry breaking operators for the tensor product representations.
We set $X:=D\times D$ and $Y:=\Delta(D)$. Thus,
we have the following diagram:
$$
\begin{matrix}
X=D\times D & \subset & \C^n\times\C^n &\simeq&\mathfrak n_-&\subset& \mathbb P^n\C\times\mathbb P^n\C\\
\cup& &\cup& & \cup& &\cup\\
Y=\Delta(D)&\subset &\Delta(\C^n)&\simeq&\mathfrak n_-'&\subset&
\Delta(\mathbb P^n\C).
\end{matrix}
$$
We also set 
\[
   G:=U(n,1)\times U(n,1),
\]
 and
let $\tau$ be the involution of $G$ acting by $\tau:(g,h)\mapsto (h,g)$. Then the fixed point subgroup $G^\tau$ is
isomorphic to $\Delta(U(n,1))$. Its identity component $G'$ coincides with $G^\tau$ which is already connected.
We consider the symmetric pair of holomorphic type $(G,G')$.

According to the branching law in Fact
\ref{fact:B}, for $(\lambda_1',\lambda_2',\lambda_1'',\lambda_2'')\in\Z^4$ with $\lambda_1'-\lambda_2'>n$
and $\lambda_1''-\lambda_2''>n$,
 there exists a non-trivial $G'$-intertwining operator
$D_{X\to Y}(\varphi)$ from $\mathcal O(X,\mathcal L_{(\lambda_1',\lambda_2')}
\boxtimes \mathcal L_{(\lambda_1'',\lambda_2'')} )$ to $ \mathcal O(Y,\mathcal W_Y)$ 
if and
only the irreducible representation $W$ of $U(n)\times U(1)$ has the highest weight
$(-\lambda_1,\cdots,-\lambda_1,-\lambda_1-a;-\lambda_2+a)$ for some $a\in\N$.
We denote it by $W_{(\lambda_1,\lambda_2)}^a$ and realize on the space $\mathrm{Pol}^a[v_1,\cdots,v_n]$ of homogeneous
polynomials of degree $a$
where $(v_1,\ldots,v_{n})$
are the standard coordinates on $\mathfrak n_-^{-\tau}\simeq\C^{n}$.  Then the vector-valued differential
symmetry breaking operators can be thought of
as elements of 
\begin{equation}
\label{eqn:W12}
\C\left[\frac{\partial}{\partial z_{1}'},\ldots,
\frac{\partial}{\partial z_{n}'},\frac{\partial}{\partial z_{1}''},\ldots,
\frac{\partial}{\partial z_{n}''}
\right]\otimes
\operatorname{Pol}^a[v_1,\ldots,v_{n}],
\end{equation}
 where
$z_{i}', z_j''$ ($1\leq i,j\leq n$) are the standard coordinates on $\mathfrak n_-\simeq\C^n\oplus\C^n.$

Let $P_\ell^{\alpha,\beta}(t)$ be the Jacobi polynomial defined by
\begin{eqnarray}\label{eqn:Ptwo}
 P^{\alpha,\beta}_\ell(t)&=&
 \frac{\Gamma(\alpha+\ell+1)}{\Gamma(\alpha+\beta+\ell+1)}\sum_{m=0}^\ell
 \begin{pmatrix}
 \ell\\m
 \end{pmatrix}
\frac{\Gamma(\alpha+\beta+\ell+m+1)}{\ell!\Gamma(\alpha+m+1)}\left(\frac{t-1}2\right)^m,
\end{eqnarray}
see Appendix \ref{sec:A1} for more details.
We inflate it to a homogeneous polynomial
 of two variables
 $x$ and $y$
 by 
\begin{equation}\label{eqn:inflJacobi}
P_\ell^{\alpha,\beta}(x,y):=y^\ell P_\ell^{\alpha,\beta}\left(2\frac xy+1\right).
\end{equation}
For instance, $P_0^{\alpha,\beta}(x,y)=1$, $P_1^{\alpha,\beta}(x,y)=(2+\alpha+\beta)x+(\alpha+1)y$, etc.

We write $\widetilde{U(n,1)}$ for the universal covering of the group $U(n,1)$. Then we can define a 
$\widetilde{U(n,1)}$-equivariant holomorphic line bundle $\mathcal L_{(\lambda_1,\lambda_2)}$ over $D$ for
all $\lambda_1,\lambda_2\in\C$, as well as a representation of $\widetilde{U(n,1)}$ on $\mathcal O(D,\mathcal L_{(\lambda_1,\lambda_2)})$.

We denote by $\widehat\otimes$ the completion of the tensor product of two nuclear spaces.

\begin{thm}\label{thm:U(n,1)}
Suppose that $\lambda_1',\lambda_2', \lambda_1'',\lambda_2''\in\C$ and $a\in \N$. We set
 $\lambda':=\lambda_1'-\lambda_2'$
 and $\lambda'':=\lambda_1''-\lambda_2''$. 
\begin{enumerate}
\item The dimension of the vector space
 $$
 \mathrm{Hom}_{\widetilde{U(n,1)}}\left(
 {\mathcal{O}}
(D, {\mathcal{L}}_{(\lambda_1',\lambda_2')} )\,\widehat\otimes \,\mathcal{O}
(D, {\mathcal{L}}_{(\lambda_1'',\lambda_2'')}),{\mathcal{O}}(D, {\mathcal{W}}_{(\lambda_1'+\lambda_1'', \lambda_2'+\lambda_2'')}^a)\right)
$$
is either one or two. It is equal to two if and only if
\begin{equation}\label{eqn:Udim}
\lambda',\lambda''\in\{-1,-2,\cdots\}\quad\mathrm{and}\quad a\geq \lambda'+\lambda''+2a-1\geq\vert\lambda'-\lambda''\vert.
\end{equation}
\item The vector-valued differential
operator from $\mathcal O(D\times D)$ to $\mathcal O(D)\otimes
\mathrm{Pol}^a[v_1,\cdots,v_n]$ defined by
\begin{equation}\label{eqn:DXYUN1}
 D_{X\to Y,a}:=
  P_a^{\lambda'-1,-\lambda'-\lambda''-2a+1}\left(
\sum_{i=1}^{n}v_i \frac{\partial}{\partial z_i},
\sum_{j=1}^{n}v_j \frac{\partial}{\partial z_j}
\right)
\end{equation}
intertwines $\pi_{\lambda_1',\lambda_2'}\boxtimes \pi_{\lambda_1'',\lambda_2''}\Big\vert_{G'}$ and $\pi_W$, where $W\simeq W^a_{\lambda_1'+\lambda_1'',\lambda_2'+\lambda_2''}$.

\item If the triple $(\lambda',\lambda'',a)$ satisfies \eqref{eqn:Udim}, then $ D_{X\to Y,a}=0$. Otherwise,
any symmetry breaking operator  
is proportional to $D_{X\to Y,a}$.
\end{enumerate}
\end{thm}

\begin{rem}
${}$
\begin{enumerate}

\item The representation theoretic interpretation of the condition \eqref{eqn:Udim} will be clarified in 
Section \ref{sec:8} 
in the case $n=1$, where we construct three
 symmetry breaking operators for singular parameters satisfying \eqref{eqn:Udim} and discuss their linear relations.

\item The fiber of the vector bundle
 ${\mathcal{W}}_{(\lambda_1, \lambda_2)}^{a}$
 is isomorphic to the space $S^a({\mathbb{C}}^n)$
 of symmetric tensors
 of degree $a$.  
It is a line bundle
 if and only if $a=0$
 or $n=1$.  
In the case $n=1$, 
the formula \eqref{eqn:DXYUN1} reduces to
the classical Rankin--Cohen bidifferential operators (see \eqref{rcb}) with an appropriate
choice of spectral parameters, namely, for $a:=\frac12(\lambda'''-\lambda'-\lambda'')\in\N$, the following identity holds:
 \begin{equation}\label{eqn:JacobiRankin}
 \mathcal{RC}_{\lambda',\lambda''}^{\lambda'''}=(-1)^a
P^{\lambda'-1, 1-\lambda'''}_a
\left(\frac{\partial}{\partial z_1}, \frac{\partial}{\partial z_2}\right)\big\vert_{z_1=z_2=z}.
 \end{equation}
 
 \end{enumerate}
 \end{rem}
 \begin{rem}
\begin{enumerate}
 \item If  $\lambda_1',\lambda_2', \lambda_1'',\lambda_2''\in\Z$ and $a\in\N$, then the 
 linear groups $G$ and $G'$ act equivariantly on the two bundles $\mathcal L_{(\lambda_1',\lambda_2')}\boxtimes 
\mathcal L_{(\lambda_1'',\lambda_2'')}\to D\times D$ and $\mathcal W_{(\lambda_1,\lambda_2)}^a\to D$, respectively.
 
\item If $\lambda',\lambda''>n$, then analogous statements as in Theorem \ref{thm:U(n,1)} remain true
 for continuous $G'$-homomorphisms
 between the Hilbert spaces
${\mathcal{H}}^2
\left(X, {\mathcal{L}}_{(\lambda_1',\lambda_2')} \otimes {\mathcal{L}}_{(\lambda_1'',\lambda_2'')})\right)$
 and 
${\mathcal{H}}^2
\left(Y, {\mathcal{W}}^a_{(\lambda_1'+\lambda_1'' , \lambda_2'+\lambda_2'')}\right)$.
\item Similar statements hold for continuous $G'$-homomorphisms
  between the Casselman--Wallach globalizations
 by the localness theorem
 \cite[Theorem \ref{thm:C}]{PART1}.

\end{enumerate}
\end{rem}

In order to prove Theorem \ref{thm:U(n,1)}, we apply again the F-method. Its Step 1 is given by

\begin{lem}
For $(\lambda_1',\lambda_2')\in\C^2$, we set $(\mu_1',\mu_2'):=(-\lambda_1'+1,-\lambda_2'-n)$ and likewise we define
$(\mu_1'',\mu_2'')$ from $(\lambda_1'',\lambda_2'')$.
Let $C:=C'+C''=(c_1',\ldots, c_n')+(c_1'',\ldots, c_n'')\in \mathfrak n_+\simeq\C^n\oplus\C^n$.  
Then 
\begin{eqnarray*}
d\pi_{\mu_1',\mu_2'}( C')\oplus d\pi_{\mu_1'',\mu_2''}( C'' )&=&
\sum_{i=1}^nc_i'z_i'(E_{z'}-\lambda'+n+1)+\sum_{j=1}^nc_j''z_j''(E_{z''}-\lambda''+n+1), 
\\
\widehat{d\pi}_{\mu_1',\mu_2'}(C' )\oplus \widehat{d\pi}_{\mu_1'',\mu_2''} (C'' )
&=& -\left(\lambda'
\sum_{i=1}^nc_i'\frac\partial{\partial \zeta_i'}
+\sum_{i,j=1}^nc_i'\zeta'_j\frac{\partial^2}{\partial\zeta_i'\partial\zeta'_j} \right)\\
   &&- \left(\lambda''
   \sum_{j=1}^nc_j''\frac\partial{\partial \zeta_j''}
   +\sum_{i,j=1}^nc_i''\zeta_j''\frac{\partial^2}{\partial\zeta_i''\partial\zeta_j''} \right).
\end{eqnarray*} 
\end{lem}

For the Step 2 we apply Lemma \ref{lem:64}.
\begin{prop}\label{prop:74prime}
Assume $\lambda'=\lambda'_1-\lambda'_2>n$ and $\lambda''=\lambda''_1-\lambda''_2>n$ . If
$$
\mathrm{Hom}_{G'}(\mathcal O(G/K,{\mathcal{L}}_{(\lambda_1',\lambda_2')} \otimes {\mathcal{L}}_{(\lambda_1'',\lambda_2'')}
, \mathcal O(G'/K',\mathcal W))\neq\{0\}
$$
for an irreducible representation $W$ of $K'$, then $W$ is of the form
\begin{eqnarray}\label{eqn:UnW}
W&=&W_{(\lambda_1'+\lambda_1'', \lambda_2'+\lambda_2'')}^a=S^a(\mathfrak n_+^{-\tau})\otimes\C_{(\lambda_1'+\lambda_1'', \lambda_2'+\lambda_2'')}\\
&\simeq&\nonumber
\left(S^a\left((\C^{n})^\vee\right)\otimes(-\lambda_1\,\mathrm{Trace}_{n})\right)\boxtimes F(\mathfrak{gl}_1,(-\lambda_2+a)e_{n+1})
\nonumber
\end{eqnarray}
for some $a\in\N$.
\end{prop}

For Step 3 we apply Lemma \ref{lem:412part2}
and we get:

\begin{lem}\label{lem:86}
Suppose $\lambda_1',\lambda_2',\lambda_1'',\lambda_2''\in\C$ and $a\in \N$.
Let $V$ be the one-dimensional representation $\C_{(\lambda_1',\lambda_2')}\boxtimes
\C_{(\lambda_1'',\lambda_2'')}$ of $\mathfrak k$, and $W$ the irreducible representation
of $\mathfrak k'\simeq\mathfrak{gl}_n(\C)\oplus\mathfrak{gl}_1(\C)$ defined in \eqref{eqn:UnW}.
\begin{enumerate}
\item The highest weight of the contragredient representation $W^\vee$
 with respect to the standard Borel subalgebra $\mathfrak b(\mathfrak k')$ of $\mathfrak k'$
 is given by
$$
\chi=(a,0,\cdots,0;-a)+(\lambda_1'+\lambda_1'',\cdots,\lambda_1'+\lambda_1';\lambda_2'+\lambda_2'').
$$
\item We regard the $\mathfrak k$-module $\operatorname{Pol}(\mathfrak n_+)
\otimes V^\vee$ as a $\mathfrak b(\mathfrak k')$-module. Then the $\chi$-weight space  is given by
\begin{equation}\label{eqn:polUn1}
\left(\operatorname{Pol}(\mathfrak n_+)\otimes V^\vee\right)_\chi\simeq\bigoplus_{i+j=a}\C(\zeta'_1)^i(\zeta_1'')^j,
\end{equation}
where we identify $\operatorname{Pol}(\mathfrak n_+)\otimes V^\vee$ with $\operatorname{Pol}(\mathfrak n_+)$ as vector spaces.
\end{enumerate}
\end{lem}

\begin{proof}
1) Since the highest weight of $W$ is given by
$$
(-\lambda_1'-\lambda_1'',\cdots,-\lambda_1'-\lambda_1'';-\lambda_2'-\lambda_2'')+
(0,\cdots,0,-a;a),
$$
see \eqref{eqn:611}, the first statement is clear.

2) The Lie algebra $\mathfrak k'\simeq \mathfrak{gl}_n(\C)\oplus\mathfrak{gl}_1(\C)$ acts on $\mathfrak n_+\simeq
\C^n\oplus\C^n$ as the direct sum of two copies of irreducible representations
$$
F(\mathfrak{gl}_n(\C), (0,\cdots,0;-1))\boxtimes F(\mathfrak{gl}_1(\C), 1),
$$
and thus one has the following irreducible decomposition
\begin{eqnarray*}
\operatorname{Pol}(\mathfrak n_+)&\simeq&
\bigoplus_{i,j}\operatorname{Pol}^i(\C^n)\otimes \operatorname{Pol}^j(\C^n)\\
&\simeq&\bigoplus_{i,j}\left( F(\mathfrak{gl}_n(\C), (i,0,\cdots,0))\otimes F(\mathfrak{gl}_n(\C),(j,0,\cdots,0))\right)
\boxtimes F(\mathfrak{gl}_1(\C),-(i+j))\\
&\simeq& \bigoplus_{i,j}\bigoplus_{\underline s}
F(\mathfrak{gl}_n(\C),(s_1,s_2,0,\cdots,0))\otimes
F(\mathfrak{gl}_1(\C),-(i+j)),
\end{eqnarray*}
where the sum in the last line is taken over all $\underline s=(s_1,s_2,0,\cdots, 0)\in\N^n$ satisfying $s_1\geq 
s_2\geq0$, and $i+j\geq s_1\geq\max(i,j)$ and $s_1+s_2=i+j$. 
In particular, 
the weight $\chi$ occurs a highest weight in 
$\operatorname{Pol}(\mathfrak n_+)\otimes V^\vee$, or equivalently, the one-dimensional
$\mathfrak b(\mathfrak k')$-module $(a,0,\cdots,0;-a)$ occurs in $\operatorname{Pol}(\mathfrak n_+)$, if and only if 
$i+j=a$ and $s_2=0$. In this case the weight vectors are the monomials $ (\zeta_1')^i(\zeta_1'')^j$.
Lemma follows.
\end{proof}

As Step 4, we reduce the
system of differential equations \eqref{eqn:Q2} to an ordinary differential equation.
For this, we recall from \eqref{eqn:pola} that $\mathrm{Pol}_a[t]$ is the space of polynomials
in one variable $t$ of degree at most $a$. We identify $\mathrm{Pol}(\mathfrak n_+)\otimes V^\vee$ with the space
of polynomials in $(\zeta',\zeta'')$ on $\mathfrak n_+\simeq\C^n\oplus\C^n$. For $g\in\mathrm{Pol}_a[t]$ we set
$$
(T_ag)(\zeta',\zeta''):=(\zeta_1'')^ag\left(\frac{\zeta_1'}{\zeta_1''}\right).
$$

\begin{prop}\label{prop:723}
Let $\chi$ be the character of $\mathfrak b(\mathfrak k')$ given in Lemma \ref{lem:86}.
\begin{enumerate}
\item The map $T_a$ induces an isomorphism
 $
 T_a:\mathrm{Pol}_a[t]\stackrel{\sim}{\to}\left(\mathrm{Pol}(\mathfrak n_+)\otimes V^\vee\right)_\chi.
 $

\item The polynomial $T_ag$ satisfies the system of partial differential equations \eqref{eqn:Q2}
 if and only if the polynomial $g(t)$ solves
the single ordinary differential equation 
\begin{equation}\label{eqn:Tsat3}
\left((t+t^2)\frac{d^2}{dt^2} + (\lambda'-(\lambda''-2a+2)t)\frac{d}{dt}+a(\lambda''+a-1) \right)g(t)=0.
\end{equation}
\end{enumerate}
\end{prop}

For the proof of Proposition \ref{prop:723} we use the following identities
for $T_a$-saturated operators whose verification is similar to the one for Lemma \ref{lem:T}.

\begin{lem}\label{lem:T3}
One has:
\begin{enumerate}
\item $T_a^\sharp\left(\zeta_1''\frac{\partial}{\partial \zeta_1'}\right)
=\frac d{dt}$.
\item $T_a^\sharp\left(\zeta_1'\zeta''_1\frac{\partial^2}{\partial (\zeta_1')^2}\right)
=t\frac {d^2}{dt^2}$.
\item $T_a^\sharp\left(\zeta_1''\frac{\partial}{\partial \zeta_1''}\right)
=a-t\frac d{dt}$.
\item $T_a^\sharp\left((\zeta_1'')^2\frac{\partial^2}{\partial (\zeta_1'')^2}\right)
=a(a-1)-2(a-1)t\frac d{dt}+t^2\frac{d^2}{dt^2}$.
\end{enumerate}
\end{lem}

\begin{proof}[Proof of Proposition \ref{prop:723}]
The general condition (\ref{eqn:Q2})
of the F-method amounts to the following differential equation:
\begin{equation}\label{eqn:DE1}
\left(\lambda'\frac\partial{\partial\zeta_i'}+\zeta_i'\frac{\partial^2}{\partial(\zeta_i')^2}
+
\lambda''\frac\partial{\partial\zeta_i''}+\zeta_i''\frac{\partial^2}{\partial(\zeta_i'')^2}\right)\psi(\zeta',\zeta'')=0,
\end{equation}
for
 $C_i=(\underbrace{0,\ldots,0,}_{i-1}1,0,\cdots,0)+(\underbrace{0,\ldots,0,}_{i-1}1,0,\cdots,0)\in\Delta(\mathfrak n_+)\simeq\mathfrak n_+'\simeq\C^n$ ($1\leq i\leq n$). Applying this to $\psi=T_ag$, and using Lemma \ref{lem:T3}, we obtain the 
 differential equation \eqref{eqn:Tsat3} for $g(t)$.
\end{proof}
We give a proof of Theorem \ref{thm:U(n,1)} below. Note that the proof requires some general argument on the Jacobi polynomials, which is summarized in Appendix, namely, Section \ref{sec:A1}. We naturally quote necessary facts from the section, although they are discussed later.
\begin{proof}[Proof of Theorem \ref{thm:U(n,1)}]
We set
$$
h(s):=g\left(\frac{s-1}2\right).
$$
Then $g(t)\in\mathrm{Pol}_a[t]$ if and only if $h(s)\in\mathrm{Pol}_a[s]$, and $g(t)$
satisfies \eqref{eqn:Tsat3} if and only if $h(s)$ satisfies 
\begin{equation}\label{eqn:JacobiDE1}
\left((1-s^2)\frac{d^2}{ds^2}+(\beta-\alpha-(\alpha+\beta+2)s)\frac d{ds}
+{a(a+\alpha+\beta+1)}\right)h(s)=0,
\end{equation}
where $\alpha:=\lambda'-1$ and $\beta:=-\lambda'-\lambda''-2a+1$. Thus, combining with
Theorem \ref{thm:FGK}, we have shown the following bijection
\begin{eqnarray}
&&\mathrm{Hom}_{\widetilde{U(n,1)}}\left(
 {\mathcal{O}}
(D, {\mathcal{L}}_{(\lambda_1',\lambda_2')} )\widehat\otimes\, \mathcal{O}
(D, {\mathcal{L}}_{(\lambda_1'',\lambda_2'')}),{\mathcal{O}}(D, {\mathcal{W}}_{(\lambda_1'+\lambda_1'', \lambda_2'+\lambda_2'')}^a)\right)\nonumber\\
&\simeq&
\mathrm{Sol}_{\mathrm{Jacobi}}(\lambda'-1,-\lambda'-\lambda''-2a+1,a)\cap \mathrm{Pol}_a[s], \label{eqn:tensorP}
\end{eqnarray}
where $\mathrm{Sol}_{\mathrm{Jacobi}}(\alpha,\beta,\ell)\cap \mathrm{Pol}_a[s]$ denotes the space of polynomials of degree at most $a$ satisfying the Jacobi differential equation \eqref{eqn:JacobiDE}.

By the bijection \eqref{eqn:tensorP} the first statement is reduced to Theorem \ref{thm:Jacobizero} in Appendix on the dimension of polynomial solutions to the 
Jacobi differential equation.

Since the Jacobi polynomial $P^{\lambda'-1,-\lambda'-\lambda''-2a+1}_a(s)$ belongs to the right-hand side of \eqref{eqn:tensorP},
it follows from Theorem \ref{thm:FGK} (2) and Lemma \ref{lem:hwvpart2} that $D_{X\to Y,a}$ is a symmetry breaking operator.
The last statement follows from the fact that
Jacobi polynomial $P_a^{\lambda'-1,-\lambda'-\lambda''-2a+1}(t)$ is identically zero
as a polynomial of $t$ if and only if the triple $(\lambda',\lambda'',a)$ satisfies \eqref{eqn:Udim}, by Theorem \ref{thm:Jacobizero} (1) in Appendix.
 \end{proof}
\begin{rem}\label{rem:sing}
In all the three cases
 we have reduced a system of partial differential equations
 to a single ordinary differential equation
 in Step 4
 of the F-method.  
The latter equation has regular singularities
 at $t = \pm 1$
and $\infty$.  
We describe the corresponding singularities
 via the map $T_a$ as follows:
\begin{enumerate}
\item The singularities of the differential equation (\ref{eqn:JacobiDE3})
correspond to the varieties given by $\zeta_n=0$ and $Q_{n-1}(\zeta')=0$.
\item The singularities of the differential equation (\ref{eqn:JacobiDE2})
correspond to the varieties given by $\zeta_{1n}=0$ and $\det
\left\vert\begin{matrix}
\zeta_{11}&\zeta_{1n}\\\zeta_{1n}&\zeta_{nn}
\end{matrix}\right\vert=0$.
\item The singularities of the differential equation (\ref{eqn:JacobiDE1})
correspond to the varieties given by $\zeta_1'=0$ and $\zeta_1'=\pm \zeta_1''$.
\end{enumerate}
\end{rem}

%---- section 9

\section{Higher multiplicity phenomenon for singular parameter}\label{sec:8}

It is well-known that the branching law for the tensor product of two holomorphic discrete series representations
of $SL(2,\R)\,(\simeq SU(1,1))$ is multiplicity free. More generally, the branching laws for holomorphic
discrete series representations of scalar type in the setting of reductive symmetric pairs remain multiplicity free for positive parameters \cite{K08}, as well as their counterpart for generalized Verma modules for \emph{generic}
parameters \cite{K12}. However, we discover that such multiplicity one results may fail for singular parameters. In this section, we examine why and how it happens in the example of $SL(2,\R)$. We shall see that the F-method reduces it to the question of finding polynomial
solutions to the Gauss hypergeometric equation with all the parameters being negative integers.
We give a complete answer to this question in Appendix.

\subsection{Multiplicity two results for singular parameters}

From now on, we consider the setting of the previous section for $n=1$, and let $G=SU(1,1)$
rather than $U(1,1)$.

For $\lambda\in\Z$, we write $\mathcal L_\lambda$ for the $G$-equivariant holomorphic line bundle over the unit disk $D=\{z\in\C:\vert z\vert<1\}$, where $\lambda=\lambda_1-\lambda_2$ in the notations of the previous section. 
Using the Bruhat decomposition, we trivialize the line bindle $\mathcal L_\lambda$ and identify
the regular representation of $G$ on
$\mathcal O(D,\mathcal L_\lambda)$ with the following multiplier representation on $\mathcal O(D)$:
$$
\left(\pi_\lambda(g)F\right)(z)=(cz+d)^{-\lambda}F\left(\frac{az+b}{cz+d}\right),
\quad\mathrm{for}\quad g^{-1}=\begin{pmatrix}
a&b\\c&d
\end{pmatrix}
\,\mathrm{and}\,\, F\in\mathcal O(D).
$$

For $\lambda\in\C$, we extend $\pi_\lambda$ to a representation of the universal covering group $\widetilde G=\widetilde{SU(1,1)} $.

We write $\indbg(\nu)$ for the Verma module $U(\mathfrak g)\otimes_{U(\mathfrak b)}\C_\nu$
of the Lie algebra $\mathfrak g=\mathfrak {sl}(2,\C)$.
In our parametrization, if $\lambda=1-k$ ($k\in\N$), then
the $k$-dimensional irreducible representation occurs as a subrepresentation of
$(\pi_\lambda,\mathcal O(D))$ and as a quotient of $\indbg(-\lambda)$.

We consider symmetry breaking operators from the tensor product representation
$\mathcal O(\mathcal L_{\lambda'})\,\widehat\otimes\,\mathcal O(\mathcal L_{\lambda''})$ to $\mathcal O(\mathcal L_{\lambda'''})$, where $\widehat\otimes$ denotes the completion of 
the tensor product of two nuclear spaces.
As we saw in \eqref{rcb}, the Rankin--Cohen bidifferential operator $\mathcal{RC}_{\lambda',\lambda''}^{\lambda'''}$ is an example of such an operator when $\lambda'''-\lambda'-\lambda''\in2\N$ (see also Example \ref{ex:RC} below).

For $(\lambda',\lambda'',\lambda''')\in\C^3$, we set
\begin{eqnarray*}
H(\lambda',\lambda'',\lambda''')&:=&\mathrm{Hom}_{\widetilde G}(\mathcal O(\mathcal L_{\lambda'})\widehat\otimes\mathcal O(\mathcal L_{\lambda''}), \mathcal O(\mathcal L_{\lambda'''}))\\
&=&\mathrm{Diff}_{\widetilde G}(\mathcal O(\mathcal L_{\lambda'})\widehat\otimes\mathcal O(\mathcal L_{\lambda''}), \mathcal O(\mathcal L_{\lambda'''}))\\
&\simeq& \mathrm{Hom}_{\mathfrak g}(\indbg(-\lambda'''),\indbg(-\lambda')\otimes
\indbg(-\lambda'')),
\end{eqnarray*}
where the second equality and the third isomorphism follow from Theorem \ref{thm:locGK}. The general
theory (see Fact \ref{fact:A}) shows that $H(\lambda',\lambda'',\lambda''')$ is generically equal to 0 or 1. Here is a precise dimension formula:

\begin{thm}\label{thm:Hlambdas}
The vector space $H(\lambda',\lambda'',\lambda''')$ is finite dimensional for any
$(\lambda',\lambda'',\lambda''')\in\C^3$. More precisely,
\begin{enumerate}
\item $\dim_\C H(\lambda',\lambda'',\lambda''')\in\{0,1,2\}$.
\item $H(\lambda',\lambda'',\lambda''')\neq\{0\}$ if and only if
\begin{equation}\label{eqn:leven}
\lambda'''-\lambda'-\lambda''\in 2\N.
\end{equation}
\item Suppose \eqref{eqn:leven} is satisfied. Then the following three conditions are equivalent:
\begin{enumerate}
\item[(i)] $\dim_\C H(\lambda',\lambda'',\lambda''')=2$.
\item[(ii)]
\begin{equation}\label{eqn:l2}
\lambda',\lambda'',\lambda'''\in\Z,\quad 2\geq\lambda'+\lambda''+\lambda''',\quad
\mathrm{and}\quad \lambda'''\geq\vert\lambda'-\lambda''\vert+2.
\end{equation}
\item[(iii)] $\mathcal{RC}_{\lambda',\lambda''}^{\lambda'''}=0$.
\end{enumerate}
\end{enumerate}
\end{thm}

Next, let us give an explicit basis of $H(\lambda',\lambda'',\lambda''')$. For this consider the polynomials of one variable $\widetilde g_j$ ($j=1,2,3$) which will be defined in Lemma \ref{lem:gjz} with
$$
\alpha=\lambda'-1,\, \beta=1-\lambda''',\quad\mathrm{and}\quad \ell=\frac12(-\lambda'-\lambda''+\lambda''').
$$
We inflate $\widetilde g_j$ into homogeneous polynomials of degree $\ell$ of two variables by
$$
G_j(x,y):=(-y)^\ell \widetilde g_j\left(1+\frac{2x}y\right),
$$ 
and set
$$
D_j:=\mathrm{Rest}_{z_1=z_2=z}\circ G_j\left(\frac{\partial}{\partial z_1},\frac{\partial}{\partial z_2}\right),
$$
for $j=1,2,3$.

\begin{thm}\label{thm:Djs}
Suppose the conditions \eqref{eqn:leven} and \eqref{eqn:l2} hold. 
\begin{enumerate}
\item The operators $D_j$ ($j=1,2,3$) are $G$-homomorphisms from $\mathcal O(\mathcal L_{\lambda'})\widehat\otimes \mathcal O(\mathcal L_{\lambda''})$ to $\mathcal O(\mathcal L_{\lambda'''})$.
\item
 $1-\lambda', 1-\lambda''$ and
$1-\lambda'''\in\N_+$, and the operators $D_j$ ($j=1,2,3$) factorize into two natural intertwining operators as follows:
\begin{eqnarray*}
D_1&=&\mathcal{RC}_{2-\lambda',\lambda''}^{\lambda'''}\circ
\left(\left(\frac{\partial}{\partial z_1}\right)^{1-\lambda'}\otimes\mathrm{id}\right),\\
D_2&=&
\mathcal{RC}_{\lambda',2-\lambda''}^{\lambda'''}\circ\left(
\mathrm{id}\otimes\left(\frac{\partial}{\partial z_2}\right)^{1-\lambda''}\right),\\
D_3&=&
\left(\frac{d}{dz}\right)^{\lambda'''-1}\circ
\mathcal{RC}_{\lambda',\lambda''}^{2-\lambda'''}.
\end{eqnarray*}
\item
The following linear relation holds:
\begin{eqnarray*}
D_1-D_2+(-1)^{\lambda'}D_3=0.
\end{eqnarray*}
\end{enumerate}
\end{thm}
The factorizations  in Theorem \ref{thm:Djs} are illustrated by the following diagram:
\begin{equation}\label{eqn:diagram3}
\xymatrix{
 & &\qquad\quad\mathcal O(\mathcal L_{2-\lambda'})\, \widehat\otimes\,
\mathcal O(\mathcal L_{\lambda''})\ar[drr]^{\mathcal{RC}_{2-\lambda',\lambda''}^{\lambda'''}}\qquad\qquad&  &\\
\mathcal O(\mathcal L_{\lambda'}) \widehat\otimes
\mathcal O(\mathcal L_{\lambda''})\ar[urr]^{\left(\frac{\partial}{\partial z_1}\right)^{1-\lambda'}\otimes\,\mathrm{id}}
\ar[rr]^{\mathrm{id}\,\otimes\left(\frac{\partial}{\partial z_2}\right)^{1-\lambda''}}
\ar[drr]_{\mathcal{RC}_{\lambda',\lambda''}^{2-\lambda'''}} 
& & \mathcal O(\mathcal L_{\lambda'}) \,\widehat\otimes\,
\mathcal O(\mathcal L_{2-\lambda''})\ar[rr]^{\mathcal{RC}_{\lambda',2-\lambda''}^{\lambda'''}} & &\mathcal O(\mathcal L_{\lambda'''}),\\
& &\mathcal O(\mathcal L_{2-\lambda'''})\ar[urr]_{\left(\frac{d}{d z}\right)^{\lambda'''-1}}& &
}
\end{equation}

To summarize we consider the following three cases.
\begin{enumerate}
\item[Case 0.] $\lambda'''-\lambda'-\lambda''\not\in 2\N$.
\item[Case 1.] $\lambda'''-\lambda'-\lambda''\in 2\N$ but the condition \eqref{eqn:l2} is not fulfilled.
\item[Case 2.] $\lambda'''-\lambda'-\lambda''\in 2\N$ and the condition \eqref{eqn:l2} is satisfied.
\end{enumerate}

\begin{cor}\label{cor:83}
$$
H(\lambda',\lambda'',\lambda''')=\left\{
\begin{matrix*}[l]
\{0\} & \mathrm{Case}\,0,\\
\C\cdot\mathcal{RC}_{\lambda',\lambda''}^{\lambda'''}&\mathrm{Case}\, 1,\\
\C\langle D_1,D_2\rangle=\C\langle D_1,D_3\rangle=\C\langle D_2,D_3\rangle&\mathrm{Case}\, 2.
\end{matrix*}\right.
$$
\end{cor}
The rest of this section is devoted to the proof of Theorems \ref{thm:Hlambdas} and
\ref{thm:Djs}.
\subsection{Application of the F-method}
For $\alpha,\beta\in\C$, and $\ell\in\N$, we denote by $\mathrm{Sol}_{\mathrm{Jacobi}}(\alpha,\beta,\ell)\cap\mathrm{Pol}_\ell[t]$
the space of polynomials $g(t)$ of degree at most $\ell$ satisfying the Jacobi differential equation (see Appendix \ref{sec:A1}):
$$
(1-t^2)g''(t)+(\beta-\alpha-(\alpha+\beta+2)t)g'(t)+\ell(\ell+\alpha+\beta+1)g(t)=0.
$$

\begin{lem}\label{lem:94}
Suppose $(\lambda',\lambda'',\lambda''')\in\C^3$. Then,
\begin{enumerate}
\item $H(\lambda',\lambda'',\lambda''')=\{0\}$ if $\lambda'''-\lambda'-\lambda''\not\in2\N$.
\item Suppose  $\lambda'''-\lambda'-\lambda''\in2\N$.
Then the F-method gives a bijection 
$$
H(\lambda',\lambda'',\lambda''')
\stackrel{\sim}{\to}\mathrm{Sol}_{\mathrm{Jacobi}}(\alpha,\beta,\ell)
\cap\mathrm{Pol}_\ell[t],
$$
with $\alpha=\lambda'-1$, $\beta=1-\lambda'''$, and $\ell=\frac12(\lambda'''-\lambda'-\lambda'')\in\N$.
\end{enumerate}
\end{lem}
\begin{proof}
By Step 3 of the F-method, the symbol map induces a bijection between $H(\lambda',\lambda'',\lambda''')$ and the space of polynomials $\psi(\zeta_1,\zeta_2)$ of two variables satisfying the following two conditions
\begin{itemize}
\item $\psi(\zeta_1,\zeta_2)$ is homogeneous of degree $\frac12(\lambda'''-\lambda'-\lambda'')$,
\item $\left(\lambda'\frac{\partial}{\partial\zeta_1}+\zeta_1\frac{\partial^2}{\partial\zeta_1^2}\right)
\psi=\left(\lambda''\frac{\partial}{\partial\zeta_2}+\zeta_2\frac{\partial^2}{\partial\zeta_2^2}\right)
\psi=0$,
\end{itemize}
corresponding to \eqref{eqn:Fmethod1part2} and \eqref{eqn:Fmethod2part2}, respectively.
Hence the first statement follows.

The second statement follows from Step 4 of the F-method, namely, Proposition \ref{prop:723}
with $n=1$ shows that there is a correspondence between $\psi(\zeta_1,\zeta_2)$ and
$g(t)\in \mathrm{Sol}_{\mathrm{Jacobi}}(\alpha,\beta,\ell)\cap\mathrm{Pol}_\ell[t]$ with $\alpha,\beta$ and $\ell$ as above given by
$$
\psi(\zeta_1,\zeta_2)=\zeta_2^\ell\, g\left(\frac{2\zeta_1}{\zeta_2}+1\right).
$$
\end{proof}
We consider the transformation $(\lambda',\lambda'',\lambda''')\mapsto(\alpha,\beta,\ell)$
given by
\begin{equation}\label{eqn:abclmd}
\alpha:=\lambda'-1,\quad \beta:=1-\lambda''',\quad \ell:=\frac12(\lambda'''-\lambda'-\lambda'').
\end{equation}

For $\ell\in\N$, we define a finite set by
\begin{equation}\label{eqn:lmdl}
\Lambda_\ell:=\{(\alpha,\beta)\in\Z^2:\alpha+\ell\geq 0, \beta+\ell\geq0,\alpha+\beta\leq-(\ell+1)\}.
\end{equation}
We note that $\Lambda_\ell\in(-\N_+)\times(-\N_+)$ and $\#\Lambda_\ell=\frac12\ell(\ell+1)$.

\begin{lem}\label{lem:95}
Suppose $\alpha,\beta,\ell$ are given by \eqref{eqn:abclmd}. Then $\ell\in\N$ and
$(\alpha,\beta)\in\Lambda_\ell$ 
if and only if $(\lambda',\lambda'',\lambda''')\in\C^3$ satisfies the following two conditions:
\begin{eqnarray}
&&\lambda',\lambda'',\lambda'''\in\Z,\quad \lambda'+\lambda''\equiv\lambda'''\,\mathrm{mod}\,2,\label{eqn:tri1}\\
&&-(\lambda'+\lambda'')\geq\lambda'''-2\geq\vert\lambda'-\lambda''\vert. \label{eqn:tri2}
\end{eqnarray}
\end{lem}
Since the proof is elementary and follows from the definition, we omit it. Note that the conditions
\eqref{eqn:tri1} and \eqref{eqn:tri2} imply that
$$
\lambda'\leq0,\quad\lambda''\geq0,\quad\mathrm{and}\quad 2\leq\lambda''',
$$
which are equivalent to $\alpha\leq-1$, $\alpha+\beta+2\ell\geq0$, and $\beta\leq-1$, respectively.

\begin{proof}[Proof of Theorem \ref{thm:Hlambdas}]
By Lemma \ref{lem:94}, the proof is reduced to the computation of the dimension of
$\mathrm{Sol}_{\mathrm{Jacobi}}(\alpha,\beta,\ell)\cap\mathrm{Pol}_\ell[t]$.

1) Since the Jacobi differential equation is of second order, the space of its polynomial solutions is at most two-dimensional.

2) If $\ell=\frac12(\lambda'''-\lambda'-\lambda'')\in\N$, then Theorem \ref{thm:Fabc} (1) shows that $\dim \mathrm{Sol}_{\mathrm{Jacobi}}(\alpha,\beta,\ell)\cap\mathrm{Pol}_\ell[t]\geq 1$ for any $\alpha,\beta\in\C$.

3) The equivalence follows from Theorem \ref{thm:Jacobizero} (1) in light of Lemma \ref{lem:95}.
\end{proof}

\subsection{Factorization of symmetry breaking operators}
We have seen in Theorem \ref{thm:Hlambdas} that
$$
\dim_\C\mathrm{Hom}_{\widetilde G}(\mathcal O(\mathcal L_{\lambda'})\widehat\otimes\mathcal O(\mathcal L_{\lambda''}), \mathcal O(\mathcal L_{\lambda'''}))=2,
$$
when $(\lambda',\lambda'',\lambda''')$ satisfies \eqref{eqn:tri1} and \eqref{eqn:tri2}. In this subsection, we show that the other three symmetry breaking operators in the diagram \eqref{eqn:diagram3} are unique up to scalars. To be precise, we prove the following.
\begin{prop}\label{prop:alphalambdas}
Suppose $(\lambda',\lambda'',\lambda''')$ satisfies \eqref{eqn:tri1} and \eqref{eqn:tri2}. Then
\begin{eqnarray*}
&&\dim_\C\mathrm{Hom}_{\widetilde G}(\mathcal O(\mathcal L_{2-\lambda'})\widehat\otimes
\mathcal O(\mathcal L_{\lambda''}), \mathcal O(\mathcal L_{\lambda'''}))\\
&=&
\dim_\C\mathrm{Hom}_{\widetilde G}(\mathcal O(\mathcal L_{\lambda'})\widehat\otimes
\mathcal O(\mathcal L_{2-\lambda''}), \mathcal O(\mathcal L_{\lambda'''}))\\
&=&
\dim_\C\mathrm{Hom}_{\widetilde G}(\mathcal O(\mathcal L_{\lambda'})\widehat\otimes
\mathcal O(\mathcal L_{\lambda''}), \mathcal O(\mathcal L_{2-\lambda'''}))=1.
\end{eqnarray*}
\end{prop}
\begin{proof}
The transformation $(\lambda',\lambda'',\lambda''')\mapsto(\alpha,\beta,\ell)$
given by
\eqref{eqn:abclmd} yields
\begin{eqnarray*}
(2-\lambda',\lambda'',\lambda''')&\mapsto&(-\alpha,\beta,\alpha+\ell),\\
(\lambda',2-\lambda'',\lambda''')&\mapsto&(\alpha,\beta,-\alpha-\beta-\ell-1),\\
(\lambda',\lambda'',2-\lambda''')&\mapsto&(\alpha,-\beta,\beta+\ell).
\end{eqnarray*}
Moreover, if $(\alpha,\beta)\in\Lambda_\ell$ for some $\ell\in\N$, then
\begin{enumerate}
\item $\alpha+\ell\in\N$ and $(-\alpha,\beta)\not\in\Lambda_{\alpha+\ell}$,
\item $-\alpha-\beta-\ell-1\in\N$ and $(\alpha,\beta)\not\in\Lambda_{-\alpha-\beta-\ell-1}$,
\item $\beta+\ell\in\N$ and $(\alpha,-\beta)\not\in\Lambda_{\beta+\ell}$.
\end{enumerate}
Then the proposition follows from Lemma \ref{lem:94} (2) and Theorem \ref{thm:Jacobizero} (1).
\end{proof}

\subsection{Differential intertwining operators for $SL_2$}
Obviously, both the $F$-method and the localness theorem hold in the case when $G=G'$, for which symmetry breaking operators are usual intertwining operators, and have been extensively studied. Lemma \ref{lem:sl2_1} below is well-known, but we illustrate its proof by using the F-method. The operators  $\left(\frac{d}{dz}\right)^k$ are used for the factorization
of $D_j$ ($j=1,2,3$) in Theorem \ref{thm:Djs}.

For $(\lambda,\nu)\in\C^2$, we set
\begin{eqnarray*}
H(\lambda,\nu)&:=&\mathrm{Hom}_{\widetilde G}(\mathcal O(\mathcal L_\lambda), \mathcal O(\mathcal L_\nu))\\
&=&\mathrm{Diff}_{\widetilde G}(\mathcal O(\mathcal L_\lambda), \mathcal O(\mathcal L_\nu))\\
&\simeq&\mathrm{Hom}_{\mathfrak g}(\indbg(-\nu),\indbg(-\lambda)).
\end{eqnarray*}

\begin{lem}\label{lem:sl2_1}\noindent
\begin{enumerate}
\item $\dim_\C H(\lambda,\nu)\leq1$, and the equality holds if and only if $\lambda=\nu$
or $(\lambda,\nu)= (1-k,1+k)$ for some $k\in\N$.
\item If $(\lambda,\nu)=(1-k,1+k)$ for some $k\in\N$, then 
$$
H(\lambda,\nu)=\C\left(\frac{d}{dz}\right)^k.
$$
\end{enumerate}
\end{lem}
\begin{proof}
By the F-method, we have the following bijection between $H(\lambda,\nu)$ and the space of polynomials $g(t)$ of one variable satisfying the following two conditions
\begin{itemize}
\item $g(t)$ is a monomial of degree $\frac12(\nu-\lambda)$, \emph{i.e.}
$g(t)=C\,t^{\frac{\nu-\lambda}2}$ for some $C\in\C$,
\item 
$\left(\lambda\frac d{dt}+\frac{d^2}{dt^2}\right)g(t)=0$,
\end{itemize}
according to \eqref{eqn:Fmethod1part2} and \eqref{eqn:Fmethod2part2}.

The first condition forces $\nu-\lambda$ to be in $2\N$ in order to have $H(\lambda,\nu)$
not reduced to zero, whereas the second one implies $(\nu-\lambda)(\lambda+\nu-2)=0$. Hence either $\lambda=\nu$ or $(\lambda,\nu)= (1+k,1-k)$ for some $k\in\N$.
In the latter case, $g(t)=Ct^k$ for some $k\in\N$, which yields $\left(\frac{d}{dz}\right)^k$
as a $\widetilde G$-intertwining operator from $\mathcal O(\mathcal L_\lambda)$ to
$\mathcal O(\mathcal L_\nu)$.
\end{proof}

\subsection{Construction of homogeneous polynomials by inflation}
In order to analyze symmetry breaking operators in the setting when the Rankin--Cohen bidifferential operators $\mathcal{RC}_{\lambda',\lambda''}^{\lambda'''}$ vanish identically, we introduce the following notation.

For a polynomial $g(s)$ of degree \emph{at most $\ell$}, we set a polynomial of two variables
$$
(\mathcal I_\ell g)(x,y)=(-y)^\ell g\left(-\frac xy\right).
$$

The proof of factorization of symmetry breaking operators will be reduced to the following elementary factorization of homogeneous polynomials $(\mathcal I_\ell g)(x,y)$. The following observation follows immediately from the definition.

\begin{lem}\label{ex:Iellg}\noindent
\begin{enumerate}
\item Suppose $g_1(s)$ is of the form $g_1(s)=s^m h_1(s)$ for some polynomial $h_1(s)$
of degree $\ell-m$, then
$$
(\mathcal I_\ell g_1)(x,y)=(-x)^m(\mathcal I_{\ell-m}h_1)(x,y).
$$

\item Suppose $g_2(s)$ is a polynomial
of degree $\ell-m$, then
$$
(\mathcal I_\ell g_2)(x,y)=(-y)^m(\mathcal I_{\ell-m}g_2)(x,y).
$$
\item Suppose $g_3(s)$ is a polynomial of the form $g_3(s)=(1-s)^m h_3(s)$ for some polynomial $h_3(s)$
of degree $\ell-m$, then
$$
(\mathcal I_\ell g_3)(x,y)=(-1)^m(x+y)^m(\mathcal I_{\ell-m}h_3)(x,y).
$$
\end{enumerate}
\end{lem}

Suppose $\ell=\frac12(\lambda'''-\lambda'-\lambda'')\in\N$. Then it follows from the proof of Lemma \ref{lem:94} that the inverse of the following bijection
\begin{equation}\label{eqn:SoltoH}
H(\lambda',\lambda'',\lambda''')\stackrel{\sim}{\to}\mathrm{Sol}_{\mathrm{Jacobi}}(\lambda'-1,1-\lambda''',\ell)\cap\mathrm{Pol}_\ell[t],\quad D\mapsto g
\end{equation}
is given (up to multiplication by $(-1)^\ell$) by
$$
D=\mathrm{Rest}_{z_1=z_2=z}\circ\left(\mathcal I_\ell g(1-2s)\right)\left(\frac{\partial}{\partial z_1},\frac{\partial}{\partial z_2}\right).
$$
\begin{example}\label{ex:RC}
The Rankin--Cohen bidifferential operator \eqref{rcb} is given for $(\lambda',\lambda'',\lambda''')\in\C^3$ with $\ell:=\frac12(\lambda'''-\lambda'-\lambda'')\in\N$ by
\begin{equation}\label{RCIell}
\mathcal{RC}_{\lambda',\lambda''}^{\lambda'''}=
\mathrm{Rest}_{z_1=z_2=z}\circ\left(\mathcal I_\ell P_\ell^{\lambda'-1,1-\lambda'''}(1-2s)\right)\left(\frac{\partial}{\partial z_1},\frac{\partial}{\partial z_2}\right).
\end{equation}
\end{example}
 
\begin{proof}[Proof of Theorem \ref{thm:Djs}]
1) Since $\displaystyle\widetilde g_j\in\mathrm{Sol}_{\mathrm{Jacobi}}(\lambda'-1,1-\lambda''',\ell)\cap\mathrm{Pol}_\ell[t]$ with $\ell=\frac12(\lambda'''-\lambda'-\lambda'')\in\N$ by Theorem \ref{thm:Jacobizero} in the Appendix,
we have $D_j\in H(\lambda',\lambda'',\lambda''')$ by \eqref{eqn:SoltoH}.

2) Combining Lemmas \ref{ex:Iellg} and \ref{lem:gjz} we have the following identities
of the homogeneous polynomials $G_j(x,y)$:
\begin{eqnarray*}
G_1(x,y)&=&(-x)^{-\alpha}\left(\mathcal I_{\alpha+\ell} P_{\alpha+\ell}^{-\alpha,\beta}(1-2s)\right)(x,y),\\
G_2(x,y)&=&(-y)^{-\beta}\left(\mathcal I_{\beta+\ell} P_{-\alpha-\beta-\ell-1}^{\alpha,\beta}(1-2s)\right)(x,y),\\
G_3(x,y)&=&(-x-y)^{-\beta}\left(\mathcal I_{\beta+\ell} P_{\beta+\ell}^{\alpha,-\beta}(1-2s)\right)(x,y).
\end{eqnarray*}
The first two identities yield the factorization of $D_1$ and $D_2$, and the last one yields the factorization of $G_3$ in light of the formula:
$$
\mathrm{Rest}_{z_1=z_2=z}\circ\left(\frac{\partial}{\partial z_1}+\frac{\partial}{\partial z_2}\right)^j=\left(\frac{d}{d z}\right)^j\circ\mathrm{Rest}_{z_1=z_2=z},\quad\mathrm{for}\,
\mathrm{all}\, j\in\N.
$$
3) The identity is reduced to the linear relations among the polynomials $\widetilde g_j(s)$ ($j=1,2,3$) (see Lemma \ref{lem:gjz}) which are obtained by Kummer's connection formula for the Gauss hypergeometric function at the regular singularities $s=0$ ($\widetilde g_1(s)$ and
$\widetilde g_2(s)$) and $s=1$ ($\widetilde g_3(s)$). Hence Theorem \ref{thm:Djs} is proved.
\end{proof}

\section{An application of differential symmetry breaking operators}\label{sec:5}

\subsection{Remark on the discrete spectrum of the branching rule for complementary series for $O(n+1,1)\downarrow O(n,1)$}\label{sec:52}

B. Kostant proved in \cite{K69} the existence of the \lq\lq{long}\rq\rq\ complementary series representations of $SO(n,1)$ and $SU(n,1)$. 
In general, branching problems for the complementary series are more involved than the ones for principal series representations because
 the Mackey machinery does not apply.

In this section we explain briefly  how the differential operators $D_{X\to Y,a}$ $(a\in\N)$ given in Theorem \ref{thm:63} explicitly characterize discrete summands in the branching laws of
the complementary series representations of $O(n+1,1)$ when restricted to the subgroup $O(n,1)$.

For this we first observe that $G'_\C$-equivariant holomorphic differential operators $D_{X\to Y,a}$ associated to the embedding of 
complex flag varieties $G_\C/P_\C\hookleftarrow G'_\C/P'_\C$ 
induce $G_\R$-equivariant differential operators  associated to the 
embedding of the real flag varieties $G_\R/P_\R\hookleftarrow G'_\R/P'_\R$ for any pair $(G_\R,G'_\R)$ of real forms of $(G_\C, G'_\C)$ as far as $(P_\C, P'_\C)$ have real forms $(P_\R,P'_\R)$ in $(G_\R,G'_\R)$.

In particular, for the pair $(G,G')=(SO_o(n,2), SO_o(n-1,2))$ and $(G_\R,G'_\R):=(SO_o(n+1,1), SO_o(n,1))$ 
whose complexifications are the same, we see that $G$-equivariant holomorphic differential operators $D_{X\to Y,a}:\mathcal O(G/K,\mathcal L_\lambda)\to
\mathcal O(G'/K',\mathcal L_{\lambda+a})$ induce a $G'_\R$-equivariant differential operators 
\begin{equation}\label{eqn:dxy}
D_{X_\R\to Y_\R,a}: C^\infty(G_\R/P_\R,\mathcal L_\lambda)\to
C^\infty(G'_\R/P'_\R,\mathcal L_{\lambda+a}),
\end{equation}
for two spherical principal series representations
of $G_\R$ and $G'_\R$, owing to \cite[Theorem \ref{thm:C} (2)]{PART1} (extension theorem).
In our parametrization, for $0<\lambda<n$, there is a complementary series
$\mathcal H_\lambda$ that contains $C^\infty(G_\R/P_\R,\mathcal L_\lambda)$ as a dense subset.

We define a family of Hilbert spaces $L^2(\R^n)_s$ with parameter $s\in\R$ by
$$
L^2(\R^n)_s:=L^2 (\R^n, (\xi_1^2+\cdots+\xi_n^2)^{\frac s2}d\xi_1\cdots d\xi_n).
$$
Then, for $0<\lambda<n$, the Euclidean Fourier transform $\mathcal F_{\R^n}$ on the $N$-picture gives a unitary isomorphism
$$
\mathcal F_{\R^n}:\mathcal H_{n-\lambda} \xrightarrow{\sim}L^2(\R^n)_{2\lambda-n}.
$$

Correspondingly to the explicit formula
$$
D_{X_\R\to Y_\R,a}=\widetilde C_a^{\lambda-\frac{n-1}2}\left(-\Delta_{\C^{n-1}},\frac{\partial}{\partial z_n}\right)
$$
that was established in Theorem \ref{thm:63}, we see that the multiplication of the inflated Gegenbauer polynomial 
$\widetilde C_a^{\lambda-\frac{n-1}2}\left(\vert\xi\vert^2,\xi_n\right)$ (see \eqref{eqn:Cxy}) yields an explicit construction of discrete summands of the branching law for the restriction of complementary series as follows:

\begin{prop}\label{prop:compbran}
Suppose $a\in\N$ and $0<\lambda<\frac{n-1}2-a$. For $\xi=(\xi_1,\cdots,\xi_{n-1})\in\R^{n-1}$,
we set $\vert\xi\vert:=(\xi_1^2+\cdots+\xi_{n-1}^2)^\frac12$.
Then, 
$$
L^2(\R^{n-1})_{2(\lambda+a)-n-1}\hookrightarrow L^2(\R^n)_{2\lambda-n},\quad v(\xi)\mapsto C_a^{\lambda-\frac{n-1}2}\left(\vert\xi\vert^2,\xi_n\right)v(\xi)
$$
is an isometric and $G'_\R$-intertwining map from the complementary series of $G'_\R=SO_o(n,1)$ to that of $G_\R=SO_o(n+1,1)$.
\end{prop}
See \cite[Chapter 15]{KS13} for the proof that 
\eqref{eqn:dxy} implies the proposition in the case $a\in2\N$ (with both
$G_\R$ and $G'_\R$ replaced by disconnected groups $O(n+1,1)$ and $O(n,1)$, respectively).

%\appendix

\section{Appendix: Jacobi polynomials and Gegenbauer polynomials}\label{sec:10}
\subsection{Polynomial solutions to the hypergeometric differential equation}

In this subsection we discuss polynomial solutions to the Gauss hypergeometric differential equation
\begin{equation}\label{eqn:Fabc}
\left( z(1-z)\frac{d^2}{dz^2}-(c-(a+b+1)z)\frac{d}{dz} -ab\right)u(z)=0.
\end{equation}
For $c\not\in-\N$, the hypergeometric series
\begin{equation}\label{eqn:FTaylor}
{}_2F_1(a,b;c;z)=\sum_{j=0}^\infty \frac{(a)_j(b)_j}{(c )_j j!}z^j
\end{equation}
is a non-zero solution to \eqref{eqn:Fabc}. It is easy to see from \eqref{eqn:FTaylor}
that ${}_2F_1(a,b;c;z)$ is a polynomial if and only if $a\in-\N$ or $b\in-\N$.

Furthermore, we may ask if there exist two linearly independent polynomial solutions to \eqref{eqn:Fabc}.
In fact, this never happens when $c\not\in-\N$. More precisely, we have the following:
\begin{thm}\label{thm:Fabc}
Suppose $a,b,c\in\C$.
\begin{enumerate}
\item The following two conditions are equivalent.
\begin{enumerate}
\item[(i)] There exists a non-zero polynomial solution to \eqref{eqn:Fabc}.
\item[(ii)] $a\in-\N$ or $b\in-\N$.
\end{enumerate}
\item The following two conditions are equivalent.
\begin{enumerate}
\item[(iii)] There exist two linearly independent polynomial solutions to \eqref{eqn:Fabc}.
\item[(iv)] $a,b,c\in-\N$ and either \emph{(iv-a)} or \emph{(iv-b)} holds:
\begin{enumerate}
\item[(iv-a)] $a\geq c>b$,
\item[(iv-b)] $b\geq c>a$.
\end{enumerate}
\end{enumerate}
In this case the two linearly independent polynomial solutions are of degree $-a$ and $-b$.
\end{enumerate}
\end{thm}
\begin{proof}
(1) We have already discussed the case where $c\not\in -\N$. Suppose now that $c\in-\N$.
Since $1-c>0$, we have linearly independent solutions to \eqref{eqn:Fabc} near $z=0$ as follows
\begin{eqnarray*}
h_1(z)&=&z^{1-c}{}_2F_1(a-c+1,b-c+1;2-c;z),\\
h_2(z)&=&g(z)+\left(\Res_{\gamma=c}{}_2F_1(a,b;\gamma;z)\right)\log z,
\end{eqnarray*}
where $g(z)$ is a holomorphic function near $z=0$ satisfying $g(0)=1$. 
We divide the proof into two cases depending on whether 
 $\Res_{\gamma=c}{}_2F_1(a,b; \gamma;z)=0$  or not.
\vskip5pt

Case 1. Assume $\Res_{\gamma=c}{}_2F_1(a,b; \gamma;z)=0$. In view of the residue formula
$$
\Res_{\gamma=c}{}_2F_1(a,b;\gamma;z)=\frac{(-1)^c(a)_{1-c}(b)_{1-c}}{(-c)! (1-c)!}z^{1-c}
{}_2F_1(a+1-c,b+1-c;2-c;z)
$$
this expression vanishes if and only if $(a)_{1-c}(b)_{1-c}=0$, namely
$$
-\N\ni a\geq c\quad\mathrm{or}\quad -\N\ni b\geq c.
$$
In this case ${}_2F_1(a,b;\gamma;z)$ is holomorphic in $\gamma$ near $\gamma=c$, and
$$
\lim_{\gamma\to c}{}_2F_1(a,b;\gamma;z)=\sum_{j=0}^L \frac{(a)_j(b)_j}{(c )_jj!}z^j,
$$
where $L=-a$ or $-b$,
is a polynomial solution to \eqref{eqn:Fabc}.

Case 2. Assume $\Res_{\gamma=c}{}_2F_1(a,b; \gamma;z)\neq0$. Since the logarithmic term does not vanish, there exists a non-zero polynomial solution
to \eqref{eqn:Fabc} if and only if $h_1(z)$ is a polynomial, or equivalently,
$$
a-c+1\in-\N\quad\mathrm{or}\quad b-c+1\in-\N,
$$
namely,
$$
-\N\ni a<c\quad\mathrm{or}\quad -\N\ni b<c.
$$
Combining Case 1 and Case 2, we conclude the equivalence of (i) and (ii) in (1) for $c\in-\N$.

(2) We recall that the differential equation \eqref{eqn:Fabc} has regular singularities at
$z=0,1$, and $\infty$, and its characteristic exponents are indicated in the Riemann scheme
$$
P\left\{
\begin{matrix}
z=0 &1&\infty\\0&0&a\\1-c&c-a-b&b
\end{matrix}; z\right\}.
$$
(iii)$\Rightarrow$(iv). Suppose (iii) holds. Since the space of local solutions to \eqref{eqn:Fabc} is two
dimensional, any solution must be a polynomial. This forces the characteristic exponents to satisfy the following
conditions:
$$
1-c, c-a-b\in\N,\quad\mathrm{and}\quad a,b\in\N.
$$
Furthermore, the condition (iii) shows that there is no local solution which involves a non-zero logarithmic term
near each regular singularity point, which in particular implies that the two characteristic exponents at $z=0,1$ or
$\infty$ cannot coincide. Hence we get 
$$
1-c\neq0,\, c-a-b\neq0,\quad\mathrm{and}\quad a\neq b.
$$
Thus we have shown that the condition (iii) implies
\begin{equation}\label{eqn:pol2}
a,b,c\in-\N.
\end{equation}
From now we assume $c\in-\N$. As in the proof of (1), the condition (iii) implies that
$\Res_{\gamma=c}{}_2F_1(a,b;\gamma;z)=0$, and $h_1(z)$ is a polynomial. The latter conditions amount to
\begin{eqnarray*}
-\N\ni a\geq c\quad &\mathrm{or}&\quad -\N\ni b\geq c,\\
-\N\ni a<c\quad &\mathrm{or}&\quad -\N\ni b<c,
\end{eqnarray*}
respectively. Equivalently, we have either $a\geq c>b$ or $b\geq c>a$ under the condition that $a,b,c\in-\N$ (see
\eqref{eqn:pol2}). Hence the implication (iii)$\Rightarrow$(iv) is proved.

(iv)$\rightarrow$(iii). Conversely, suppose (iv) holds. Then as we saw in the proof of (1), $h_1(z)$ and
$$
\lim_{\gamma\to c}{}_2F_1(a,b;\gamma;z)=\sum_{j=0}^{\min(-a,-b)}\frac{(a)_j(b)_j}{(c )_jj!}z^j
$$
are both polynomial solutions to \eqref{eqn:Fabc}, corresponding to the characteristic exponents $1-c$ and $0$,
respectively. Thus they are linearly independent, and we have completed the proof of the equivalence of (iii) and (iv).
\end{proof}

\subsection{Jacobi polynomials}\label{sec:A1}
In this subsection, we discuss polynomial solutions to the Jacobi differential equation with emphasis on singular parameters where the corresponding Jacobi polynomial $P^{\alpha,\beta}_\ell(t)$ vanishes. In particular, we give a criterion for the space of polynomial solutions
to be two-dimensional, and find its explicit basis.

First we quickly review the classical facts on Jacobi polynomials.
Suppose $\alpha,\beta\in\C$ and $\ell\in\N$. The Jacobi differential equation
\begin{equation}\label{eqn:JacobiDE}
\left((1-t^2)\frac{d^2}{dt^2}+(\beta-\alpha-(\alpha+\beta+2)t)\frac d{dt}+\ell(\ell+\alpha+\beta+1)\right)y=0
\end{equation}
is a particular case of the Gauss hypergeometric equation \eqref{eqn:Fabc}, and has at least one non-zero polynomial solution by Theorem \ref{thm:Fabc} (1).

The Jacobi polynomial $P^{\alpha,\beta}_\ell(t)$ is the normalized polynomial solution to \eqref{eqn:JacobiDE} that
is subject to the Rodrigues formula
$$
(1-t)^\alpha(1+t)^\beta P_\ell^{\alpha,\beta}(t)=\frac{(-1)^\ell}{2^\ell\ell!}\left(\frac d{dt}\right)^\ell\left((1-t)^{\ell+\alpha}(1+t)^{\ell+\beta} \right),
$$
from which we have
\begin{equation}\label{eqn:Parity}
P^{\beta,\alpha}_\ell(-t)=(-1)^\ell P^{\alpha,\beta}_\ell(t).
\end{equation}
The Jacobi polynomial
 $P^{\alpha,\beta}_\ell(t)$ is
generically non-zero (see Theorem \ref{thm:Jacobizero} below for a precise condition) and is
 a polynomial of degree $\ell$ satisfying $P^{\alpha,\beta}_\ell(1)=\frac{\Gamma(\alpha+\ell+1)}{\Gamma(\alpha+1)\ell!}$.
Explicitly, for $\alpha\not\in-\N_+$,
\begin{eqnarray}\label{eqn:PF}
 P^{\alpha,\beta}_\ell(t)&=&
 \frac{\Gamma(\alpha+\ell+1)}{\Gamma(\alpha+1)\ell!}
 {}_2F_1\left(-\ell,\alpha+\beta+\ell+1;\alpha+1;\frac{1-t}2\right)\\
 &=&\frac{\Gamma(\alpha+\ell+1)}{\Gamma(\alpha+\beta+\ell+1)}\sum_{m=0}^\ell
 \begin{pmatrix}
 \ell\\m
 \end{pmatrix}
\frac{\Gamma(\alpha+\beta+\ell+m+1)}{\Gamma(\alpha+m+1)\ell!}\left(\frac{t-1}2\right)^m.\nonumber
\end{eqnarray}

Here are the first three Jacobi polynomials.
\begin{itemize}
\item $P_0^{\alpha,\beta}(t)= 1.$
\item
$P_1^{\alpha,\beta}(t)=\frac{1}{2} (\alpha-\beta+(2+\alpha+\beta) t)$.
\item$P_2^{\alpha,\beta}(t)=
\frac{1}{2} (1+\alpha) (2+\alpha)+\frac{1}{2} (2+\alpha) (3+\alpha+\beta) (t-1)+\frac{1}{8} (3+\alpha+\beta) (4+\alpha+\beta) (t-1)^2$.
\end{itemize}
If $\alpha>-1$ and $\beta>-1$, then the
Jacobi polynomials $ P^{\alpha,\beta}_\ell(t)$ ($\ell\in\N$) form an orthogonal basis in $L^2([-1,1],
(1-t)^\alpha(1+t)^\beta dt)$.

When $\alpha=\beta$ these polynomials yield Gegenbauer polynomials (see the next section for more details), and they further reduce to Legendre polynomials in the case when $\alpha=\beta=0$.

\begin{thm}\label{thm:Jacobizero}
Suppose $\ell\in\N$. We recall from \eqref{eqn:lmdl} that $\Lambda_\ell\subset (-\N)^2$ is
a finite set of the cardinality $\frac12\ell(\ell+1)$.
\begin{enumerate}
\item The following three conditions on $(\alpha,\beta)\in\C^2$ are equivalent:
\begin{enumerate}
\item[(i)] The Jacobi polynomial $P_\ell^{\alpha,\beta}(t)$ is equal to zero as a polynomial of $t$.
\item[(ii)] There exist two linearly independent polynomial solutions to \eqref{eqn:JacobiDE} of degree less than or equal to $\ell$, namely,
$$\dim_\C(\mathrm{Sol}_{\mathrm{Jacobi}}(\alpha,\beta,\ell)\cap\mathrm{Pol}_\ell[t])=2.$$
\item[(iii)] $(\alpha,\beta)\in\Lambda_\ell$. 
\end{enumerate}
\item If one of (therefore any of) the equivalent conditions (i)-(iii) is satisfied, then
\begin{equation}\label{eqn:Fcnegative}
\lim_{\varepsilon\to0}{}_2F_1(-\ell,\alpha+\beta+1;\alpha+\varepsilon+1;z)
\end{equation}
exists and is a polynomial in $z$, which we denote by
${}_2F_1(-\ell,\alpha+\beta+1;\alpha+1;z)$.
Then any two of the following three polynomials
\begin{eqnarray}
g_1(z)&:=& z^{-\alpha} {}_2F_1(-\alpha-\ell, \beta+\ell+1; 1-\alpha;z),\label{eqn:g1}\\
g_2(z)&:=&{}_2F_1(-\ell,\alpha+\beta+\ell+1;\alpha+1;z),\label{eqn:g2}\\
g_3(z)&:=& (1-z)^{-\beta}{}_2F_1(-\beta-\ell,\alpha+\ell+1;1-\beta;1-z),\label{eqn:g3}
\end{eqnarray}
with $z=\frac12(1-t)$
are linearly independent polynomial solutions to \eqref{eqn:JacobiDE}
of degree $\ell$, $-(\alpha+\beta+\ell+1)$, and $\ell$, respectively. In particular, any polynomial solution is of degree at most $\ell$.
\end{enumerate}
\end{thm}
\begin{proof}
(1). (i)$\Leftrightarrow$(iii). By the expression
$$
P_\ell^{\alpha,\beta}(t)=\sum_{j=0}^\ell
\frac{(\alpha+j+1)_{\ell-j}(\alpha+\beta+\ell+1)_j}{j!(\ell-j)!}\left(\frac{t-1}2\right)^j,
$$
one has $P_\ell^{\alpha,\beta}(t)\equiv 0$ as a polynomial of $t$ if and only if 
\begin{equation}\label{eqn:abj}
\underbrace{(\alpha+j+1)\cdots(\alpha+\ell)}_{\ell-j}\underbrace{(\alpha+\beta+\ell+1)\cdots(\alpha+\beta+\ell+j)}_j=0,
\mathrm{for}\,\mathrm{all}\, j\, (0\leq j\leq\ell).
\end{equation}

The condition \eqref{eqn:abj} implies $\alpha\in\{-1,\cdots,-\ell\}$ by taking $j=0$.
Conversely, if $\alpha\in\{-1,\cdots,-\ell\}$, then $(\alpha+j+1)\cdots(\alpha+\ell)=0$ for all
$j$ ($0\leq j\leq\ell$), and therefore \eqref{eqn:abj} is equivalent to $(\alpha+\beta+\ell+1)\cdots(\alpha+\beta+\ell+j)=0$ with $j=1-\alpha$, namely, $\alpha+\beta+\ell+1\leq0\leq\beta+\ell+1$. Hence  the equivalence of (i) and
(iii) is proved.

(ii)$\Leftrightarrow$(iii). We recall from Theorem \ref{thm:Fabc} that if the condition (iii), or equivalently
(iv), is satisfied, then there are two linearly independent polynomial solutions
 to \eqref{eqn:Fabc} of degrees $-a$ and $-b$, respectively. Applying Theorem \ref{thm:Fabc}
 (2) with
 $$
 a=-\ell,\quad b=\alpha+\beta+\ell+1,\quad\mathrm{and}\quad c=1+\alpha,
 $$
 we see that the condition on the degree of polynomials in (ii) corresponds to the
 condition $-a\geq-b$, which excludes (iv-b) in Theorem \ref{thm:Fabc}, and therefore, the condition (ii) is equivalent to
 $$
 -\ell, \alpha+\beta+\ell+1, 1+\alpha\in-\N,\quad
 \alpha+\beta+\ell+1\geq1+\alpha>-\ell,
 $$
 which is nothing but $(\alpha,\beta)\in\Lambda_\ell$.
 
 (2). Suppose $(\alpha,\beta)\in\Lambda_\ell$ for some $\ell\in\N$.
 
 Since $-\alpha-\ell\in-\N$ and $\beta+\ell+1, 1-\alpha\not\in-\N$, the polynomial $g_1(z)$ is of degree $-\alpha+(\alpha+\ell)=\ell$. 
 
 Secondly, the
 expression $-(\alpha+\beta+\ell+1)$ defines a non-negative integer smaller than $-\ell$ and
 we have:
 $$
 {}_2F_1(-\ell,\alpha+\beta+1;\alpha+\varepsilon+1;z)=
 \sum_{j=0}^{-(\alpha+\beta+\ell+1)}\frac{(-\ell)_j (\alpha+\beta+\ell+1)_j}{(\alpha+\varepsilon+1)_jj!}z^j.
 $$
 Since $\alpha+j\leq-(\beta+\ell+1)<0$ for all $j$ with $0\leq j\leq-(\alpha+\beta+\ell+1)$,
 the denominator in each summand does not vanish at $\varepsilon=0$, and therefore, $g_2(z)$ is
 well-defined and
is a polynomial of degree
$-(\alpha+\beta+\ell+1)$.
 
 Thirdly, since $-\beta-\ell\in-\N$ and $\alpha+\ell+1, 1-\beta\in\N_+$, the function
 ${}_2F_1(-\beta-\ell,\alpha+\ell+1;1-\beta;1-z)$ is a polynomial of homogeneous degree $\ell+\beta$, and thus $g_3(z)$ is a polynomial of degree $\ell$.
 
 Moreover, since $g_j(z)$ $(j=1,2,3)$ are local solutions to 
 \begin{equation}\label{eqn:Falphabeta}
\left( z(1-z)\frac{d^2}{dz^2}-((\alpha+1)-(\alpha+\beta+2)z)\frac{d}{dz} +\ell(\alpha+\beta+\ell+1)\right)u(z)=0
\end{equation}
 near zero depending meromorphically on parameters $(\alpha,\beta)\in\C^2$, and since they do not admit poles at any point of $\Lambda_\ell$, they are actually solutions
 to \eqref{eqn:Falphabeta}. Since $g_1(0)=0$ and $g_2(0)=1$, these functions are
 linearly independent.
 
 Finally, we apply Kummer's connection formula (see \cite[2.9 (4.3)]{Er})
 \begin{eqnarray*}
&& (1-z)^{c-a-b}{}_2F_1(c-a,c-b;c-a-b+1;1-z)\\
&=& \frac{\Gamma(c-1)\Gamma(c-a-b+1)}{\Gamma(c-a)\Gamma(c-b)} z^{1-c}
{}_2F_1(a+1-c,b+1-c;2-c;z)\\
&+&\frac{\Gamma(1-c)\Gamma(c-a-b+1)}{\Gamma(1-a)\Gamma(1-b)} 
{}_2F_1(a,b;c;z)
\end{eqnarray*}
with
$$
a=-\ell,\quad b=\alpha+\beta+\ell+1,\quad c=1+\alpha+\varepsilon,
$$
and taking the limit $\varepsilon\to0$, we obtain
\begin{equation}\label{eqn:Kummer}
g_3(z)=(-1)^{\alpha+\beta+\ell}\frac{(-\beta)!(\beta+\ell)!}{(-\alpha)!(\alpha+\ell)!}\,g_1(z)+
\frac{(-\alpha-1)!(-\beta)!}{l!(-\alpha-\beta-\ell-1)!}\,g_2(z).
\end{equation}
Since the scalars of this linear combination are non-zero, both pairs  $\{g_1(z), g_3(z)\}$
and $\{g_2(z),g_3(z)\}$ are linearly independent.
\end{proof}

To end this subsection, we express $g_j(z)$ ($j=1,2,3$) in terms of the Jacobi polynomials.
As a byproduct, we also give an identity among the Jacobi polynomials
when $(\alpha,\beta)\in\Lambda_\ell$, or equivalently, when $P_\ell^{\alpha,\beta}(t)\equiv0$ (Theorem \ref{thm:Jacobizero}).
\begin{lem}\label{lem:gjz}
Suppose $(\alpha,\beta)\in\Lambda_\ell$. Then,
\begin{eqnarray*}
&(1)& \widetilde g_1(z):=
\binom{\ell}{-\alpha}\cdot g_1(z)=z^{-\alpha}P_{\ell+\alpha}^{-\alpha,\beta}(1-2z);\\
&&\widetilde g_2(z):=(-1)^{-\ell-\alpha-\beta-1}\binom{-\alpha-1}{\ell+\beta}\cdot g_2(z)=P_{-\ell-\alpha-\beta-1}^{\alpha,\beta}(1-2z);\\
&&\widetilde g_3(z):= (-1)^{\beta+\ell}\binom{\ell}{-\beta}\cdot g_3(z)=(1-z)^{-\beta}P_{\ell+\beta}^{\alpha,-\beta}(1-2z).\\ \\
&(2)&(-1)^\alpha\,\widetilde g_3(z)=\widetilde g_1(z)-\widetilde g_2(z),\quad\mathrm{namely,}\\
&&P_{-\ell-\alpha-\beta-1}^{\alpha,\beta}(t)=
(-1)^{\alpha+1}\left(\frac{1+t}2\right)^{-\beta}P_{\beta+\ell}^{\alpha,-\beta}(t)
+\left(\frac{1-t}2\right)^{-\alpha}P_{\alpha+\ell}^{-\alpha,\beta}(t).
\end{eqnarray*}
\end{lem}
\begin{proof}
1) The first and third formul\ae\, follow from the equation \eqref{eqn:PF} and the identity
$\Gamma(\lambda)\Gamma(1-\lambda)=\frac{\lambda}{\sin\pi\lambda}$. The second one
is more subtle because $g_2(z)$ is defined as
 the limit of the Gauss hypergeometric function in a specific direction (see \eqref{eqn:Fcnegative}). Taking this into account, we deduce the second formula from \eqref{eqn:PF}.

2) The second identity follows directly from the first statement and \eqref{eqn:Kummer}.
\end{proof}

\subsection{Gegenbauer Polynomials}\label{sec:A2}
Let $\vartheta_t:=t\frac{t}{dt}$. For $\alpha\in\C$ and $\ell\in\N$,
the Gegenbauer differential equation
$$
\left((1-t^2)\frac{d^2}{dt^2}-(2\alpha+1)t\frac d{dt}+\ell(\ell+2\alpha)\right)y=0
$$
or, equivalently,
\begin{equation}
\label{eqn:ODEC}
\left((1-t^2)\vartheta_t^2-(1+2\alpha t^2)\vartheta_t +\ell(\ell+2\alpha)t^2\right)y=0
\end{equation}
is a particular case of 
the Jacobi
differential equation \eqref{eqn:JacobiDE} where $(\alpha,\beta)$  
are set to be $(\alpha-\frac12,\alpha-\frac12)$, and has at least one non-zero 
 polynomial solution owing to Theorem \ref{thm:Fabc} (1).
The Gegenbauer (or ultraspherical) polynomial $C_\ell^\alpha(t)$ is a 
solution to \eqref{eqn:ODEC} given
by the following formula:
\begin{eqnarray*}\label{eqn:Gegen-def2}
C_\ell^\alpha(t)&=&\frac{\Gamma(\ell+2\alpha)}{\Gamma(2\alpha)\Gamma(\ell+1)}
{}_2F_1\left(-\ell,\ell+2\alpha;\alpha+\frac12;\frac{1-t}2\right)\\
&=&\sum_{k=0}^{\left[\frac \ell2\right]}(-1)^k\frac{\Gamma(\ell-k+\alpha)}
{\Gamma(\alpha)\Gamma(k+1)\Gamma(\ell-2k+1)}(2t)^{\ell-2k}.
\end{eqnarray*}
It is
a specialization of the Jacobi polynomial
\begin{equation}\label{eqn:Gegen-def}
C_\ell^\alpha(t)=\frac{\Gamma(\alpha+\frac12)\Gamma(\ell+2\alpha)}
{\Gamma(2\alpha)\Gamma(\ell+\alpha+\frac12)}P_\ell^{\alpha-\frac12,\alpha-\frac12}(t).
\end{equation}

The Gegenbauer polynomial $C_\ell^\alpha(t)$ is a polynomial of degree $\ell$.
Here are the first five Gegenbauer polynomials.
\begin{itemize}
\item $C_0^\alpha(t)=1.$
\item $C_1^\alpha(t)=2\alpha t.$
\item $C_2^\alpha(t)=-\alpha(1-2(\alpha+1)t^2)$.
\item $C_3^\alpha(t)=-2\alpha(\alpha+1)(t-\frac23(\alpha+2)t^3)$.
\item $C_4^\alpha(t)=
\frac12 \alpha (\alpha+1)(1-4(\alpha+2)t^2 + 
 \frac43 (\alpha+2) (\alpha+3) t^4)$.
\end{itemize}
We note that $C_\ell^\alpha(t)\equiv0$ if $\ell\geq1$ and $\alpha=0,-1,-2,\cdots,-\left[\frac{\ell-1}2\right].$
Slightly differently from the usual notation in the literature, we renormalize the Gegenbauer polynomial by
\begin{equation}\label{eqn:Gegen2}
\widetilde C_\ell^\alpha(t):=\frac{\Gamma(\alpha)}{\Gamma\left(\alpha+ \left[\frac{\ell+1}2\right]\right)}C_\ell^\alpha(t).
\end{equation}
Then $\widetilde C_\ell^\alpha(t)$ is a non-zero solution to \eqref{eqn:ODEC} for all $\alpha\in\C$ and $\ell\in\N$.

As in the case of the Jacobi differential equation, there are some exceptional parameters
$(\alpha,\ell)$ for which the Gegenbauer differential equation \eqref{eqn:ODEC} has two linearly
independent polynomial solutions. For this
we denote by
$$
\mathrm{Sol}_{\mathrm{Gegen}}\left(\alpha,\ell\right)\cap\operatorname{Pol}[t]
$$
the space of polynomial solutions to \eqref{eqn:ODEC}, and consider its subspace 
$\mathrm{Sol}_{\mathrm{Gegen}}\left(\alpha,\ell\right)\cap\operatorname{Pol}_\ell[t]_{\mathrm{even}}$ where
$\operatorname{Pol}_\ell[t]_{\mathrm{even}}=\C\operatorname{-span}\left\langle t^{\ell-2j}:0\leq j\leq\left[\frac \ell2\right]
\right\rangle$. Then we have the following:
\begin{thm}\label{thm:Gegen} (1) Suppose $\ell\in\N$ and $\alpha\in\C$. 
Then 
$$
\dim_\C (\mathrm{Sol}_{\mathrm{Gegen}}\left(\alpha,\ell\right)\cap\operatorname{Pol}[t])=2
$$
if and only if $(\alpha,\ell)$
satisfies
\begin{equation}\label{eqn:Gegentwo}
\alpha\in\Z+\frac12\quad\mathrm{and}\quad 1-2\ell\leq2\alpha\leq-\ell.
\end{equation}
(2) For any $\ell\in\N$ and any $\alpha\in\C$, the space
$\mathrm{Sol}_{\mathrm{Gegen}}\left(\alpha,\ell\right)\cap\operatorname{Pol}_\ell[t]_{\mathrm{even}}$
is one-dimensional, and is spanned by $\widetilde{C}_\ell^\alpha(t)$. 
\end{thm}
\begin{proof}
(1) The first statement follows immediately from Theorem \ref{thm:Jacobizero} by replacing
$(\alpha,\beta)$ with $(\alpha-\frac12,\alpha-\frac12)$.

(2) Clearly, $\widetilde C_\ell^\alpha(t)\in \mathrm{Sol}_{\mathrm{Gegen}}\left(\alpha,\ell\right)\cap \operatorname{Pol}_\ell[t]_{\mathrm{even}}$ for all $\alpha\in\C$ and $\ell\in\N$.
 Hence it suffices to show that another solution (see Theorem \ref{thm:Jacobizero} and \eqref{eqn:Fcnegative})
 $$
 {}_2F_1\left(-\ell,2\alpha+\ell;\alpha+\frac12;\frac{1-t}2\right)\not \in
 \operatorname{Pol}_\ell[t]_{\mathrm{even}}
 $$
 when $\alpha$ satisfies \eqref{eqn:Gegentwo}. Indeed $ {}_2F_1\left(-\ell,2\alpha+\ell;\alpha+\frac12;\frac{1-t}2\right)$ is a polynomial in $t$ whose top term is a non-zero multiple of $t^{-(2\alpha+\ell)}$, but $-(2\alpha+\ell)\not\equiv\ell$ mod $2$ because $\alpha\in\Z+\frac12$.
 Hence Theorem is proved.
\end{proof}

\subsection*{Acknowledgements}
T. Kobayashi was partially supported by Institut des Hautes \'Etudes Scientifiques, France and
Grant-in-Aid for Scientific Research (B) (22340026)
and (A) (25247006), Japan Society for the Promotion of
Science.
Both authors were partially supported by Max Planck Institute for Mathematics (Bonn) where a large part of this work was done.
\textheight=220mm

\vskip7pt

\footnotesize{ \noindent % Addresses:
T. Kobayashi.  Kavli IPMU and Graduate School of Mathematical Sciences, The University of
 Tokyo,
3-8-1 Komaba, Meguro, Tokyo, 153-8914 Japan; \texttt{{
 toshi@ms.u-tokyo.ac.jp}}.\vskip5pt

\noindent M. Pevzner.
Laboratoire de Math\'ematiques de Reims, Universit\'e
de Reims-Champagne-Ardenne, FR 3399 CNRS, F-51687, Reims, France; \texttt{{
 pevzner@univ-reims.fr.}}
}
\end{document}

\begin{abstract}
 We prove a one-to-one correspondence between differential symmetry breaking operators for equivariant vector bundles over two homogeneous spaces and
 certain homomorphisms for representations of two Lie algebras, in connection with branching problems of the restriction of representations.

We develop a new method (F-method) based on the \emph{algebraic
Fourier transform for generalized Verma modules},
 which characterizes differential symmetry breaking operators by means of certain systems of partial differential 
equations.

In contrast to the setting of real flag varieties, continuous symmetry breaking operators
of Hermitian symmetric spaces
 are proved
 to be differential operators
 in the holomorphic setting.  
In this case symmetry breaking operators
 are characterized by differential equations
 of second order via the F-method.  

\vskip7pt

\noindent Key words and phrases: \emph{branching laws, F-method, symmetric pair, invariant theory, 
Verma modules, Hermitian symmetric spaces.}
\end{abstract}
\maketitle
\tableofcontents
\section{Introduction}
Let ${\mathcal{W}} \to Y$
 and ${\mathcal{V}} \to X$ be two vector bundles
 with a smooth map $p:Y \to X$.  
Then we can define
 \lq\lq{differential operators}\rq\rq\
$D:C^{\infty}(X, {\mathcal{V}}) \to C^{\infty}(Y, {\mathcal{W}})$
 between the spaces
 of smooth sections (Definition \ref{def:21}).  

Suppose that $G' \subset G$ is a pair 
 of Lie groups
 acting equivariantly 
 on ${\mathcal{W}} \to Y$ and ${\mathcal{V}} \to X$, 
 respectively,
 and that $p$ is $G'$-equivariant.  
The object
 of the present work is the study of $G'$-intertwining differential operators
 ({\it{differential symmetry breaking operators}}).  
If ${\mathcal{W}}$ is isomorphic
 to the pull-back $p^{\ast} {\mathcal{V}}$,  
then the restriction map $f \mapsto f|_Y$
 is obviously a $G'$-intertwining operator
 (and a differential operator of order zero).  
In the general setting where there is no morphism from $p^{\ast} {\mathcal{V}}$
 to ${\mathcal{W}}$, 
  non-zero $G'$-intertwining differential operators may and may not exist.  

Suppose that $G$ acts transitively on $X$
and $G'$ acts transitively on $Y$.  
We write $X=G/H$ and $Y=G'/H'$ as homogeneous spaces.  
The first main result
 is a duality theorem 
 that gives a one-to-one correspondence 
 between $G'$-intertwining differential operators
 and $({\mathfrak {g}}', H')$-homo\-morphisms
 for induced representations
 of Lie algebras
 (see Corollary \ref{cor:2.9}
for the precise notation):
\begin{thmintro}
\label{thm:thma}
Suppose $H' \subset H$.  
Then there is a natural bijection:
\begin{equation}
\label{eqn:ThmA}
D_{X\to Y}:
\operatorname{Hom}_{(\mathfrak g',H')}(\indhgprime(W^\vee),\indhg(V^\vee))
\stackrel{\sim}{\longrightarrow}
\operatorname{Diff}_{G'}\left( \mathcal V_X,
\mathcal W_Y\right).  
\end{equation}
\end{thmintro}
This generalizes a well-known result in the case
 where $G$ and $G'$ 
 are the same reductive group and where $X$ and $Y$
 are the same flag variety
 (\cite{Ks74,HJ}).  
\vskip7pt

By a \emph{branching problem} we wish to understand how a given representation of a group $G$ behaves when restricted to a subgroup $G'$. For a unitary representation $\pi$ of $G$, branching problems concern a decomposition of $\pi$ into the direct integral of irreducible unitary representations of $G'$ (branching law). 

More generally, for non-unitary representations $\pi$ and $\tau$ of $G$ and $G'$, 
 respectively,
 we may consider the space $\operatorname{Hom}_{G'}\left(\pi\vert_{G'},\tau\right)$  of continuous $G'$-homomorphisms. 
The right-hand side of \eqref{eqn:ThmA} concerns
 branching problems
 with respect to the restriction from $G$ to $G'$, 
 whereas the left-hand side of \eqref{eqn:ThmA} concerns branching laws
 of \lq\lq{generalized Verma modules}\rq\rq.  

\vskip 7pt

If $\operatorname{Diff}_{G'}\left( \mathcal V_X,
\mathcal W_Y\right)$ in \eqref{eqn:ThmA} is one-dimensional,
 we may regard its generator
 as canonical up to a scalar and be tempted to find an explicit description for such a natural differential symmetry breaking operator.
It should be noted that seeking explicit formul\ae\
of intertwining operators
is much more involved 
 than finding abstract
branching laws,
 as we may observe with the celebrated Rankin--Cohen brackets which appear as symmetry breaking operators in the decomposition of
the tensor product of two holomorphic discrete series representations
 of $SL(2,\R)$ (see \cite{vDP,PART2} for a detailed discussion).

The condition $\dim \operatorname{Diff}_{G'}\left( \mathcal V_X,
\mathcal W_Y\right) \le 1$ is often fulfilled
 when $\mathfrak h$ is
a parabolic subalgebra of $\mathfrak g$ with abelian nilradical, 
 (see \cite[Theorem 2.7]{K14}).
Moreover,  finding all bundles $\mathcal W_Y$ for which such nontrivial intertwining operators exist
 is a part of the initial problem, 
which reduces to abstract branching problems (see \cite[Fact \ref{fact:B}]{PART2}).

We propose a new method to find explicit expressions for differential symmetry breaking operators appearing in this geometric setting.
We call it the \emph{F-method}, where F stands for the Fourier transform.
More precisely,
 we consider an \lq\lq{algebraic Fourier transform}\rq\rq\
 of generalized Verma modules,
 and characterize symmetry breaking operators 
 by means of certain systems of partial differential equations.  
If ${\mathfrak {h}}$ is a parabolic subalgebra 
 with abelian nilradical,
 then the system 
 is of second order 
although the resulting differential symmetry breaking operators
 may be of any higher order. 
The characterization is performed by
applying an algebraic Fourier transform  (see Definition \ref{def:hat}).
A detailed recipe of the {\it{$F$-method}} is described in Section \ref{sec:44}
relying on Theorem \ref{thm:313} and Proposition \ref{prop:Fc}.

In general, the symmetry breaking operators between two principal series representations of real reductive Lie groups $G'\subset G$
are given by integro-differential operators
in geometric models.  
Among them, 
 equivariant differential operators are very
special (\emph{e.g}. \cite{Knapp-Stein} for $G'=G$ and \cite{KS} for $G'\subsetneqq G$). 
However, in the  case where $X$ is a Hermitian symmetric space,
 $Y$ a subsymmetric space, $G'\subset G$ are the groups of biholomorphic transformations of $Y\hookrightarrow X$,
respectively, 
 we prove the following \emph{localness and extension theorem}:
\begin{thmintro}
\label{thm:intro}
 Any continuous $G'$-homomorphism from
 $\mathcal O(X,\mathcal V)$
to $\mathcal O(Y,\mathcal W)$
is given by a holomorphic differential operator, which extends to the whole flag variety.
\end{thmintro}
See Theorem \ref{thm:C} for the precise statement. Theorem \ref{thm:intro} includes the case of the tensor product of two holomorphic discrete series representations corresponding to the setting where
$G \simeq G' \times G'$ and
$X \simeq Y \times Y$ as a special case.

In the second part of the work \cite{PART2} we apply  the F-method to Hermitian
symmetric spaces to find explicit formul\ae\, of differential symmetry breaking operators in the six parabolic geometries arising from symmetric pairs of split rank one.

\vskip10pt The authors are grateful to the referee
for his/her enlightening remarks and for suggesting to divide the original
manuscript into two parts
 and to write more detailed proofs
 and explanations
 for the first part
 not only for specialists
 but also for broader audience. Special thanks are also due to Dr. T. Kubo who read very carefully the revised manuscript and made constructive suggestions on its readability. \vskip10pt

Notation: $\N=\{0,1,2,\cdots\}$.

\section{Differential intertwining operators}\label{sec:2part1}

In this section we discuss equivariant differential operators
between sections of homogeneous vector bundles in a more general setting than the usual.
Namely, we consider vector bundles admitting a morphism between their
 base spaces.
In this generality, we establish a natural bijection 
between such  differential operators
 ({\it{differential symmetry breaking operators}})
 and certain homomorphisms arising from the branching problems for infinite-dimensional representations of Lie algebras,
 see
Theorem \ref{thm:surject} (\emph{duality theorem}).

\subsection{ Differential operators between two manifolds}\label{sec:21}
 We understand the notion of differential operators between two vector bundles in the usual sense when the bundles are defined over the same base space. We extend this terminology in a more general setting,
where there exists a morphism between base spaces. 
Let $\mathcal V\to X$ be a vector bundle
 over a smooth manifold $X$.  
We write $C^\infty(X,\mathcal V)$
 for the space of smooth sections,
 which is endowed with the Fr{\'e}chet topology
 of uniform convergence
 of sections 
 and their derivatives
 of finite order
 on compact sets.  
Let $\mathcal W \to Y$ be another vector bundle,
 and $p: Y \to X$ a smooth map between the base spaces.  

\begin{defn}\label{def:21}
We say that a continuous linear map $T:C^\infty(X,\mathcal V)\to C^\infty(Y,\mathcal W)$ is a \emph{differential
operator} if $T$ satisfies 
\begin{equation}
\label{eqn:suppT}
p\left(\operatorname{Supp} T f\right)
 \subset
 {\operatorname{Supp}f}
 \qquad \mathrm{for\, any}\,
 f\in C^\infty (X,\mathcal V).
\end{equation}
We write $\operatorname{Diff}(\mathcal V_X,\mathcal W_Y)$ for the vector space
 of differential operators
from $C^\infty(X,\mathcal V)$ to $C^\infty(Y,\mathcal W)$.
\end{defn}

The condition \eqref{eqn:suppT} shows that $T$ is a local operator in the sense that for any open subset $U$ of $X$, 
$T$ induces a continuous linear map:
$$
T_U: C^\infty(U,\mathcal V\vert_U)\longrightarrow C^\infty\left(p^{-1}(U),\mathcal W\vert_{p^{-1}(U)}\right).
$$
\begin{rem}
\label{rem:Peetre}
If $X=Y$
 and $p$ is the identity map,
 then the condition \eqref{eqn:suppT} is equivalent 
 to $T$ being a differential operator
 in the usual sense
 owing to Peetre's celebrated theorem \cite{Pee59}.  
Our proof of Lemma \ref{lem:1}
 in this special case
 gives an account
 of this classical theorem
 by using the theory
 of distributions
 due to L.~Schwartz \cite{SCH}.  
\end{rem}

Let $\Omega_X:=\vert\bigwedge^{\mathrm {top}}T^{\vee}(X)\vert$ be the bundle of densities. 
For a vector bundle $\mathcal{V} \to X$, 
 we set $\mathcal{V}^{\vee}:=\coprod_{x \in X}\mathcal{V}_x^{\vee}$
 where $\mathcal{V}_x^{\vee}:= \operatorname{Hom}_{\mathbb{C}} (\mathcal{V}_x, 
 {\mathbb{C}})$,
 and denote by $\mathcal V^*$
the dualizing bundle $\mathcal V^\vee\otimes\Omega_X$.
In what follows
$\mathcal{D}'(X,\mathcal{V}^*)$
(respectively, $\mathcal{E}'(X,\mathcal{V}^*)$)
denotes the space of $\mathcal{V}^*$-valued distributions
(respectively, those with compact support). 
We shall regard distributions as generalized functions \`a la Gelfand rather than continuous linear forms on
$C_c^\infty(X)$ or $C^\infty(X)$.
In particular,
 we sometimes write as 
\begin{equation}
\label{eqn:intE41}
\mathcal E'(X,\Omega_X)\to\C,
\quad 
\omega\mapsto \int_X\omega,
\end{equation}
to denote the natural pairing $\langle\omega,\mathbf 1_X\rangle$
 of $\omega$ with
 the constant function $\mathbf 1_X$ 
on $X$.  
Composing \eqref{eqn:intE41} with the contraction on the fiber, we get
a natural bilinear map 
\begin{equation}\label{eqn:intE42}
  C^{\infty}(X,\mathcal{V}) \times 
  {\mathcal{E}}'(X,\mathcal{V}^*)
  \to {\mathbb{C}}, 
\quad
  (f, \omega) \mapsto \langle f , \omega \rangle=\int_X f \omega.  
\end{equation}
Let $\mathcal V^*\boxtimes\mathcal W$ denote the tensor product bundle over $X\times Y$
of the two vector bundles $\mathcal V^*\to X$ and $\mathcal W\to Y$. Then for
 any continuous linear map $T:C^\infty(X,\mathcal V)\to C^\infty(Y,\mathcal W)$ 
there exists a unique
 distribution $K_T\in \mathcal D'(X\times Y,\mathcal V^*\boxtimes \mathcal W)$ such that the projection
 on the second factor $\mathrm{pr}_2:X\times Y\to Y$ is proper on the support of $K_T$ and such that
$$
(Tf)(y)= \langle K_T(\cdot,y),f(\cdot)\rangle\,\,\mathrm{for}\,\mathrm{any}\,f\in C^\infty(X,\mathcal V),
$$
by the Schwartz kernel theorem.

Given a map $p:Y \to X$, 
 we set 
\[
\Delta(Y):=
\{(p(y),y):y\in Y\}\subset X\times Y.  
\] 
 The following lemma characterizes differential operators by means of the distribution kernels $K_T$.

\begin{lem} \label{lem:1}
Let $p:Y\to X$ be a smooth map. 
A continuous operator $T:C^\infty(X,\mathcal V)\to C^\infty(Y,\mathcal W)$ is a differential operator in the sense of Definition \ref{def:21} if and only if
$\mathrm{Supp}\,K_T\subset\Delta(Y)$.
\end{lem}
\begin{proof}
Suppose $\mathrm{Supp}\,K_T\subset\Delta(Y)$. 
Let $(x_o,y_o)\in\Delta(Y)$ and take a neighborhood $U$ of $x_o=p(y_o)$ in $X$ and a neighborhood $U'$ of $y_o$ in $Y$ such that
$U'\subset p^{-1}(U)$. We trivialize the bundles locally as 
$\mathcal V\vert_U\simeq U\times V$ and $\mathcal W\vert_{U'}\simeq
U'\times W$. Let $(x_1,\cdots,x_m)$ be the coordinates in $U$. According to
 the structural
theory of distributions supported on a submanifold $\Delta Y\subset X\times Y$
\cite[Chapter III, Th\'eor\`eme XXXVII]{SCH}, 
there exists a unique family  $h_\alpha(y)\in \mathcal D'(U')\otimes W$
for a finite number of multi-indices $\alpha=(\alpha_1,\cdots,\alpha_m)\in\N^m$,
such that
 $\langle K_T,f\rangle\in\mathcal D'(U')\otimes \operatorname{Hom}_\C(V,W)$ is locally given as
 a finite sum
\begin{equation}\label{eqn:halpha}
\sum_\alpha h_\alpha(y)\frac{\partial^{\vert\alpha\vert}f}{\partial x^\alpha}(p(y)),
\end{equation}
for every $f\in C^\infty(X,\mathcal V)$. Hence $\langle K_T,f\rangle\vert_{U'}=0$ if $f\vert_U=0$. Thus $T$ is a differential operator in the sense of
Definition \ref{def:21}.

Conversely, 
 take any $(x_o,y_o)\in \mathrm{Supp}\,K_T$.  
By the definition of the distribution kernel $K_T$, 
 for any neighborhood $S$ of  $x_o$ in $X$ there exists $f\in C^\infty(X,\mathcal V)$ such that $\mathrm{Supp}f\subset S$ and $(x_o,y_o)\in
\mathrm{Supp}\,f\times \mathrm{Supp}\,Tf$. 
If $T$ is a differential operator then by (\ref{eqn:suppT}) 
$$
p(\mathrm{Supp}\,Tf) \subset \mathrm{Supp}\,f \subset S.
$$
Since $S$ is an arbitrary neighborhood of $x_o$, $p(y_o)$ must
coincide with $x_o$.  
Hence $\mathrm{Supp}\,K_T\subset\Delta(Y)$.
\end{proof}
By \eqref{eqn:halpha}, 
 the terminology
 \lq\lq{differential operators}\rq\rq\
 in Definition \ref{def:21}
 is justified as follows:
\begin{example}
\begin{enumerate}
\item Let $p:Y\twoheadrightarrow X$ be a submersion. Choose an atlas of local coordinates $\{(x_i,z_j)\}$ on $Y$ in
such a way that  $\{x_i\}$ form an atlas on $X$. Then, every $T\in \operatorname{Diff}(\mathcal V_X,\mathcal W_Y)$
 is locally of the form
$$
\sum_{\alpha\in\N^{\dim X}}h_\alpha(x,z)\frac{\partial^{|\alpha|}}{\partial x^\alpha}\qquad\mathrm{(finite}\,\,\, \mathrm{sum)},
$$
where $h_\alpha(x,z)$ are $\operatorname{Hom}(V,W)$-valued smooth functions on $Y$.

\item Let $i: Y\hookrightarrow X$ be an immersion. Choose an atlas of local coordinates $\{(y_i,z_j)\}$ on $X$ in
such a way that  $\{y_i\}$ form an atlas on $Y$. Then, every $T\in \operatorname{Diff}(\mathcal V_X,\mathcal W_Y)$
 is locally of the form
$$
\sum_{(\alpha,\beta)\in\N^{\dim X}}g_{\alpha\beta}(y)\frac{\partial^{|\alpha|+|\beta|}}{\partial y^\alpha\partial z^\beta}
\qquad\mathrm{(finite}\,\,\, \mathrm{sum)},
$$
where $g_{\alpha,\beta}(y)$ are $\operatorname{Hom}(V,W)$-valued smooth functions on $Y$.
\end{enumerate}
\end{example}

Next, suppose that the two
vector bundles $\mathcal V\to X$ and $\mathcal W\to Y$ are equivariant with respect
to a given Lie group $G$.
Then we have natural actions of $G$ on the Fr{\'e}chet spaces
 $C^\infty(X,\mathcal{V})$
and $C^\infty(Y,\mathcal{W})$ by translations.
Denote by $\operatorname{Hom}_G(C^\infty(X,\mathcal V),C^\infty(Y,\mathcal W))$
 the space of continuous $G$-homomorphisms.  
We set 
\begin{eqnarray}
\label{eqn:DVW}
\operatorname{Diff}_G(\mathcal V_X,\mathcal W_Y):=
\operatorname{Diff}(\mathcal V_X,\mathcal W_Y)\cap
\operatorname{Hom}_G(C^\infty(X,\mathcal V),C^\infty(Y,\mathcal W)).
\end{eqnarray}

\begin{example}
Suppose $X$ and $Y$ are both Euclidean vector spaces
 with an injective linear map $p: Y \hookrightarrow X$.  
If $G$ contains the subgroup of all translations of $Y$ then
 $\operatorname{Diff}_G(\mathcal V_X,\mathcal W_Y)$ is a subspace of the space of
differential operators with constant coefficients.
\end{example}

An analogous notion can be defined
 in the holomorphic setting.  
Let ${\mathcal {V}} \to X$ and ${\mathcal {W}} \to Y$
 be two holomorphic vector bundles with a holomorphic map $p:Y\to X$ between the complex base
manifolds $X$ and $Y$.  
We say a differential operator $T: C^\infty(X,\mathcal V)\to C^\infty(Y,\mathcal W)$ is {\it{holomorphic}} if
$$
T_U(\mathcal O(U,\mathcal V\vert_U))
 \subset\mathcal O(p^{-1}(U),\mathcal W\vert_{p^{-1}(U)})
$$
for any open subset $U$ of $X$. We denote by $\operatorname{Diff}^{\mathrm{hol}}(\mathcal V_X,\mathcal W_Y)$ the vector space of holomorphic differential operators.
When a Lie group $G$ acts biholomorphically on the two holomorphic vector bundles $\mathcal V\to X$ and $\mathcal W\to Y$, we set
$$
\operatorname{Diff}_G^{\mathrm{hol}}(\mathcal V_X,\mathcal W_Y):=
\operatorname{Diff}^{\mathrm{hol}}(\mathcal V_X,\mathcal W_Y)
\cap
\operatorname{Hom}_G(C^\infty(X,\mathcal V), C^\infty (Y,\mathcal W)).
$$
\subsection{Induced modules}
Let $\mathfrak g$ be a Lie algebra over $\C$, and $U(\mathfrak g)$ its universal enveloping algebra.
Let $\mathfrak h$ be a Lie subalgebra of $\mathfrak g$.

\begin{defn}
  For an
$\mathfrak h$-module $V$, we define the induced $U(\mathfrak g)$-module $\indhg(V)$ as
$$
\indhg(V):=U(\mathfrak g)\otimes_{U(\mathfrak h)}V.
$$
If $\mathfrak h$ is a Borel subalgebra and $\dim V = 1$,
then the $\mathfrak g$-module $\indhg(V)$ is the Verma module.
\end{defn}

For later purposes we formulate the following statement in terms of the contragredient representation $V^\vee$.
Let $\mathfrak h'$ be another Lie subalgebra of $\mathfrak g$.

\begin{prop} \label{prop:35}
For a
finite-dimensional  $\mathfrak h'$-module $W$ we have:
\begin{enumerate}
\item
$
\operatorname{Hom}_{\mathfrak g}(\operatorname{ind}_{\mathfrak h'}^{\mathfrak g}(W^\vee),\indhg(V^\vee))\simeq
\operatorname{Hom}_{\mathfrak h'}(W^\vee,\indhg(V^\vee)).
$
\item If $\mathfrak h'\not\subset\mathfrak h$, then $\operatorname{Hom}_{\mathfrak h'}(W^\vee,\indhg(V^\vee))=\{0\}$.

\end{enumerate}
\end{prop}
\begin{proof}
The first statement is due to the functoriality of the tensor product.

For the second statement it suffices to treat the case where $\mathfrak h'$ is one-dimensional.
Then the assumption $\mathfrak h'\not\subset\mathfrak h$ implies that 
$\mathfrak h'\cap\mathfrak h=\{0\}$, and therefore there is a direct sum decomposition
of vector spaces:
$$
\mathfrak g=\mathfrak h' +\mathfrak q+\mathfrak h,
$$
for some subspace $\mathfrak q$ in $\mathfrak g$.
We fix a basis $X_1,\cdots,X_n$ of $\mathfrak q$, and define a subspace
of $U(\mathfrak g)$ by
$$
U'(\mathfrak q):=\C\mhyphen\mathrm{span}\left\{X_1^{\alpha_1}\cdots X_n^{\alpha_n}:(\alpha_1,\cdots,\alpha_n)\in\N^n\right\}.
$$
 Then, by the Poincar\'e--Birkhoff--Witt theorem we have an isomorphism of $\mathfrak h'$-modules:
$$
\indhg(V^\vee)\simeq U(\mathfrak h')\otimes_\C U'(\mathfrak q)\otimes_\C V^\vee.
$$
In particular, $\indhg(V^\vee)$ is a free $U(\mathfrak h')$-module. Hence
there does not exist a non-zero
{\it{finite-dimensional}}  $\mathfrak h'$-submodule
in the $\mathfrak g$-module $\indhg(V^\vee)$.
\end{proof}

\begin{rem}\label{rem:231}
We shall see in Theorem \ref{thm:surject} that
$\dim_\C\operatorname{Hom}_{\mathfrak g'}(\indhgprime(W^\vee),
\indhg(V^\vee))$ is equal to the dimension of the space of differential symmetry breaking operators from $C^\infty(X,\mathcal V)$ to $C^\infty(Y,\mathcal W)$
 when $H'$ is connected.  
In \cite[Section \ref{sec:4}]{PART2}, we give a family of sextuples
$(\mathfrak g, \mathfrak g', \mathfrak h, \mathfrak h', V, W)$ such that
this dimension is one.  
\end{rem}

\subsection{Duality theorem for differential operators
 between two homogeneous spaces}\label{sec:33}

Let $G$ be a real Lie group, 
 and $\mathfrak g(\R):=\LieG$.  
We denote by $U(\mathfrak g)$
 the universal enveloping algebra
 of the complexified Lie algebra $\mathfrak g:=
\mathfrak g(\R)\otimes_\R\C$. Analogous notations will be applied to other
Lie groups.

Let $H$ be a closed subgroup of $G$. Given a finite-dimensional  representation $\lambda:H\to GL_\C(V)$ we define the homogeneous  vector bundle 
$\mathcal V_X\equiv\mathcal V:= G\times_HV$ over $X:=G/H$.
As a $G$-module, the space $C^\infty(X,\mathcal V)$ of smooth sections  is
identified with the following
subspace of $C^{\infty}(G,V) \simeq C^\infty(G)\otimes V$:
\begin{align*}
C^\infty(G,V)^H:=
&\{f\in C^{\infty}(G,V) : f(gh)=\lambda(h)^{-1}f(g) \,\mathrm{for}\,\mathrm{any}\,g\in G, h\in H\}
\\
\simeq &
\{F \in C^{\infty}(G) \otimes  V
  :
  \lambda(h) F(g h) =F(g)
\quad
\text{for any $g \in G$, $h \in H$}\}.  
\end{align*}

In dealing with a representation $V$ of a disconnected subgroup $H$
(\emph{e.g.} $H$ is a parabolic subgroup of a real reductive Lie group $G$),
we notice that the diagonal $H$-action on $U(\mathfrak{g}) \otimes_{\mathbb{C}} V^\vee$
defines a representation of $H$ on $\operatorname{ind}_{\mathfrak{h}}^{\mathfrak{g}}(V^\vee)=U(\mathfrak g)\otimes_{\mathfrak h} V^\vee$ and thus $\indhg(V^\vee)$ is endowed with a $(\mathfrak g, H)$-module structure.

\begin{thm}[Duality theorem]\label{thm:surject}
Let $H' \subset H$ be (possibly disconnected) closed subgroups
 of a Lie group $G$ with Lie algebras
$\mathfrak{h}' \subset \mathfrak{h}$,
respectively.
Suppose  
$V$ and $W$ are finite-dimensional  representations of 
$H$ and $H'$, respectively.
Let $G'$ be any subgroup of $G$ containing $H'$, and $\mathcal V_X:=G\times_{H} V$ and 
 $\mathcal W_Y:=G'\times_{H'} W$ be the corresponding
homogeneous
vector bundles.
 Then, there is a natural linear isomorphism:
\begin{equation}
\label{eqn:DXYH}
D_{X\to Y}:
\operatorname{Hom}_{H'}(W^\vee,\indhg(V^\vee))
\stackrel{\sim}{\longrightarrow}
\operatorname{Diff}_{G'}\left( \mathcal V_X,
\mathcal W_Y\right),
\end{equation}
or equivalently,
$$
D_{X\to Y}:
\operatorname{Hom}_{(\mathfrak g',H')}(\indhgprime(W^\vee),\indhg(V^\vee))
\stackrel{\sim}{\longrightarrow}
\operatorname{Diff}_{G'}\left( \mathcal V_X,
\mathcal W_Y\right).
$$
For $\varphi \in \operatorname{Hom}_{H'}(W^\vee,\indhg(V^\vee))$
 and $F \in C^{\infty}(X, {\mathcal{V}}) \simeq C^{\infty}(G,V)^H$, 
 $D_{X \to Y}(\varphi)F \in C^{\infty}(Y, {\mathcal{W}}) \simeq C^{\infty}(G',W)^{H'}$
 is given by the following formula:
\begin{equation}
\label{eqn:Dxy}
\langle 
D_{X \to Y}(\varphi)F, 
w^{\vee}
\rangle
=
\sum_j
\langle 
d R(u_j)F
, 
v_j^{\vee}
\rangle
|_{G'}
\quad
\text{for $w^{\vee} \in W^{\vee}$}, 
\end{equation}
where 
 $\varphi(w^{\vee})=\sum_j u_j v_j^{\vee}$
 $\in \indhg(V^\vee)$
 ($u_j \in U({\mathfrak {g}})$, $v_j^{\vee} \in V^{\vee}$).  

\end{thm}

When $H'$ is connected,
 we can write the left-hand side
 of \eqref{eqn:DXYH}
 by means
 of Lie algebras.  
\begin{cor}
\label{cor:2.9}
Suppose we are in the setting  of
 Theorem \ref{thm:surject}.  
Assume
 that $H'$ is connected.  
Then there is a natural linear isomorphism:
\begin{eqnarray}\label{eqn:isoo}
D_{X\to Y}:\operatorname{Hom}_{\mathfrak h'}(W^\vee,\indhg(V^\vee))
\stackrel{\sim}{\longrightarrow}
\operatorname{Diff}_{G'}( \mathcal V_X,\mathcal W_Y), 
\end{eqnarray}
or equivalently, 
\begin{equation*}
D_{X\to Y}:\operatorname{Hom}_{\mathfrak g'}(\operatorname{ind}_{\mathfrak{h}'}^{\mathfrak{g}'}(W^\vee),\indhg(V^\vee))
\stackrel{\sim}{\longrightarrow}
\operatorname{Diff}_{G'}( \mathcal V_X,\mathcal W_Y).
\leqno{(\ref{eqn:isoo})'}
\end{equation*}
\end{cor}

The construction of $D_{X \to Y}$
and the fact that the formula \eqref{eqn:Dxy} is well-defined 
 will be explained in Section \ref{subsec:DXY}.

\begin{rem}
\begin{enumerate}
\item Corollary \ref{cor:2.9} is known when $X=Y$,
 \emph{i.e.} $G' = G$ and $H' = H$, especially in the setting
 of complex flag varieties, 
 see \emph{e.g.} \cite{Ks74, HJ}.
\item 
 When ${\mathfrak {g}}'$ is a reductive subalgebra
 and ${\mathfrak {h}}'$ is a parabolic subalgebra,
 the existence
 of an ${\mathfrak {h}}'$-module $W$
 for which the left-hand side of \eqref{eqn:isoo}
 is non-zero,
 is closely
 related to the \lq\lq{discretely decomposability}\rq\rq\
 of the ${\mathfrak {g}}$-module
 $\indhg (V^{\vee})$
 when restricted to the subalgebra ${\mathfrak {g}}'$
 (\cite[Part III]{kdeco98}, \cite{K12}).  This relationship
 will be used in Section \ref{subsec:4.2} in proving that any continuous
 symmetry breaking operator in a holomorphic setting is given by 
 a differential operator (\emph{localness theorem}).

\item
Owing to Proposition \ref{prop:35},
 the left-hand side 
of (\ref{eqn:isoo})$'$
 is non-zero
 only when ${\mathfrak {h}}'
\subset {\mathfrak {h}}$.  
Conversely,
 if $H' \subset H \cap G'$, 
then there is a natural morphism
 $Y=G'/H' \to X = G/H$
 and therefore
 \lq\lq{differential operators}\rq\rq\
 (in the sense of Definition \ref{def:21}) from 
 $C^{\infty}(X, {\mathcal{V}})$ to $C^{\infty}(Y, {\mathcal{W}})$
 are defined.  
\item
We shall consider 
the case
 where $H'=H \cap G'$
 in later applications,
 however,
 Theorem \ref{thm:surject} 
 also covers the cases
 where the natural morphism $Y\to X$ is not injective, 
 \emph{i.e.} where $H'\subsetneqq H \cap G'$.

\end{enumerate}
\end{rem}

An analogous result to 
Theorem \ref{thm:surject} holds in the holomorphic setting as well. 
To be precise,
 let $G_{\C}$ be a complex Lie group, $G_{\C}'$, $H_{\C}$
 and $H_{\C}'$ be closed complex subgroups
 such that $H_{\C}'\subset H_{\C} \cap G_{\C}'$. 
We write ${\mathfrak {g}}$, ${\mathfrak {h}}$, 
 $\ldots$
 for the Lie algebras
 of the complex Lie groups $G_{\C}$, $H_{\C}$, $\ldots$, respectively.
Given finite-dimensional holomorphic representations $V$ of $H_{\C}$
 and $W$ of $H_{\C}'$,
 we form holomorphic vector bundles $\mathcal V:= G_{\C}\times_{H_\C} V$ over
$X_{\C}=G_{\C}/H_{\C}$ and $\mathcal W:= G_{\C}'\times_{H_{\C}'}W$ over $Y_{\C}=G_{\C}'/H_{\C}'$.

For simplicity,
 we assume
 that $H_{\C}'$ is connected.  
(This is always the case
 if $G_{\C}'$ is a connected complex reductive Lie group 
 and $H_{\C}'$ is a parabolic subgroup of $G_{\C}'$.)
Then we have:
\begin{thm}[Duality theorem in the holomorphic setting]\label{prop:2.10} 
There is a canonical linear isomorphism:
$$
D_{X\to Y}:\operatorname{Hom}_{\mathfrak g'}\left(
\mathrm{ind}_{\mathfrak h'}^{\mathfrak g'}(W^\vee),
\mathrm{ind}_{\mathfrak h}^{\mathfrak g}(V^\vee)\right)
\stackrel{\sim}{\longrightarrow}
\operatorname{Diff}_{G_{\C}'}^{\mathrm{hol}}(\mathcal V_{X_{\C}},\mathcal W_{Y_{\C}}).
$$
\end{thm}

Suppose furthermore
 that $G$, $G'$, $H$ and $H'$ are real forms
 of the complex Lie groups
 $G_{\C}$, $G_{\C}'$, $H_{\C}$
 and $H_{\C}'$, 
 respectively.  
We regard $V$ and $W$ as $H$- and $H'$-modules
 by the restriction,
 and form vector bundles 
 ${\mathcal{V}}=G \times_H V$
 over $X=G/H$
 and ${\mathcal{W}}=G' \times_{H'} W$
 over $Y=G'/H'$.

We ask whether or not all symmetry breaking operators
 have holomorphic extensions.  
Here is a simple sufficient condition:
\begin{cor}
\label{cor:holoreal}
If $H'$ is contained 
 in the connected complexification $H_{\C}'$, 
 then we have a natural bijection:
\[
\operatorname{Diff}_{G_{\C}'}^{\operatorname{hol}}({\mathcal{V}}_{X_{\C}},{\mathcal{W}}_{Y_{\C}})
 \overset \sim \to 
\operatorname{Diff}_{G'}({\mathcal{V}}_{X},{\mathcal{W}}_{Y}).  
\]
\end{cor}
\begin{proof}
Comparing Theorem \ref{thm:surject} with Theorem \ref{prop:2.10}, 
 the proof of Corollary \ref{cor:holoreal} reduces
 to the surjectivity
 of the inclusion
\begin{equation}
\label{eqn:gHhom}
\operatorname{Hom}_{({\mathfrak {g}'},H')}
({\operatorname{ind}}_{{\mathfrak {h}'}}^{{\mathfrak {g}'}}(W^{\vee}), 
{\operatorname{ind}}_{{\mathfrak {h}}}^{{\mathfrak {g}}}(V^{\vee})
)
\hookrightarrow
\operatorname{Hom}_{{\mathfrak {g}'}}
({\operatorname{ind}}_{{\mathfrak {h}'}}^{{\mathfrak {g}'}}(W^{\vee}), 
{\operatorname{ind}}_{{\mathfrak {h}}}^{{\mathfrak {g}}}(V^{\vee})
).  
\end{equation}
We note that 
$\operatorname{Hom}_{({\mathfrak {g}'},{H_{\C}'})}
({\operatorname{ind}}_{{\mathfrak {h}'}}^{{\mathfrak {g}'}}(W^{\vee}), 
{\operatorname{ind}}_{{\mathfrak {h}}}^{{\mathfrak {g}}}(V^{\vee})
)$
 is a subspace of the left-hand side
 of \eqref{eqn:gHhom}
 because $H' \subset H_{\C}'$, 
 whereas it coincides with the right-hand side
 of \eqref{eqn:gHhom}
 if $H_{\C}'$ is connected.  
Hence \eqref{eqn:gHhom} is surjective.  
Thus Corollary \ref{cor:holoreal} is proved.  
\end{proof}

The rest of this section
 is devoted to the proof 
 of Theorem \ref{thm:surject}.  
For Theorem \ref{prop:2.10}, since the argument
 is parallel to that of Theorem \ref{thm:surject}, 
 we omit the proof.

\subsection{Construction 
 of $D_{X \to Y}$}
\label{subsec:DXY}

This subsection gives the definition
 of the linear map $D_{X \to Y}$
 in Theorem \ref{thm:surject}.

Consider two actions $dR$ and $dL$ of the universal enveloping algebra $U(\mathfrak g)$ on the space 
$C^\infty (G)$ of smooth complex-valued functions on $G$ induced by the regular representation $L\times R$ of $G\times G$ on $C^\infty(G)$:
\begin{equation}\label{eqn:dR}
 (dR(Z)f)(x):=\left.\frac {d}{dt}\right\vert_{t=0}f(xe^{tZ})\quad {\mathrm {and}}\quad
 (dL(Z)f)(x):=\left.\frac {d}{dt}\right\vert_{t=0}f(e^{-tZ}x),
\end{equation}
for $Z\in\mathfrak g(\R)$.

The right differentiation (\ref{eqn:dR}) defines a bilinear map
\begin{equation*}
\Phi:C^\infty(G)\times U(\mathfrak g)\to C^\infty(G),
\quad
(F,u)\mapsto dR(u)F,
\end{equation*}
with the following properties
\begin{eqnarray}\label{eqn:Phi1}
\Phi(L(g)F,u)&=&L(g)\Phi(F,u),\\
\Phi(F,u'u)&=&dR(u')\Phi(F,u)\label{eqn:Phi2},
\end{eqnarray}
for any $g\in G$ and $u,u'\in U(\mathfrak g)$.

Combining $\Phi$ with the canonical pairing $V \times V^\vee \to \mathbb{C}$, we obtain a bilinear map
$$
\Phi_V:C^\infty(G)\otimes V\times U(\mathfrak g)\otimes_\C V^\vee\to
C^\infty(G).
$$
Then we have the following:
\begin{lem}\label{lem:GHbilinear}
The map $\Phi_V$ induces
a well-defined diagram of maps:
$$
\begin{matrix}
C^\infty(G)\otimes V\times U(\mathfrak g)\otimes_\C V^\vee
&
\overset{\Phi_V}
{\longrightarrow}
&
C^\infty(G)\\
\lhookuparrow\qquad\qquad\twoheaddownarrow&
&
\parallel\\
C^\infty(X,\mathcal V)\times \indhg(V^\vee)
&
\, \rule[.55ex]{.23em}{.5pt} \, \rule[.55ex]{.23em}{.5pt} \!\dashrightarrow&
C^\infty(G).
\end{matrix}
$$
\end{lem}
\begin{proof}
Denote by $\lambda^\vee$ the contragredient representation of the representation $(\lambda,V)$ of $H$, and
by $d\lambda^\vee$ the infinitesimal representation of $\mathfrak h$.
The kernel of the natural quotient map $U(\mathfrak g)\otimes_\C V^\vee\to\indhg(V^\vee)$ is generated by
$$
-uY\otimes v^\vee+u\otimes d\lambda^\vee(Y)v^\vee
$$
with $u\in U(\mathfrak g), Y\in\mathfrak h$ and $v^\vee\in V^\vee$. Hence it suffices to show 
$$
\Phi_V(f,-uY\otimes v^\vee+u\otimes d\lambda^\vee(Y)v^\vee)=0 
$$
for any $f\in
 C^\infty(X,\mathcal V)\simeq
C^\infty(G,V)^H$.

Since $f \in C^\infty(G,V)^H$ satisfies $dR(Y)f=-d\lambda(Y)f$ for $Y\in\mathfrak h$, we
have
\begin{eqnarray*}
\Phi_V(f,uY\otimes v^\vee)&=&\langle dR(u)dR(Y)f,v^\vee\rangle\\
&=&\langle dR(u)f,d\lambda^\vee(Y)v^\vee\rangle\\
&=&\Phi_V(f,u\otimes d\lambda^\vee(Y)v^\vee\rangle.
\end{eqnarray*}
Thus the lemma is proved.
\end{proof}

\begin{lem}\label{lem:indPhi}
{\rm{1)}} The bilinear map 
\begin{equation}
\label{eqn:CVind}
     C^\infty(X,\mathcal V)\times\indhg(V^\vee)\to \C,
 (f,m)\mapsto\Phi_V(f,m)(e)
\end{equation}
is $({\mathfrak {g}}, H)$-invariant.  

{\rm{2)}}
 If $m\in\indhg(V^\vee)$ satisfies $\Phi_V(f,m)(e)=0$ for all $f\in C^\infty(X,\mathcal V)$ then $m=0$.
\end{lem}
 \begin{proof} 
1)
Let $f \in C^{\infty}(X, {\mathcal{V}})$
 and $m \in \indhg(V^{\vee})$.  
It follows from \eqref{eqn:Phi1} and \eqref{eqn:Phi2}
 that
\begin{align*}
\Phi_V(dL(Z)f, m)=&dL(Z) \Phi_V(f,m)
\\
\Phi_V(f,Zm)=&dR(Z) \Phi_V(f,m)
\end{align*}
for any $Z \in {\mathfrak {g}}$.  
Since
\[
 (dL(Z)+dR(Z))F(e)=0
\]
 for any $F \in C^{\infty}(G)$, 
 we have shown 
 the ${\mathfrak {g}}$-invariance
 of the bilinear map \eqref{eqn:CVind}:
\[
 \Phi_V(dL(Z)f,m)(e)+\Phi_V(f,Zm)(e)=0. 
\]
The proof for the $H$-invariance of \eqref{eqn:CVind} is similar.

2)
 Take a basis $\{v_1,\cdots, v_k\}$ of $V$, and let $\{v_1^\vee,\cdots,v_k^\vee\}$ be the dual basis in $V^\vee$. Choose a complementary subspace $\mathfrak q$ of $\mathfrak h$ in $\mathfrak g$, and fix a basis 
 $\{X_1,\cdots,X_n\}$ of $\mathfrak q$. Then by the Poincar\'e--Birkhoff--Witt theorem, we can write $m \in \operatorname{ind}_{{\mathfrak {h}}}^{{\mathfrak {g}}}(V^{\vee})$ as a finite sum:
 $$
 m=\sum_{j=1}^k\sum_{\alpha=(\alpha_1,\cdots,\alpha_n)}a_{\alpha,j}
 X_1^{\alpha_1}\cdots X_n^{\alpha_n}v_j^\vee.
 $$
If $m$ is non-zero,
 we can find a multi-index $\beta$ and $j_o\,
 (1\leq j_o\leq k)$ 
 such that $a_{\beta,j_o}\not=0$
 and that $a_{\alpha,j_0}=0$ for any multi-index $\alpha$
 satisfying $\vert\alpha\vert>\vert\beta\vert$ and for any $j$. 
 Here $\vert\alpha\vert=\sum_{i=1}^n\alpha_i$ for $\alpha\in\N^n$.
 We take
 $f\in C^\infty(G,V)^H\simeq C^\infty(X,\mathcal V)$ such that $f$ is given in a right $H$-invariant neighborhood of $H$ in $G$ by
 $$
 f\left(\exp\left(\sum_{i=1}^n x_iX_i\right)h\right)=
 x^\beta\lambda(h)^{-1}v_{j_o}\quad\mathrm{for}\, x=(x_1,\cdots,x_n)\in\R^n\,\mathrm{and}\,\, h\in H.  
 $$
Then $\Phi_V(f,m)(e)=a_{\beta,j_o}\beta_1!\cdots\beta_k!\neq0$. 
The contraposition completes the proof.
 \end{proof}

We regard $C^\infty (G)$ as a $G \times \mathfrak{g}$-module via the $(L\times dR)$-action.
Then the space $\operatorname{Hom}_G(C^\infty(X,\mathcal V), C^\infty(G))$ of continuous $G$-homomorphisms becomes a $\mathfrak g$-module
 by the remaining $dR$-action on the target space. 
By \eqref{eqn:Phi1},
\eqref{eqn:Phi2} and Lemma \ref{lem:GHbilinear},
we get the following $\mathfrak g$-homomorphism:
\begin{equation}\label{eqn:1234}
\indhg(V^\vee)\To\operatorname{Hom}_G(C^\infty(X,\mathcal V),
C^\infty(G)), 
\quad
u \otimes v^{\vee}
 \mapsto 
 (f \mapsto \langle d R(u)f, v^{\vee}\rangle).  
\end{equation}

Furthermore,
 it is actually a $({\mathfrak {g}},H)$-homomorphism,
 where the group $H$ acts 
 on $\indhg(V^{\vee})=U({\mathfrak {g}}) \otimes_{U({\mathfrak {h}})} V^{\vee}$
 diagonally
 and acts on ${\operatorname{Hom}}_G(C^{\infty}(X,{\mathcal{V}}), C^{\infty}(G))$
 via the $R$-action on $C^{\infty}(G)$.

Let $H'$ be a connected closed Lie subgroup of $G$. 
Given a finite-dimensional  representation $W$ of $H'$, 
 we form a
homogeneous vector bundle $\mathcal W_Z\equiv\mathcal W:=G\times_{H'}W$ over $Z:=G/H'$.

Taking 
the tensor product of the $({\mathfrak {g}},H)$-modules
 in (\ref{eqn:1234})
 with the $H'$-module $W$,
 we get an $(H'\times(\mathfrak g,H))$-homomorphism:
\begin{equation*}
\operatorname{Hom}_\C(W^\vee,\indhg(V^\vee))\To
\operatorname{Hom}_G(C^\infty(X,\mathcal V),
C^\infty(G,W)).
\end{equation*}

Let $\Delta(H')$ be a subgroup of $H'\times H$
 defined by $\{(h,h):h \in H'\}$. 
Taking $\Delta(H')$-invariants, 
 we obtain the following $\C$-linear map:
\begin{equation}\label{eqn:morf}
\operatorname{Hom}_{H'}(W^\vee,\indhg(V^\vee))\To
\operatorname{Hom}_G(C^\infty(X,\mathcal V),
C^\infty(Z,\mathcal W)),\quad \varphi\mapsto D_\varphi,
\end{equation}
where $D_{\varphi}$ satisfies
\begin{equation}\label{eqn:Dfw}
\langle D_{\varphi}f, w^{\vee}\rangle
=
\Phi_V(f, {\varphi}(w^{\vee}))
\end{equation}
for any $f \in C^{\infty}(X,{\mathcal{V}})$
 and any $w^{\vee} \in W^{\vee}$.
\begin{rem}
If $H'$ is connected,
 then we can replace
 $\operatorname{Hom}_{H'}$
 by $\operatorname{Hom}_{\mathfrak{h}'}$
 in \eqref{eqn:morf}.  
\end{rem}
\begin{lem}\label{lem:Ding}
 The map \eqref{eqn:morf} is injective.
\end{lem}
\begin{proof}
By \eqref{eqn:Dfw}, Lemma \ref{lem:Ding} is derived from the second statement
 of the Lemma \ref{lem:indPhi}.
 \end{proof}
\vskip7pt

Take any subgroup $G'$ of $G$ containing $H'$ and form a homogeneous vector bundle $\mathcal W_Y:= G'\times_{H'} W$ over $Y=G'/H'$. 
Then,
 the vector bundle $\mathcal W_Y$ is isomorphic to the restriction $\mathcal W_Z\vert_Y$ of the
vector bundle $\mathcal W_Z$ to the submanifold $Y$ of the base space $Z$.
Let 
\[
     R_{Z\to Y}:C^\infty(Z,\mathcal W_Z)\to C^\infty(Y,\mathcal W_Y)
\] 
be the restriction map of sections.
For $\varphi\in \operatorname{Hom}_{\mathfrak h'}(W^\vee,\indhg(V^\vee))$ we set
\begin{equation}\label{eqn:DXY}
D_{X\to Y}(\varphi):=R_{Z\to Y}\circ D_\varphi.
\end{equation}
Then $D_{X\to Y}(\varphi):C^\infty(X,\mathcal V)\to C^\infty(Y,\mathcal W)$
 is a $G'$-equivariant differential operator, \emph{i.e.} $D_{X\to Y}$ defines
 a linear map 
 $ \operatorname{Hom}_{\mathfrak h'}(W^\vee,\indhg(V^\vee))\to
  \operatorname{Diff}_{G'}(\mathcal V_X, \mathcal W_Y)$.  
Theorem \ref{thm:surject} describes explicitly the image $D_{X\to Y}$
 when $H' \subset H \cap G'$, 
namely, 
when the following diagram exists:
$$
\xymatrix{
&Z=G/H'\ar@{->>}[dr]&\\
Y=G'/H'\ar@{^{(}->}[ur]\ar[rr]&{}&X=G/H
}.
$$

\begin{rem}
The left-hand side of \eqref{eqn:isoo} does not depend on the choice of $G'$. This fact is reflected by the 
 commutativity of the following diagram.
\begin{equation}\label{eqn:DXYdiag}
\xymatrix{
\operatorname{Hom}_{\mathfrak h'}( W^\vee,\indhg(V^\vee))\ar[dr]_{D_{X\to Y}}^\sim
\ar[rr]^{{\quad\sim}}&{}&
\operatorname{Diff}_{G}(\mathcal V_X,{\mathcal W_Z}\ar[dl]^{R_{Z\to Y}})\\
& \operatorname{Diff}_{G'}(\mathcal V_X,{\mathcal W_Y})&
}
\end{equation}
\end{rem}

\subsection{Proof of Theorem \ref{thm:surject}}

We have already seen in Lemma \ref{lem:Ding}
 that $D_{X \to Y}$ is injective.  
In order to prove the surjectivity of the linear map $D_{X\to Y}$, 
 we realize the induced $U(\mathfrak g)$-module $\indhg(V^\vee)$
in the space of distributions.

We recall that $\mathcal V^*=\mathcal V^\vee\otimes\Omega_X$
is the dualizing bundle of a vector bundle $\mathcal V$ over $X$.
For a closed subset $S$ and an open subset $U$ in $X$ containing $S$,
we write $\mathcal{D}'_S(U,\mathcal{V}^*)$
for the space of $\mathcal V^*$-valued distributions on $U$
with support in  $S$. 
Obviously, $\mathcal{D}'_S(U,\mathcal{V}^*)=\mathcal{D}'_S(X,\mathcal{V}^*)$. If $S$ is compact, then $\mathcal{D}'_S(U,\mathcal{V}^*)$
is contained in the space $\mathcal{E}'(U,\mathcal{V}^*)$
of distributions on $U$ with compact support,
 and thus coincides with
$\mathcal{E}'_S(U,\mathcal{V}^*):=\mathcal{D}'_S(U,\mathcal{V}^*)
\cap \mathcal{E}'(U,\mathcal{V}^*)$.

We return to the setting of Theorem \ref{thm:surject}, where $\mathcal V$ is a $G$-equivariant vector bundle over $X=G/H$. Then the
 Lie group $G$ acts on $C^\infty(X,\mathcal V)$ and $\mathcal E'(X,\mathcal V^*)$ by the pull-back of smooth sections and distributions, respectively. The infinitesimal action defines representations of the Lie algebra $\mathfrak g$ on  $C^\infty(U,\mathcal V)$ and
 $\mathcal{E}'_S(U,\mathcal{V}^*)$.

The ``integration map" \eqref{eqn:intE41}
\begin{equation}\label{eqn:intE13}
 \mathcal E'(X,\Omega_X)\to\C,\qquad
\omega\mapsto \int_X\omega
\end{equation}
 is $G$-invariant. 
Composing this with the 
$G$-invariant bilinear map ({\it{contraction}}):
$$
C^\infty(X,\mathcal V)\times \mathcal E'(X,\mathcal V^*)\longrightarrow
\mathcal E'(X,\Omega_X),\qquad (f,h)\mapsto\langle f,h\rangle,
$$
we obtain the following $G$-invariant bilinear form 

\begin{equation}
\label{eqn:Gbilinear}
C^\infty(X,\mathcal V)\times \mathcal E'(X,\mathcal V^*)\To\C,
\qquad
(f,h)\mapsto \int_X\langle f,h\rangle.
\end{equation}

Similarly, we obtain the following local version:

\begin{lem}\label{lem:EHom}
Let $S$ be a closed subset of $X$ and $U$ an open neighborhood of $S$ in $X$. Then, 
we have the natural $\mathfrak g$-invariant bilinear form:
$$
C^\infty(U,\mathcal V)\times \mathcal E'_S(U,\mathcal V^*)\To\C,
\qquad
(f,h)\mapsto \int_U\langle f,h\rangle.
$$
Moreover,
 if $S \subset U$ are both $H$-invariant subsets in $X$, 
 then the bilinear form is also $H$-invariant.  
\end{lem}

We write $o=eH\in X$ for the origin. 
By Lemmas \ref{lem:indPhi} and \ref{lem:EHom}, 
 we have obtained two $(\mathfrak g,H)$-invariant pairings:
\begin{eqnarray*}
C^\infty(X,\mathcal V)\times \indhg(V^\vee)&\To&\C, \qquad
(f,m)\mapsto \Phi_V(f,m)(e),\\
C^\infty(X,\mathcal V)\times\mathcal E'_{\{o\}}(X,\mathcal V^*)&\To&\C
\qquad (f,h)\mapsto \int_X\langle f,h\rangle.
\end{eqnarray*}

Let us show that there is a natural $(\mathfrak g,H)$-isomorphism
between $\indhg(V^\vee)$ and $\mathcal E'_{\{o\}}(X,\mathcal V^*)$.  
In fact,
 it follows from Lemma \ref{lem:indPhi}
that there exists an injective $(\mathfrak g,H)$-homomorphism
\[
   A:\indhg (V^{\vee}) \to {\mathcal{E}}_{\{o\}}(X,{\mathcal{V}}^{\ast})
\]
such that 
\[
\Phi_V(f,m)(e)=\int_X \langle f, A(m) \rangle 
\quad
\text{for all }
m \in \indhg (V^{\vee})
\text{ and }
f \in C^{\infty}(X,{\mathcal{V}}).  
\]

For a homogeneous vector bundle $\mathcal V=G \times_H V$ we define
 a vector-valued Dirac $\delta$-function
$
\delta\otimes v^\vee\in \mathcal E'_{\{o\}}(X,\mathcal V^\ast)$,
for $v\in V^\vee$
by
\begin{equation}\label{eqn:deltaV}
\langle f,\delta\otimes v^\vee\rangle:=\langle f(e),v^\vee\rangle
\quad
\text{for $f\in  C^\infty(X,\mathcal V)\simeq C^\infty(G,V)^H$}.
\end{equation}
By the definition of $\Phi_V$, 
 we have
\[
  \Phi_V(f, 1 \otimes v^{\vee})(e)=\langle f(e), v^{\vee}\rangle.  
\]
Hence $A(1 \otimes v^{\vee})=\delta \otimes v$
 by \eqref{eqn:deltaV}.  
Since $A$ is a ${\mathfrak {g}}$-homomorphism,
 we have shown that
\[
   A(u \otimes v^{\vee})
  =
   dL(u) (\delta \otimes v^{\vee}), 
\qquad
 \text{for }u \in U({\mathfrak {g}}),  v \in V^{\vee}.  
\]
\begin{lem}\label{lem:indE}
The $({\mathfrak {g}},H)$-homomorphism
\begin{equation}\label{eqn:isoMh}
A:\indhg(V^\vee){\stackrel{}{\To}}\mathcal E'_{\{o\}}(X,\mathcal V^*),
\qquad u\otimes v^\vee\mapsto dL(u)\left(\delta\otimes v^\vee\right),
\end{equation}
is bijective.
\end{lem}
\begin{proof}
By Lemma \ref{lem:indPhi} the map \eqref{eqn:isoMh} is injective. Let us show that it is also surjective. By the structural theorem of (scalar-valued) 
distributions \cite[Chapter III, Th\'eor\`eme XXXVII]{SCH},
distributions supported on the singleton
$\{o\}$ are obtained as a finite sum of derivatives of the Dirac's delta
function. An analogous statement holds for vector-bundle valued 
distributions supported on 
$\{o\}$, as we can see by trivializing the bundle near the origin $o$. 
Choose a complementary subspace $\mathfrak q(\R)$ of $\mathfrak h(\R)= \mathrm{Lie} (H)$ in $\mathfrak g(\R)=\mathrm{Lie}(G)$.
Since
$dL(Z)\, (Z\in\mathfrak q(\R))$ spans the tangent space $T_o(G/H)\simeq\mathfrak q(\R)$, 
 any derivative of the vector-valued Dirac's delta function
 is given as a linear combination of elements of the form
$dL(u)(\delta\otimes v^\vee)\, (u\in U(\mathfrak g), v^\vee\in V^\vee)$. Thus the map \eqref{eqn:isoMh} is surjective.
\end{proof}

Let
 ${\mathbb{C}}_{2\rho}$ denote the one-dimensional representation
 of $H$ defined
 by
\[
  h \mapsto \vert \det (\operatorname{Ad}_{G/H}(h):
{\mathfrak {g}}/{\mathfrak {h}}
\to
{\mathfrak {g}}/{\mathfrak {h}})
 \vert^{-1}.
\]
If $H$ is a parabolic subgroup of $G$ with Langlands decomposition $P=MAN_+$
 then the infinitesimal representation of $\C_{2\rho}$ is given
 by the sum of the roots for ${\mathfrak {n}}_+=\mathrm{Lie}(N_+)$.
The bundle of densities $\Omega_{G/H}$ is given as a
$G$-equivariant line bundle,
$$
   \Omega_{G/H}
     \simeq G\times_H\vert \det\nolimits^{-1}\operatorname{Ad}_{G/H}
     \vert\simeq G\times_H\C_{2\rho}.
$$
For an $H$-module $(\lambda,V)$,
we define a ``twist" of the contragredient representation
 $\lambda_{2\rho}^{\vee}$ 
 on the dual space $V^{\vee}$ 
 (or simply denoted by $V_{2\rho}^{\vee}$) by
 the formula
\[
 \lambda^*\equiv \lambda_{2\rho}^{\vee}:=\lambda^{\vee} \otimes {\mathbb{C}}_{2\rho}
                         =\lambda^{\vee}\otimes\vert\det{}^{-1}\operatorname{Ad}_{G/H}\vert.
\]
Then the dualizing bundle $\mathcal V^*=\mathcal V^\vee\otimes\Omega_{G/H}$
 of the vector bundle $\mathcal V=G \times_H V$ is given, as a homogeneous vector bundle, by:
\begin{equation}\label{eqn:V2rho}
\mathcal V^*\equiv\mathcal V^\vee_{2\rho}\simeq G\times_HV_{2\rho}^\vee.
\end{equation}
Then ${\mathcal{D}}'(X,{\mathcal{V}}^{\ast})$ is identified
 with 
\[
  ({\mathcal{D}}'(G) \otimes V_{2 \rho}^{\vee})^{\Delta(H)}
  =
  \{F \in {\mathcal{D}}'(G) \otimes V^{\vee}:
    \lambda_{2 \rho}^{\vee}(h)F(\cdot h)=F(\cdot)
   \quad
   \text{for any }h \in H
  \}.  
\]

\vskip5pt

Now let us consider the setting of Theorem \ref{thm:surject}
 where we have a $G'$-equivariant 
(but not necessarily injective) morphism from $Y=G'/H'$ to $X=G/H$.
\begin{lem}\label{lem:A}
Suppose that $G'$ is a subgroup of $G$.
Then the multiplication map 
$$
m:G\times G'\to G,\quad (g,g')\mapsto (g')^{-1}g,
$$
 induces the isomorphism:
$$
m^*:\left(\mathcal D'(X,\mathcal V^\ast)\otimes W\right)^{\Delta(H')}
\stackrel{\sim}{\To}
\mathcal D'(X\times Y,\mathcal V^\ast \boxtimes \mathcal W)^{\Delta(G')}.
$$
\end{lem}
\begin{proof}
The image of the pull-back
 $m^*: \mathcal D'(G)\to\mathcal D'(G\times G')$ is $\mathcal D'(G\times G')^{\Delta(G')}$, where $G'$ acts diagonally from the left.
 Thus, considering the remaining $G\times G'$ action from the right,
 we take  $H\times H'$-invariants with respect to the diagonal action in
 the $(G\times G'\times H\times H')$-isomorphism:
$$
m^*\otimes\mathrm{id}\otimes\mathrm{id}:
\mathcal D'(G)\otimes V_{2\rho}^\vee\otimes W
\stackrel{\sim}{\To}\mathcal D'(G\times G')^{\Delta(G')}\otimes V_{2\rho}^\vee\otimes W,
$$
and therefore we get the lemma.
\end{proof}

We recall from Section \ref{sec:21} that any continuous linear map
$T:C^\infty(X,\mathcal V)\to C^\infty(Y,\mathcal W)$ is given by a unique distribution kernel $K_T\in\mathcal D'(X\times Y,\mathcal V^*\boxtimes\mathcal W)$.
The following lemma gives
 a necessary and sufficient condition
 on the distribution $K_T$ 
for the linear map $T$ to be  a $G'$-equivariant differential operator.  

\begin{lem}\label{lem:312}
There is a natural linear isomorphism:
\begin{eqnarray}\label{eqn:KT}
\operatorname{Diff}_{G'}(\mathcal V_X,\mathcal W_Y)
\stackrel{\sim}{\To}
(\mathcal D'_{\{o\}}(X,\mathcal V^*)\otimes W)^{\Delta(H')},\quad
T\mapsto (m^*)^{-1}(K_T).
\end{eqnarray}
\end{lem}
\begin{proof}
First, we show that the map \eqref{eqn:KT} is well-defined. Suppose
 $T\in
\operatorname{Diff}_{G'}( \mathcal V_X,\mathcal W_Y)$.
Since $K_T$ is uniquely determined by $T$,  the operator $T$ is $G'$-equivariant, \emph{i.e.} $L(g)\circ T\circ L(g^{-1})=T$ for all $g\in G'$
if and only if $K_T\in\mathcal D'(X\times Y, \mathcal V^*\boxtimes\mathcal W)^{\Delta(G')}.$
By Lemma \ref{lem:1} the distribution kernel $K_T$ is supported
on the diagonal set $\Delta(Y)=\{(p(y),y):y\in Y\}\subset X\times Y$. Via the bijection $m^*$ given in Lemma \ref{lem:A} we thus have
$$
\operatorname{Supp}((m^*)^{-1}K_T) \subset\{o\}.
$$
Hence the map \eqref{eqn:KT} is well-defined. The injectivity of 
\eqref{eqn:KT} is clear.

Conversely, take any element $k\in (\mathcal D'_{\{o\}}(X,\mathcal V^*)\otimes W)^{\Delta(H')}$.
We set $K:=m^*(k)\in \mathcal D'(X\times Y, \mathcal V^*\boxtimes\mathcal W)^{\Delta(G')}$, and define a linear map
$$
T:C^\infty(X,\mathcal V)\To\mathcal D'(Y,\mathcal W),\quad f\mapsto
\int_Xf(x)K(x,\cdot).
$$
Then $T$ is $G'$-equivariant because $K$ is $\Delta(G')$-invariant.

Let us show that $Tf\in C^\infty(Y,\mathcal W)$ for any $f\in C^\infty(X,\mathcal V)$. To see this, we take neighborhoods $U, U'$ and $U''$ of $x_o=p(y_o)$ in $X$, $y_o$ in $Y$, and $e$ in $G'$, respectively, such that $gU'\subset p^{-1}(U)$ for any $g\in U''$. 
Since the kernel $K$ is supported on the diagonal set $\Delta(Y)$, 
$TF\vert_{U'}$ is locally of the form \eqref{eqn:halpha} as in the proof of Lemma \ref{lem:1}.

Since $T$ is $G'$-equivariant, we have
$$
\sum_\alpha h_\alpha(y)\frac{\partial^{\vert\alpha\vert}f}{\partial x^\alpha}\left(gp(y)\right)=\sum_\alpha h_\alpha(gy)
\frac{\partial^{\vert\alpha\vert}f}{\partial x^\alpha}(p(y)),
$$
for any $y\in U', g\in U''$, and $f\in C^\infty(U)\otimes V$. 
By taking $f(x)=x^\alpha\otimes v$ ($\alpha\in\N^n$ and $v\in V$) as test functions, 
there are some $\varphi_{\alpha\beta}\in C^\infty(U''\times U')$
 for $\vert \beta \vert< \vert \alpha \vert$
 such that 
$$
h_\alpha(gy)=h_\alpha(y)+\sum_{\vert\beta\vert<\vert\alpha\vert}
\varphi_{\alpha\beta}(g,y)h_\beta(y).  
$$
Therefore we see inductively on $\vert\alpha\vert$ that
$h_\alpha(y)\in C^\infty(U')\otimes\operatorname{Hom}(V,W)$ for
all $\alpha$ because $G'$ acts transitively on $Y$. Hence $Tf\vert_{U'}\in
C^\infty(U')\otimes W$. Thus we have shown that $T$ maps $C^\infty(X,\mathcal V)$ into $C^\infty(Y,\mathcal W)$.

Finally, it follows from Lemma \ref{lem:1} that $T$ is a differential operator because $\mathrm{Supp}K\subset\Delta(Y)$. Now we have proved
the lemma.
\end{proof}

\begin{proof}[Proof of Theorem \ref{thm:surject}]
Taking the tensor product of each term in (\ref{eqn:isoMh}) with the finite-dimensional representation $W$ of $H'$, we get a bijection between the subspaces of $\mathfrak h'$-invariants:
$$
\operatorname{Hom}_{\mathfrak h'}(W^\vee,\indhg(V^\vee))
\stackrel{\sim}{\longrightarrow}
(\mathcal D'_{\{o\}}(X,\mathcal V_X^*)\otimes W)^{\Delta(H')}.
$$

Composing this with the bijection in Lemma \ref{lem:312}, we obtain a bijection from 
$\operatorname{Hom}_{\mathfrak h'}(W^\vee,\indhg(V^\vee))$
to $\operatorname{Diff}_{G'}(\mathcal V_X,\mathcal W_Y)$, which is
by construction nothing but $D_{X\to Y}$ in Theorem \ref{thm:surject}.
\end{proof}

\section{Algebraic Fourier transform for generalized Verma modules}
The duality theorem (Theorem \ref{thm:surject}) states that, to obtain a differential symmetry breaking operator $D\in
\mathrm{Diff}_{G'}(\mathcal V_X,\mathcal W_Y)$, it suffices to find $\varphi\in\operatorname{Hom}_{H'}(W^\vee,\indhg(V^\vee))$.
In Section \ref{sec:Fmethod}, we shall present a new method (\emph{F-method}) which characterizes the ``algebraic Fourier transform" of 
$\varphi$ as a solution to a certain system of partial differential equations.

In this section we introduce and study  the ``algebraic Fourier transform" 
of generalized Verma modules. Proposition \ref{prop:Fc} is particularly important to the F-method.

\subsection{Weyl algebra and algebraic Fourier transform}

Let $E$ be a vector space over $\C$. 
The Weyl algebra
$\mathcal D(E)$
is the ring of holomorphic differential operators on $E$
with polynomial coefficients. 

\begin{defn}\label{def:hat}
 We define the \emph{algebraic Fourier transform} as an isomorphism of two Weyl algebras on $E$ and its dual space $E^\vee$:
$$
\mathcal D(E)\to\mathcal D(E^\vee), \qquad T\mapsto \widehat T,
$$
induced by
\begin{equation}\label{eqn:FTdiff}
\widehat{\frac\partial{\partial z_j}}:=-\zeta_j,\quad
\widehat z_j:=\frac\partial{\partial\zeta_j},\quad 1\leq j\leq n=\dim E.
\end{equation}
where $(z_1,\ldots,z_n)$ are coordinates on $E$ and $(\zeta_1,\ldots,\zeta_n)$ are the dual coordinates on $E^\vee$.
\end{defn}

\begin{example}
Let $E_z:=\sum_{j=1}^n z_j\frac\partial{\partial z_j}$ be the Euler operator on $E$. Then, by the 
commutation relations
\begin{equation}\label{eqn:commrel}
\frac{\partial}
{\partial \zeta_i}\zeta_j-\zeta_j\frac{\partial}
{\partial \zeta_i}=\delta_{ij},
\end{equation}
in the Weyl algebra $\mathcal D(E^\vee)$, where $\delta_{ij}$ is the Kronecker delta. Hence we have
 $\widehat{E_z}=-E_\zeta-n$.
\end{example}

The isomorphism $T\mapsto\widehat T$ in 
Definition \ref{def:hat} does not depend on the choice of coordinates. To see this, we consider
the natural action of the general linear group $GL(E)$ on $E$, which yields automorphisms of 
the ring $\operatorname{Pol}(E)$ of polynomials of $E$ and the Weyl algebra $\mathcal D(E)$. For $A\in GL(E)$, we set
\begin{eqnarray*}
A_{\#}&:&\operatorname{Pol}(E)\To\operatorname{Pol}(E), \quad F\mapsto F(A^{-1}\cdot),\\
A_*&:&\mathcal D(E)\To\mathcal D(E),\quad T\mapsto A_{\#}\circ T\circ A_{\#}^{-1}.
\end{eqnarray*}
We denote by $\trans A\in GL(E^\vee)$ the dual map of $A$. Then we have
\begin{lem}\label{lem:Faffine}
For any $A\in GL(E)$ and $T\in\mathcal D(E)$,
$$
\widehat{A_*T}=\left(\trans A^{-1}\right)_*\widehat T.
$$
\end{lem}
The proof is straightforward from the definition \eqref{def:hat}, and we omit it.

Next we consider the group homomorphism $GL(E)\To GL(\operatorname{Pol}(E))$, $A\mapsto A_{\#}$. Taking the differential,
 we get a Lie algebra homomorphism $\operatorname{End}(E)\to\mathcal D(E)$. In the coordinates, we write $Z={}^t(z_1,\cdots,z_n)$ and $\partial_Z=
{}^t\left(\frac\partial{\partial z_1},\cdots, \frac\partial{\partial z_n}\right)$. Then this homomorphism amounts to
\begin{equation}\label{eqn:endo1}
\Psi_E:\operatorname{End}(E)\to\mathcal D(E),\quad
A\mapsto -\trans Z\,\trans A\,\partial_Z\equiv - \sum_{i,j}A_{ij}z_j\frac{\partial}
{\partial z_i}.
\end{equation}

Let $\sigma:\mathfrak g\to\operatorname{End}(E)$ be a representation
of a Lie algebra
 $\mathfrak g$ on $E$, and $\mathfrak\sigma^\vee:\mathfrak g\to
 \operatorname{End}(E^\vee)$ the contragredient representation. Then 
the algebraic Fourier transform $T\mapsto\widehat T$ relates
 the two
 Lie algebra homomorphisms $\Psi_E\circ\sigma:\mathfrak g\to
\mathcal D(E)$ and $\Psi_{E^\vee}\circ\sigma^\vee:\mathfrak g\to
\mathcal D(E^\vee)$ as follows:
\begin{lem}
\label{prop:Phihat}
$$
\widehat {\Psi_E\circ\sigma}
=
\Psi_{E^\vee}\circ\sigma^\vee
+
(\mathrm{Trace}\circ\sigma) \cdot {\operatorname{id}}_{E^{\vee}}.
$$
\end{lem}

\begin{proof}
In the coordinates, we write $A:=\sigma(Z)\in\operatorname{End}(E)\simeq M(n,\C)$ for $Z\in\mathfrak g$. Then,
\begin{eqnarray*}
\widehat{\Psi_E\circ\sigma(Z)}-\Psi_{E^\vee}\circ\sigma^\vee(Z)&=&
-\widehat{\hphantom{,} \trans Z\hphantom{,} \trans A\partial_Z}-\hphantom{,} \trans\zeta A\partial_\zeta\\
&=&\hphantom{,} \trans\partial_\zeta \hphantom{,} \trans A\zeta-\hphantom{,} \trans\zeta A\partial_\zeta
\\
&=&(\operatorname{Trace}A)\hphantom{,}{\operatorname{id}}_{E^{\vee}},
\end{eqnarray*}
where the last equality follows from
the commutation relations \eqref{eqn:commrel}.
\end{proof}

For actual computations that will be undertaken in a subsequent paper
\cite{PART2},
 it is convenient to
give another interpretation of the algebraic Fourier transform by using
real forms of $E$.

\begin{defn}\label{defn:Fourier}
Fix a real form $E(\R)$
 of the complex vector space $E$. 
Let $\mathcal E'_{\{0\}}(E({\mathbb{R}}))$ be the space of distributions on the vector space $E(\R)$
 supported at the origin $0$. 
We define a \lq\lq{Fourier transform}\rq\rq\
  $\mathcal F_{c}:\mathcal E'_{\{0\}}(E(\R))\to\mathrm{Pol}(E^\vee)$ 
  by the following formula:
\begin{equation}\label{eqn:algF}
\mathcal F_{c}f(\zeta):=\langle f(\cdot),e^{\langle\cdot,\zeta\rangle}\rangle=\int_{E(\R)} e^{\langle x,\zeta\rangle}f(x)
\qquad \mathrm{for}\,\zeta\in E^\vee.
\end{equation}

We have used the function 
$e^{\langle x,\zeta\rangle}$
 in \eqref{eqn:algF} rather than $e^{-\sqrt{-1}\langle x,\zeta\rangle}$ or
$e^{-\langle x,\zeta\rangle}$ which are involved in the usual Fourier transform or the
Laplace transform, respectively. 
This convention makes later computations simpler (see Remark \ref{rem:ipower}).

Furthermore, with our convention 
\begin{equation}\label{eqn:changevar}
\mathcal F_c(f(A\cdot))=(\mathcal F_cf)(\hphantom{,} \trans A^{-1}\cdot),
\end{equation} for any $A\in GL_\R(E(\R))$.

The Fourier transform $\mathcal F_c$ induces an algebra isomorphism 
\[
\mathcal F_{c}: \mathcal E'_{\{0\}}(E({\mathbb{R}}))  \xrightarrow{\sim} \operatorname{Pol}(E^\vee)
\]
between the polynomial algebra $\operatorname{Pol}(E^\vee)$ with unit $\mathbf 1$, the constant function on $E^\vee$, and the convolution algebra $\mathcal E'_{\{0\}}(E(\R))$ with unit $\delta$, the Dirac
delta function.
We write $\mathcal F_{c}^{-1}:  \operatorname{Pol}(E^\vee) \xrightarrow{\sim}
\mathcal E'_{\{0\}}(E({\mathbb{R}}))$ for the inverse ``Fourier transform":
$$\mathcal F_c^{-1}(\mathbf 1)=\delta.$$
\end{defn}

\begin{rem}\label{rem:45}
The Weyl algebra $\mathcal D(E)$ acts naturally on the space of distributions on $E(\R)$, and in particular, on $\mathcal
E'_{\{0\}}(E(\R))$.
The algebraic Fourier transform defined in Definition \ref{def:hat} satisfies
\begin{equation}\label{eqn:F-trans}
 \widehat T=\mathcal F_{c}\circ T\circ \mathcal F_{c}^{-1}
\quad\text{for $T \in \mathcal{D}(E)$},
\end{equation}
and the formula \eqref{eqn:F-trans} characterizes $\widehat{T}$.
To see this, we take coordinates $(x_1,\cdots,x_n)$ on $E(\R)$, and extend them to the complex coordinates $(z_1,\cdots,z_n)$ on $E$ and the dual ones
$(\zeta_1,\cdots,\zeta_n)$ on $E^\vee$. Let $P(\zeta)=\zeta^\alpha\in\mathrm{Pol}(E^\vee)$ and 
$\displaystyle T=\sum_{\beta,\gamma}a_{\beta,\gamma}z^\beta\frac{\partial^{\vert\gamma\vert}}{\partial z^\gamma}\in\mathcal D(E)$. Then we have
$$
\widehat TP=\sum_{\beta,\gamma}(-1)^{\vert\gamma\vert}a_{\beta,\gamma}
\frac{\partial^{\vert\beta\vert}}{\partial z^\beta}\zeta^{\alpha+\gamma},
$$
and on the other hand,
\begin{eqnarray*}
\mathcal F_c\circ T\circ\mathcal F_c^{-1}P&=&(-1)^{\vert\alpha\vert}\mathcal F_c\circ T(\delta^\alpha(x))
=(-1)^{\vert\alpha\vert}\mathcal F_c\left(\sum_{\beta,\gamma}a_{\beta,\gamma}x^\beta\delta^{\alpha+\gamma}(x)\right)\\
&=&(-1)^{\vert\alpha\vert}\sum_{\beta,\gamma}(-1)^{\vert\alpha\vert+\vert\gamma\vert}a_{\beta,\gamma}
\frac{\partial^{\vert\beta\vert}}{\partial z^\beta}\zeta^{\alpha+\gamma}.
\end{eqnarray*}
Hence the identity \eqref{eqn:F-trans} holds on $\mathrm{Pol}(E^\vee)$. Since
the Weyl algebra $\mathcal D(E^\vee)$ acts faithfully
 on $\operatorname{Pol}(E^\vee)$, we have shown \eqref{eqn:F-trans}. In particular, the composition 
 $\mathcal F_{c}\circ T\circ \mathcal F_{c}^{-1}$ does not depend
on the choice of a real form $E(\R)$.
\end{rem}

\subsection{Holomorphic vector fields associated to the Gelfand--Naimark decomposition}
It is convenient to prepare some notation in the \emph{complex} reductive Lie algebras for later purpose.

Let $\mathfrak g$ be a complex reductive Lie algebra, and $\mathfrak p=\mathfrak l+\mathfrak n_+$ a Levi decomposition of a parabolic subalgebra. Let $G_\C$ be a connected complex Lie group with Lie algebra $\mathfrak g$, and $P_\C=L_\C\exp\mathfrak n_+$ the parabolic subgroup with Lie algebra $\mathfrak p=\mathfrak l+\mathfrak n_+$.
According to the Gelfand--Naimark decomposition $\mathfrak g=\mathfrak n_-+\mathfrak l+\mathfrak n_+$
of the Lie algebra $\mathfrak{g}$, 
we have a diffeomorphism
$$
\mathfrak n_-\times L_\C\times \mathfrak n_+\to G_\C,\quad
(Z,\ell, Y)\mapsto (\exp Z)\ell(\exp Y),
$$
into an open dense subset $G^{\mathrm{reg}}_\C$ of $G_\C$. 
Let
$$
p_\pm: G^{\mathrm{reg}}_\C\To\mathfrak n_\pm,\qquad p_o:G^{\mathrm{reg}}_\C\to L_\C,
$$
 be the projections characterized by the identity
$$
\exp(p_-(g))p_o(g)\exp(p_+(g))=g.
$$
We set
\begin{align}\label{eqn:alpha}
\alpha:\mathfrak g\times \mathfrak n_-&\to\mathfrak l,\quad
 &&(Y,Z)\mapsto \left.\frac {d}{dt}\right\vert_{t=0}p_o\left(e^{tY}e^Z\right),\\
\beta:\mathfrak g\times \mathfrak n_-&\to\mathfrak n_-,\quad
 &&(Y,Z)\mapsto \left.\frac {d}{dt}\right\vert_{t=0}p_-\left(e^{tY}e^Z\right).\label{eqn:beta}\\
 \gamma:\mathfrak g\times \mathfrak n_-&\to\mathfrak l+\mathfrak n_+,\quad
 &&(Y,Z)\mapsto \alpha (Y,Z)+\left.\frac {d}{dt}\right\vert_{t=0}p_+\left(e^{tY}e^Z\right).\label{eqn:gamma}
\end{align}

We regard $\beta(Y,\cdot\,)$ as a holomorphic vector field on $\mathfrak n_-$ through the
following identification.  
\[
   \mathfrak n_-\ni Z \mapsto \beta(Y,Z)\in\mathfrak n_-\simeq T_Z\mathfrak n_-.\]

\begin{example}\label{example:410}
$G_\C=GL(p+q,\C)$, $L_\C=GL(p,\C)\times GL(q,\C)$, and $\mathfrak n_-\simeq M(p,q;\C)$.

We note that $\mathfrak n_-$ is realized as upper block matrices. Then for
$g^{-1}=
\begin{pmatrix}
a&b\\c&d
\end{pmatrix}
\in G_\C
$, 
$
Y=
\begin{pmatrix}
A&B\\C&D
\end{pmatrix}
\in M(p+q;\C)
$
and $Z \in M(p,q;\C)$
we have
\begin{alignat*}{2}
p_-(g^{-1})=&bd^{-1}, &&
\\
p_o(g^{-1})=&(a-bd^{-1}c,d)
&&\in GL(p,\C)\times GL(q,\C),
\\
\alpha(Y,Z)=&(A-ZC,CZ+D)
&&\in  \mathfrak{gl}_p(\C) \oplus \mathfrak{gl}_q(\C),
\\
\beta(Y,Z)=&AZ+B-ZCZ-ZD.
&&
\end{alignat*}
Then $\beta(Y,\cdot)$ is regarded as
 the following holomorphic vector field on $\mathfrak n_-\simeq M(p,q;\C)$ given by

\begin{align*}
\operatorname{Trace}(\beta(Y,Z){\hphantom{,}} {}^{t\!} \partial_Z)
=&
\sum_{a=1}^{p}\sum_{b=1}^{q}\beta(Y,Z)_{ab} \frac{\partial}{\partial z_{ab}}
\\
=&
\sum_{a=1}^{p}\sum_{b=1}^{q}
(\sum_{i=1}^{p} A_{ai}z_{ib} + B_{ab}
-\sum_{i=1}^{p}\sum_{j=1}^{q} z_{aj} C_{ji}z_{ib}
-\sum_{j=1}^{q} z_{aj}D_{jb})
\frac{\partial}{\partial z_{ab}}.
\end{align*}
\end{example}

A reductive Lie algebra $\mathfrak g$ is said to be $k$-\emph{graded} if it admits a direct sum decomposition $\mathfrak g=\displaystyle\oplus_{j=-k}^k \mathfrak g(j)$ such that
$[\mathfrak g(i),\mathfrak g(j)]\subset\mathfrak g(i+j)$ for all $i,j$. Any parabolic subalgebra $\mathfrak p=\mathfrak l+\mathfrak n_+$ of $\mathfrak g$ is given by $\mathfrak l=\mathfrak g(0)$ and $\mathfrak n_+=\oplus_{j>0}\mathfrak g(j)$ for some $k$-gradation of $\mathfrak g$. We then have the following estimates of coefficients of holomorphic differential operators $d\pi_\mu(Y)$.
\begin{lem}\label{lem:dpipoly}
According to the direct sum decomposition $\mathfrak g=\displaystyle\oplus_{j=-k}^k \mathfrak g(j)$, we
write $\gamma(Y,Z)=\sum_{\ell=0}^k\gamma_\ell$ and $\beta(Y,Z)=\sum_{\ell=-k}^{-1}\beta_\ell$, where $\gamma_\ell\in\mathfrak g(\ell)$ for $0\leq\ell\leq k$ and $\beta_\ell\in\mathfrak g(\ell)$ for $-k\leq\ell\leq-1$.
Then $\gamma_\ell$ and $\beta_\ell$ are polynomials in $Z$ of degree
at most $k-\ell$.
\end{lem} 
\begin{proof}
Since the map $N_{-,\C}\times P_\C \overset{\sim}{\To} G_\C^{\mathrm{reg}}$ is an analytic diffeomorphism, we have
\begin{equation}\label{eqn:tYZ}
e^{tY}e^Z= e^{Z+t\beta(Y,Z)+o(t)}e^{t\alpha(Y,Z)+o(t)}
\end{equation}
for sufficiently small $t\in\C$, where we use the Landau symbol 
$o(t)$ for a $\mathfrak g$-valued function dominated by $t$ when $t$ tends to be zero. Multiplying \eqref{eqn:tYZ} by $e^{-Z}$
from the left, and taking the differential at $t=0$, we get
$$
\mathrm{Ad}(e^{-Z})Y=\gamma(Y,Z)+\frac{e^{\mathrm{ad}(Z)}-1}{{\mathrm{ad}(Z)}}\beta(Y,Z),
$$
because $\displaystyle \left.\frac d{dt}\right|_{t=0}e^{-Z}e^{Z+tW}=\frac{e^{\mathrm{ad}(Z)}-1}{{\mathrm{ad}(Z)}}W.$
We note that $\mathrm{ad}(Z)$ lowers the grading of $\mathfrak g$, namely $\mathrm{ad}(Z)\mathfrak g(j)\subset
\oplus_{i=-k}^{j-1}\mathfrak g(i)$ because $Z\in\mathfrak n_-$. In particular, we have
\begin{equation}\label{eqn:YZ}
\sum_{j=0}^{2k}\frac{(-1)^j}{j!}\mathrm{ad}(Z)^jY=\gamma(Y,Z)+\sum_{i=0}^{k-1}\frac{\mathrm{ad}(Z)^i}{(i+1)!}\beta(Y,Z).
\end{equation}
Let $q_\ell:\mathfrak g\to\mathfrak g(\ell)$ be the projection according to the direct sum decomposition 
$\displaystyle\mathfrak g=\displaystyle\bigoplus_{j=-k}^k \mathfrak g(j)$. 
Suppose $\ell\geq 0$. Applying $q_\ell$ to \eqref{eqn:YZ}, we have
$$
\gamma_\ell=q_\ell\left(\sum_{j=0}^{k-\ell}\frac{(-1)^j\mathrm{ad}(Z)^j}{j!}Y\right).
$$
Hence $\gamma_\ell$ is a polynomial in $Z$ of degree at most $k-\ell$.

Suppose $\ell<0$. Applying $q_\ell$ to \eqref{eqn:YZ}, we get
$$
\beta_\ell= q_\ell\left(\sum_{j=0}^{k-\ell}\frac{(-1)^j\mathrm{ad}(Z)^j}{j!}Y\right)-
q_\ell\left(\sum_{i=0}^{j-\ell}\sum_{j=\ell}^{-1}\frac{\mathrm{ad}(Z)^i}{(i+1)!}\beta_j\right).
$$
By the downward induction on $\ell$, we see that $\beta_\ell$ is a polynomial in $Z$ of degree
at most $k-\ell$ for $-k\leq\ell\leq-1$.
\end{proof}

\subsection{Fourier transform of principal series representations}
\label{subsec:3.4}

Suppose $\mathfrak g$ is a complex reductive Lie algebra, $\mathfrak p=\mathfrak l+\mathfrak n_+$ a parabolic
subalgebra, and $\lambda:\mathfrak p \to\mathrm{End}_\C(V)$ a finite-dimensional representation.

We use the letter $\mu$ to denote the representation of $\mathfrak p$ on the dual space $V^\vee$ given
by
\begin{equation}\label{eqn:mucompl}
\mu:=\lambda^*\equiv\lambda^\vee\otimes\mathrm{Trace} (\mathrm{ad}(\cdot):\mathfrak n_+\to\mathfrak n_+).
\end{equation}

By applying the (algebraic) Fourier transform of the Weyl algebra, we define
a Lie algebra homomorphism
$$
\widehat{d\pi_\mu}:\mathfrak g\to\mathcal D(\mathfrak n_+)\otimes\mathrm{End}_\C(V^\vee),
$$
by using the complex flag variety $G_C/P_\C$ in this subsection. In Section \ref{sec:33part1}, we relate $\widehat{d\pi_\mu}$
with the ``algebraic Fourier transform" of a generalized Verma module
$$
F_c: \indpg(V^\vee)\stackrel{\sim}{\to}\operatorname{Pol}(\mathfrak n_+)\otimes V^\vee,
$$
which is defined by using a real flag variety $G/P$, see \eqref{eqn:indPol}.

Let $G_\C$ be a connected complex reductive Lie group with complex reductive Lie algebra $\mathfrak g$, and
$P_{\C} = L_{\C} N_{+,\C}$ be the parabolic subgroup with Lie algebra $\mathfrak p$.
Let $\Omega_{X_\C}$ be the canonical line bundle of the complex generalized flag variety $X_\C=G_\C/P_\C$.
\vskip10pt

Suppose $\lambda$ lifts to a holomorphic representation of $P_{\C}$,
then so does $\mu$. 
We form a $G_{\C}$-equivariant holomorphic vector bundle $\mathcal{V}$ and
$\mathcal{V}^\vee\otimes\Omega_{X_\C}$
over $X_\C$ associated to $\lambda$ and $\mu$, respectively.

We consider the regular representation $\pi_\mu$ of $G_{\C}$ on $C^\infty(G_{\C}/P_{\C},\mathcal{V}^\vee\otimes\Omega_{X_\C})$.
The infinitesimal action will be denoted
 by $d \pi_{\mu}$, 
 which is defined on
$C^\infty(U,\mathcal V^\vee\otimes\Omega_{X_\C}\vert_U)$
 for any open subset $U$ of $G_\C/P_\C$. In particular, we take $U$ to be the open Bruhat cell
 ${\mathfrak{n}}_- \hookrightarrow G_{\C}/P_{\C}$,
 $Z \mapsto \exp Z \cdot o$,
 where $o =eP_{\C} \in G_{\C}/P_{\C}$.
By  trivializing the holomorphic vector bundle
 ${\mathcal{V}}^{\vee} \otimes\Omega_{X_\C}\to G_{\C}/P_{\C}$
 on it, we define a function $F \in C^\infty(\mathfrak{n}_-,V^\vee)$
 for a section $f \in C^\infty(G_{\C}/P_{\C},\mathcal{V}^\vee\otimes\Omega_{X_\C})$ by
$$
F(Z):=f(\exp Z) \qquad\text{for $Z\in\mathfrak n_-$}.
$$
Then the action of $\mathfrak g$ on $C^\infty(\mathfrak{n}_-,V^\vee)$  given by
\begin{eqnarray} \label{eqn:35}
\left(d\pi_\mu(Y)F\right)(Z)=
&=&\left.\frac {d}{dt}\right\vert_{t=0}f(e^{-tY}e^Z)\nonumber\\
&=&\mu(\gamma(Y,Z))F(Z)-(\beta(Y,\cdot\,)F)(Z)
\quad\text{for $Y \in \mathfrak{g}$},
\end{eqnarray}
where by a little abuse of notation $\mu$ stands for the infinitesimal action.
The right-hand side of (\ref{eqn:35}) defines a representation of Lie algebra $\mathfrak g$ whenever $\mu$ (or $\lambda$)
 is a representation 
of the Lie algebra $\mathfrak p$
 without assuming that it lifts to a holomorphic representation
 of the  complex reductive group $P_\C$.

It follows from \eqref{eqn:35} and Lemma \ref{lem:dpipoly} that we obtain a Lie algebra homomorphism
\begin{equation}\label{eqn:dpi1}
d\pi_\mu:\mathfrak g\to\mathcal D(\mathfrak n_-)\otimes\operatorname{End}(V^\vee),
\end{equation}
for any representation $\lambda$ of the Lie algebra $\mathfrak p$.
By taking the algebraic Fourier transform
 on the Weyl algebra $\mathcal D(\mathfrak n_-)$
(see Definition \ref{def:hat}), 
 we get another Lie algebra homomorphism:
\begin{equation}\label{eqn:dpi2}
\widehat{d\pi_\mu}:\mathfrak g\to\mathcal D(\mathfrak n_+)\otimes\operatorname{End}(V^\vee).
\end{equation}
We use the same letter $\pi_\mu$ to denote the \lq\lq{action}\rq\rq\
 of $G_{\C}$ on
$C^\infty(\mathfrak{n}_-,V^\vee)$
given as 
\begin{equation}
\label{eqn:gNpic}
\left(\pi_\mu(g)F\right)(Z)=\mu(p_o(g^{-1}\exp Z)
\exp(p_+(g^{-1}\exp Z)))^{-1} F(p_-(g^{-1}\exp Z)).  
\end{equation}
This formula makes sense
 if $F$ comes from 
 $C^{\infty}(G_{\C}/P_{\C}, {\mathcal{V}}\otimes\Omega_{X_\C}^{\vee})$, 
 or if
 $F\in C^{\infty}({\mathfrak{n}}_-,V^{\vee})$ and 
 $g\in G_\C$ and $Z\in\mathfrak n_-$ 
 satisfy 
$g^{-1}\exp Z\in G^{\mathrm{reg}}_\C$.
In particular, if $\lambda$ is trivial on the nilpotent radical $\mathfrak n_+$
 for $g=m\exp W$ with $m\in L_\C$ and $W\in\mathfrak n_-$, and if $\mathfrak n_+$ is abelian we have 
\begin{equation}\label{eqn:simpler}
\left(\pi_\mu(g)F\right)(Z)=\mu(m)F(\mathrm{Ad}(m)^{-1}Z-W).
\end{equation}

Let us analyze  $d\pi_\mu(Y)$ for $Y \in \mathfrak{l}+\mathfrak n_+$.
We begin with the case $Y\in \mathfrak{l}$.
We let the Levi subgroup $L_{\C}$
 act on $\operatorname{Pol}({\mathfrak {n}}_+)$
 by 
\[
\operatorname{Ad}_{\#}(l):
  f(\cdot) \mapsto f({\operatorname{Ad}}(l^{-1})\cdot), 
\qquad
  l \in L_{\C}.  
\]
Since this action is algebraic,
the infinitesimal action defines a Lie algebra homomorphism
into the Weyl algebra:
$$
\operatorname{ad}_{\#}:\mathfrak l\to\mathcal D(\mathfrak n_+), 
\quad
Y\mapsto \operatorname{ad}_{\#}(Y),
$$
where $\operatorname{ad}_{\#}(Y)$ is a holomorphic vector field on
$\mathfrak n_+$ given by $\operatorname{ad}_{\#}(Y)_x:=\left.\frac d{dt}\right|_{t=0}\operatorname{Ad}(e^{-tY})x
\in T_x(\mathfrak n_+)$ for $x\in\mathfrak n_+$.

\begin{lem}\label{lem:410}
Let $\lambda$ be a representation of the parabolic Lie algebra $\mathfrak p= \mathfrak l+ \mathfrak n_+$, and $\mu\equiv\lambda^*$ be as in \eqref{eqn:mucompl}. Then
the following two representations of $\mathfrak l$ on  $\operatorname{Pol}(\mathfrak n_+)\otimes V^\vee $are isomorphic:
\begin{equation}\label{eqn:dpiell}
\widehat{d\pi_\mu}\big\vert_{\mathfrak l}
\simeq
\operatorname{ad}_{\#}
\otimes \mathrm{id}+\mathrm{id}\otimes(\mu-\mathrm{Trace}\circ\mathrm{ad}\Big\vert_{\mathfrak n_-})
=\operatorname{ad}_{\#}(Y)\otimes\mathrm{id}+\mathrm{id}\otimes(-\lambda).
\end{equation}
\end{lem}
In particular, if $\lambda$ lifts to a holomorphic representation of $P_\C$ then the right-hand side is the infinitesimal action
 of ${\operatorname{Ad}}_{\#} \otimes \lambda^{\vee}$
 of $L_{\C}$ on $\operatorname{Pol}({\mathfrak {n}}_+) \otimes V^{\vee}$.  
\begin{proof}
For $Y\in\mathfrak l,\,X\in\mathfrak n_-$ we have $\gamma(Y,X)=Y$,
 and the formula \eqref{eqn:35} reduces, in $\mathcal D(\mathfrak n_-)\otimes\mathrm{End}(V^\vee)$,  to
$$
d\pi_\mu(Y)=\mathrm{id}\otimes \mu(Y)-\beta(Y,\,\cdot\,)\otimes
\mathrm{id}.
$$

We apply Lemma \ref{prop:Phihat} to the case where $(\sigma, E)$ is the
adjoint representation of $\mathfrak l$ on $\mathfrak n_-$. 
Since $\beta(Y,\cdot)=-dL(Y)$ for $Y\in\mathfrak l$, we have $\Psi_{\mathfrak n_-}\circ\mathrm{ad}=-\beta$ on
$\mathfrak l$,
with the notation therein. 
Moreover, via the identification $\mathfrak n_-^\vee\simeq \mathfrak n_+$, the map $\Psi_{\mathfrak n_-^\vee}\circ
\mathrm{ad}^\vee$ amounts to $\Psi_{\mathfrak n_+}\circ\mathrm{ad}=\mathrm{ad}_{\#}$.
Therefore, we get
\begin{eqnarray*}
\widehat{d\pi_\mu}(Y)
&=&
\mathrm{id}\otimes\mu(Y)+
\Psi_{\mathfrak n_-^\vee}\circ\mathrm{ad}^\vee(Y)\otimes \mathrm{id}+\left(\mathrm{Trace}\circ\mathrm{ad}(Y)\Big\vert_{\mathfrak n_-}\right)\mathrm{id}\otimes\mathrm{id}
\\
&=&
\mathrm{id}\otimes\mu(Y)
+\operatorname{ad}_{\#}(Y)\otimes \mathrm{id}-\left(\mathrm{Trace}\circ\mathrm{ad}(Y)\Big\vert_{\mathfrak n_+}\right)\mathrm{id}\otimes\mathrm{id}
.
\end{eqnarray*}
Thus, the lemma follows.
\end{proof}

The differential operators $\widehat{d\pi_\mu}(Y)$ with $Y\in\mathfrak n_+$ play a central role in the F-method. 
If the parabolic subalgebra $\mathfrak p$ is associated to a $k$-gradation of $\mathfrak g$, then these differential operators
are at most  of order $2k$ by Lemma \ref{lem:dpipoly}. We
 describe their structure in the case where $k=1$, namely $\mathfrak n_+$ is abelian.

\begin{prop}\label{lem:315} Assume that $\mathfrak n_+$ is abelian. 
Let $(\lambda, V)$ be a representation of $\mathfrak l$, extended trivially on $\mathfrak n_+$, and $\mu=\lambda^*$
be as in \eqref{eqn:mucompl}.
For every
$Y\in\mathfrak n_+$ the operator $\widehat{d\pi_\mu}(Y)$ is of the form
\begin{equation}\label{eqn:degree}
\sum a_i^{jk}\zeta^i\frac{\partial ^2}{\partial\zeta^j\partial\zeta^k}+ \sum b^j\frac{\partial}{\partial\zeta^j},
\end{equation}
where $a_i^{jk}$ and $b^j\in\mathrm{End}(V^\vee)$ are constants depending on $Y$.
\end{prop}
\begin{proof}
Since $\mathfrak n_+$ is abelian, we can take a characteristic element $H$
such that
$$
\mathrm{Ad}(e^{sH})Y=e^sY\quad\mathrm{for\, any}\,Y\in \mathfrak n_+.
$$
We set $m:= e^{sH}$. Then $\trans \mathrm{Ad}(m)^{-1}=e^{-s}\mathrm{id}$ on $\mathfrak n_-\simeq\mathfrak n_+^\vee$.

Taking the algebraic Fourier transform of the formula
$$
d\pi_\mu(\mathrm{Ad}(m)Y)=\pi_\mu(m)d\pi_\mu(Y)\pi_\mu(m^{-1}),
$$
where $(\pi_\mu(m)F)(Z)=\mu(m)F(\mathrm{Ad}(m^{-1})Z)=\mu(m)F(\mathrm{Ad}(m)_{\#} F)(Z)$ by
\eqref{eqn:simpler}, we get
$$
\widehat{d\pi_\mu}(\mathrm{Ad}(m)Y)=\left(\trans \mathrm{Ad}(m)^{-1}\right)_*\widehat{d\pi_\mu}(Y)
$$
by Lemma \ref{lem:Faffine}. Hence
\begin{equation}\label{eqn:Rt}
e^s\widehat{d\pi_\mu}(Y)=\left(e^{-s}\mathrm{id}\right)_*\widehat{d\pi_\mu}(Y)
\end{equation}
If we write $\widehat{d\pi_\mu}(Y)$ in the form
$$
\sum_{\alpha,\beta\in\N^n}C_{\alpha,\beta}\zeta^\alpha\frac{\partial^{\vert\beta\vert}}{\partial\zeta^\beta}
$$
then \eqref{eqn:Rt} implies that $C_{\alpha,\beta}\neq0$ only when $\vert\alpha\vert+\vert\beta\vert=-1$
because $(e^{-s}\mathrm{id})_*\frac{\partial}{\partial\zeta_j}=e^{-s}\frac{\partial}{\partial\zeta_j}$ and $(e^{-s}\mathrm{id})_*\zeta_j=e^{s}\zeta_j$ ($1\leq j\leq n$).
As $d\pi_\mu(Y)$ is a vector field there is no
term for $\vert\alpha\vert>1$. Hence we get the expression (\ref{eqn:degree}).
\end{proof}

\subsection{Fourier transform on the real flag varieties}\label{sec:33part1}
In this subsection we define ``algebraic Fourier transform" of generalized Verma modules,
see \eqref{eqn:indPol}:
$$
F_c:\indpg(V^\vee)\stackrel{\sim}{\to} \operatorname{Pol}(\mathfrak n_+)\otimes V^\vee.
$$
As we shall prove in Proposition \ref{prop:Fc}, the Lie algebra homomorphism $\widehat{d\pi_{\lambda^*}}:U(\mathfrak g)\To
\mathcal D(\mathfrak n_+)\otimes\mathrm{End}(V^\vee)$ defined in
\eqref{eqn:dpi2} in the previous section can be reconstructed from $F_c$, namely, $\widehat{d\pi_{\lambda^*}}(u)$ ($u\in U(\mathfrak g))$ is the operator $S$ that is characterized by
$$
S F_c(v)=F_c(u\cdot v) \quad\mathrm{for}\,\,\mathrm{any}\,\,
v\in\indpg(V^\vee).
$$

For later purpose, we work with a real form $G$ of $G_\C$.
From now on, let $G$ be a real semisimple Lie group, $P$  a parabolic subgroup of $G$ with Levi decomposition $P=LN_+$,
and $V$ a finite-dimensional  representation of $P$. 

Let $LN_-$ be the opposite parabolic subgroup of $P=L N_+$. 
We write $\mathfrak n_+(\R)$ and $\mathfrak n_-(\R)$ for the Lie algebras
 of $N_+$ and $N_-$, respectively, and set $\mathfrak n_+=\mathfrak n_+(\R)\otimes_\R\C$.
The open Bruhat cell is given as the image of
 the following 
 embedding
$$
\iota:\mathfrak n_-(\R)\hookrightarrow G/P,\quad
X\mapsto \exp(X)\cdot o,
$$ where $o=eP\in G/P$.

Let $\lambda: P\to GL_\C(V)$ be a finite-dimensional representation of $P$, and
$\mathcal V=G\times_P V$ the $G$-equivariant vector bundle over the real flag variety $G/P$.
The pullback of the dualizing bundle $\mathcal V^*\equiv\mathcal V^\vee_{2\rho}\to G/P$ via $\iota$ is
trivialized into the direct product bundle $\mathfrak{n}_-({\R})\times V^\vee\to \mathfrak{n}_-({\R})$ and thus we have a linear isomorphism:
\begin{equation}\label{eqn:iso1}
\iota^*:\mathcal E'_{\{o\}}(G/P,\mathcal V^\vee_{2\rho})\stackrel{\sim}{\longrightarrow}
 \mathcal E'_{\{0\}}(\mathfrak n_-(\R))\otimes V^\vee,
\end{equation}
through which we induce the $(\mathfrak g, P)$-action on $\mathcal E'_{\{0\}}(\mathfrak n_-(\R))\otimes V^\vee$ 
from $\mathcal E'_{\{o\}}(G/P,\mathcal V^\vee_{2\rho})$.

The Killing form of $\mathfrak g$ identifies the dual space $\mathfrak n_-(\R)^\vee$ with $\mathfrak n_+(\R)$, 
and thus the Fourier
transform $\mathcal F_c$ in (\ref{eqn:algF}) gives rise to a linear
isomorphism:
\begin{equation}\label{eqn:iso11}
\mathcal F_c\otimes \operatorname{id}: \mathcal E'_{\{0\}}(\mathfrak n_-(\R))\otimes V^\vee\stackrel{\sim}{\longrightarrow} \operatorname{Pol}(\mathfrak n_+)\otimes V^\vee,
\end{equation}
through which we induce the $(\mathfrak g, P)$-action further on the right-hand side.

In summary we have the following $(\mathfrak g,P)$-isomorphisms:
\begin{equation}\label{eqn:indPol}
F_c:\indpg(V^\vee)
{\stackrel{(\ref{eqn:isoMh})}{\overset{\sim}{\To}}}
\mathcal E'_{\{o\}}(G/P,\mathcal V^\vee_{2\rho}){\stackrel{(\ref{eqn:iso1})}{\overset{\sim}{\To}}} \mathcal E'_{\{0\}}(\mathfrak n_-(\R))\otimes V^\vee
{\stackrel{\mathcal F_c\otimes \operatorname{id}}{\stackrel{\sim}{\longrightarrow}}}\operatorname{Pol}(\mathfrak n_+)\otimes V^\vee.
\end{equation}
We say that $F_c$ is the {\it{algebraic Fourier transform of a generalized Verma module}}.

The $(\mathfrak g, P)$-module structure of $\operatorname{Pol}(\mathfrak n_+)\otimes V^\vee$ is described by the 
following proposition.

\begin{prop}\label{prop:Fc}
Let $(\lambda,V)$ be a finite-dimensional representation of $P$
and define another representation of $P$ on the dual space $V^\vee$ by $\mu:=\lambda^*\equiv\lambda^\vee\otimes\mathbb{C}_{2\rho}$.
Then,
\begin{enumerate}
\item[1)] The $\mathfrak g$-action on $\operatorname{Pol}(\mathfrak n_+)\otimes V^\vee$ induced by $F_c$
in \eqref{eqn:indPol} coincides with the one given by $\widehat{d\pi_\mu}$ in \eqref{eqn:dpi2}.
\item[2)] The $L$ action on $\operatorname{Pol}(\mathfrak n_+)\otimes V^\vee$ induced by $F_c$
in \eqref{eqn:indPol} coincides with the one given by $\mathrm{Ad}_{\#}\otimes\lambda^\vee$.
\end{enumerate}
\end{prop}
\begin{proof}
1) Let $G_\C$ be a complexification of $G$ and $P_\C$ the connected
subgroup of $G_\C$ with Lie algebra $\mathfrak p=\mathrm{Lie}(P)\otimes_\R\C$.
First we assume that $\lambda$ extends to a holomorphic representation of $P_\C$. Then the $G$-equivariant
vector bundle
$\mathcal V^\vee_{2\rho}$ over $X=G/P$ is the restriction of
the
$G_\C$-equivariant holomorphic vector bundle $\mathcal V^\vee\otimes\Omega_{X_\C}$ over $X_\C=G_\C/P_\C$ that was introduced in the previous subsection. Therefore, the action
of $Y\in\mathfrak g$ on $\mathcal E'_{\{0\}}(\mathfrak n_-(\R)\otimes V^\vee)$ induced by $\iota^*$ in \eqref{eqn:iso1} is given
by the restriction of the holomorphic differential operator $d\pi_\mu(Y)$.

In turn, the action of $Y\in\mathfrak g$ on $\mathrm{Pol}(\mathfrak n_+)\otimes V^\vee$ induced by the isomorphism
\eqref{eqn:iso11} is given by 
$$
(\mathcal F_c\otimes \mathrm{id})\circ d\pi_\mu(Y)\circ (\mathcal F_c^{-1}\otimes \mathrm{id}),
$$
which is equal to $\widehat{d\pi_\mu(Y)}$ by Remark \ref{rem:45}.

To complete the proof in the general case we denote by $\mathrm{Hom}(P_\C,GL_\C(V))$ the set of holomorphic
representations of $P_\C$ on $V$ and by $\mathrm{Hom}(\mathfrak p, \mathrm{End}(V))$ the set of Lie algebra
representations of $\mathfrak p$. Since the former is Zariski dense in the latter, the two $\mathfrak g$-actions
on $\mathrm{Pol}(\mathfrak n_+)\otimes V^\vee$ coincide for all $\lambda$ because both depend algebraically (actually affinely) on $\lambda\in\mathrm{Hom}(\mathfrak p, \mathrm{End}(V))$.

2) This statement is the analogue of Lemma \ref{lem:410} for the Lie group $L$. Indeed, since the group $L$ normalizes
$\mathfrak n_-(\R)$ and fixes the origin $0$,
the isomorphism $\iota^*$ in \eqref{eqn:iso1} respects the
 $L$-action when $L$ acts diagonally on $\mathcal E'_{\{0\}}(\mathfrak n_-(\R))\otimes V^\vee$. To conclude the proof we use \eqref{eqn:changevar}.
\end{proof}

The map $F_c$  does not depend on the choice of a real form $G$ of
$G_\C$ that appears in the two middle terms of (\ref{eqn:indPol}).
Moreover, the isomorphism $F_c:\indpg(V^\vee)\stackrel{\sim}{\longrightarrow}
 \operatorname{Pol}(\mathfrak n_+)\otimes V^\vee$ depends only on the infinitesimal action of $P$ on $V$.
In fact, the following corollary follows immediately from the statement 1) of Proposition \ref{prop:Fc}.
\begin{cor}\label{cor:Fc}
 The algebraic Fourier transform 
 of generalized Verma modules
 (see \eqref{eqn:indPol})
\[
F_c: \operatorname{ind}_{\mathfrak{p}}^{\mathfrak{g}}(V^\vee) 
\overset{\sim}{\longrightarrow} 
\operatorname{Pol}(\mathfrak{n}_+)\otimes V^\vee
\]
is given by 
$$
(u\otimes v^\vee)\mapsto \widehat{d\pi_{\lambda^*}(u)}(1\otimes v^\vee),\quad u\in U(\mathfrak g),\,v^\vee\in V^\vee.
$$
\end{cor}

\section{F-method}\label{sec:Fmethod}

In Section \ref{sec:2part1} we have established a one-to-one correspondence between differential symmetry breaking operators for vector bundles and certain Lie algebra homomorphisms (Theorem \ref{thm:surject}). 
Using this framework our aim is to find explicit formul\ae\,
 for such operators, 
 in particular,
 when such operators are a priori known 
 to be unique up to scalar. For this purpose we
 propose a new method, which we call the \emph{F-method}. 
 Its theoretical foundation is summarized
 in Theorem \ref{thm:313}. This method 
  becomes particularly simple
 when $\mathfrak h$ is a parabolic subalgebra with abelian nilradical. In this case
we develop the F-method
 in more details, and give its \emph{recipe}
in Section \ref{sec:44}.  
Some useful lemmas
 for actual computations
 for vector-valued 
 differential operators
 are collected in 
Section \ref{subsec:36}.  
 
\subsection{Construction of equivariant differential operators by
algebraic Fourier transform}
Let $E$ be a finite-dimensional vector spaces over $\C$ and $E^\vee$ its dual space.
Let $\operatorname{Diff}^{\mathrm{const}}(E)$ denote the ring of holomorphic differential operators on $E$  with constant coefficients.
 We define the symbol map
\begin{eqnarray*}
\operatorname{Symb}:
 \operatorname{Diff}^{\mathrm {const}}(E){\stackrel{\sim}{\longrightarrow}}\operatorname{Pol}(E^\vee),
\quad
D_z\mapsto Q(\zeta)
\end{eqnarray*}
by the following characterization
$$
D_ze^{\langle z,\zeta\rangle}=Q(\zeta)e^{\langle z,\zeta\rangle}.
$$
Then $\mathrm{Symb}$ is an algebra isomorphism.
The differential operator on $E$  with symbol $Q(\zeta)$ will be denoted by $\partial Q_z$.

By the definition of the algebraic Fourier transform (Definition \ref{def:hat}) one has
\begin{equation}\label{eqn:aFP}
\widehat{\partial P_z}=(-1)^\ell P(\zeta),\qquad
\widehat{Q(z)}=\partial Q_\zeta
\end{equation}
for any homogeneous polynomial $P$ on $E^\vee$ of degree $\ell$ and any polynomial $Q$ on $E$ seen as a multiplication
  operator.

We recall from Corollary \ref{cor:Fc} that $\operatorname{Pol}(\mathfrak n_+)\otimes V^\vee$ is a $(\mathfrak g,P)$-module if $V$ is a $P$-module. Note that the action of $\exp(\mathfrak n_+)$ ($\subset P$) 
on $\operatorname{Pol}(\mathfrak n_+)\otimes V^\vee$
is not geometric, namely, it is not given by the pull-back of polynomials via
the action on the base space $\mathfrak n_+$.

The key tool for the F-method that we explain in Section \ref{sec:44}
 is the following assertion. We note
 that the two approaches 
 (the canonical invariant pairing \eqref{eqn:Gbilinear})
 and the algebraic Fourier transform \eqref{eqn:indPol})
 give rise to the same differential operators, provided that $\mathfrak n_+$ is abelian:
\begin{thm}\label{thm:313}
Suppose that ${\mathfrak{p}}$ is
 a parabolic subalgebra ${\mathfrak {g}}$ and that $P=L\exp(\mathfrak n_+)$ is its Levi decomposition.
 Let $P'$ be a closed subgroup of $P$ such that
 $P'$ has a decomposition $P'=L'\exp(\mathfrak n_+')$ with
 $L'\subset L$ and $\mathfrak n_+'\subset \mathfrak n_+$.
 Let $G'$ be
 an arbitrary
 subgroup of $G$ containing $P'$. For a representation $(\lambda, V)$ of $P$ and a representation
 $(\nu, W)$ of $P'$, we form a 
 $G$-equivariant vector bundle $\mathcal V=G\times_PV$ over $X=G/P$ and a
 $G'$-equivariant vector bundle
 $\mathcal W= G'\times_{P'} W$ over $Y=G'/P'$, respectively. Let $\mu:=\lambda^*$ be as in \eqref{eqn:mucompl}.
 \begin{enumerate}
 \item There is a natural isomorphism
  \begin{eqnarray}\label{eqn:thm41}
 \operatorname{Diff}_{G'}(\mathcal V_X,\mathcal W_Y)&\simeq&
 \left(\operatorname{Pol}(\mathfrak n_+)\otimes\operatorname{Hom}_\C
 (V,W)\right)^{L', \widehat{d\pi_\mu}(\mathfrak n'_+)}\\&\simeq&\left(\mathrm{Hom}_{L'}
 (V\otimes \operatorname{Pol}(\mathfrak n_+),W)\right)^{\widehat{d\pi_\mu}(\mathfrak n'_+)}.\nonumber
 \end{eqnarray}
 Here the right-hand side of \eqref{eqn:thm41} consists of $\psi\in
 \operatorname{Pol}(\mathfrak n_+)\otimes \mathrm{Hom}_\C(V,W)$
 satisfying 
 \begin{eqnarray}
\nu(\ell)\circ\mathrm{Ad}_{\#}(\ell)\psi\circ\lambda(\ell^{-1})=\psi
&&\mathrm{for\,\, all}\, \ell\in L',\label{eqn:Fmethod1}\\
(\widehat{d\pi_\mu}(C)\otimes\operatorname{id}_W+\operatorname{id}\otimes\nu(C)) \psi=0 &&\mathrm{for\,\, all}\, C\in\mathfrak n_+'.\label{eqn:Fmethod2}
\end{eqnarray}

 \item
  Assume that the nilradical\/ $\mathfrak n_+$ is abelian. Then the following diagram commutes:
$$
\begin{matrix}
\operatorname{Hom}_\C (W^\vee,\indpg(V^\vee)) \,
{\stackrel{F_c\otimes \mathrm{id}}{\overset{\sim}{\To}}}
&\operatorname{Pol}(\mathfrak n_+)\otimes \operatorname{Hom}_\C(V,W)
& \stackrel{\operatorname{Symb}\otimes \mathrm{id}}{\stackrel{\sim}{\longleftarrow}}\operatorname{Diff}^{\mathrm {const}}(\mathfrak n_-)\otimes \operatorname{Hom}_\C(V,W)\\
\cup & \circlearrowright& \cup\\
\operatorname{Hom}_{P'}(W^\vee,\indpg(V^\vee))
&\underset{D_{X\to Y}}{\overset{\sim}{\longrightarrow}}&
\operatorname{Diff}_{G'}(\mathcal V_X,\mathcal W_Y).
\end{matrix}
$$
\end{enumerate}
\end{thm}
\begin{rem}\label{rem:ipower}
The convention on the Fourier transform $\mathcal F_c$ in Definition \ref{defn:Fourier}
makes the diagram in Theorem \ref{thm:313} commutative without additional powers of
$\sqrt{-1}$.
\end{rem}

Theorem \ref{thm:313} may be regarded as a 
construction of symmetry breaking operators by using the Fourier
transform of generalized Verma modules.

\begin{cor}
\label{cor:Fmethod}
Assume
 that ${\mathfrak {n}}_+$ is abelian and that $P'=L'\exp(\mathfrak n_+')$ with $L'\subset L$ and
 $\mathfrak n_+'\subset\mathfrak n_+$.  
Then the following diagram of three isomorphisms commutes.
$$
\xymatrix{
{}
&
\mathrm{Hom}_{L'}\left(V\otimes\operatorname{Pol} (\mathfrak n_+),W)\right)^{\widehat{d\pi_\mu}(\mathfrak n_+')} 
& {}
\\
\operatorname{Hom}_{P'}(W^\vee,\indpg(V^\vee)) 
\ar[ur]^{F_c\otimes\mathrm{id}}
\ar[rr]_{D_{X\to Y}}^\sim
&
{}
&
 \operatorname{Diff}_{G'}(\mathcal V_X, \mathcal W_Y)
\ar[ul]_{\mathrm{Symb}\otimes\mathrm{id}}
}
$$
\end{cor}
In the above corollary,
 $\mathrm{Hom}_{L'}\left(V\otimes\operatorname{Pol} (\mathfrak n_+),W)\right)^{\widehat{d\pi_\mu}(\mathfrak n_+')} $ consists of $L'$-equivariant,
 $\operatorname{Hom}_{\C}(V,W)$-valued polynomial solutions $\psi$ on $\mathfrak n_+$
 to a system of partial differential equations
 of second order,
see Sections \ref{subsec:3.4} and \ref{sec:44}.  
Corollary \ref{cor:Fmethod} implies that, once we find such a polynomial solution $\psi$, we obtain  a $P'$-submodule $W^\vee$ in $\indpg(V^\vee)$ (sometimes referred to as \emph{singular vectors}) by
$(F_c\otimes\mathrm{id})^{-1}(\psi)$, and a differential symmetry breaking operator
by
$(\mathrm{Symb}\otimes\mathrm{id})^{-1}(\psi)$.

We first give proofs for the first statement of
Theorem \ref{thm:313} here. The proof of the second statement is postponed until the next subsection.

\begin{proof}[Proof of Theorem \ref{thm:313} (1)]
Combining the duality theorem (Theorem \ref{thm:surject}) with the algebraic Fourier transform (Corollary \ref{cor:Fc})
we have an isomorphism
$$
\operatorname{Hom}_{P'}(W^\vee,
\operatorname{Pol}(\mathfrak n_+)\otimes V^\vee)
\stackrel{\sim}{\To}
\operatorname{Diff}_{G'}(\mathcal V_X,\mathcal W_Y)
$$
where the $P'$-action on $\operatorname{Pol}(\mathfrak n_+)\otimes V^\vee$ is defined via the algebraic Fourier transform $F_c$, namely,
the left-hand side consists of $\psi\in\mathrm{Hom}_\C(W^\vee, \operatorname{Pol}(\mathfrak n_+)\otimes V^\vee)\simeq
\operatorname{Pol}(\mathfrak n_+)\otimes \mathrm{Hom}_\C(V,W)$
satisfying
$$
\left(\widehat{d\pi_\mu}(C)\otimes\mathrm{id}_W+
\mathrm{id}\otimes\nu(C )\right)\psi=0\quad
\mathrm{for}\,\,\mathrm{all}\,\,C\in\mathfrak l'+\mathfrak n_+,
$$
provided $L'$ is connected. Owing to Lemma \ref{lem:410}, the condition for $C\in\mathfrak l'$ is equivalent to that $\psi\in
(\operatorname{Pol}(\mathfrak{n}_+) \otimes \operatorname{Hom}_{\mathbb{C}}(V,W))^{\mathfrak{l}'}$,
where $\mathfrak{l}'$ acts on $\operatorname{Pol}(\mathfrak{n}_+) \otimes \operatorname{Hom}_{\mathbb{C}}(V,W)$
by 
$
   {\operatorname{ad}}_{\#}\otimes\operatorname{id}
 + \operatorname{id} \otimes (\lambda^\vee \otimes \operatorname{id}
 + \operatorname{id} \otimes \nu)
$.

In a more general setting where we allow $L'$ to be disconnected, by
 the same argument as in the proof of Lemma \ref{lem:410}, we see that the $P$-action on 
$\operatorname{Pol}(\mathfrak n_+)\otimes V^\vee$
via the algebraic Fourier transform $F_c$ of generalized Verma modules (Corollary \ref{cor:Fc})
coincides with the tensor product representation $\mathrm{Ad}_{\#}\otimes\lambda^\vee$
when restricted to the Levi subgroup $L$.
Thus the isomorphism \eqref{eqn:thm41} is proved.
\end{proof}

\subsection{Symbol map and reversing signatures}

The purpose of this section is to 
 carefully and clearly set up 
relations involving various signatures in connection with the algebraic Fourier transform in a coordinates-free fashion.

  \vskip10pt

Denote by $\gamma:S(E)\xrightarrow{\sim} \operatorname{Pol}
(E^\vee)$ the canonical isomorphism, and define another algebra
isomorphism
$$
\gamma_{sgn}: S(E)\xrightarrow{\sim} \operatorname{Pol}(E^\vee),
$$
by $\gamma\circ a$, where $a:S(E)\to S(E)$ denotes the automorphism 
of the symmetric algebra $S(E)$ 
induced
by the linear map $X\mapsto -X$ for $X\in E$.

Now we regard $E$ as an abelian Lie algebra over $\C$, and identify its enveloping algebra $U(E)$
with the symmetric algebra $S(E)$. Then, the right and left-infinitesimal actions
induce two isomorphisms:
$$
dR: S(E)\xrightarrow{\sim}\operatorname{Diff}^{\operatorname{const}}(E),
\qquad
dL: S(E)\xrightarrow{\sim}\operatorname{Diff}^{\operatorname{const}}(E).
$$

By the definition of the symbol map, we get,
$$
\operatorname{Symb}\circ dR=\gamma,\qquad\operatorname{Symb}\circ dL=\gamma_{sgn}.
$$
On the other hand, it follows from (\ref{eqn:aFP}) that
$$
\widehat{dL(u)}=\gamma(u),\qquad \widehat{dR(u)}=\gamma_{sgn}(u),
$$
for every $u\in S(E)\simeq U(E)$, where polynomials are regarded as multiplication
operators.
Hence we have proved
\begin{lem}\label{lem:RLFc}
Let $E$ be an abelian Lie algebra over $\C$.
For any $u\in U(E)$, 
$$
\operatorname{Symb}\circ dR(u)=
\widehat{dL(u)},\qquad \operatorname{Symb}\circ dL(u)=\widehat{dR(u)}.
$$
\end{lem}

\subsection{Proof of Theorem \ref{thm:313} (2)}
We are ready to complete the proof of Theorem \ref{thm:313} (and Corollary \ref{cor:Fmethod}).

\begin{proof}
Take an arbitrary 
$\varphi\in \operatorname{Hom}_{\C}
 (W^\vee,\indpg(V^\vee))$,
 which may be written as a finite sum 
 $$
 \varphi=\sum_j u_j\otimes\psi_j\in U(\mathfrak n_-)\otimes
 \operatorname{Hom}_\C(V,W)
 $$
by the Poincar\'{e}--Birkhoff--Witt theorem
$U(\mathfrak{g}) \simeq U(\mathfrak{n}_-) \otimes U(\mathfrak{p})$.
 Then it follows from (\ref{eqn:isoMh}) and (\ref{eqn:indPol}) that
 $$
 F_c\varphi=\sum_j\mathcal F_c\, (dL(u_j) \delta)\otimes \psi_j
 \in \operatorname{Pol}(\mathfrak n_+)\otimes\operatorname{Hom}_\C(V,W).
 $$
Since $\delta=\mathcal F_c^{-1}(\mathbf 1)$, we get
$$
F_c\varphi=\sum_j\widehat{dL(u_j)}\otimes \psi_j.
$$
On the other hand, by the construction (\ref{eqn:DXYdiag}), 
$$D_{X\to Y}(\varphi)=\sum_j dR(u_j)\otimes \psi_j.$$ 
Now we use the assumption that $\mathfrak n_+$ or equivalently  $\mathfrak n_-$ is abelian. Then, in the coordinates $\mathfrak n_-(\R)\hookrightarrow G/P$ the operator $dR(u_j)$ for
$u_j\in U(\mathfrak n_-)$ defines a constant coefficient differential operator on $\mathfrak n_-$. Thus $ D_{X\to Y}(\varphi)$ can be regarded as an element of
$\operatorname{Diff}^{\operatorname{const}}(\mathfrak n_-(\R))\otimes
\operatorname{Hom}_\C (V,W)$.

Applying the symbol map we have
$$
(\operatorname{Symb}\otimes\operatorname{id})\circ D_{X\to Y}(\varphi)
=\sum_j \operatorname{Symb}\circ dR(u_j)\otimes\psi_j
=\sum_j\widehat{dL(u_j)}\otimes \psi_j,
$$
where the last equation follows from Lemma \ref{lem:RLFc}. 
Thus we have proved that
$$
(F_c\otimes{\rm{id}})\varphi
=({\operatorname{Symb}}\otimes {\operatorname{id}})\circ D_{X\to Y}(\varphi),
$$
whence the second statement of Theorem \ref{thm:313}.
\end{proof}

\subsection{Recipe of the F-method
 for abelian nilradical ${\mathfrak{n}}_+$}\label{sec:44}
Our goal is to find an explicit form of a  differential 
symmetry breaking operator from $\mathcal V_X$ to $\mathcal W_Y$. Equivalently, what we call F-method 
provides a way to find
 an explicit element in the space $\displaystyle\operatorname{Hom}_{\mathfrak g'}(\indpgprime(W^\vee),
\indpg(V^\vee))  \simeq \operatorname{Hom}_{\mathfrak{p}'}(W^\vee,\operatorname{ind}_{\mathfrak{p}}^{\mathfrak{g}}(V^\vee))$.

A semisimple element $Z$ in $\mathfrak g$ is called \emph{hyperbolic} if all the
eigenvalues of ad$(Z)$ are real. A hyperbolic element $Z$ defines
a parabolic subalgebra
 ${\mathfrak {p}}(Z)={\mathfrak {l}}(Z)+{\mathfrak {n}}_+(Z)$, 
where
 ${\mathfrak {l}}(Z)$ and ${\mathfrak {n}}_+(Z)$
 are the sum of eigenspaces
 of $\operatorname{ad}(Z)$
 with zero and positive eigenvalues, respectively.

Let 
$\mathfrak g'$ be a reductive subalgebra in $\mathfrak g$, in the sense that $\mathfrak g'$ itself is reductive and the
adjoint representation of $\mathfrak g'$ on $\mathfrak g$ is completely reducible.

\begin{defn}
\label{def:3.2}
A parabolic subalgebra ${\mathfrak {p}}$
 is said to be ${\mathfrak {g}}^\prime$-{\it{compatible}} if there exists a
hyperbolic element $Z\in{\mathfrak {g}}^\prime$ such that ${\mathfrak {p}}={\mathfrak {p}}(Z)$. 
\end{defn}
If ${\mathfrak {p}}={\mathfrak {l}} +{\mathfrak {n}}_+$ is ${\mathfrak {g}}^\prime$-compatible, then ${\mathfrak {p}}^\prime :={\mathfrak {p}}\cap{\mathfrak {g}}^\prime$
becomes a parabolic subalgebra of ${\mathfrak {g}}^\prime$ with 
the following Levi decomposition:
$$
{\mathfrak {p}}^\prime ={\mathfrak {l}}^\prime +{\mathfrak {n}}^\prime_+:=({\mathfrak {l}}\cap{\mathfrak {g}}^\prime)+({\mathfrak {n}}_+\cap{\mathfrak {g}}^\prime),
$$
which satisfies the assumptions of Theorem \ref{thm:313} 2).

In this case the space $\operatorname{Diff}_{G'}(\mathcal V_X,\mathcal W_Y)$ of differential symmetry breaking operators is always finite-dimensional owing to Corollary \ref{cor:2.9} because:
$$
\dim_\C\operatorname{Hom}_{\mathfrak g'}
(\indpgprime(W^\vee),\indpg(V^\vee))<\infty
$$
for any finite-dimensional representations $V$ and $W$ of 
$\mathfrak p$ and $\mathfrak p'$, respectively \cite[Proposition 2.8]{K14}.

Our assumption here is that $\mathfrak{p}=\mathfrak{l}+\mathfrak{n}_+$
is a $\mathfrak{g}'$-compatible parabolic subalgebra of $\mathfrak{g}$
with abelian nilradical $\mathfrak{n}_+$.
Based on the following diagram (see Corollary \ref{cor:Fmethod}),
\begin{equation}\label{eqn:diagram}
\xymatrix{
&\left(\operatorname{Pol} (\mathfrak n_+)\otimes \operatorname{Hom}_\C(V,W)\right)^{L',\widehat{d\pi_\mu}(\mathfrak n_+')} & {}\\
\operatorname{Hom}_{P'}(W^\vee,\indpg(V^\vee)) \ar[ur]^{F_c\otimes\mathrm{id}}\ar[rr]_{D_{X\to Y}}^\sim
&{} & \operatorname{Diff}_{G'}(\mathcal V_X, \mathcal W_Y)\ar[ul]_{\mathrm{Symb}\otimes\mathrm{id}}
}
\end{equation}
 we develop a method as follows:
\begin{enumerate}
\item[Step 0.] Fix a finite-dimensional  representation $(\lambda,V)$ of the parabolic subgroup $P$.
It defines a $G$-equivariant vector bundle $\mathcal V_X=G\times_P V$ over $X=G/P$.
\vskip5pt
\item[Step 1.] Let $\mu:=\lambda^\vee\,\otimes\, \C_{2\rho}$
and compute (see (\ref{eqn:dpi1}) and (\ref{eqn:dpi2})),
\begin{eqnarray*}
d\pi_\mu:\mathfrak g&\to&\mathcal D(\mathfrak n_-)\otimes\operatorname{End}(V^\vee),\\
\widehat{d\pi_\mu}:\mathfrak g&\to&\mathcal D(\mathfrak n_+)\otimes\operatorname{End}(V^\vee).
\end{eqnarray*}
According to \eqref{eqn:35}, $\widehat{d\pi_\mu}$ only depends on the infinitesimal representation $\lambda$
of the parabolic subalgebra $\mathfrak p$.\vskip5pt
\item[Step 2.] Find a finite-dimensional  representation $(\nu,W)$ of the Lie group $P'$
such that 
$$
\operatorname{Hom}_{P'}(W^\vee,\indpg(V^\vee))\neq\{0\}.
$$
It defines a $G'$-equivariant vector bundle $\mathcal W_Y=G'\times_{P'}W$ over $Y=G'/P'$ such that $\operatorname{Diff}_{G'}(\mathcal V_X, \mathcal W_Y)$ is non-trivial.\vskip5pt

\item[Step 3.]  Consider $\psi\in \operatorname{Pol}(\mathfrak n_+)\otimes\operatorname{Hom}_{\C}(V,W)$ satisfying \eqref{eqn:Fmethod1} and \eqref{eqn:Fmethod2}.
 Note that the system of partial differential equations
(\ref{eqn:Fmethod2}) is of second order (see Proposition \ref{lem:315}).
\vskip5pt
\item[Step 4.] 
Take a slice $S$ for generic $L'_\C$-orbits on $\mathfrak n_+$.
Use invariant theory  for (\ref{eqn:Fmethod1}) and consider the system of differential equations
 on $S$ induced from (\ref{eqn:Fmethod2}). 
Find polynomials $\psi\in \mathrm{Pol}(\mathfrak n_+)\otimes\mathrm{Hom}(V,W)$ satisfying
(\ref{eqn:Fmethod1}) and (\ref{eqn:Fmethod2}) by solving those equations on $S$.
\vskip5pt

\item[Step 5.] Let $\psi$ be a polynomial solution to (\ref{eqn:Fmethod1}) and (\ref{eqn:Fmethod2}) obtained in Step 4.
 In the diagram \eqref{eqn:diagram}, $(\operatorname {Symb}\otimes \operatorname{id})^{-1}(\psi)$ gives 
the desired differential symmetry breaking operator in the coordinates
$\mathfrak{n}_-$ of $X$
 by Theorem \ref{thm:313}.
In the same diagram,
$(F_c\otimes\operatorname{id})^{-1}(\psi)$
gives an explicit element in
$\operatorname{Hom}_{\mathfrak{p}'}(W^\vee,\operatorname{ind}_{\mathfrak{p}}^{\mathfrak{g}}(V^\vee))$
($\simeq \operatorname{Hom}_{\mathfrak{g}'} (\operatorname{ind}_{\mathfrak{p}'}^{\mathfrak{g}'}(W^\vee),\operatorname{ind}_{\mathfrak{p}}^{\mathfrak{g}}(V^\vee))$),
which is sometimes referred to as a \textit{singular vector}.
 \end{enumerate}\vskip5pt
   
This method gives all non-trivial differential symmetry breaking operators for given data $(Y\hookrightarrow X, \mathcal V_X)$ 
by providing $G'$-equivariant vector bundles $\mathcal W_Y$ and explicit elements in $\mathrm{Diff}_{G'}(\mathcal V_X,
\mathcal W_Y)$. In fact, Step 2 based on Theorem \ref{thm:surject} gives a necessary and sufficient condition for a
$P'$-module $W$ to ensure that $\mathrm{Diff}_{G'}(\mathcal V_X,
\mathcal W_Y)$ is non-zero.
Steps 1, 3 and 5 based on Theorem \ref{thm:313} show that any differential symmetry breaking operator is of
the form 
$(F_c\otimes\operatorname{id})^{-1}(\psi)$ where $\psi$ is a polynomial solution to 
\eqref{eqn:Fmethod1} and \eqref{eqn:Fmethod2}.   
   
Actual applications
 of the F-method
 include the following cases:
 
1.  Holomorphic discrete series representations.  

2. Principal series representations
 of real reductive groups
 (Corollary \ref{cor:holoreal}).

The latter is related to questions in conformal geometry (more generally parabolic geometry), see \cite{Juhl, KOSS}. 
The former case includes the classical Rankin--Cohen bidifferential operators as a prototype, and it is the main object of the second part of this work \cite{PART2}. The connection between these
two is discussed in \cite{KoKuPe}.
 
Here we give some comments on the actual applications of the F-method when $X$ and $Y$ are Hermitian symmetric spaces.
In Theorem \ref{thm:C} we prove that all continuous symmetry breaking operators in this case are given
by  holomorphic differential operators that extend to the complex flag varieties, so that the F-method for
a parabolic subalgebra with abelian nilradical applies.

Furthermore, if $(G,G')$ is a reductive symmetric pair, 
we know a priori that $\mathrm{Diff}_{G'}(\mathcal V_X,
\mathcal W_Y)$ is one-dimensional for line bundles $\mathcal V_X$ with generic parameter
\cite[Theorem 2.7]{K14}.
Thus, it is natural to look for explicit formul\ae\, for such canonical operators.
 In Step 2 we can use explicit branching laws (see \cite[Fact \ref{fact:A}]{PART2}) to find all $W$ 
 such that $\operatorname{Hom}_{\mathfrak{p}'}(W^\vee,\operatorname{ind}_{\mathfrak{p}}^{\mathfrak{g}}
 (V^\vee))$ is non-zero.
Conversely, 
the differential equations in Step 3 are useful in certain cases to get a finer structure of branching laws, 
{\it{e.g.}}, to find
 the Jordan--H\"older series of the restriction for exceptional parameters $\lambda$ (see \cite{KOSS}).
 
 The Rankin--Cohen operators as well as Juhl's conformally covariant differential operators are
 recovered by the F-method as a special case where generic $L_\C'$-orbits on $\mathfrak n_+$ are of codimension one.
 The induced system of \eqref{eqn:Fmethod2} reduces to ordinary differential equations on the one
 dimensional complex manifold $S$.
 In the second part of this work \cite{PART2} we shall treat all the six geometries with a one-dimensional slice $S$.
 
 \subsection{F-method -- supplement for vector valued cases.}
\label{subsec:36}
 In order to deal with the general case where the target $\mathcal{W}_Y$ is no longer a line bundle but a vector bundle,
{\it{i.e.}}, where $W$ is
 an arbitrary finite-dimensional
 irreducible $\mathfrak l'$-module, we may find the condition
 \eqref{eqn:Fmethod1} somewhat complicated in practice, even though
 it is a system of differential equations of first order.
  In this section we give two useful lemmas to simplify Step 3 in the recipe by reducing \eqref{eqn:Fmethod1} to a simpler algebraic question on polynomial rings,  so that we can focus on the crucial part consisting of a system of differential equations of second order \eqref{eqn:Fmethod2}.
The results here will be used 
 in \cite[Sections \ref{sec:sp} and \ref{sec:un1}]{PART2}.

We fix a Borel subalgebra $\mathfrak b(\mathfrak l')$ of $\mathfrak l'$.
Let $\chi: \mathfrak b(\mathfrak l')\to\C$ be a character. For an $\mathfrak{l}'$-module $U$, 
we set 
$$
U_\chi:=
 \{u\in U\,:\, Zu=\chi(Z)u \quad\text{for any $Z\in\mathfrak b(\mathfrak l')$}\}.
$$

Suppose that $W$ is an irreducible representation of $\mathfrak l'$ with lowest weight $-\chi$. Then
 the contragredient representation $W^\vee$ has a highest weight $\chi$. We fix 
  a non-zero highest weight vector $w^\vee\in (W^\vee)_\chi$. 
 Then the contraction map
 $$
U\otimes W \to U, 
\quad
\psi\mapsto \langle\psi,w^\vee\rangle,
$$
induces a bijection between the following two subspaces:
\begin{equation}\label{eqn:111}
(U\otimes W)^{\mathfrak{l}'}
\stackrel{\sim}{\To}
U_\chi,
\end{equation}
if $U$ is completely reducible as an $\mathfrak{l}'$-module. By using the isomorphism \eqref{eqn:111}, we reformulate Step 3 of the recipe for the F-method as follows:

\begin{lem}\label{lem:412}
Suppose we are in the setting of Section \ref{sec:44}. Assume that $W$ is an irreducible representation of the parabolic subalgebra $\mathfrak p'$. Let $-\chi$ be the lowest weight of $W$ as an $\mathfrak l'$-module. Then we have a natural injective homomorphism
$$
\mathrm{Diff}_{G'}(\mathcal V_X,\mathcal W_Y)\hookrightarrow\left\{Q\in\left(\mathrm{Pol}(\mathfrak n_+)\otimes V^\vee\right)_\chi:
\widehat{d\pi_\mu}({C})Q=0\quad\mathrm{for}\,\mathrm{all}\,\,C\in\mathfrak n'_+\right\},
$$
which is bijective if $L'$ is connected.
\end{lem}

\begin{proof}
Applying \eqref{eqn:111} to the $\mathfrak l'$-module
$U:=\operatorname{Pol}(\mathfrak{n}_+) \otimes \operatorname{Hom}_{\mathbb{C}}(V,W)$,
we get an isomorphism:
\begin{equation}\label{eqn:PVWchi}
\left(\mathrm{Pol}(\mathfrak n_+)\otimes \mathrm{Hom}(V,W)\right)^{\mathfrak l'}
\stackrel{\sim}{\To}\left(\mathrm{Pol}(\mathfrak n_+)\otimes V^\vee\right)_\chi.
\end{equation}
Since $W$ is an irreducible $\mathfrak p'$-module, the Lie subalgebra $\mathfrak n'$ acts
trivially on $W$ and $\mathfrak l'$ acts irreducibly. In particular, the condition \eqref{eqn:Fmethod2} amounts to
$$
\left(\widehat{d\pi_\mu}({C})\otimes \mathrm{id}_W\right)\psi=0\quad\mathrm{for}\,\mathrm{all}\,\,C\in\mathfrak n'_+.
$$
Therefore, the isomorphism \eqref{eqn:PVWchi} induces a bijection
\begin{eqnarray*}
&&\left\{\psi\in \left(\mathrm{Pol}(\mathfrak n_+)\otimes \mathrm{Hom}(V,W)\right)^{\mathfrak l'}: \psi\,\,\mathrm{satisfies}\,\,\eqref{eqn:Fmethod2}\right\}\\
&\stackrel{\sim}{\to}&
\left\{Q\in \left(\mathrm{Pol}(\mathfrak n_+)\otimes V^\vee\right)_\chi: \widehat{d\pi_\mu}({C})Q=0\quad\mathrm{for}\,\mathrm{all}\,\,C\in\mathfrak n'_+\right\}.
\end{eqnarray*}
Now Lemma follows from Theorem \ref{thm:313}.
\end{proof}

Since any non-zero vector in $W^\vee$ is cyclic, 
the next lemma explains how to recover $
 D_{X\to Y}(\varphi)$ from $Q$ given in Lemma \ref{lem:412}.

We assume, for simplicity,
that the $\mathfrak{l}$-module $(\lambda,V)$ lifts to $L_\C$, 
the $\mathfrak{l}'$-module $(\nu,W)$ lifts to $L'_\C$, and use the same letters to denote their liftings.
\begin{lem}\label{lem:hwv}
For any $\varphi\in\operatorname{Hom}_{\mathfrak p'}(W^\vee,\indpg(V^\vee))$,
$\ell\in L'_\C$ and $w^\vee\in W^\vee$,
\begin{equation}\label{eqn:Scalea}
\left\langle D_{X\to Y}(\varphi), \nu^\vee(\ell)w^\vee\right\rangle=
(\mathrm{Ad}(\ell)\otimes \lambda^\vee(\ell))
\left\langle D_{X\to Y}(\varphi), w^\vee\right\rangle.
\end{equation}
\end{lem}

\begin{proof}
We write $\varphi=\sum_j u_j\otimes \psi_j\in U(\mathfrak n_-)\otimes \operatorname
{Hom}_\C (V,W)$. Since $\varphi$ is $\mathfrak p'$-invariant, we have the identity:
$$
\sum_j u_j\otimes \psi_j=\sum_j \mathrm{Ad}(\ell)u_j\otimes \nu(\ell)\circ\psi_j\circ\lambda(\ell^{-1})
\quad\text{for $\mathfrak{l}\in L'_{\mathbb{C}}$}.
$$
In turn, we have
\begin{eqnarray*}
\left\langle D_{X\to Y}(\varphi),\nu^\vee(\ell)w^\vee\right\rangle &=&
\sum_j dR(\mathrm{Ad}(\ell)u_j)\otimes \langle \psi_j, w^\vee\rangle\circ \lambda(\ell^{-1})\\
&=&\left((\mathrm{Ad}(\ell)\otimes\lambda^\vee(\ell)\right)
\left\langle D_{X\to Y}(\varphi),w^\vee\right\rangle.
\end{eqnarray*}
Thus, we have proved Lemma.
\end{proof}

We notice that the right-hand side of \eqref{eqn:Scalea} can be computed
by using the identity in $\mathrm{Diff}^{\mathrm{const}}(\mathfrak n_-)\otimes V^\vee$:
 $$
 \langle  D_{X\to Y}(\varphi),w^\vee\rangle=(\mathrm{Symb}^{-1}\otimes
 \mathrm{id}_{V^\vee})(Q),
 $$
once we know the polynomial $Q = \langle \psi,w^\vee\rangle$ with
$\psi = (F_c \otimes \operatorname{id}) (\varphi)$
(see Theorem \ref{thm:313}).
In \cite[Sections \ref{sec:sp} and \ref{sec:un1}]{PART2},
we find explicit formul\ae\
 for vector-bundle valued equivariant differential operators
by solving equations for the polynomials $Q$.

\section{Localness and extension theorem for symmetry breaking operators}
\label{subsec:4.2}

Let $G\supset G'$ be a pair of real reductive Lie groups. In general, continuous symmetry breaking operators between two principal series representations of $G$ and $G'$ are not always given by
differential operators. Actually, generic ones are supposed to be given by integral transforms and their meromorphic continuation,  as one can see  from a classification result \cite{KS}.
In this section, however, we formulate and prove a quite remarkable phenomenon 
({\it{localness theorem}})
that any continuous $G'$-intertwining operator between two representation spaces
consisting of holomorphic sections over Hermitian symmetric spaces is given by differential 
operators,
see Theorem \ref{thm:C}. In particular, the covariant holomorphic differential operators which we shall obtain explicitly in the second part \cite{PART2} of this work exhaust all continuous symmetry breaking operators.

\subsection{Formulation of the localness theorem}\label{sec:511}

Let $G$ be a connected reductive Lie group, $\theta$ a Cartan involution, and $G/K$ the associated Riemannian symmetric space.
We write ${\mathfrak {c}}({\mathfrak {k}})$
for the center
 of the complexified Lie algebra
 $\mathfrak k:=\operatorname{Lie}(K)\otimes_{\mathbb{R}} {\mathbb{C}}
 \equiv\mathfrak k(\R)\otimes_\R\C$.  
In order to formulate a localness theorem, we suppose that $G/K$ is a Hermitian symmetric space. This means that
  there exists a characteristic element $Z \in 
{\mathfrak {c}}({\mathfrak {k}})$
 such that the eigenvalues
 of ${\operatorname{ad}} (Z) \in {\operatorname{End}}({\mathfrak {g}})$
 is 0 or $\pm 1$ and that we have an eigenspace decomposition
\[
{\mathfrak {g}}={\mathfrak {k}}+{\mathfrak {n}}_+
 +{\mathfrak {n}}_-
\]
 of ${\operatorname{ad}} (Z)$ 
 with eigenvalues
 $0$, $1$, and $-1$, 
respectively.  
We note that ${\mathfrak {c}}({\mathfrak {k}})$ is one-dimensional if $G$ is simple. With the notation of the previous
sections, the complex Lie algebra $\mathfrak k$ plays the role of the Levi subalgebra
${\mathfrak {l}}$.

 Let $G_\C$ be a complex reductive Lie group with Lie algebra $\mathfrak g$, and $P_\C$ the maximal parabolic subgroup with Lie algebra
 ${\mathfrak {p}}:= {\mathfrak {k}}+{\mathfrak {n}}_+$, 
 with abelian nilradical ${\mathfrak {n}}_+$. The complex structure
 of the homogeneous $G/K$ is induced from
the open embedding
$$
G/K\subset G_\C/K_\C\exp\mathfrak n_+ = G_{\mathbb{C}}/P_{\mathbb{C}}.
$$

Let $G'$ be a connected reductive subgroup of $G$.
Without loss of generality we may and do assume 
that $G'$ is $\theta$-stable. We set
 $K':=K\cap G'$.
 Our crucial assumption throughout this section is
 \begin{equation}\label{eqn:ckk}
 Z\in\mathfrak k'.
 \end{equation}
 \begin{lem}\label{lem:subherm}
 If \eqref{eqn:ckk} holds, then the parabolic subalgebra $\mathfrak p$ is $\mathfrak g'$-compatible (see Definition \ref{def:3.2}), and the homogeneous space
$G'/K'$ is a Hermitian sub-symmetric space of $G/K$ such that the embedding
$G'/K' \hookrightarrow G/K$ 
 is holomorphic.
 \end{lem}
\begin{proof}
Let $G'_\C$ be the connected complex subgroup of $G_\C$ with Lie
algebra $\mathfrak g':=\mathrm{Lie}(G')\otimes_\R\C$. Then $\mathfrak p':=\mathfrak k'+\mathfrak n'_+\equiv(\mathfrak k\cap\mathfrak g')+
(\mathfrak n_+\cap\mathfrak g')$ is the sum of the eigenspaces of ad$(Z)$ in $\mathfrak g'$ with $0$ and $+1$ eigenvalues, respectively, and therefore is a parabolic subalgebra of $\mathfrak g'$.
  We set $P'_\C:=P_\C\cap G'$. Then, the Riemannian symmetric space $G'/K'$
  becomes a Hermitian symmetric space, for which the complex structure is induced from the open embedding
  in the complex flag variety $Y_\C:=G_\C'/P_\C' $:
  $$
  \begin{matrix}
 & G'/K' &\hookrightarrow &G/K\\
&\llap{${}_  {\mathrm{open}}$}\,\bigcap\ & &\bigcap\,\rlap{${}_{\mathrm{open}}$}\\
  Y_{\C}=&G_\C'/P_\C'&\hookrightarrow&G_\C/P_\C&=X_{\C}.
  \end{matrix}
$$
Since $Y_\C$ is a complex submanifold of $X_\C=G_\C/P_\C$, the embedding $G'/K'\hookrightarrow G/K$ is holomorphic.
\end{proof}

Notice that in the setting of Lemma \ref{lem:subherm} the complexified Lie algebra of $K'$ 
is a Levi subalgebra
of the parabolic subalgebra $\mathfrak p'$. \vskip5pt

 \begin{example}\label{ex:ckk} 
 \begin{enumerate}
\item  Let $G'$ be a connected simple Lie group such that the associated Riemannian
 symmetric space $G'/K'$ is a Hermitian symmetric space. We take a characteristic element $Z'\in\mathfrak c (\mathfrak k')$.
 Let $G:=G'\times G'$, and we realize $G'$ as the diagonal subgroup $\Delta(G'):=\{(g,g):\,g\in G'\}$ of $G$. Then 
 $Z:=(Z',Z')\in\mathfrak c(\mathfrak k)$ satisfies \eqref{eqn:ckk}, yielding a holomorphic embedding $\Delta: G'/K'\hookrightarrow G/K=G'/K'\times G'/K'$.  
\item Let $G$ be a connected simple Lie group such that the associated symmetric space $G/K$ is a Hermitian symmetric space with $Z$ a characteristic element in $\mathfrak c(\mathfrak k)$. Suppose $\tau$ is an automorphism of $G$ of finite order such that $\tau(Z)=Z$. Let $G'$ be the identity component of the subgroup $G^\tau:=\{g\in G:\,\tau(g)=g\}$,
 and $K':=G'\cap K$. Then the assumption \eqref{eqn:ckk} is satisfied, and $G'/K'$ is a Hermitian sub-symmetric space
 of $G/K$. We shall focus on the case where
 $(G,G^\tau)$ is a symmetric pair, namely, $\tau$ is of order two in \cite{PART2} for detailed analysis.
 \end{enumerate}
 \end{example}

Consider a finite-dimensional  representation of $K$ on a complex vector space $V$. We extend it to a holomorphic representation of $P_\C$ 
 by letting the unipotent subgroup $\exp(\mathfrak n_+)$ act trivially,
 and form a holomorphic
 vector bundle $\mathcal V_{X_\C}=G_\C\times_{P_\C}V$ over $X_{\C}=G_\C/P_\C$.
 The restriction to the open set $G/K$ defines a $G$-equivariant holomorphic vector bundle $\mathcal V:=G\times_K V$. We then have
 a natural representation of $G$ on the vector space $\mathcal O(G/K,\mathcal V)$ of global holomorphic sections endowed with
 the Fr\'echet topology of uniform convergence on compact sets.
 
 Likewise, given a finite-dimensional  representation $W$ of $K'$, we form
 the $G'$-equivariant holomorphic vector bundle $\mathcal W=G'\times_{K'}W$ and consider
 the representation of $G'$ on $\mathcal O(G'/K',\mathcal W)$.
By definition, it is clear that
\begin{equation}\label{eqn:dcont}
\operatorname{Diff}_{G'}^{\mathrm{hol}}(\mathcal V_{X},\mathcal W_{Y}) 
\subset
\operatorname{Hom}_{G'}\left(\mathcal{O}\left(G/K,\mathcal V\right),\mathcal O(G'/K',\mathcal W)\right).
\end{equation}
Theorem \ref{thm:C} below shows that the two spaces do coincide.

\begin{thm}\label{thm:C}
Let  $G'$ be a reductive
subgroup of $G$ satisfying \eqref{eqn:ckk}.
Let $V$ and $W$ be any finite-dimensional  representations of $K$ and
$K'$, respectively. Then, 

\begin{enumerate} 
\item \emph{(localness theorem)} any continuous $G'$-homomorphism from
$\mathcal O(G/K,\mathcal V)$ to $\mathcal O( G'/K',\mathcal W)$ is
given by a holomorphic differential operator, in the sense
of Definition \ref{def:21},
 with respect to a holomorphic map between
  the Hermitian symmetric
spaces $G'/K'\hookrightarrow G/K$, that is,
$$\operatorname{Diff}_{G'}^{\mathrm{hol}}(\mathcal V_{X},\mathcal W_{Y}) 
=
\operatorname{Hom}_{G'}\left(\mathcal{O}\left(G/K,\mathcal V\right),\mathcal O(G'/K',\mathcal W)\right);
$$
\item \emph{(extension theorem)} any such a differential operator 
 (or equivalently,
 any continuous $G'$-homomorphism)
 extends to a $G'_\C$-equivariant holomorphic differential operator
 with respect to a holomorphic map between the flag varieties $
 Y_{\C}=G'_\C/P_\C'\hookrightarrow  X_{\C}=G_\C/P_\C$, 
 namely,
 the injection 
\begin{equation}\label{eqn:inj}
\operatorname{Diff}_{G'_\C}^{\mathrm{hol}}(\mathcal V_{X_{\C}},\mathcal W_{Y_{\C}}) 
\hookrightarrow
\operatorname{Diff}_{G'}^{\mathrm{hol}}\left(\mathcal V_{X},\mathcal W_{Y}\right)
\end{equation}
is bijective.
\end{enumerate}
\end{thm}

\begin{rem}
More generally, 
we may ask whether an analogous statement to Theorem \ref{thm:C} (1)
holds or not if we replace
$\mathcal{O}\left(G/K, \mathcal V\right)$ and $\mathcal O(G'/K',\mathcal W)$
by some other topological vector spaces having the same underlying $(\mathfrak g,K)$-module
and $(\mathfrak g',K')$-module, respectively (\emph{e.g.} the Casselman--Wallach globalization, Hilbert space globalization, etc.).
This question was raised by D. Vogan in May 2014. 
It turns out
 that this generalization is also true, as we shall show in the proof of Theorem \ref{thm:C}, that the natural injection 
\begin{equation}\label{eqn:gKsurj}
\operatorname{Diff}_{G'_\C}^{\mathrm{hol}}(\mathcal V_{X_\C},\mathcal W_{Y_\C}) 
\hookrightarrow
\operatorname{Hom}_{(\mathfrak g',K')}\left(\mathcal{O}\left(G/K,\mathcal V\right)_{K\mhyphen\mathrm{finite}},
\mathcal O(G'/K',\mathcal W)_{K'\mhyphen\mathrm{finite}}\right)
\end{equation}
is surjective if the assumption \eqref{eqn:ckk} is satisfied.  
\end{rem}

\begin{rem}
An analogous statement
for real parabolic subgroups
is not true. For instance, for the pair $(G,G')=(O(n+1,1),O(n,1))$ there always exists
a non-zero continuous $G'$-equivariant map from the spherical principal series representations $C^\infty(G/P,\mathcal L_\lambda)$ of $G$ to the
one $C^\infty(G'/P',\mathcal L_\nu)$ of $G'$
 for any $(\lambda,\nu)\in\C^2$, 
however, 
 non-zero $G'$-equivariant {\it{differential}} operators exist if and only if $\nu-\lambda\in 2\N$
 \cite{KS}.  
\end{rem}

\begin{rem}
Suppose that $V$ is
a generic character of $K$ and $(G,G')$ is a symmetric pair. Then owing to Theorems \ref{prop:2.10} and
\ref{thm:C} (2),
  $\operatorname{Diff}_{G'}^{\mathrm{hol}}(\mathcal V_{X_\C},\mathcal W_{Y_\C})$ is at most 
  one-dimensional  for any irreducible $K'$-module $W$, and \cite[Fact \ref{fact:A}]{PART2} tells us precisely when it is non-zero.
\end{rem}
In \cite{PART2} we describe explicit formul\ae\, of such differential operators by using the F-method
(Theorem \ref{thm:313})
for the six complex geometries arising from symmetric pairs of split rank one.

\subsection{Proof of the localness theorem}
Theorem \ref{thm:C} is a reflection of the theory of \emph{discretely decomposable restrictions} (see \cite{kdeco94,kdeco98}).
The proof is based on
 a careful analysis of the 
following three objects: 
\[
\text{
$(\mathfrak g,K)$-modules, $(\mathfrak g,K')$-modules, 
 and $(\mathfrak g',K')$-modules.
}
\]
\vskip6pt

We say that a $K'$-module $Z$ is $K'$-{\it{admissible}}
 if the multiplicity
$$
[M:F]:= \dim\operatorname{Hom}_{K'}(F,M\big\vert_{K'})
$$
is finite for any $F\in\widehat{K'}$.
Then, $K'$-admissibility is preserved by taking the tensor product with finite-dimensional representations.

We write ${\mathcal{O}}(G/K,{\mathcal{V}})_{\text{$K$-finite}}$
for the space of $K$-finite vectors
 of ${\mathcal{O}}(G/K,{\mathcal{V}})$, 
 which becomes naturally a
 $({\mathfrak {g}}, K)$-module.  

\begin{lem}\label{lem:431}
The $(\mathfrak g, K)$-module $\mathcal O(G/K,\mathcal V)_{K\mhyphen\mathrm{finite}}$ is $K'$-admissible if $Z\in\mathfrak g'$.
\end{lem}
\begin{proof}
As a $K$-module, we have the following isomorphism
\begin{eqnarray*}
\mathcal O(G/K,\mathcal V)_{K-{\mathrm{finite}}}
&\simeq& S(\mathfrak n_+)\otimes V\\
&\simeq&\left( \bigoplus_{a\geq 0}S^a(\mathfrak n_+)\right)\otimes V,
\end{eqnarray*}
where $S^a(\mathfrak n_+)$ denotes the space of symmetric tensors of
homogeneous degree $a$. Since 
$\exp(\R\sqrt{-1}Z)$ acts on $S^a(\mathfrak n_+)$ as the scalar
$e^{\sqrt{-1}at}$ $(t\in\R)$, the whole $S(\mathfrak n_+)$ is 
admissible as a module of the one-dimensional subgroup $\exp(\R\sqrt{-1}Z)$,
and so is $\mathcal O(G/K,\mathcal V)_{K-{\mathrm{finite}}}$.
Hence it is
also admissible as a $K'$-module by \cite[Theorem 1.2]{kdeco94}.
Alternatively, the lemma follows as a special case of the general
result \cite[Theorem 2.7]{kdeco94} or \cite[II, Theorem 4.1]{kdeco98}.
\end{proof}

Given a $(\mathfrak g,K')$-module $M$, we consider the contragredient representation on the dual space $M^\vee:=\operatorname{Hom}_\C(M,\C)$. Collecting $K'$-finite vectors in $M^\vee$, we get a $(\mathfrak g,K')$-module $\left(M^\vee\right)_{K'\mhyphen\mathrm{finite}}$.

\begin{lem}\label{lem:432}
Let $M$ be a $K'$-admissible $(\mathfrak g, K')$-module. Then,
\begin{enumerate}
\item $M$ is discretely decomposable as a $(\mathfrak g', K')$-module. 
\item The $(\mathfrak g, K')$-module $\left (M^\vee\right)_{K'\mhyphen\mathrm{finite}}$ is $K'$-admissible and one has the following $K'$-isomorphism 
$$
\left (M^\vee\right)_{K'\mhyphen\mathrm{finite}}\simeq\bigoplus_{F\in\widehat{K'}} [M:F]\, F^\vee.
$$
\end{enumerate}
\end{lem}
For the proof we refer to \cite[Part III, Proposition 1.6]{kdeco98}.

\begin{lem}\label{lem:433}
Let $M$ be a $K'$-admissible $(\mathfrak g, K)$-module. Then, 
$$
\left (M^\vee\right)_{K\mhyphen\mathrm{finite}}=\left (M^\vee\right)_{K'\mhyphen\mathrm{finite}}.
$$
\end{lem}
\begin{proof}
There is an obvious inclusion $\left (M^\vee\right)_ {K\mhyphen\mathrm{finite}}
\subset \left (M^\vee\right)_ {K'\mhyphen\mathrm{finite}}$. We shall prove that
the multiplicities in $\left (M^\vee\right)_ {K\mhyphen\mathrm{finite}}$ and
$\left (M^\vee\right)_ {K'\mhyphen\mathrm{finite}}$ are both finite and are the same.
Indeed, $M$ being $K'$-admissible, one has 
$$
[M:F]=\bigoplus_{E\in\widehat K}[M:E][E:F]<\infty.
$$
Conversely, 
$
\left (M^\vee\right)_{K'\mhyphen\mathrm{finite}}\simeq\bigoplus_{F\in\widehat{K'}}[M:F]\,F^\vee$ and thus,
$$
\left (M^\vee\right)_{K\mhyphen\mathrm{finite}}\simeq\bigoplus_{E\in\widehat{K}}[M:E]\,E^\vee\simeq
\bigoplus_{F\in\widehat{K'}}\left(\bigoplus_{E\in\widehat{K}}[M:E][E:F]\right)F^\vee,
$$
which concludes the proof.
\end{proof}

The next lemma is known to experts, but for the sake of completeness, we give a proof.

\begin{lem}\label{lem:434} 
There is a natural $(\mathfrak g,K)$-isomorphism:
$$
\left(\left(\mathcal O(G/K,\mathcal V)_{K\mhyphen\mathrm{finite}}\right)^\vee\right)_{K\mhyphen\mathrm{finite}}\simeq
\indpg(V^\vee).
$$
\end{lem}
\begin{proof}
As in Lemma \ref{lem:indPhi}, there is a natural non-degenerate $\mathfrak g$-invariant bilinear form
$$
\mathcal O(G/K,\mathcal V)_{K\mhyphen\mathrm{finite}}\times\indpg(V^\vee)\to\mathbb C.
$$
Hence, we have an injective
$(\mathfrak g,K)$-homomorphism
  $\indpg(V^\vee)\subset \left( \mathcal O(G/K,\mathcal V)_{K\mhyphen\mathrm{finite}}\right)^\vee$. Taking $K$-finite vectors we get the following commutative diagram of $K$-modules isomorphisms:
$$
\xymatrix{
  \indpg(V^\vee)\eq[d]&\subset  &\left(\left(\mathcal O(G/K,\mathcal V)_{K\mhyphen\mathrm{finite}}\right)^\vee\right)_{K\mhyphen\mathrm{finite}}\eq[d]\\
  S(\mathfrak n_-)\otimes V^\vee&\simeq&\left((\operatorname{Pol}(\mathfrak n_-)\otimes V)^\vee\right)_{K\mhyphen\mathrm{finite}}.
 }
$$ 
Hence the first row is also bijective. 
\end{proof}

Combining Lemmas\, \ref{lem:431}, \ref{lem:433} and \ref{lem:434}
we have shown  the following key result:

\begin{prop}\label{prop:Kprimeind}
There is a natural $(\mathfrak g,K)$-isomorphism:
$$
\left(\left(\mathcal O(G/K,\mathcal V)_{K\mhyphen\mathrm{finite}}\right)^\vee\right)_{K'\mhyphen\mathrm{finite}}\simeq
\indpg(V^\vee).
$$
\end{prop}

\begin{proof}[Proof of Theorem \ref{thm:C}] Let $T:\mathcal O(G/K,\mathcal V)\to \mathcal O(G'/K',\mathcal W)$ be a continuous
$G'$-intertwining operator. It induces a $(\mathfrak g',K')$-homomorphism
\begin{equation}\label{eqn:TK}
T_K: \mathcal O(G/K,\mathcal V)_{K\mhyphen\mathrm{finite}}\to \mathcal O(G'/K',\mathcal W)_{K'\mhyphen\mathrm{finite}}.
\end{equation}
We shall prove that any such $(\mathfrak g',K')$-homomorphism
$T_K$ comes from a $G'_\C$-equivariant differential operator on the flag variety. 

To see this, we take the dual map \eqref{eqn:TK}, and apply Lemma \ref{lem:434} and Proposition \ref{prop:Kprimeind}. Then there is a 
$(\mathfrak g',K')$-homomorphism $\psi:\indpgprime(W^\vee)\to\indpg(V^\vee)$ such that the following diagram commutes:
$$
\xymatrix{
\left(\mathcal O(G'/K',\mathcal W)_{K'\mhyphen\mathrm{finite}}^\vee\right)_{K'\mhyphen\mathrm{finite}}\ar[r]^{T_K^\vee} 
\eq[d]_{\mathrm{Lemma}\,\ref{lem:434} }
&
\left( \mathcal O(G/K,\mathcal V)_{K\mhyphen\mathrm{finite}}^\vee\right)_{K'\mhyphen\mathrm{finite}}
\eq[d]^{\mathrm{Proposition\, \ref{prop:Kprimeind} }}\\
\indpgprime(W^\vee)\ar[r]^\psi&\indpg(V^\vee)
}
$$
The correspondence $T_K\mapsto \psi$ is one-to-one, and thus we have
obtained the following natural injective map
$$
\operatorname{Hom}_{(\mathfrak g',K')}(\mathcal O(G/K,\mathcal V)_
{K\mhyphen\mathrm{finite}},\mathcal O(G'/K',\mathcal W)_{K'\mhyphen\mathrm{finite}})\hookrightarrow 
\operatorname{Hom}_{\mathfrak g'}(\indpgprime(W^\vee), \indpg(V^\vee)).
$$
According to Theorems \ref{prop:2.10} and \ref{thm:313} the latter space is isomorphic to $\operatorname{Diff}_{G'}^{\mathrm{hol}}(\mathcal V_X,\mathcal W_Y).$
This shows that \eqref{eqn:gKsurj} is surjective.

Since the injective map \eqref{eqn:gKsurj} factors the two injective maps
\eqref{eqn:dcont} and \eqref{eqn:inj}, both \eqref{eqn:dcont} and \eqref{eqn:inj} are bijective.
\end{proof}

\subsection{Automatic continuity theorem in the unitary case}
Any unitary highest weight module is realized as a subrepresentation of $\mathcal O(G/K,\mathcal V)$ for
some $G$-equivariant holomorphic vector bundle $\mathcal V$ over $G/K$. In this subsection, we
prove that any continuous homomorphism between Fr\'echet modules $\mathcal O(G/K,\mathcal V)$ 
and $\mathcal O(G'/K',\mathcal W)$ induces a continuous homomorphism between their unitary submodules. 
\begin{defn}
For a Fr\'echet $G$-module $\mathcal F$, we say a $G$-submodule $\mathcal H$ is a \emph{unitary submodule} if $\mathcal H$ is a Hilbert space such that the inclusion map $\mathcal H\hookrightarrow\mathcal F$
is continuous and that $G$ acts unitarily on $\mathcal H$.
\end{defn}

If $V$ is an irreducible $K$-module, then there exists at most one non-zero
unitary
submodule (up to a scaling of the inner product) of $\mathcal O(G/K,\mathcal V)$.
We denote by $\mathcal H_{V}^G$ the unitary
submodule of $\mathcal O(G/K,\mathcal V)$. 
The classification of irreducible $K$-modules $V$ for which $\mathcal H_{V}^G\neq\{0\}$
was accomplished in \cite{EHW}. We shall prove that
any $G'$-equivariant differential operator in Theorem \ref{thm:C} preserves the unitary submodules in the following sense:

\begin{thm}\label{thm:CC}
Let $G'$ be a reductive subgroup of $G$ satisfying \eqref{eqn:ckk}. Let 
$V$ and $W$ be any irreducible finite-dimensional representations of $K$ and $K'$, respectively. Suppose that $T:\mathcal O(G/K,\mathcal V)\longrightarrow \mathcal O(G'/K',\mathcal W)$ is a $G'$-equivariant differential operator such that $T\Big\vert_{\mathcal H_{V}^G}
\not\equiv0$.
Then $\mathcal H_{W}^{G'}\not=\{0\}$ and $T$ induces a continuous $G'$-equivariant linear map from the Hilbert space $\mathcal H_{V}^G$ onto the Hilbert space $\mathcal H_{W}^{G'}$.
\end{thm}

Applying Theorems \ref{thm:C} and \ref{thm:CC} to the setting of Example \ref{ex:ckk} (1), we have:

\begin{example}
Any symmetry breaking operator for the tensor product of two holomorphic discrete series representations is given by
 a holomorphic differential operator
 if those representations are realized in the space of holomorphic sections for $G$-equivariant holomorphic vector bundles over the Hermitian symmetric space $G/K$. The Rankin--Cohen bidifferential operators are such operators for $G=SL(2,\R)$.
\end{example}

\begin{rem}
As we shall see in the proof, the unitary representation $\mathcal H_{V}^G$ decomposes discretely when restricted
to $G'$ if the condition \eqref{eqn:ckk} is satisfied.
The unitary submodule $\mathcal H_{W}^{G'}$ occurs as a discrete summand of the restriction of the unitary representation
$\mathcal H_{V}^G$ of $G$ to the subgroup $G'$.
\end{rem}

Let $V$ be an irreducible representation of $K$ as before. Then, there exists a unique 
 $K$-submodule of $\mathcal O(G/K,\mathcal V)_{K\mhyphen\mathrm{finite}}\simeq S(\mathfrak n_+)\otimes V$ isomorphic to $V$, namely
 $ S^0(\mathfrak
n_+)\otimes V\simeq V.$

\begin{lem}
\label{lem:unitsubmod}
Let $M$ be a non-zero $(\mathfrak g, K)$-submodule
of $\mathcal O(G/K,\mathcal V)_{K\mhyphen\mathrm{finite}}$. Then,
\begin{enumerate}
\item[1)] The module $M$ contains $V$.
\item[2)] If $M$ is unitarizable, then its Hilbert completion can be realized in
$\mathcal O(G/K,\mathcal V)$ and $M=\left(\mathcal H_V^G\right)_{K\mhyphen\mathrm{finite}}$.
\end{enumerate}
\end{lem}
\begin{proof}
1) Since any non-zero quotient of the (generalized) Verma module $\indpg(V^\vee)$ contains $V^\vee$, the first statement follows from
Lemma \ref{lem:434}. Alternatively, since the infinitesimal action of $\mathfrak n_-$ on $\mathcal O(G/K,\mathcal V)_{K\mhyphen\mathrm{finite}}\simeq\operatorname{Pol}(\mathfrak n_-)\otimes V$ is given by directional derivatives, iterated operators of $\mathfrak n_-$ yield non-zero elements in $V$.

2) Denote by $(\pi,\overline M)$ the unitary representation of $G$ obtained as an (abstract) Hilbert completion of the $(\mathfrak g,K)$-module $M$. 
We regard $V$ as a $K$-submodule of $M$, and also of $\overline M$.
Then the map
$$
G\times \overline M\times V\longrightarrow\C,\qquad
(g,w,v)\mapsto (w,\pi(g)v)_{\overline M},
$$
induces an injective $G$-homomorphism $\iota:\overline M\longrightarrow \mathcal O(G/K,\mathcal V)$. Since $\mathcal H_V^G$
is the unique non-zero unitary submodule, $\iota$ is an isomorphism onto 
$\mathcal H^G_V$.
\end{proof}
\begin{proof}[Proof of Theorem \ref{thm:CC}]
By Lemma \ref{lem:431}, the module $\left(\mathcal H_V^G\right)_{K\mhyphen\mathrm{finite}}$ is $K'$-admissible. Therefore, the unitary representation $\mathcal H_V^G$ decomposes into a Hilbert direct sum of irreducible unitary representations $\{U_j\}$ of $G'$:
\begin{equation}\label{eqn:431}
\mathcal H_V^G\simeq \sideset{}{^\oplus}\sum_j m_j U_j,
\end{equation}
with $m_j<\infty$ for all $j$ (\cite[Theorem 1.1]{kdeco94}) and the underlying
$(\mathfrak g,K)$-module $\left(\mathcal H_V^G\right)_{K\mhyphen\mathrm{finite}}$ is isomorphic to an algebraic direct sum of irreducible and 
unitarizable $(\mathfrak g',K')$-modules
\begin{equation}\label{eqn:432}
\left(\mathcal H_V^G\right)_{K\mhyphen\mathrm{finite}}
=\left(\mathcal H_V^G\right)_{K'\mhyphen\mathrm{finite}}
\simeq \bigoplus_j m_j 
\left(U_j\right)_{K'\mhyphen\mathrm{finite}},
\end{equation}
with the same multiplicities \cite[Part III]{kdeco98}. (We remark that an
analogous statement fails for the restriction $\pi \vert_{G'}$ of an irreducible 
unitary representation $\pi$ of $G$ if the branching law contains continuous spectrum). 

As we saw in the proof of Theorem \ref{thm:C}, the $G'$-equivariant differential operators $T$ induces a $(\mathfrak g',K')$-homomorphism
$$
T_K: \left(\mathcal H_V^G\right)_{K\mhyphen\mathrm{finite}}\longrightarrow
\mathcal O(G'/K',\mathcal W)_{K'\mhyphen\mathrm{finite}}.
$$
By \eqref{eqn:432}, $M:=T_K\left(\left(\mathcal H_V^G\right)_{K\mhyphen\mathrm{finite}}\right)$ is an algebraic direct sum of some irreducible unitarizable
$(\mathfrak g',K')$-modules. Since $\mathcal O(G'/K',\mathcal W)_{K'\mhyphen\mathrm{finite}}$ contains at most one irreducible unitarizable 
$(\mathfrak g',K')$-module, $M$ is irreducible
as a $(\mathfrak g',K')$-module, and we can realize its Hilbert completion as $\mathcal H_W^{G'}$
by Lemma \ref{lem:unitsubmod}.

In view of \eqref{eqn:431} and \eqref{eqn:432}, there exists a continuous
$G'$-homomorphism between Hilbert spaces:
$$
\widetilde T: \mathcal H_V^G\longrightarrow \mathcal H_W^{G'}
$$
such that $\widetilde T\vert_{\left(\mathcal H_V^G\right)_{K\mhyphen\mathrm{finite}}}=
T\vert_{\left(\mathcal H_V^G\right)_{K\mhyphen\mathrm{finite}}}$.
Since the inclusion map $\mathcal H_V^G\hookrightarrow\mathcal O(G/K,\mathcal V)$ and the differential
operator $T:\mathcal O(G/K,\mathcal V)\longrightarrow\mathcal O(G'/K',\mathcal W)$ are both continuous, we get
$\widetilde T=T$ on $\mathcal H_V^G$. Hence Theorem is proved.
\end{proof}

\subsection{Orthogonal projectors}

If $V$ is one-dimensional and $(G,G')$ is a reductive symmetric pair
satisfying \eqref{eqn:ckk}, then all the multiplicities $m_j$ in \eqref{eqn:431} are equal to one (see \cite{K08}) and it becomes meaningful to describe
the projector from $\mathcal H_V^G$ to each $G'$-irreducible summand. We explain briefly the relationship
between the projector for the unitary representation and the symmetry breaking operator.

For this, suppose $T:\mathcal O(G/K,\mathcal V)\to O(G'/K',
\mathcal W)$ is a $G'$-equivariant differential operator such that
$T\vert_{\mathcal H_V^G}\not\equiv0$. 
By Theorem \ref{thm:CC}, $T$ induces a continuous map $T:\mathcal H_V^G\to\mathcal H_W^{G'}$ .
Let $T^*:\mathcal H_W^{G'}\to\mathcal H_V^G$ be its the adjoint operator.
Then the composition $T^*T:\mathcal H_V^G\to\mathcal H_V^G$ is
a $G'$-intertwining operator onto the $G'$-irreducible summand which is isomorphic to $\mathcal H_W^{G'}$. Since $T$ vanishes on the orthogonal complement to $T^*\left(\mathcal H_W^{G'}\right)$, it is (up to scaling) the orthogonal projector onto $\mathcal H_W^{G'}$.

Explicit description of such differential operators $T$ will be the main concern of the second part \cite{PART2} of this work.

\subsection*{Acknowledgements}
T. Kobayashi was partially supported by Institut des Hautes \'Etudes Scientifiques, France and
Grant-in-Aid for Scientific Research (B) (22340026) and (A) (25247006), Japan Society for the Promotion of
Science.
Both authors were partially supported by Max Planck Institute for Mathematics (Bonn) where a large part of this work was done.
\textheight=220mm

\vskip7pt

\footnotesize{ \noindent % Addresses:
T. Kobayashi.  Kavli IPMU and Graduate School of Mathematical Sciences, The University of
 Tokyo,
3-8-1 Komaba, Meguro, Tokyo, 153-8914 Japan; \texttt{{
 toshi@ms.u-tokyo.ac.jp}}.\vskip5pt

\noindent M. Pevzner.
Laboratoire de Math\'ematiques de Reims, Universit\'e
de Reims-Champagne-Ardenne, FR 3399 CNRS, F-51687, Reims, France; \texttt{{
 pevzner@univ-reims.fr.}}
}